\documentclass[12pt]{article}

\usepackage{amssymb,amsmath,amsthm,amsopn,amsfonts}
\usepackage[T1]{fontenc}
\usepackage{ae,aecompl}
\usepackage{bm}

\usepackage[utf8]{inputenc}

\usepackage{dutchcal} 

\usepackage{graphicx}

\usepackage[margin=20pt, justification=justified, font=small,labelfont=bf,labelsep=space]{caption}
\usepackage{float}
\usepackage{placeins}
\usepackage[percent]{overpic}  
\usepackage{morefloats}

\usepackage{fancybox,layout}
\usepackage{fancyhdr}

\usepackage{scrtime}

\usepackage{setspace}	

\usepackage{url}
\usepackage{placeins}

\usepackage[normalem]{ulem}  
\usepackage{cancel}  

\usepackage{calc}
\newsavebox\CBox
\newcommand\hcancel[2][0.5pt]{%
  \ifmmode\sbox\CBox{$#2$}\else\sbox\CBox{#2}\fi%
  \makebox[0pt][l]{\usebox\CBox}%
  \rule[0.5\ht\CBox-#1/2]{\wd\CBox}{#1}}

\usepackage[ruled,vlined]{algorithm2e} 

\SetKwBlock{Set}{set}{end}
\SetKwBlock{Comp}{compute}{end}
\SetKwBlock{Update}{update}{end}
\SetKwBlock{Solve}{solve}{end}
\SetKwInOut{Input}{input}

\usepackage{mdframed}
\newmdenv[
  leftmargin = 3em,
  innerleftmargin = 1em,
  innertopmargin = 0pt,
  innerbottommargin = 0pt,
  innerrightmargin = 0pt,
  rightmargin = 0pt,
  linewidth = 1pt,
  topline = false,
  rightline = false,
  bottomline = false
  ]{leftbar}

\numberwithin{equation}{section}

\newtheorem{proposition}{Proposition}[section]

\newtheorem{definition}{Definition}

\newtheorem{remark}{Remark}[section]
\newtheorem{problem}{Problem}[section]

\newcommand{\bmA}{\ensuremath{\bm{A}} } 
\newcommand{\bmB}{\ensuremath{\bm{B}} }

\newcommand{\bmH}{\ensuremath{\bm{H}} }

\newcommand{\bmN}{\ensuremath{\bm{N}} }

\newcommand{\bmR}{\ensuremath{\bm{R}} } 
 
\newcommand{\bmU}{\ensuremath{\bm{U}} }

\newcommand{\bmf}{\ensuremath{\bm{f}} }

\newcommand{\bmn}{\ensuremath{\bm{n}} }

\newcommand{\bmt}{\ensuremath{\bm{t}} }
\newcommand{\bmu}{\ensuremath{\bm{u}} }
\newcommand{\bmv}{\ensuremath{\bm{v}} } 
\newcommand{\bmx}{\ensuremath{\bm{x}} } 
 
\newcommand{\bmw}{\ensuremath{\bm{w}} }
\newcommand{\bmz}{\ensuremath{\bm{z}} }

\newcommand{\varepsilonb}{\ensuremath{\bm{\varepsilon}} }

\newcommand{\sigmab}{\ensuremath{\bm{\sigma}} }

\newcommand{\R}{\ensuremath{\mathbb{R}} }

\DeclareMathOperator{\tr}{\mathsf{tr}}
\DeclareMathOperator{\ddiv}{\mathsf{div}}

\DeclareMathOperator{\ds}{\mathrm{d}\mathit{s}}
\DeclareMathOperator{\dx}{\mathrm{d}\bm{\mathit{x}}}
\DeclareMathOperator{\dt}{\mathrm{d}\mathit{t}}

\newcommand{\abs}[1]{\ensuremath{|#1|}}

\newcounter{box}

\newcommand{\boxx}[4]{
  \begin{figure}[H]
    \refstepcounter{box}\label{#2} 
\begin{minipage}{#1}
    \begin{center}
      {\footnotesize \bf Box $\bm{\thebox}$. #3}
    \end{center}
\end{minipage}

\fbox{ 
\begin{minipage}{#1}      
    \begin{center}      
      {#4}  
    \end{center}
\end{minipage}
}    
  \end{figure}
}

\usepackage{indentfirst}

\title{A variational asymmetric phase-field model of quasi-brittle fracture:\\ 
Energetic solutions and their computation.}

\author{\normalsize
	Mariela Luege\thanks{CONICET, Instituto de Estructuras, FACET, 
		Universidad Nacional de Tucum\'an, Argentina}\, and
	Antonio Orlando\thanks{CONICET, Departamento de Bioingenier\`ia,
		Universidad Nacional de Tucum\'an, Argentina}
}

\date{ }

\begin{document}

\maketitle


\onehalfspacing


\pagestyle{fancy}
\fancyhead{}
\cfoot{\thepage}


\begin{abstract} 
	We derive the variational formulation of a gradient damage model by applying the energetic formulation 
	of rate-independent processes and obtain a regularized formulation of fracture. 
	The model exhibits different behaviour at traction and compression and has a state-dependent dissipation 
	potential which induces a path-independent work. 
	We will show how such formulation provides the natural framework for setting up a consistent numerical scheme
	with the underlying variational structure and for the derivation of additional necessary conditions
	of global optimality in the form of a two-sided energetic inequality. These conditions will form 
	our criteria for making a better choice of the starting guess in the application of 
	the alternating minimization scheme to describe crack propagation as quasistatic evolution of global minimizers of the underlying 
	incremental functional. We will apply the procedure 
	for two- and three-dimensional benchmark problems and we will compare the results with the solution of the weak form of the
	Euler-Lagrange equations. We will observe that by including the two-sided energetic inequality in our solution method, 
	we describe, for some of the benchmark problems, an equilibrium path when damage starts to manifest, 
	which is different from the one obtained by solving simply the stationariety conditions of the underlying functional.
\end{abstract}

\medskip
\footnotesize{
	{\bf Keywords}:\textit{Phase field variable. Generalized standard material. Energetic formulation. 
		Two sided energetic inequality. Alternating minimization. Backtracking algorithm.}

\medskip
\noindent ${}^{\ast}$ Corresponding author.\\
	{\bf Email}: \{mluege, aorlando\}@herrera.unt.edu.ar}

\normalsize
\section{Introduction}\label{Sec.INTRO}
In computational fracture mechanics, the variational phase-field models of fracture have received a considerable 
and increased attention as approximation models of fracture since the seminal works \cite{FM98,BFM08} where the
classical concept of Griffith's critical energy release rate \cite{Gri21} is replaced by a least energy principle, making it possible to 
capture otherwise characteristic features of the fracture process.
These models of fracture can appear as regularization formulations of 
free-discontinuity problems in the context of the variational approach to fracture \cite{FM98,CCF18,FR10,DT02,Gia05},  or they can result from  
the modelling of gradient damage as application of material constitutive theories in terms of two potentials, 
the free energy density and the dissipation potential \cite{Fre02,FN96,LA99,LA03,AMM09,Mie11,MHW10,MWH10}.
Excellent reviews on the application of these two approaches to approximate quasi-brittle fracture according to the above sense 
can be found in \cite{AGDl15,BA18,PAMM11,URSS19,DeG20} whereas \cite{EPAK19,DbV16,Wu17,WNNSSB20} provide an extensive overview
of also other phase-field models, not only the variational ones, that have been lately proposed for providing a more 
accurate description of the fracture process. 
In the variational formulations the cracks are represented by a continuum variable, namely a phase-field variable, that can be identified 
with the damage variable $\beta$ which describes the damaged and the undamaged phases, whereas
their propagation is described by the quasistatic evolution of the critical points of an energetic functional 
which accounts for the stored elastic energy and the dissipation associated with the variation of $\beta$. 
The main advantages of these formulations, compared for instance to the discrete approaches to fracture \cite{MDB99},
which relies on explicit modelling of the displacement field discontinuity jump produced by the crack, is that 
phase-field formulations can handle the evolution of complex crack patterns, can account for crack initiation and propagation without initial defects 
and prescribed crack path, and can be implemented without any particular consideration of what the crack pattern will be. This is because one
deals with the search for critical points of functionals defined over Sobolev spaces which can be easily discretized by standard finite elements spaces \cite{Qua17}
and crack initiation and propagation appear as a result of a competition between the different energetic terms \cite{TLBMM18}.

The existence of a variational structure for the models which we consider in this paper is basically a consequence of the 
rate-independence and associativity of the evolution process, thus the need to work with standard damage models 
\cite{MMP16,Ngu00,BCCFB10}. 
In this case, the variational formulation can be derived quite naturally by a general theoretical 
framework of clear mechanical interpretation given by the energetic formulation 
proposed by Mielke and coworkers \cite{MR15,MRZ10}. The use of the energetic formulation indirectly defines 
also the type of critical point which must be considered for the description of the evolution process. 
Since the existence of energetic solutions is proved by considering the 
evolution along global minimizers of discrete functionals, which are those that we use in the numerical simulation,
the concept of global minimizer is therefore the appropriate critical point we will use in this paper. 
This modelling assumption, which represents a milestone of the variational approach of fracture advanced
in \cite{FM98}, 
has been analysed
and justified
theoretically,
for instance, in \cite{DT02,BFM00,FL03,Gia05,Mie11,Mar10}, 
though the evolution via global minimzation is also subject of debate because it is not always able 
to produce physical solutions \cite{BDaMG99,Ste09,Ale16}. For instance, in \cite{Ale16} it is shown that by applying 
the energetic formulation, the evolution along global minimizers could produce early 
and unphysical discontinuity jumps, thus 
the relevance of an evolution along other type of critical points
is also beeing investigated as in  \cite{DT02b,AN20,Ale16,Bra14} 
to the fraction of cost of adding further physical based conditions about 
the choice of the particular critical point or of a redefinition of the concept of evolution
as in \cite{AN20,Ale16}.

From the numerical standpoint, the computation of global 
minimizers of the discrete energetic functional, which is separately convex in the displacement field 
$\bmu$ and in the phase-field damage variable $\beta$, poses the problem to ensure the global 
optimality of the critical points that one computes.  
The application of brute-force global optimization algorithms, such as clustering like or stocastic 
methods does not represent, at the moment, a viable option. In practice, one considers the Euler-Lagrange 
equations of this minimizing principle, and then apply the finite element method 
to their weak formulations \cite{MHW10,MWH10,GDl16,GDl19,Wic17,WNNSSB20}
or apply an alternating minimization method (referred to also as staggered scheme) to the finite element 
discrete energetic functional \cite{BWBNR20,FM17,URSS19}.
However, this methodology contrasts with the underlying modelling assumption of 
global minimization. Since the functional is non-convex, the Euler-Lagrange 
equations represent only stationariety conditions, thus their satisfaction cannot guarantee the
global optimality of the computed solution. Same conclusion holds by applying the staggered procedure 
given that the sequence of iterates, eventually up to a subsequence, converges to a critical point
of the discrete energetic functional \cite{AN20,KN17}.
Notably exceptions to this approach are those methods 
where the search of a global minimizer is realized still by local optimization algorithms which 
are however augmented by conditions met by the global minimizers \cite{MRZ10,Bou07,CLR16,MBK15}. 
On the basis of these additional necessary conditions
of global optimality, one basically tries to make a better choice of the starting guess so that it 
falls within the attraction basin of a global minimizer. 
This procedure has been applied successfully to the simulation of isotropic damage in \cite{Bou07} 
and \cite{MRZ10} by the variational and energetic formulation, respectively; to the simulation of the energetic formulation of
a delamination and adhesive contact model in \cite{RKZ13,ZG10}
and to the simulation of hysteresis in magnetic shape memory composites by the energetic formulation 
in \cite{CLR16}.
A similar idea, though applied in a different context, 
has been used in \cite{CCO08}
to obtain fast and efficient numerical relaxation algorithms
for the simulation of microstructures
in single crystals with one active slip system.
Here the authors exploit the structure of the modelling problem to 
obtain a better starting guess for the minimization of the nonconvex functional 
that models the problem at hand.

In this paper, starting from the mechanical model of \cite{AMM09,MWH10}, we illustrate the complete 
procedure for the derivation of the corresponding energetic formulation, of the 
additional optimality conditions of the discrete energetic solutions in the form of 
a two-sided energy inequality, 
and the ensuing energy-balance-based backtracking strategy for the numerical 
simulation of the phase-field  damage model, characterized by a different behaviour between 
traction and compression and with a state-dependent  
dissipation potential which induces a path-independent dissipated work. 
The discrete energetic functional we obtain 
is the same as the regularized functional considered in \cite{CCF18}
which is usually known as \textsc{AT2}-model in the literature on the regularized formulations of fracture
\cite{AMM09,MMP16,DeG20} and has been shown in \cite{CCF18} to $\Gamma-$converge 
(in the appropriate topology) to the free-discontinuity functional of \cite{FM98} with the additional 
non-interpenetration constraint of the crack faces under compression. 
The term \textsc{AT} stands for Ambrosio-Torterelli who were the first to propose an elliptic regularization of 
the Mumford-Shah free-discontinuity  functional \cite{AT90}.
To ensure that the discrete energetic solutions meet the additional conditions of globality, and to develop 
in this manner a computational strategy consistent with the modelling paradigma of evolution along global minimizers,   
we apply the strategy of backtracking in the context of an 
alternating minimization of the separately convex discrete energetic functional.
By such algorithm adapted to rate-independent processes \cite{Ben11},   
we go back over the time steps, whenever the two-sided-energy inequality is violated at the current time, 
to restart the simulation with a different initial guess which is built on the basis of the computed states 
that violate the check test given by the energetic bounds. 
By comparing our simulations to those based on the solution of the weak form of the Euler-Lagrange equations, 
we observe that for some of our benchmark problems the energetic solutions describe an equilibrium path that deviates 
from the standard one when damages starts to manifest, though eventually the two paths coincide. 
The two-sided energy inequality does not only represent a quick test 
of whether the candidate solution can be completed to a valid solution, but it has also a relevant physical meaning. 
This condition in fact represents a discrete equivalent of the conservation of energy \cite{MRZ10,MR15}.
Compared to \cite{MRZ10}, the present work enlarges the field of applications to state dependent dissipation potentials 
and to  phase-field models 
that account for the non-interpenetration condition when they are considered as fracture approximation models. 
It also proposes a numerical procedure which is consistent with 
the underlying globality assumption of the model by exploiting properties of the global minimizers. 

After this brief introduction, in the next Section we derive the 
phase-field model of fracture introduced in \cite{AMM09,MWH10} by applying 
the constitutive material theory based on the extended virtual power developed in \cite{Fre02}. We will 
introduce therefore the additive decomposition of the free 
elastic energy $\psi_0$ into a `compressive' and `tensile', with only the 
tensile contribution degraded by damage development. We use such decomposition to enforce in the limit the
non--interpenetration condition in view of the $\Gamma-$convergence result of \cite{CCF18}, though the assumption of the decomposition
of the free energy for the formulation of regularized variational formulations of fracture that accounts for the non--interpenetration constraint
has been debated, for instance, in \cite{LR09,FR10}.
At this stage, we do not go into the specific of such decomposition which is not relevant for the subsequent 
theoretical developments,
though, when we consider the actual implementation of the model,
we will then refer to the decomposition of $\psi_0$ that results from
the spectral splitting of the strain $\varepsilonb$ proposed in \cite{MWH10}.
The objective of Section \ref{Sec.StrngFormPFM} is to relate the mechanical model as is given in the literature to the 
energetic formulation which is the subject of Section \ref{Sec.EnergForm}. In this Section, we give first the continuous formulation 
which describes the evolution in time of the rate-independent system, and then we present the discrete energetic functional 
as approximation of the continuous energetic formulation. We will also mention therein the relation between this approach and the 
ones in the literature \cite{FN96,LA99,LCK11,Ngu16,LOAP18}, and derive the important energetic bounds of the discrete energetic solution
by exploiting the property of global optimality. Section \ref{Sec.AMMVF} describes then the alternating minimization algorithm. 
The corresponding finite element discrete equations
and how the energetic based backtracking algorithm is used in the whole numerical strategy is explained in Section \ref{Sec.FDS}
whereas Section \ref{Sec.NumEx} gives applications of the full procedure to the numerical solution of $2d-$ and $3d-$benchmark problems. 
The results are compared to the numerical solutions obtained without the activation of the backtracking algorithm, that is, the solution
of the weak form of the Euler-Langrange equations of the discrete energetic functional. 
Section \ref{Sec.Conc} concludes the paper with some final remarks about the energetic formulation 
and the proposed procedure.


\section{Mechanical derivation of the phase-field model of fracture}\label{Sec.StrngFormPFM}
In this section we derive the phase-field model of fracture introduced in \cite{AMM09,MWH10} by applying 
the constitutive material theory developed in \cite{Fre02}.

\subsection{Notations, main assumptions and field equations}
Let $\Omega\subset\R^n$, $n=1,2,3$,  be a bounded open domain which we take as reference configuration
of an homogeneous body made of brittle damaging material. We denote by $\partial\Omega$ the boundary of the 
domain $\Omega$, 
and assume that $\partial\Omega$ is split into two parts: a Dirichlet boundary $\partial \Omega_D$  
and the remaining Neumann boundary $\partial \Omega_N := \partial \Omega \setminus \partial\Omega_D$
where displacements $\bmw$ and surface tractions $\bmt$ are prescribed, respectively.
The boundary $\partial\Omega$ is such that the outward normal $\bmn$ can be defined 
almost everywhere (a.e.) on $\partial\Omega$. 
We assume the displacement field $\bmu$ to be small and the system to undergo an isothermal 
quasi-static evolution over the time interval of interest $[0,\,T]$, $T>0$ and with
uniform temperature in $\Omega$. The state of the system is then characterized by the 
linearized strain $\varepsilonb(\bmu)=\nabla_s\bmu:=(\nabla\bmu+\nabla\bmu^T)/2$, where 
$\nabla$ denotes the gradient operator, 
and additional variables which are introduced to capture the effects of microfractures 
at the material point on its macroscopic properties.  As such additional variables we consider the 
damage variable field $\beta$ and its gradient $\nabla\beta$. The field variable $\beta$ can take values 
in $[0,\,1]$ with $\beta=0$ when the material is undamaged and $\beta=1$
for completed damaged material, i.e. when the material is not able to sustain any stress. 
Its gradient $\nabla \beta$ is introduced to account for the influence of the damage at a point on damage 
of its neighborhood. 
Following the method of virtual power, we assume as in Fr\'emond \cite{Fre02,Fre17} that damage is 
produced by microscopic motions which break 
bonds among particles and such motion is described on the macroscopic level by the rate quantities
$\dot\beta:=d\beta/\dt$ and $\nabla \dot\beta:=\nabla (d\beta/\dt)$. The underlying assumption of the
theory is that the power of these motions must be taken into account in the power of the internal forces.
For background on continuum mechanics and corresponding notation, 
we refer to \cite{Gur70,Fre02}.

To define the functional setting where to formulate our model, we introduce
the standard Sobolev spaces $W^{1,\infty}(\Omega;\,\mathbb{R}^n)$
and $H^1(\Omega;\,\mathbb{R}^n)$ of functions defined a.e. in $\Omega$
and with values in $\mathbb{R}^n$, $n=1,2,3$. We denote then by $\mathcal{V}\subset H^1(\Omega;\,\mathbb{R}^n)$ 
the space of all displacement fields that generate 
compatible strain fields, by $\mathcal{V}_{D}\subset\mathcal{V}$ 
the affine space of kinematically admissible fields,
i.e. for any $t\in[0,\,T]$, $\bmu(\cdot,t)\in\mathcal{V}$ such that $\bmu=\bmw$ on $\partial\Omega_D$ in the sense of trace, 
and by $\mathcal{V}_{D,0}\subset\mathcal{V}$ the linear space of the kinematically 
admissible virtual displacement fields, that is, 
the space of the displacement fields $\bmu_{D,0}$  meeting the homogeneous kinematic boundary conditions 
on $\partial\Omega_D$, i.e. $\bmu_{D,0}=\bm{0}$ on $\partial\Omega_D$. 
We then introduce the linear space of the damage fields
$\mathcal{B}=\{\gamma\in W^{1,\infty}(\Omega;\,\mathbb{R}):\,\nabla\gamma\cdot\bmn=0\text{ on }\partial\Omega\}$
and consider $\beta\in W^{1,1}([0,\,T];\, \mathcal{B})$ as function of $t$, 
\cite[page 120]{MR15}.

For any $t\in[0,\,T]$, let $\bmu_{D,\bmw}(\cdot,t)\in\mathcal{V}_D$ 
be a lifting of the 
Dirichlet boundary data $\bmw(\cdot,t)$ \cite{Qua17}, that is, 
$\bmu_{D,\bmw}(\cdot,t)$ is a given (fixed) extension 
of $\bmw(\cdot,t)$ onto $\overline\Omega$, the closure of $\Omega$ with 
$\overline\Omega=\Omega\cup\partial\Omega_D\cup\partial\Omega_N$. 
Such extension can be obtained, for instance, 
by taking an interpolation of $\bmw$ onto $\overline\Omega$ by finite element shape functions
and must be considered as a known function once $\bmw(\bmx,t)$ is given.
We use also the notation  $\varepsilonb_{D,0}=\nabla_s\bmu_{D,0}$ to refer to the linearized strain of virtual 
admissible displacement fields.

The internal virtual power is then defined by the linear form
\[
	\mathcal{P}_i(\dot\bmu_{D,0},\,\dot\beta)= \int_{\Omega} \left(\sigmab\colon\dot\varepsilonb_{D,0} + 
	V\dot\beta+\bmH\cdot\nabla\dot\beta\right)\,\dx
\]
which defines the field variables $\sigmab$, $V$ and $\bmH$ dual of $\dot\varepsilonb$, $\dot\beta$ and 
$\nabla\dot \beta$, respectively. In this paper only volume forces $\bmf$, 
surface tractions $\bmt$ on $\partial \Omega_N$ and prescribed displacements $\bmw$ on $\partial \Omega_D$
are accounted for producing damage, thus the external virtual power is represented by the linear form
\[
	\mathcal{P}_e(\dot\bmu_{D,0})= \int_{\Omega} \bmf\cdot\dot\bmu_{D,0} \,\dx + \int_{\partial \Omega_N} \bmt\cdot\dot\bmu_{D,0} \,dx\,.
\]
The principle of the virtual power then states
\[
	\mathcal{P}_e(\dot\bmu_{D,0})=\mathcal{P}_i(\dot\bmu_{D,0},\,\dot\beta)
\]
which must hold for any admissible $\dot\bmu_{D,0}\in\mathcal{V}_{D,0}$ and  $\dot\beta\in\mathcal{B}$. 
By applying the Gauss-Green theorem and then the 
fundamental lemma of the calculus of variations, we obtain two field equations, 
one is the balance equations of linear momentum

\begin{equation}\label{Eq.2.EquilEq}
\begin{array}{ll}
	\displaystyle \ddiv\sigmab+\bmf= \bm{0} & \text{ in }\Omega\\
	\displaystyle \sigmab \bmn=\bmt    & \text{ on }\partial\Omega_N
\end{array}
\end{equation}
with the corresponding Neumann boundary condition given by the Cauchy Theorem \cite{Gur70}, 
whereas the other is the microforce balance equations

\begin{equation}\label{Eq.2.MicroEq}
\begin{array}{ll}
	\displaystyle \ddiv\bmH - V= 0 & \text{ in }\Omega\\
	\displaystyle \bmH\cdot \bmn=0    & \text{ on }\partial\Omega
\end{array}
\end{equation}
with the corresponding boundary value.

\subsection{The differential constitutive model}

We consider the coupled elasto--damage model defined by the following potentials

\begin{subequations}\label{FreModel}
\begin{align}
	&\displaystyle\psi(\varepsilonb,\,\beta,\,\nabla\beta)=
\displaystyle	R(\beta)\psi_0^{+}(\varepsilonb)+\psi_0^{-}(\varepsilonb)+
\displaystyle	\frac{g_c\ell}{2}|\nabla\beta|^2+I_{[0,\,1]}(\beta)\,,\label{FreModelFreeErg}\\
	&\displaystyle\phi(\dot\beta;\,\beta)=\frac{g_c}{\ell}\beta\dot\beta
		+I_{\mathbb{R}^+}(\dot\beta)\,, \label{FreModelDisPot}
\end{align}
\end{subequations}
where $I_{\mathbb{A}}$ is the indicator function of the set $\mathbb{A}$ and is defined by 
$I_{\mathbb{A}}(x)=0$ if $x\in\mathbb{A}$ and  $I_{\mathbb{A}}(x)=+\infty$ if
$x\not\in\mathbb{A}$. We use the notation  $\mathbb{R}^+=\{x\in\mathbb{R}:x\geq 0\}$ and denote by
$R(\beta)=g(\beta)+k$ the degradation function where
$g(\beta)$ is decreasing with $\beta$, convex and Lipschitz and such that $g(0)=1$
and $g(\beta)=1$ whereas $k$ is a small positive parameter that precludes complete damage by ensuring 
an artificial residual stiffness of a totally broken phase when $\beta=1$. 
The symbol $g_c$ is the fracture toughness whereas 
$\ell>0$ has the dimension of a length and controls the width of the transition zone of $\beta$. Such parameter 
identifies with the regularization parameter in the
variational model of fracture. Finally $\psi_0^{+}$ and $\psi_0^{-}$
are the `tensile' and `compressive' parts of the elastic strain energy density $\psi_0$
which can be derived, for instance, from the so-called volumetric-deviatoric splitting of the strain tensor 
$\varepsilonb$ \cite{AMM09,FR10,LR09} or the spectral decomposition of $\varepsilonb$ \cite{MWH10,MHW10},
though also other options for constructing $\psi_0^{+}$ and $\psi_0^{-}$ have been suggested in \cite{Li16,NYWH20}.
By \eqref{FreModelFreeErg}, it is assumed that only the positive part of the energy is degraded 
by the occurrence of damage, whereas the negative part remains unaffected by it.

We next show that this model is the same as the one proposed 
by \cite[Eq.~(61)]{MWH10} and \cite[Eq.~(38)]{MHW10} and corresponds to the \textsc{AT2} regularized formulation 
of fracture of \cite{BFM00} apart from the presence of the splitting of the free elastic
energy term \cite{AMM09,MMP16,DeG20}.

\begin{proposition}\label{Prop.MieheFre}
The differential constitutive model defined by the potentials \eqref{FreModel} is given by the state laws
\begin{equation}\label{FreStatLaw}
	\sigmab=\frac{\partial\psi}{\partial\varepsilonb}
\end{equation}
and the following evolution laws
\begin{equation}\label{MieheRateEq}
	\left|\begin{array}{l}
	\displaystyle \dot\beta\geq 0\\[1.5ex]
	\displaystyle -\frac{\partial\psi}{\partial\beta}-\frac{g_c}{\ell}(\beta-\ell^2\Delta\beta)\leq 0\\[1.5ex]
	\displaystyle \dot\beta\Big(-\frac{\partial\psi}{\partial\beta}-\frac{g_c}{\ell}(\beta-\ell^2\Delta\beta)\Big)= 0\,.
		\end{array}\right.
\end{equation}
\end{proposition}

\begin{proof}
We start from the Clausius-Duhem inequality for isothermal processes
\begin{equation}\label{FreCDineq}
	\sigmab\colon\dot\varepsilonb+\bmH\cdot\nabla\dot\beta+V\dot\beta-\rho\dot\psi\geq 0\,,
\end{equation}
and make the following constitutive assumptions
\begin{equation}\label{ConstAsFre}
	\sigmab=\sigmab^{nd}+\sigmab^{d}\,,\quad V=V^{nd}+V^{d}\quad\text{and}\quad
	\bmH=\bmH^{nd}+\bmH^{d}
\end{equation}
which distinguish the components that are 
responsable of the dissipative and reversible mechanisms.
By replacing \eqref{ConstAsFre} into \eqref{FreCDineq} and by defining
\begin{equation}\label{FreStateLaws}
	\sigmab^{nd}=\frac{\partial\psi}{\partial\varepsilonb}\,,\quad
	V^{nd}=\frac{\partial\psi}{\partial\beta}\quad\text{and}\quad
	\bmH^{nd}=\frac{\partial\psi}{\partial(\nabla\beta)}
\end{equation}
the Clausius-Duhem inequality reduces to the following expression
\begin{equation}\label{RedCDFreIneq}
	\sigmab^d\colon\dot\varepsilonb+V^d\dot\beta+\bmH^d\cdot\nabla\dot\beta\geq 0\,.
\end{equation}
For our model we assume 
\begin{equation}\label{AsDispMech}
	\sigmab^d=\bm{0}\quad\text{and}\quad\bmH^{d}=\bm{0}\,,
\end{equation}
thus \eqref{RedCDFreIneq} becomes
\begin{equation}\label{RedCDFreIneqOurModel}
	V^d\dot\beta\geq 0\,.
\end{equation}
We can meet \eqref{RedCDFreIneqOurModel} by taking 
\begin{equation}\label{EvLawFreDis}
	V^d\in \partial_{\dot\beta}\phi(\dot\beta;\,\beta)
\end{equation}
given that the function \eqref{FreModelDisPot} is a dissipation potential. 
Now, by the constitutive assumptions \eqref{ConstAsFre} and \eqref{AsDispMech}, 
and given the expression \eqref{FreModelFreeErg} of $\psi$, we have that 
\[	
	\bmH=\bmH^{nd}=g_c\ell\nabla\beta\,, 
\]
which replaced in \eqref{Eq.2.MicroEq}, yields 
\[
	V=\ddiv\bmH=g_c\ell\Delta\beta\,.
\]
By accounting for the expression \eqref{FreStateLaws}
of $V^{nd}$, we have thus 
\begin{equation}\label{DisV}
	V^d=V-V^{nd}=g_c\ell\Delta\beta-\frac{\partial\psi}{\partial\beta}\,.
\end{equation}
By computing $\partial_{\dot\beta}\phi$ we obtain
\begin{equation}\label{SubFreDis}
	\partial_{\dot\beta}\phi(\dot\beta;\,\beta)=\frac{g_c}{\ell}\beta+\partial_{\dot\beta}I_{\mathbb{R}^+}(\dot\beta)\,,
\end{equation}
thus \eqref{EvLawFreDis} reads as
\begin{equation}\label{EvLawFre}
	g_c\ell\Delta\beta-\frac{\partial\psi}{\partial\beta}-\frac{g_c}{\ell}\beta
	\in\partial_{\dot\beta}I_{\mathbb{R}^+}(\dot\beta)\,,
\end{equation}
where we have accounted for \eqref{DisV} and \eqref{SubFreDis}. 
By the definition of subdifferential of the indicator function \cite[Appendix A.1.3]{Fre02}, 
we have that \eqref{EvLawFre} means 
\begin{equation}
	\begin{array}{ll}
\displaystyle -\frac{\partial\psi}{\partial\beta}-\frac{g_c}{\ell}(\beta-\ell^2\Delta\beta)=0&\text{ if }\dot\beta>0\,,	\\[1.5ex]
\displaystyle -\frac{\partial\psi}{\partial\beta}-\frac{g_c}{\ell}(\beta-\ell^2\Delta\beta)<0&\text{ if }\dot\beta=0\,,	\\[1.5ex]
\displaystyle \varnothing &\text{ if }\dot\beta<0\,,	
	\end{array}
\end{equation}
which can then be expressed in the form given by \eqref{MieheRateEq}.
\end{proof}

\begin{remark}
	The model defined by \eqref{FreStatLaw} and \eqref{MieheRateEq} is a generalized standard material in the meaning of \cite[page 37]{Ngu00} 
	and \cite[page 69]{BCCFB10}
	given that it can be defined by the two potentials, the free energy potential $\psi(\varepsilonb,\,\beta,\,\nabla\beta)$
	and the dissipation potential $\phi(\dot\beta,\,\beta)$ using \eqref{ConstAsFre}, \eqref{FreStateLaws}, \eqref{AsDispMech} and \eqref{EvLawFreDis}.
	Since for $\dot\beta\geq 0$, the dissipation potential $\phi$ is a gauge \cite{HR13}, the model is associative. 
	The free energy and the dissipation potential are the only ingredients we need to set up the energetic formulation, henceforth 
	to derive the incremental variational formulation. This will be shown in Section \ref{Sec.IncrMin}.
\end{remark}

If in place of the dissipation potential \eqref{FreModelDisPot} we take \cite{MR06,MRZ10,BMMS14,TLBMM18}
\begin{equation}\label{Eq:AT1DisPot}
	\phi(\dot\beta)=\frac{g_c}{\ell}\kappa\dot\beta+I_{\mathbb{R}^+}(\dot\beta)\,,
\end{equation}
with $\kappa>0$ a  phenomenological constant such that $\kappa g_c/\ell$ is an activation threshold 
that represents the 
specific energy dissipated by fully damaging the bulk material, then by
the arguments of the proof of Proposition \ref{Prop.MieheFre}  we obtain the following evolution laws
\begin{equation}\label{MielkeRateEq}
	\left|\begin{array}{l}
	\displaystyle \dot\beta\geq 0\\[1.5ex]
	\displaystyle -\frac{\partial\psi}{\partial\beta}-\frac{g_c}{\ell}(\kappa-\ell^2\Delta\beta)\leq 0\\[1.5ex]
	\displaystyle \dot\beta\Big(-\frac{\partial\psi}{\partial\beta}-\frac{g_c}{\ell}(\kappa-\ell^2\Delta\beta)\Big)= 0\,.
		\end{array}\right.
\end{equation}
We will see later in Remark \ref{Rem:IncMinProAT}$(ii)$ that the dissipation potential \eqref{Eq:AT1DisPot} defines 
the so-called  \textsc{AT1} regularized formulation of fracture \cite{TLBMM18}. Also, it is not difficult
to verify that \eqref{MielkeRateEq} leads to the existence of an elastic stage before onset of damage, 
that is, as long 
as $-\text{d}R(\beta)/\text{d}\beta\,\psi^{+}_0(\varepsilonb)\leq \kappa g_c/\ell$, $\beta=0$ is solution
of \eqref{MielkeRateEq} whereas with the \textsc{AT2} regularized formulation, as soon as 
$\psi_0^{+}(\varepsilonb)\not=0$, the damage variable $\beta$ starts to evolve.


\subsection{The incremental boundary value problem}

Let $\mathcal{P}=\{0=t_0<t_1,\,\ldots,\,t_{N}=T\},\,N\in\mathbb{N}$ 
be a discrete set of time instants that realize a partition of 
the time interval of interest $[0,\,T]$ and $\Delta t=\max_{n=0,1,\ldots, N-1}\{t_{n+1}-t_n\}$.
Denote by $\bmz$ a kinematically admissible displacement field, that is, a 
displacement field that meets the Dirichlet boundary conditions at the current time $t$.
Given $(\bmz_{n},\,\beta_n)$ an approximation of the fields $\bmz$ and $\beta$ at the time instant $t_n$, 
we consider the incremental boundary value problem associated with the time step $[t_n,\,t_{n+1}]$ 
obtained by an Euler implicit time discretization of the constitutive equations \eqref{FreStatLaw} 
and \eqref{MieheRateEq}, and of the momentum balance equations \eqref{Eq.2.EquilEq} and \eqref{Eq.2.MicroEq}.
This problem consists of finding $(\bmz_{n+1},\,\beta_{n+1})$ 
such that the following relations are met

\begin{subequations}\label{Eq.IncConstModel}
	\begin{align}
		&-\ddiv\sigmab_{n+1}-\bmf_{n+1}=\bm{0}\quad\text{in }\Omega\,,	\label{Eq.IncConstModel.01}\\[1.5ex]
		&\varepsilonb_{n+1}=\nabla_s\bmz_{n+1},\quad\sigmab_{n+1}=\frac{\partial\psi}{\partial\varepsilonb}(\varepsilonb_{n+1},\beta_{n+1})\quad\text{in }\Omega\,, \label{Eq.IncConstModel.02}	\\[1.5ex]
		&\bmz_{n+1}=\bmu_{D,n+1}\text{ on }\partial\Omega_D\,,\quad\sigmab_{n+1}\bmn=\bmt_{n+1}\text{ on }\partial\Omega_N\,,	\label{Eq.IncConstModel.03}\\[1.5ex]
		&0\leq \beta_{n+1}\leq 1\quad\text{in }\Omega\,,	\label{Eq.IncConstModel5}\\[1.5ex]
		&\beta_{n+1}\geq\beta_n\quad\text{in }\Omega\,,	\label{Eq.IncConstModel4}\\[1.5ex]
		&-\frac{\partial\psi}{\partial\beta}(\varepsilonb_{n+1},\beta_{n+1})-\frac{g_c}{\ell}(\beta_{n+1}-\ell^2\Delta\beta_{n+1})\leq 0\quad\text{in }\Omega\,,  \label{Eq.IncConstModel6}\\[1.5ex]
		&(\beta_{n+1}-\beta_n)\Big(-\frac{\partial\psi}{\partial\beta}(\varepsilonb_{n+1},\beta_{n+1})-\frac{g_c}{\ell}(\beta_{n+1}-\ell^2\Delta\beta_{n+1})\Big)= 0\quad\text{in }\Omega\,,  \label{Eq.IncConstModel7}\\[1.5ex]
		&\nabla\beta_{n+1}\cdot\bmn=0\text{ on }\partial\Omega\quad\text{in }\Omega\,,\label{Eq.IncConstModel8}
	\end{align}
\end{subequations}
where $\psi$ is given by \eqref{FreModelFreeErg} 
without the indicator function $I_{[0,\,1]}(\beta)$ given that the constraint enforced by this function 
has been accounted explicitly by \eqref{Eq.IncConstModel5}. The condition \eqref{Eq.IncConstModel4}
is referred to as irreversibility condition and prevents material healing. 

\begin{remark}
For any given $\varepsilonb_{n+1}$, the model defined by \eqref{Eq.IncConstModel4}, \eqref{Eq.IncConstModel6} and \eqref{Eq.IncConstModel7}
is similar to the one that describes the deformation of a membrane over a linear elastic obstacle represented by 
$\beta_n$ and loaded by $\partial\psi/\partial\beta$. For instance, using the classical assumption for 
$g(\beta)$ in $R(\beta)$ due to Kachanov \cite{Kac58}, with $g(\beta)=1-\beta$,
we have $\partial\psi/\partial\beta=-\psi_0^+(\varepsilonb)$. For more 
general expressions of $R(\beta)$, such as the ones in \cite{LCK11,LOAP18}, $\partial\psi/\partial\beta$ has always a term
that depends only on $\varepsilonb$ and another one that depends also on $\beta$. The latter would then modify the 
bilinear form associated with $\beta g_c/\ell-g_c\ell\Delta\beta$. 
\end{remark}

If we denote by $\mathbb{C}$ the convex set of admissible solutions for $\beta$

\begin{equation}\label{Eq.DefAdmSet}
	\mathbb{C}=\bigg\{\beta\in \mathcal{B}:\beta(x)\in[0,\,1]\text{ a.e. in }\Omega
	\text{ and }\beta\geq\beta_n
	\text{ a.e. in }\Omega\bigg\}\,,
\end{equation}
and take $\bmz$ and $\beta$ as primary variables, we can consider the following weak formulation of \eqref{Eq.IncConstModel}:

\begin{subequations}\label{Eq.WeakForm}
	\begin{align}
		&\text{Find }(\bmz_{n+1},\,\beta_{n+1})\in\mathcal{V}_D\times\mathbb{C}:\nonumber\\[1.5ex]
		&\phantom{xxxxx}\int_{\Omega}\frac{\partial\psi}{\partial\varepsilonb}(\varepsilonb_{n+1},\,\beta_{n+1})\colon\varepsilonb(\bmv)\,\dx=
			\int_{\Omega}\bmf\cdot\bmv\,\dx+\int_{\partial\Omega_N}\bmt\cdot\bmv\,\ds
			\quad\text{for all }\bmv\in\mathcal{V}_{D,0}\,,	\label{Eq.WeakForm.01}\\[1.5ex]
		&\phantom{xxxxx}\int_{\Omega}\frac{\partial\psi}{\partial\beta}(\varepsilonb_{n+1},\beta_{n+1})(\gamma-\beta_{n+1})\,\dx
		+\int_{\Omega}\frac{g_c}{\ell}\beta_{n+1}(\gamma-\beta_{n+1})\,\dx \nonumber\\[1.5ex]
		&\phantom{xxxxxxxxxxxxxxxxxxxxxxxxxxx}+\int_{\Omega}g_c\ell\nabla\beta_{n+1}\nabla(\gamma-\beta_{n+1})\,\dx
			\geq 0\quad\text{for all }\gamma\in\mathbb{C}\,, \label{Eq.WeakForm.02}
	\end{align}
\end{subequations}
where $\varepsilonb_{n+1}=\nabla_s\bmz_{n+1}$.

\begin{remark}
Since $\psi$ is not convex, problem \eqref{Eq.WeakForm} is not ensured to have a unique solution. 
We will discuss this below with reference to the minimization formulation associated with \eqref{Eq.WeakForm}. 
\end{remark}

This observation justifies therefore the following notion.
\begin{definition}\label{Def.LocSol}
	We refer to any solution of \eqref{Eq.WeakForm} as a local solution of the model \eqref{FreModel}.
\end{definition}

\begin{proposition}
	If $(\bmz_{n+1},\,\beta_{n+1})\in\mathcal{V}_D\times\mathbb{C}$ solves \eqref{Eq.WeakForm}, then the field equations
	of \eqref{Eq.IncConstModel} are met a.e. in $\Omega$ and the boundary conditions are 
	met a.e. on the corresponding part of $\partial\Omega$.
\end{proposition}

\begin{proof}
	Let $(\bmz_{n+1},\,\beta_{n+1})\in\mathcal{V}_D\times\mathbb{C}$ be a solution of \eqref{Eq.WeakForm} and denote by 
	$\mathcal{D}(\Omega)$ the space of the infinitely differentiable functions compactly supported in $\Omega$ i.e., for 
	$\varphi\in\mathcal{D}(\Omega)$, let $S=\{x\in\Omega:\varphi(x)\not =0\}$, then the closure of $S$
	is bounded and contained in $\Omega$.
	By standard arguments based on the Gauss Green theorem and the properties of the space $\mathcal{D}(\Omega)$,
	from \eqref{Eq.WeakForm.01} we derive that $(\bmz_{n+1},\,\beta_{n+1})$ meets \eqref{Eq.IncConstModel.01}, 
	\eqref{Eq.IncConstModel.02} and \eqref{Eq.IncConstModel.03}. Conditions \eqref{Eq.IncConstModel5},
	\eqref{Eq.IncConstModel4} and \eqref{Eq.IncConstModel8}
	are also met given that they are enforced by the definition of $\mathbb{C}$. Now for any $\varphi\in\mathcal{D}(\Omega)$
	such that $\varphi\geq 0$ and $0\leq \beta_{n+1}+\varphi\leq 1$, $\gamma=\varphi+\beta_{n+1}\in\mathbb{C}$. 
	Thus, from \eqref{Eq.WeakForm.02}
	we obtain
	\begin{equation}\label{Eq.PosDistr}
		\int_{\Omega}\left(\frac{\partial\psi}{\partial\beta}+\frac{g_c}{\ell}\beta-g_c\ell\Delta\beta\right)\varphi\,dx\geq 0\,.
	\end{equation}
	Since \eqref{Eq.PosDistr} holds for any $\varphi\geq 0$ meeting the above conditions, then there must hold
	\[
		\frac{\partial\psi}{\partial\beta}+\frac{g_c}{\ell}\beta-g_c\ell\Delta\beta\geq 0\quad\text{a.e. in }\Omega\,,
	\]
	which is \eqref{Eq.IncConstModel6}. To prove \eqref{Eq.IncConstModel7}, for simplicity, 
	we make the further assumption that $\beta_{n+1},\,\beta_n\in C^0(\Omega)$. In this case, then, if we let 
	$\Omega'=\{x\in\Omega:\beta_{n+1}(x)>\beta_n(x)\}$, $\Omega'$ is Lebesgue measurable and has positive measure 
	and, therefore, we can consider the space $\mathcal{D}(\Omega')$. By the introduction of the set $\Omega'$,
	condition \eqref{Eq.IncConstModel7} can also be stated as 
	\[
		-\frac{\partial\psi}{\partial\beta}-\frac{g_c}{\ell}\beta_{n+1}+g_c\ell\Delta\beta_{n+1}= 0\quad\text{in }\Omega'
	\]
	whose weak form is given by
	\begin{equation}\label{Eq.WeakForm.CoincSet}
		\int_{\Omega}\left(
			\frac{\partial\psi}{\partial\beta}+\frac{g_c}{\ell}\beta_{n+1}\varphi+g_c\ell\nabla\beta_{n+1}\nabla\varphi
		\right)\,dx=0
		\quad\forall\varphi\in \mathcal{D}(\Omega')\,, 
	\end{equation}
	where we have used the fact that for $\varphi\in \mathcal{D}(\Omega')$, $\varphi(x)=0$ for $x\in\Omega\setminus\Omega'$.
	Therefore, next we need to show that we can derive \eqref{Eq.WeakForm.CoincSet} starting from \eqref{Eq.WeakForm.02}.
	For any $\varphi\in\mathcal{D}(\Omega')$, we can choose $\epsilon>0$ such that 
	$\gamma=\beta_{n+1}+\epsilon\varphi\in\mathbb{C}$. For instance, take $\epsilon>0$ such that $\epsilon<m/M$ 
	where $m=\min_{S}(\beta_{n+1}(x)-\beta_{n}(x))$
	with $S$ the support of $\varphi$, and $M=\max|\varphi|$. In this case, then 
	\[
		-\epsilon\varphi(x)\leq \epsilon|\varphi(x)|\leq\epsilon M<m\leq (\beta_{n+1}(x)-\beta_{n}(x))\,,
	\]
	thus
	\[
		\gamma(x)=\beta_{n+1}(x)+\epsilon\varphi(x)> \beta_{n}(x)\,.
	\]
	With such test function in \eqref{Eq.WeakForm.02}, we obtain
	\begin{equation}\label{Eq.Ineq1}
		\int_{\Omega}\left(
		\frac{\partial\psi}{\partial\beta}+\frac{g_c}{\ell}\beta_{n+1}\varphi+g_c\ell\nabla\beta_{n+1}\nabla\varphi
		\right)\,dx\geq 0\,.
	\end{equation}
	Since \eqref{Eq.Ineq1} holds for any $\varphi\in\mathcal{D}(\Omega')$, then it must hold even if we take $-\varphi$, 
	which gives the opposite inequality
	\begin{equation}\label{Eq.Ineq2}
		\int_{\Omega}\left(
		\frac{\partial\psi}{\partial\beta}+\frac{g_c}{\ell}\beta_{n+1}\varphi+g_c\ell\nabla\beta_{n+1}\nabla\varphi
		\right)\,dx\leq 0\,.
	\end{equation}
	By comparing \eqref{Eq.Ineq1} and \eqref{Eq.Ineq2}, we finally conclude \eqref{Eq.WeakForm.CoincSet}.
\end{proof}

\setcounter{equation}{0}
\section{Energetic formulation}\label{Sec.EnergForm}

In this section we present the continuos and incremental energetic formulation associated with 
the differential model \eqref{Eq.IncConstModel}.


\subsection{Continuous formulation}

The energetic theory developed by \cite{MR15} applies to standard generalized models that are rate-independent.  
The state and evolution laws of such material models are defined  
in terms of  
only two potentials $\psi$ and $\phi$ \cite{BCCFB10,LC98,Ngu00} with $\phi$ non-negative, convex and positively homogeneous
with respect to the rate variables. According to this theory, the governing equations 
can be concordingly described in terms of the stored energy functional $\mathcal{E}$ and the 
dissipation distance $\mathcal{D}$.

The stored energy functional $\mathcal{E}:[0,\,T]\times\mathcal{V}_D\times\mathcal{B}\rightarrow \mathbb{R}\cup\{\infty\}$  
is defined by 

\begin{equation}\label{Eq.2.EF1}
	\mathcal{E}(t,\bmz,\beta)=\int_{\Omega}\psi(\varepsilonb,\,\beta)\,\dx-
\langle\ell(t),\,\bmz\rangle
\end{equation}
where $\varepsilonb=\nabla_s\bmz$ and the pairing
$\langle\cdot,\cdot\rangle$ is the linear form modelling the work of the external time-dependent 
loading  given by

\begin{equation}\label{Eq.2.EF2}
	\langle\ell(t),\,\bmz\rangle=\int_{\Omega}\bmf(\bmx,t)\cdot\bmz(\bmx,t)\,\dx+
		\int_{\partial\Omega_N}\bmt(s,t)\cdot\bmz(s,t)\,\ds\,.
\end{equation}
The dissipation distance $\mathcal{D}:\mathcal{B}\times\mathcal{B}\rightarrow \mathbb{R}^+\cup\{\infty\}$ is given by

\begin{equation}\label{Eq.DissDist}
	\mathcal{D}(\beta_0,\,\beta_1)=\inf_{\beta\in\mathcal{B}}\left\{
	\int_0^1\mathcal{R}(\beta(s),\,\dot\beta(s))\,\ds:\,\beta(0)=\beta_0,\,\beta(1)=\beta_1
	\right\}\,,
\end{equation}
where $\mathcal{R}:\mathcal{B}\times\mathcal{B}\rightarrow \mathbb{R}\cup\{\infty\}$ is referred to as
the dissipation functional and is related to the dissipation potential via

\begin{equation}\label{Eq.2.EF3}
	\mathcal{R}(\beta,\,\dot\beta)=
	\int_{\Omega}\phi(\beta,\,\dot\beta)\,\dx\,.
\end{equation}
\begin{remark}
	The functional $\mathcal{R}$ is non negative because of the definition \eqref{FreModelDisPot} of $\phi(\dot\beta,\,\beta)$.
\end{remark}
We refer to the triple $(\mathcal{V}_D\times\mathcal{B},\,\mathcal{E},\,\mathcal{D})$ 
as Energetic Rate-Independent System \cite{MR15} given that its specification defines
completely the evolution of the model in terms of two global energetic conditions: 
an energetic balance condition $(E)$ and a stability condition $(S)$.

\begin{definition}
	We say that for any $t\in[0,\,T]$, $\big(\bmz(\cdot,t),\beta(\cdot,t)\big)\in\mathcal{V}_D\times\mathcal{B}$ 
	is an energetic 
	solution of the system $(\mathcal{V}_D\times\mathcal{B},\,\mathcal{E},\,\mathcal{D})$ if for all $t\in[0,\,T]$, 
	the following two conditions are met

\begin{align}
	& \mathcal{E}(t,\bmz(\cdot,t),\beta(\cdot,t))+\mathcal{D}(\beta(\cdot,0),\,\beta(\cdot,t)) =
	  \mathcal{E}(0,\bmz(\cdot,0),\beta(\cdot,0))
		+\int_0^t \frac{\partial \mathcal{E}}{\partial \tau}(\tau,\,\bmz(\cdot,\,\tau),\,\beta(\cdot,\,\tau))\,d\tau\,,	
		\label{Eq.EnrgBal}\tag{E} \\
	& \forall (\tilde\bmz,\tilde\beta)\in\mathcal{V}_{D}\times\mathcal{B},\quad	
		\mathcal{E}(t,\bmz(\cdot,t),\beta(\cdot,t))\leq\mathcal{E}(t,\tilde\bmz,\tilde\beta)
		+\mathcal{D}(\tilde\beta,\,\beta(\cdot,t))\,.	
		\label{Eq.Stab}\tag{S} 
\end{align}
\end{definition}

For $\beta\in\mathcal{B}$ such that $\dot\beta\geq 0$ a.e. in $\Omega$ and for 
almost all $t\in[0,\,T]$, and $0\leq \beta\leq 1$ a.e. in $\Omega$, the energetic formulation defined by the constitutive potentials  
\eqref{FreModel} is therefore obtained by taking the following functionals

\begin{subequations}\label{EnrgFrmFreMie}
	\begin{align}
		&\mathcal{E}(t,\,\bmz,\,\beta)=\int_{\Omega}\left[R(\beta)
		\psi_0^{+}(\varepsilonb)+\psi_0^{-}(\varepsilonb)+
		\frac{g_c\ell}{2}|\nabla\beta|^2\right]\,\dx-\langle\ell(t),\,\bmz\rangle\,,  \label{EnrgFrmFreMieFree}\\[1.5ex]
		&\mathcal{R}(\beta,\,\dot\beta)=\int_{\Omega}\frac{g_c}{\ell}\beta\dot\beta\,\dx\,. \label{EnrgFrmFreMieR}
	\end{align}
\end{subequations}

Given the expression \eqref{EnrgFrmFreMieR} of $\mathcal{R}(\beta,\,\dot\beta)$, we have  

\[
	\int_0^1\frac{g_c}{\ell}\beta(s)\dot\beta(s)\,\text{d}s=\frac{g_c}{2\ell}\left(\beta^2_1-\beta^2_0\right)
\]
for any $\beta:s\in[0,1]\to\beta(s,\,\cdot)\in\mathcal{B}$ such that $\beta(0,\cdot)=\beta_0$
and $\beta(1,\cdot)=\beta_1$, thus the dissipation distance \eqref{Eq.DissDist}
becomes in our case a path-independent function that depends only on the initial and final 
state of the system, and for $\beta_1\geq\beta_0$ is given by

\begin{equation}\label{Eq:DissAT2}
	\mathcal{D}(\beta_0,\,\beta_1)=\int_{\Omega}\frac{g_c}{2\ell}\left(\beta_1^2-\beta_0^2\right)\,\dx\,.
\end{equation}

\begin{remark}
	For the relation of the energetic theory with the continuum theories developed by \cite{LA99,Fre02} for gradient damage models, 
	we refer to \cite{PMM11,Ngu16,LOAP18}.
\end{remark}
In the following, for any given $t\in[0,\,T]$, we will express any admissible displacement field $\bmz(\bmx,t)$ of $\mathcal{V}_D$ 
as the sum of a fixed element of $\mathcal{V}_{D}$, for instance the lifting $\bmu_D$ of 
$\bmw$, and elements $\bmu$ of $\mathcal{V}_{D,0}$, that is, we write

\[
	\bmz(\bmx,t)=\bmu_D(\bmx,t)+\bmu(\bmx,t)\,.
\]
As a result, when we describe the stored energy functional $\mathcal{E}$ we will also use the notation
$\mathcal{E}(t,\bmu,\beta)$ with $\bmu\in\mathcal{V}_{D,0}$ to mean that we are considering 
$\mathcal{E}(t,\bmu+\bmu_D,\beta)$ where $\bmu_D$ is a fixed lifting of the Dirichlet boundary condition.

\subsection{Incremental minimization problem}\label{Sec.IncrMin}
The time incremental minimization problems associated with the energetic rate independent system 
\eqref{Eq.2.EF1} and \eqref{Eq.DissDist} are given by
\begin{leftbar}
\begin{problem}\label{GlobOptMatModel}\mbox{}\\[1.5ex]
Let $\mathcal{P}=\{0=t_0,\,\ldots,\,t_{N}=T\},\,N\in\mathbb{N}$\\[1.5ex]
For $n=0,\ldots, N-1$\\[1.5ex]
Given $\begin{array}[t]{ll}
	\displaystyle \text{External loading: }&\begin{array}[t]{ll}
							\ell(t_{n+1})&\text{Neumann b.c.}\\[1.5ex]
							\bmu_{D,n+1}(\bmx)=\bmw(x,t_{n+1})\text{ on }\partial\Omega_D&\text{Dirichlet b.c.}
						\end{array}\\[3ex]
	\displaystyle \text{State of the system at $t_n$:}&\begin{array}[t]{l} 
								\beta_n\in\mathcal{B} 
							  \end{array} 
	\end{array}$\\[1.5ex]
Find $(\bmu_{n+1},\beta_{n+1})\in\mathcal{V}_{D,0}\times\mathcal{B}$ such that minimize

\begin{equation}\label{Eq.3.IM01}
	\mathcal{F}(t_{n+1},\bmu,\,\beta;\,\beta_n):=\mathcal{E}(t_{n+1},\bmu,\,\beta)+\mathcal{D}(\beta_n,\,\beta) 
\end{equation}
subject to

\begin{subequations}\label{Eq.3.IM01a}
	\begin{align}
		&0\leq \beta_{n+1} \leq 1\,, \label{Eq.3.IM01SB}\\[1.5ex]
		&\beta_{n+1}\geq \beta_n\,. \label{Eq.3.IM01IRR}
	\end{align}
\end{subequations}
\end{problem}
\end{leftbar}
with the functionals $\mathcal{E}$ and $\mathcal{D}$
given by \eqref{EnrgFrmFreMieFree} and \eqref{Eq:DissAT2}, respectively.

\begin{remark}\label{Rem:IncMinProAT}
\begin{itemize}
	\item[$(i)$]
	If the functional $\mathcal{F}(t_{n+1},\bmu,\,\beta)$ of Problem \ref{GlobOptMatModel} is augmented by the term
	$-\mathcal{E}(t_n,\,\bmu_n,\,\beta_n)$, the functionals $\mathcal{F}(t_{n+1},\bmu,\,\beta)$ and 
	$\mathcal{F}(t_{n+1},\bmu,\,\beta)-\mathcal{E}(t_n,\,\bmu_n,\,\beta_n)$ have clearly the same minimizers. 
	These minimizers have therefore the property to minimize the sum of the variation of the free elastic energy 
	and of the dissipation. We obtain in this manner the
	\textsc{AT2} regularized formulation of fracture considered by 
	\cite{AMM09,Mie11,MHW10,MWH10}.
	In those works, one starts from \eqref{Eq.IncConstModel} and looks for the existence of a functional such that its Euler-Lagrange 
	equations coincide with \eqref{Eq.IncConstModel}.
	\item[$(ii)$]
		By taking the dissipation potential $\phi$ as \eqref{Eq:AT1DisPot}, the dissipation distance $\mathcal{D}$ 
		\eqref{Eq.DissDist} also in this case is path-independent and is given by
		
		\[
			\mathcal{D}(\beta_0,\,\beta_1)=\int_{\Omega}\frac{\kappa g_c}{\ell}\left(\beta_1-\beta_0
			\right)\,\dx+I_{\mathbb{R}^{+}}(\beta_1-\beta_0)\,,
		\]
		thus, the functional $\mathcal{F}(t_{n+1},\bmu,\,\beta)$ that defines Problem \ref{GlobOptMatModel}
		with $\beta$ meeting the constraints \eqref{Eq.3.IM01SB} and \eqref{Eq.3.IM01IRR},
		is given by
		
		\begin{equation}\label{Eq:FuncAT1}
		\begin{split}
			\mathcal{F}(t_{n+1},\bmu,\,\beta)&=
				\int_{\Omega} [R(\beta)\psi^+_0(\varepsilonb(\bmu+\bmu_{D,t}))+
				\psi^-_0(\varepsilonb(\bmu+\bmu_{D,t}))]\,\dx	\\[1.5ex]
				&+\int_{\Omega}\frac{g_c\ell}{2}\nabla\beta\cdot\nabla\beta\,\dx
				-\langle\ell(t), \bmu+\bmu_{D,t}\rangle\\[1.5ex]
				&+\int_{\Omega}\frac{\kappa g_c}{\ell}\left(\beta-\beta_n
				\right)\,\dx\,.
		\end{split}
		\end{equation}
		The functional \eqref{Eq:FuncAT1} defines the \textsc{AT1} regularized formulation considered by \cite{MR06,MRZ10,BMMS14,TLBMM18}
		apart from the constant $\int_{\Omega}\frac{\kappa g_c}{\ell}\beta_n\,\dx$.
	\item[$(iii)$]
	The derivation of Problem \ref{GlobOptMatModel} from the energetic formulation relies basically on two theoretical considerations:
	One regards the solutions of Problem \ref{GlobOptMatModel} as approximation of the
	energetic solutions as $\Delta t\to 0$ for given $\ell>0$, and the other refers to the 
	energetic formulation as approximation of the variational formulation of fracture 
	as limit problem for $\ell\to 0$. The asymptotic behaviour of the functional \eqref{Eq.3.IM01}
	in the special case of the degradation function $g(\beta)=(1-\beta)^2$ has been analyzed in \cite{CCF18} where it 
	has been shown that as $\ell\to 0$ the family of functionals $\mathcal{F}_{\ell}(t_{n+1},\bmu,\,\beta;\,\beta_n)$ $\Gamma-$converges to the 
	functional given by the sum of the stored elastic energy in the bulk material and the Griffith surface energy. In this sense, 
	therefore, we can state that this result justifies Problem \ref{GlobOptMatModel} as a variational approximation of quasi-brittle fracture
	\cite{FM98}.
\end{itemize}
\end{remark}

\begin{remark}
\begin{itemize}
	\item[$(i)$] 
		The explicit dependence of $\mathcal{E}$ on $t$ is through the loading term $\ell(t)$ and the free 
		energy term that depends on $\bmu_{D,t}$.
	\item[$(ii)$] If we account for the irreversibility condition \eqref{Eq.3.IM01IRR} by redefining 
		$\mathcal{D}$ over $\mathcal{B}\times\mathcal{B}$ as
		\begin{equation}
			\mathcal{D}(\beta_1,\beta_2)=
			\int_{\Omega}\frac{g_c}{2\ell}(\beta_2^2-\beta_1^2)\,\dx+
			\int_{\Omega}I_{\mathbb{R}^+}(\beta_2-\beta_1)\,\dx\,,
		\end{equation}
		then it is not difficult to verify 
		that $\mathcal{D}$ is an extended quasidistance \cite{MR15}, that is, 
		it meets the following conditions

		\begin{equation*}
			\begin{array}{ll}
			\displaystyle \forall \beta_1,\,\beta_2,\,\beta_3\in\mathcal{B}:&\displaystyle\mathcal{D}(\beta_1,\,\beta_2)\geq 0\\[1.5ex]
			& \displaystyle\mathcal{D}(\beta_1,\,\beta_2)=0\Longleftrightarrow 
			\displaystyle \beta_1=\beta_2\quad\text{a.e. in }\Omega\,;			\\[1.5ex]
			&\displaystyle \mathcal{D}(\beta_1,\,\beta_2)\leq \mathcal{D}(\beta_1,\,\beta_3)+\mathcal{D}(\beta_3,\,\beta_2)\,.
			\end{array}
		\end{equation*}
	\end{itemize}
\end{remark}

The existence of minimizers of Problem \ref{GlobOptMatModel} can be established by standard compacteness 
arguments \cite{Dac08,BB92,JL98}.

\begin{proposition}\label{Prop.Exist}
	Problem \ref{GlobOptMatModel} admits at least a solution $(\bmu_{n+1},\,\beta_{n+1})\in\mathcal{V}_{D,0}\times\mathbb{C}$.
\end{proposition}

\begin{proof}
	The set of solutions of Problem \ref{GlobOptMatModel} coincides with the set of minimizers of 
	$\mathcal{F}(t_{n+1},\bmu,\,\beta)$
	over the sublevel set of $\mathcal{E}(t_{n+1},\bmu,\beta)$ 
	with threshold $\mathcal{E}(t_{n+1},\bmu_n,\beta_n)$, that is,
	\[
		\Sigma=\{(\bmu,\beta)\in\mathcal{V}_{D,0}\times\mathbb{C}:\quad
		\mathcal{E}(t_{n+1},\bmu,\beta)\leq \mathcal{E}(t_{n+1},\bmu_n,\beta_n) \}
	\]
	which is not empty and sequentially compact \cite[Proposition 3.4]{TM10}, that is, 
	for any sequence $\{\left(\bmu_{\nu},\,\beta_{\nu}\right)\}_{\nu\in\mathbb{N}}$ of points of $\Sigma$,
	there exists a subsequence, which we keep on denoting by the same notation, which converges to 
	$(\bar\bmu,\,\bar\beta)$ with respect to the weak topology of $\mathcal{V}_{D,0}\times\mathcal{B}$	
	and $(\bar\bmu,\,\bar\beta)\in\Sigma$. Furthermore, by 
	\cite[Proposition 3.4]{TM10}, we have that $\mathcal{E}(t_{n+1},\bmu,\beta)$ is weakly sequentially lowersemicontinuous, 
	and by \cite[Lemma 4.3.1]{JL98} so is also the functional 
	$\beta\in\mathcal{B}\rightarrow\mathcal{D}(\beta_n,\,\beta)$,	
	that is, for any sequence $\{\left(\bmu_{\nu},\,\beta_{\nu}\right)\}$ of points of $\mathcal{V}_{D,0}\times\mathbb{C}$
	such that, up to a subsequence, weakly converges to a certain $(\bar\bmu,\,\bar\beta)$ 
	in $\mathcal{V}_{D,0}\times\mathbb{C}$, there holds
	\[
		\underset{\nu\rightarrow\infty}{\lim\inf}\,\mathcal{E}(t_{n+1},\,\bmu_{\nu},\,\beta_{\nu})
		\geq \mathcal{E}(t_{n+1},\,\bar\bmu,\,\bar\beta)\quad\text{and}\quad
		\underset{\nu\rightarrow\infty}{\lim\inf}\,\mathcal{D}(\beta_n,\,\beta_{\nu})
		\geq \mathcal{D}(\beta_n,\,\bar\beta)
	\]
	where $\lim\inf$ denotes the lower limit. We have therefore that also
	$\mathcal{F}(t_{n+1},\bmu,\beta)$ is weakly sequentially lowersemicontinuous. Thus,
	the application of the Weierstrass Theorem \cite[Theorem 1.1.2]{BB92} with the set $\Sigma$ and the 
	functional $\mathcal{F}(t_{n+1},\bmu,\beta)$
	 concludes the proof.
\end{proof}

\begin{remark}
\begin{itemize}
	\item[$(i)$] Uniqueness of minimizers is not guaranteed given that the functional \eqref{Eq.3.IM01} is not convex.
		Neither we can rule out the absence of local solutions in the sense of 
		Definition \ref{Def.LocSol}. The precise relation between Problem \ref{GlobOptMatModel}
		and \eqref{Eq.WeakForm} is given by the content of Proposition \ref{Prop.MinWeak} below.
	\item[$(ii)$] A similar argument to the proof of Proposition \ref{Prop.Exist} can also be applied to 
		establish the existence of minimizers for the minimization problems

		\begin{equation}
			\forall \bar\beta\in\mathbb{C}\,,\quad\min_{u\in\mathcal{V}_{D,0}}\,\mathcal{E}(t_{n+1},\,u,\,\bar\beta)
			\quad\text{and}\quad\forall \bar u\in\mathcal{V}_{D,0}\,,\quad
			\min_{\beta\in\mathbb{C}}\,\mathcal{E}(t_{n+1},\,\bar u,\,\beta)+\mathcal{D}(\beta_n,\,\beta)\,,
		\end{equation}
		which will be examined in Section \ref{Sec.AMMVF}.
\end{itemize}
\end{remark}

If we consider the Euler-Lagrange equations of the functional \eqref{Eq.3.IM01}, we obtain 
the weak form \eqref{Eq.WeakForm} of the incremental boundary value problem \eqref{Eq.IncConstModel}. This result justifies 
therefore Problem \ref{GlobOptMatModel} as a minimization formulation of the equations \eqref{Eq.IncConstModel}.
More precisely, we have the following result.

\begin{proposition}\label{Prop.MinWeak}
	If $(\bmu_{n+1},\,\beta_{n+1})$ solves Problem \ref{GlobOptMatModel}, then $(\bmu_{n+1},\,\beta_{n+1})$
	is a local solution of the model \eqref{FreModel}, that is, $(\bmu_{n+1},\,\beta_{n+1})$ solves \eqref{Eq.WeakForm}.
\end{proposition}

\begin{proof}
Let $(\bmu_{n+1},\,\beta_{n+1})$ be a solution of Problem \ref{GlobOptMatModel}. Then 
$(\bmu_{n+1},\,\beta_{n+1})$ is a solution of the following stationariety conditions 

\begin{subequations}\label{Eq.StatCnd}
\begin{align}
	&D\mathcal{E}(t_{n+1},\bmu,\beta)[\bmv]+D\mathcal{D}(\beta_n,\,\beta)[\bmv]=0
		\quad\text{for all }\bmv\in\mathcal{V}_{D,0}	\label{Eq.GateauxDerE}	\\[1.5ex]
	&D\mathcal{E}(t_{n+1},\bmu,\beta)[\gamma-\beta]+D\mathcal{D}(\beta_n,\,\beta)[\gamma-\beta]\geq 0
		\quad\text{for all }\gamma\in\mathbb{C}\,, \label{Eq.GateauxDerDiss}
\end{align}
\end{subequations}
where $D\mathcal{F}(t_{n+1},\bmu,\beta)[\bmv]$ and $D\mathcal{F}(t_{n+1},\bmu,\beta)[\gamma]$
are the Gateaux derivatives of the functional $\mathcal{F}(t_{n+1},\bmu,\beta)$ with respect to $\bmu$ and $\beta$. 
Condition \eqref{Eq.GateauxDerDiss} is the variational inequality corresponding to the stationariety condition of the functional
$\mathcal{F}(t_{n+1},\bmu,\beta)$ in the variable $\beta$ defined over the convex set $\mathbb{C}$. 
The expressions of the Gauteaux derivatives of the functionals \eqref{EnrgFrmFreMieFree} and \eqref{Eq:DissAT2}
with respect to $\bmu$ and $\beta$ are given as follows

\begin{equation}\label{Eq.Gateaux}
\begin{split}
	D\mathcal{E}(t_{n+1},\bmu,\beta)[\bmv]&=\left.\frac{d}{dh}\right|_{h=0}\mathcal{E}(t_{n+1},\bmu+h\bmv,\beta)
		\\
		&=\int_{\Omega}\big[R(\beta)\sigmab_0^+(\varepsilonb(\bmu))+
		\sigmab_0^-(\varepsilonb(\bmu))\big]\colon\varepsilonb(\bmv)\,\dx\\
		&+\int_{\Omega}\big[R(\beta)\sigmab_0^+(\varepsilonb(\bmu_{D,t}))+
		\sigmab_0^-(\varepsilonb(\bmu_{D,t}))\big]\colon\varepsilonb(\bmv)\,\dx-\langle\ell(t),\,\bmv\rangle  \\[1.5ex]
	D\mathcal{E}(t_{n+1},\bmu,\beta)[\gamma]&=\left.\frac{d}{dh}\right|_{h=0}\mathcal{E}(t_{n+1},\bmu,\beta+h\gamma)\\[1.5ex]
		&=\int_{\Omega}\frac{dR}{d\beta}\psi_0^+(\varepsilonb(\bmu+\bmu_{D,t}))\,\gamma\,\dx+
		\int_{\Omega}g_c\ell\nabla\beta\cdot\nabla\gamma\,\dx\\[1.5ex]
	D\mathcal{D}(\beta_n,\,\beta)[\bmv]&=0\\[1.5ex]
	D\mathcal{D}(\beta_n,\,\beta)[\gamma]&=\left.\frac{d}{dh}\right|_{t=0}\mathcal{D}(\beta_n,\,\beta+h\gamma)
		=\int_{\Omega}\frac{g_c}{\ell}\beta\gamma\,\dx\,,
\end{split}
\end{equation}
which, replaced into \eqref{Eq.StatCnd}, give \eqref{Eq.WeakForm}, 
that is the weak form of \eqref{Eq.IncConstModel}.
\end{proof}

In our succesive developments, we will consider only the 
case of applied Dirichlet boundary conditions, zero body force and zero traction forces.


\subsection{Energetic bounds}
The solutions of the incremental minimization problem enjoy additional properties which 
will be used to build the backtracking algorithm. In the proof of these additional properties, 
it is determinant to note the role played by $(\bmu_{n+1},\,\beta_{n+1})$ as global 
optimizers of Problem \ref{GlobOptMatModel}. 

\begin{proposition}
Let $(\bmu_{n+1},\,\beta_{n+1})\in\mathcal{V}_{D,0}\times \mathbb{C}$ be solution
of Problem \ref{GlobOptMatModel} for $n=0,\ldots,N-1$. Then the following estimates hold
\begin{itemize}
\item[$(i)$] A stability condition met by $(\bmu_{n+1},\,\beta_{n+1})$ in the sense that 

		\begin{equation}\label{Eq.3.Stab}
			\boxed{\mathcal{E}(t_{n+1},\bmu_{n+1},\,\beta_{n+1})\leq \mathcal{E}(t_{n+1},\tilde\bmu,\,\tilde\beta)
					+\mathcal{D}(\beta_{n},\tilde\beta)
					\quad\text{for all }(\tilde\bmu,\,\tilde\beta)\in\mathcal{V}_{D,0}\times\mathcal{C}
					}\,.
		\end{equation}  
\item[$(ii)$] The upper bound to  
	$\mathcal{E}(t_{n+1},\bmu_{n+1},\beta_{n+1})-\mathcal{E}(t_{n},,\bmu_{n},\beta_{n})+\mathcal{D}(\beta_{n},\,\beta_{n+1})$
	given by

		\begin{equation}\label{Eq.UpBnd}
		\boxed{
			\begin{aligned}
				&\mathcal{E}(t_{n+1},\bmu_{n+1},\beta_{n+1})-\mathcal{E}(t_{n},\bmu_{n},\beta_{n})
				+\mathcal{D}(\beta_{n},\beta_{n+1})\\[1.5ex]
				&\phantom{xxxxxxx}\leq \mathcal{E}(t_{n+1},\bmu_{n},\beta_{n})-\mathcal{E}(t_n,\bmu_{n},\beta_{n}):=UB_{n,n+1}\,.
			\end{aligned}
		}
	\end{equation}
\item[$(iii)$] The lower bound to $\mathcal{E}(t_{n+1},\bmu_{n+1},\beta_{n+1})-\mathcal{E}(t_{n},,\bmu_{n},\beta_{n})
		+\mathcal{D}(\beta_{n},\,\beta_{n+1})$ given by

	\begin{equation}\label{Eq.LwBnd}
		\boxed{
		\begin{aligned}
			&\mathcal{E}(t_{n+1},\bmu_{n+1},\beta_{n+1})-\mathcal{E}(t_{n},\bmu_{n},\beta_{n})
			+\mathcal{D}(\beta_{n},\beta_{n+1})\\[1.5ex]
			&\phantom{xxxxxxx}\geq \mathcal{E}(t_{n+1},\bmu_{n+1},\beta_{n+1})-\mathcal{E}(t_n,\bmu_{n+1},\beta_{n+1}):=LB_{n,n+1}\,.
		\end{aligned}
		}
	\end{equation}
\end{itemize}
\end{proposition}

\begin{proof} 
	\textit{Part (i)}: Since $\mathcal{D}(\beta_{n},\,\beta_{n+1})\geq 0$ and from the definition of 
		$(\bmu_{n+1},\beta_{n+1})\in\mathcal{V}_{D,0}\times\mathbb{C}$ we have that

		\begin{equation}\label{Eq.GlobMin}
		\begin{split}
			\mathcal{E}(t_{n+1},\,\bmu_{n+1},\,\beta_{n+1})&\leq
			\mathcal{E}(t_{n+1},\,\bmu_{n+1},\,\beta_{n+1})+\mathcal{D}(\beta_{n},\,\beta_{n+1})\\[1.5ex]
			&\leq
			\mathcal{E}(t_{n+1},\,\tilde\bmu,\,\tilde\beta)+\mathcal{D}(\beta_{n},\,\tilde\beta)\quad\text{for any }
			\tilde\bmu\in\mathcal{V}_{D,0}\text{ and }\tilde\beta\in\mathbb{C}\,,
		\end{split}
		\end{equation}
		which is \eqref{Eq.3.Stab}.\\[1.5ex]
	\textit{Part (ii)}: 
		Using  the definition of $(\bmu_{n+1},\beta_{n+1})\in\mathcal{V}_{D,0}\times\mathbb{C}$ 
		with $\tilde\bmu=\bmu_n$ and $\tilde\beta=\beta_n$, then it is 

		\begin{equation*}
			\mathcal{E}(t_{n+1},\,\bmu_{n+1},\,\beta_{n+1})+\mathcal{D}(\beta_{n},\,\beta_{n+1})\leq
			\mathcal{E}(t_{n+1},\,\bmu_{n},\,\beta_{n})\,.
		\end{equation*}
		given that $\mathcal{D}(\beta_{n},\,\beta_{n})=0$, thus by adding $-\mathcal{E}(t_{n},\,\bmu_{n},\,\beta_{n})$ 
		to both sides, we get \eqref{Eq.UpBnd}.\\[1.5ex]
	\textit{Part (iii)}: We start from the quantity that we want to bound to which we add and subtract $\mathcal{E}(t_{n},\,\bmu_{n+1},\,\beta_{n+1})$.
	This gives

	\begin{equation}\label{Eq.LwBnd.01}
		\begin{split}
		    &\mathcal{E}(t_{n+1},\,\bmu_{n+1},\,\beta_{n+1})-\mathcal{E}(t_{n},\,\bmu_{n},\,\beta_{n})
			+\mathcal{D}(\beta_{n},\,\beta_{n+1})=\mathcal{E}(t_{n+1},\,\bmu_{n+1},\,\beta_{n+1})
			-\mathcal{E}(t_{n},\,\bmu_{n+1},\,\beta_{n+1})\\[1.5ex]
		    &\phantom{xxxxxxxxxxxx}+\mathcal{E}(t_{n},\,\bmu_{n+1},\,\beta_{n+1})-\mathcal{E}(t_{n},\,\bmu_{n},\,\beta_{n})
			+\mathcal{D}(\beta_{n},\,\beta_{n+1})\,.
		\end{split}
	\end{equation}
	By the definition of $(\bmu_n,\beta_n)\in\mathcal{V}_{D,0}\times\mathbb{C}$, 

	\begin{equation}\label{Eq.GlobMinUN}
		\mathcal{E}(t_{n},\,\bmu_{n},\,\beta_{n})+\mathcal{D}(\beta_{n-1},\,\beta_{n})\leq
		\mathcal{E}(t_{n},\,\tilde\bmu,\,\tilde\beta)+\mathcal{D}(\beta_{n-1},\,\tilde\beta)\quad\text{for any }
		\tilde\bmu\in\mathcal{V}_{D,0}\text{ and }\tilde\beta\in\mathbb{C}\,,
	\end{equation}
	specialized for $\tilde\bmu=\bmu_{n+1}$ and $\tilde\beta=\beta_{n+1}$, gives

	\begin{equation}\label{Eq.GlobMinUN}
		\mathcal{E}(t_{n},\,\bmu_{n},\,\beta_{n})+\mathcal{D}(\beta_{n-1},\,\beta_{n})\leq
		\mathcal{E}(t_{n},\,\bmu_{n+1},\,\beta_{n+1})+\mathcal{D}(\beta_{n-1},\,\beta_{n+1})\,.
	\end{equation}
	From the expression \eqref{Eq:DissAT2} of $\mathcal{D}(\beta_1,\beta_2)$, we find

	\begin{equation}
		\mathcal{D}(\beta_{n-1},\,\beta_{n+1})-\mathcal{D}(\beta_{n-1},\,\beta_{n})
		=\frac{g_c}{2\ell}\int_{\Omega}(\beta_{n+1}^2-\beta_n^2)\,\dx
		=\mathcal{D}(\beta_{n},\,\beta_{n+1})\,,
	\end{equation}
	which used in \eqref{Eq.GlobMinUN} gives 

	\begin{equation}\label{Eq.GlobMinUN1}
		\mathcal{E}(t_{n},\,\bmu_{n+1},\,\beta_{n+1})-\mathcal{E}(t_{n},\,\bmu_{n},\,\beta_{n})
		+\mathcal{D}(\beta_{n},\,\beta_{n+1})\geq 0\,.
	\end{equation}
	By comparing \eqref{Eq.GlobMinUN1}  with \eqref{Eq.LwBnd.01} we get \eqref{Eq.LwBnd}.
\end{proof}

By taking into account for the expression \eqref{EnrgFrmFreMieFree},
the terms that appear in \eqref{Eq.UpBnd} and \eqref{Eq.LwBnd} have the following explicit expressions

\begin{subequations}
\begin{align}
	\mathcal{E}(t_{n+1},\bmu_{n+1},\beta_{n+1})&=
		\int_{\Omega}\bigg[R(\beta_{n+1})\psi_0^+(\varepsilonb(\bmu_{n+1}+\bmu_{D,n+1}))
			+\psi_0^-(\varepsilonb(\bmu_{n+1}+\bmu_{D,n+1}))\bigg] \,\dx	\nonumber \\[1.5ex]
		&+\frac{g_c\ell}{2}\int_{\Omega}|\nabla\beta_{n+1}|^2\,\dx\,,\\[1.5ex]
	\mathcal{E}(t_{n},\bmu_{n},\beta_{n})&=
		\int_{\Omega}\bigg[R(\beta_{n})\psi_0^+(\varepsilonb(\bmu_{n}+\bmu_{D,n}))
			+\psi_0^-(\varepsilonb(\bmu_{n}+\bmu_{D,n}))\bigg] \,\dx	\nonumber \\[1.5ex]
		&+\frac{g_c\ell}{2}\int_{\Omega}|\nabla\beta_{n}|^2\,\dx\,,\\[1.5ex]
	\mathcal{D}(\beta_{n},\beta_{n+1})&=\frac{g_c}{2\ell}\int_{\Omega}\left(\beta^2_{n+1}-\beta^2_{n}\right)\,\dx\,,
\end{align}
\end{subequations}
whereas

\begin{equation}\label{Eq.Bounds}
\begin{split}
	\mathcal{E}(t_{n+1},\bmu,\beta)-\mathcal{E}(t_n,\bmu,\beta)&=
			\int_{\Omega}R(\beta)\big[
				\psi_0^+(\varepsilonb(\bmu+\bmu_{D,n+1}))-\psi_0^+(\varepsilonb(\bmu+\bmu_{D,n}))
			\big]\,\dx\\[1.5ex]
	&+
			\int_{\Omega}\big[
				\psi_0^-(\varepsilonb(\bmu+\bmu_{D,n+1}))-\psi_0^-(\varepsilonb(\bmu+\bmu_{D,n}))
			\big]\,\dx\,, 
\end{split}
\end{equation}
which does not contain the term with $\nabla\beta$ that cancels out. 

\begin{remark}\label{Rem.3.SC}
	In the expression of $\mathcal{R}$ and $\mathcal{E}$, we have not taken into account 
	for the indicator functions $I_{\mathbb{R}^+}(\dot\beta)$ and $I_{[0,1]}(\beta)$, respectively, 
	given that the corresponding conditions
	on $\beta$ have been explicitly enforced as side conditions on the variable $\beta$.
\end{remark}


\section{Alternate minimization}\label{Sec.AMMVF}


The alternating minimization method consists in solving separately and sequentially the minimization of 
the functional
\[
	\mathcal{F}(t_{n+1},\bmu,\beta;\,\beta_n)=\mathcal{E}(t_{n+1},\bmu,\beta)+\mathcal{D}(\beta_n,\beta)
\]
with respect to the variables $\bmu$ and $\beta$ over the set
$\mathcal{V}_{D,0}$ and $\mathbb{C}$, respectively, where $\mathbb{C}$ is the convex set defined by \eqref{Eq.DefAdmSet}.
For each time step
$[t_{n},\,t_{n+1}]$, we produce, therefore, a sequence $(\bmu_{n+1}^i\,,\beta_{n+1}^i)_{i\in\mathbb{N}}$
where each term of the sequence is obtained by solving the
following minimization problems.  

\begin{subequations}\label{Sec.AMM}
	\begin{align}
		&\text{Set }\beta^0_{n+1}\in\mathbb{C},\,i=0		\nonumber\\[1.5ex]
		&\text{Find }\bmu_{n+1}^{i+1}\in\mathcal{V}_{D,0}\text{ such that }\text{minimize }
			\mathcal{F}(t_{n+1},\bmu,\,\beta^{i}_{n+1};\,\beta_n) \label{Sec.AMM.Disp}	\\[1.5ex]
		&\text{Find }\beta_{n+1}^{i+1}\in\mathbb{C}\text{ such that minimize }
			\mathcal{F}(t_{n+1},\bmu^{i+1}_{n+1},\,\beta;\,\beta_n)	\label{Sec.AMM.Damage}\\[1.5ex]
		&i\leftarrow i+1		\nonumber
	\end{align}
\end{subequations}

The alternating minimization method has, for instance, been used also in \cite{BFM00,BFM08,MR15,MRZ10,RKZ13,ZG10}.

\begin{remark}\label{Rem.AltMin}
	In \eqref{Sec.AMM}, $\beta_{n+1}^0$ represents the initial guess for $\beta$ to start the
	alternating minimization of $\mathcal{F}(t_{n+1},\bmu,\,\beta)$ over $\mathcal{V}_{D,0}\times\mathbb{C}$,
	whereas $\beta_n$ enters in the definition \eqref{Eq.DefAdmSet} of the convex set $\mathbb{C}$ of the admissible solutions. 
	In the standard application of \eqref{Sec.AMM}, we can take $\beta_{n+1}^0=\beta_n$ whereas in applying 
	the backtracking method described below, 
	we can also have $\beta_{n+1}^0\not=\beta_n$.
\end{remark}

The realization of the scheme \eqref{Sec.AMM} for finding solutions of Problem \ref{GlobOptMatModel} gives rise to the 
questions about the convergence of the scheme and the meaning of the corresponding limit in the case of convergence. 
In finite dimensional optimization, the scheme \ref{Sec.AMM} is known as block--coordinate descent method \cite{Ber16},
whose convergence cannot be given, in general, for granted, especially when dealing with nonsmooth optimization. 
A convergence analysis of \eqref{Sec.AMM} is reported in \cite{AN20,KN17} where it is shown that, up to a subsequence, 
as $i\to\infty$, we obtain a critical point of $\mathcal{F}(t_{n+1},\bmu,\beta;\,\beta_n)$.  
In this paper, and consistently with the numerical scheme which we will use to enforce the non-interpenetration condition,
we will analyse the convergence of a regularized formulation of the scheme \eqref{Sec.AMM} where
the convex constrained optimization \eqref{Sec.AMM.Damage} is solved by a penalization method 
with the introduction of a penalty function $\varphi$ defined over $\mathcal{B}$ which is convex and smooth and such that 
$\varphi(\gamma)\geq 0$ for any $\gamma\in\mathcal{B}$ and with the property that 
$\varphi(\gamma)= 0$ if and only if $\gamma\in\mathbb{C}$ \cite[page 321]{Cia89}. More specifically,
we will enforce through penalty only the irreversibility
constraint whereas the simple bounds \eqref{Eq.3.IM01SB} on the variable $\beta$ are taking into account
in the scheme itself. We therefore replace \eqref{Sec.AMM.Damage} with the following unconstrained minimization problem.
\begin{equation}\label{Sec.AMM.Penalty}
	\text{Assume }\epsilon>0.\text{ Find }\beta_{n+1}^{i+1}\in\mathcal{B}\text{ such that minimize }
			\mathcal{F}(t_{n+1},\bmu^{i+1}_{n+1},\,\beta;\,\beta_n)+\frac{1}{\epsilon}\varphi(\beta)\,.
\end{equation}
As penalty function $\varphi$ we take 
\[
	\varphi(\beta-\beta_n)=\int_{\Omega}[\beta-\beta_n]^2_{-}\,\dx\,, 
\]
where for $x\in\mathbb{R}$, $[x]_-=(x-|x|)/2$. Indeed, we have
\[
	\varphi(\beta-\beta_n)=\int_{\Omega}[\beta-\beta_n]^2_{-}\,\dx=0\Leftrightarrow
	[\beta-\beta_n]^2_{-}=0\Leftrightarrow[\beta-\beta_n]_{-}=0\Leftrightarrow \beta\geq\beta_n\,.
\]
The alternate minimization \eqref{Sec.AMM} is thus replaced by the following scheme

\begin{subequations}\label{Eq.AMMPen}
	\begin{align}
		&\text{Set }\beta^0_{n+1}\in\mathbb{C},\,i=0		\nonumber\\[1.5ex]
		&\text{Find }\bmu_{n+1}^{i+1}\in\mathcal{V}_{D,0}\text{ such that }\text{minimize }
			\mathcal{F}(t_{n+1},\bmu,\,\beta^{i}_{n+1};\,\beta_n) \label{Sec.AMMPen.Disp}	\\[1.5ex]
		&\text{Find }\beta_{n+1}^{i+1}\in\mathcal{B}\text{ such that minimize }
			\mathcal{F}(t_{n+1},\bmu^{i+1}_{n+1},\,\beta;\,\beta_n)+\frac{1}{\epsilon}\varphi(\beta)	\label{Sec.AMMPen.Damage}\\[1.5ex]
		&i\leftarrow i+1		\nonumber
	\end{align}
\end{subequations}

Since each of the minimizations \eqref{Eq.AMMPen} is an unconstrained convex smooth optimization problem, the corresponding 
optimality conditions, given by the Euler-Lagrange equations, are also minimality conditions and are given by the following 
variational formulation.

\begin{subequations}\label{AltMinVarForm}
	\begin{align}
		&\text{Let }\varepsilon>0.\text{ Set }\beta^0_{n+1}\in\mathbb{C},\,i=0	\nonumber\\[1.5ex]
		&\text{Find }\bmu_{n+1}^{i+1}\in\mathcal{V}_{D,0}\text{ such that}		\nonumber\\[1.5ex]
		&\phantom{xxxx} \int_{\Omega}\sigmab(\varepsilonb(\bmu_{n+1}^{i+1}),\beta_{n+1}^{i})\colon\varepsilonb(\bmv)\,\dx
			= -\int_{\Omega}\sigmab(\varepsilonb(\bmu_{D,n+1},\beta_{n+1}^{i})\colon\varepsilonb(\bmv)\,\dx\quad \text{ for all }\bmv\in\mathcal{V}_{D,0}\,.	\label{eq:3.IM22.disp}\\[1.5ex]
		&\text{Find }\beta_{n+1}^{i+1}\in\mathcal{B}\text{ such that}		\nonumber\\[1.5ex]
		&\phantom{xxxx} \int_{\Omega}\left.\frac{dR}{d\beta}\right|_{\beta_{n+1}^{i+1}}
			\psi_0^+(\varepsilonb(\bmu_{n+1}^{i+1}+\bmu_{D,n+1}))\,\gamma\,\dx
			+ \int_{\Omega}	\frac{g_c}{\ell}\beta_{n+1}^{i+1}\,\gamma\,\dx
			+ \int_{\Omega}	g_c\ell\nabla\beta_{n+1}^{i+1}\cdot\nabla\gamma\,\dx	\nonumber\\[1.5ex]
		&\phantom{xxxxxxxxxxxxxxxxxxx} + \frac{1}{\epsilon}\int_{\Omega}[ \beta_{n+1}^{i+1}-\beta_n]_{-}\,\gamma	\,\dx	= 0
				\quad\text{for all }\gamma\in\mathcal{B}\,.	\label{eq:3.IM22.damage}\\[1.5ex]
		&i\leftarrow i+1\,,		\nonumber
	\end{align}
\end{subequations}
where $\sigmab(\varepsilonb,\beta)=R(\beta)\sigmab_0^+(\varepsilonb)+\sigmab_0^-(\varepsilonb)$ with 

\begin{equation}\label{Eq.ApndxA.SplitStres}
	\sigmab_0^{+}=\frac{\partial\psi_0^{+}}{\partial \varepsilonb}\quad
	\text{ and }
	\sigmab_0^{-}=\frac{\partial\psi_0^{-}}{\partial \varepsilonb}\,.
\end{equation}

In order to discuss the convergence of \eqref{AltMinVarForm}, we require an additional notion.

\begin{definition}
	We say that $(\bmu,\,\beta)\in\mathcal{V}_{D,0}\times\mathcal{B}$ is a critical point of the functional
	\begin{equation}\label{Def.Funct.Pen}
		\mathcal{F}(t_{n+1},\,\bmu,\,\beta;\,\beta_n)+\frac{1}{\varepsilon}\varphi(\beta)
	\end{equation}
	if $(\bmu,\,\beta)$ meets the following equations

\begin{subequations}\label{Eq.VariatCond}
\begin{align}
	&\int_{\Omega}\sigmab(\varepsilonb(\bmu),\beta)\colon\varepsilonb(\bmv)\,\dx
	= -\int_{\Omega}\sigmab(\varepsilonb(\bmu_{D},\beta)\colon\varepsilonb(\bmv)\,\dx
		\quad \text{ for all }\bmv\in\mathcal{V}_{D,0}\,.	\label{Eq.VariatCond.01}\\[1.5ex]
	&\int_{\Omega}\left.\frac{dR}{d\beta}\right|_{\beta}
		\psi_0^+(\varepsilonb(\bmu+\bmu_{D}))\,\gamma\,\dx
		+ \int_{\Omega}	\frac{g_c}{\ell}\beta\,\gamma\,\dx
		+ \int_{\Omega}	g_c\ell\nabla\beta\cdot\nabla\gamma\,\dx	\nonumber\\[1.5ex]
	&\phantom{xxxxxxxxxxxx} + \frac{1}{\epsilon}\int_{\Omega}[ \beta-\beta_n]_{-}\,\gamma	\,\dx	= 0
				\quad\text{for all }\gamma\in\mathcal{B}\,.	\label{Eq.VariatCond.02}
\end{align}
\end{subequations}
\end{definition}

We can then state the following result.

\begin{proposition}\label{Prop.CnvrgAMA}
	Let $(\bmu_{n+1}^i\,,\beta_{n+1}^i)_{i\in\mathbb{N}}$ be a sequence generated by the scheme 
	\eqref{AltMinVarForm}. Then, up to a subsequence, $(\bmu_{n+1}^i\,,\beta_{n+1}^i)_{i\in\mathbb{N}}$ 
	is convergent in $\mathcal{V}_{D,0}\times\mathcal{B}$ 
	and its limit $(\bmu,\,\beta)$ is a critical point of the functional \eqref{Def.Funct.Pen}.
\end{proposition}

\begin{proof}
	The proof can be obtained by an adaptation of the arguments given in \cite[Theorem 1]{Bou07}
	or \cite[Section 5.1]{AN20}
	which we refer to for the full details, and consists of obtaining first an a-priori estimate 
	of the solution of \eqref{AltMinVarForm} and then in applying compacteness arguments 
	(see also \cite[page 268]{Ber16} for an application of these arguments to the finite 
	dimensional setting of \eqref{AltMinVarForm}). 
\end{proof}

\begin{remark}
	A stronger result than the one stated in Proposition \ref{Prop.CnvrgAMA} is reported in \cite{BWBNR20}
	where, by modifying the scheme \eqref{AltMinVarForm} with the introduction of 
	coercive terms of Uzawa's like \cite{GLT81}, it is proved the convergence of the sequence of the iterates itself
	to a critical point of the discrete energetic functional.
\end{remark}

\begin{remark}
	Since the uniqueness of the solution of Problem \ref{GlobOptMatModel} is not guaranteed,
	by changing the initial value $\beta_{n+1}^0$ in \eqref{AltMinVarForm}, we will build in general 
	a different sequence which, up to a subsequence, will converge to a different critical point, which is an 
	approximation of a local solution of the model \eqref{FreModel}. 
\end{remark}


\section{Fully discrete scheme}\label{Sec.FDS}


In this section we present first the fully discrete equations obtained by a FE interpolation of the 
displacement and the damage phase-field. We then describe the numerical algorithm which we use to find
an approximate solution to these equations.


\subsection{Finite Element Discretization}


The fully discrete equations are 
obtained by replacing the infinite dimensional affine spaces $\mathcal{V}_{D,0}\times \mathcal{B}$ of the trial functions 
$(\bmu,\beta)$ and of the test functions $(\bmv,\gamma)$ with finite dimensional affine subspaces
which are taken here as finite element spaces.
Let us denote by $\bmN_{\bmu}$, and $\bmN_{\beta}$ the shape interpolation functions of
$\bmu$ and $\beta$, respectively, by $\bmU,\,\bmU_{D}\in\mathbb{R}^{n_{\bmu}}$
the displacement degree of freedom of the test functions 
$\bmu\in \mathcal{V}_{D,0}$ and of the lifting function $\bmu_{D}\in \mathcal{V}_D$, respectively,
and by $\bmA\in\mathbb{R}^{n_{\beta}}$ 
the degree of freedom of the field $\beta$. We have the following interpolations

\begin{equation}\label{Eq.3.IM23.SF}
	\bmu^h(\bmx) +\bmu^h_{D}(\bmx)= \bmN_{\bmu}(\bmx)(\bmU+\bmU_D)\quad\text{and}\quad
	\beta^h(\bmx) = \bmN_{\beta}(\bmx)\bmA.
\end{equation}
Consequently,

\begin{equation}\label{Eq.3.IM24.Grad}
	\varepsilonb^h(\bmx) =\nabla_s\bmu^h+\nabla_s\bmu_D^h=\bmB_{\bmu}(\bmx)(\bmU+\bmU_D)
	\quad\text{and}\quad
	\nabla\beta^h(\bmx)=  \bmB_{\beta}(\bmx)\bmA\,,
\end{equation}
where we  have introduced the matrices $\bmB_{\bmu}$ and $\bmB_{\beta}$ which are obtained by appropriately differentiating 
and combining rows of the matrices $\bmN_{\bmu}$ and $\bmN_{\beta}$, respectively \cite{Bat96}.
By using \eqref{Eq.3.IM23.SF} and \eqref{Eq.3.IM24.Grad}
into \eqref{eq:3.IM22.disp} and \eqref{eq:3.IM22.damage}, we obtain the following discrete variational formulation 

\begin{subequations}\label{Eq.3.IM25}
\begin{align}
	&\delta\bmU^T  \int_{\Omega}  \bmB_{\bmu}^{T}(\bmx)\; {}^h\sigmab\big(\bmB_{\bmu}(\bmx)\bmU,\,\bmN_{\beta}(\bmx)\bmA\big)\,\dx 
		+\delta\bmU^T  \int_{\Omega}  \bmB_{\bmu}^{T}(\bmx)\; {}^h\sigmab\big(\bmB_{\bmu}(\bmx)\bmU_D,\,\bmN_{\beta}(\bmx)\bmA\big)\,\dx=0
		\label{Eq.3.IM25U}\\[2ex] 
	&\delta\bmA^T \int_{\Omega} \bmN_{\beta}^{T}(\bmx)\frac{\partial\psi}{\partial\beta}(\varepsilonb^h(\bmx),\,\beta^h(\bmx))\,\dx  
		+\delta\bmA^T  \int_{\Omega} \frac{g_c}{\ell}\bmN_{\beta}^{T}(\bmx) \bmN_{\beta}(\bmx)\bmA \,\dx \notag \\[2ex] 
	&\phantom{xxx}+\delta\bmA^T \int_{\Omega} g_c\, \ell \bmB_{\beta}^{T}(\bmx)\bmB_{\beta}(\bmx)\bmA\,\dx   
		+\frac{1}{\epsilon}\delta\bmA^T  \int_{\Omega}  \bmN_{\beta}^{T}(\bmx) \big[\bmN_{\beta}(\bmx)(\bmA-\bmA_n)\big]_{-}\,\dx=0 
		\label{Eq.3.IM25A}
\end{align}
\end{subequations}
where  
\[
	\frac{\partial\psi}{\partial\beta}(\varepsilonb^h(\bmx),\,\beta^h(\bmx))=
	\left.\frac{dg}{d\beta}\right|_{\bmN_{\beta}\bmA}\,\psi^{+}_0(\bmB_{\bmu}(\bmU+\bmU_D))\,,
\]
whereas

\begin{equation*}
\begin{split}
	&{}^h\sigmab(\bmB_{\bmu}\bmU,\bmN_{\beta}\bmA)=R(\bmN_{\beta}\bmA)\sigmab_0^+(\bmB_{\bmu}\bmU)+\sigmab_0^-(\bmB_{\bmu}\bmU)\\[1.5ex]
	&{}^h\sigmab(\bmB_{\bmu}\bmU_D,\bmN_{\beta}\bmA)=R(\bmN_{\beta}\bmA)\sigmab_0^+(\bmB_{\bmu}\bmU_D)
	+\sigmab_0^-(\bmB_{\bmu}\bmU_D)\,,
\end{split}
\end{equation*}
with $\sigmab_0^{\pm}$ given by \eqref{Eq.ApndxA.SplitStres}.


\begin{remark}
The symbol ${}^h(\cdot)$ is here used to mean that, in the present formulation, the 
field $(\cdot)$ is not interpolated but it is computed by solving an equation.
\end{remark}

If we denote by $\Omega_e^h$ a generic element of the triangulation $\mathcal{T}^h$ and by $\bmx_{e,i}$ the $i^{th}$ Gauss point 
of the element $\Omega_e^h$ and $n_{gp}$ their number, the discrete variational formulations \eqref{Eq.3.IM25} are 
thus transformed into the following system of nonlinear algebraic equations

\begin{subequations}\label{Eq.3.Residual}
\begin{align}
	& \bmR_{\bmu}(t_{n+1},\bmU,\bmA):=\sum_{\Omega_e^h\in\mathcal{T}^h}\;\sum_{i=1}^{n_{gp}} w_{e,i}j_{e,i}
			\bmB_{\bmu}^{T}(\bmx_{e,i})\bigg[R(\bmN_{\beta}(\bmx_{e,i})\bmA)\sigmab_0^+(\bmB_{\bmu}(\bmx_{e,i})\bmU)  \nonumber\\[1.5ex]
	&\phantom{xxxxxxx}+\sigmab_0^-(\bmB_{\bmu}(\bmx_{e,i})\bmU)+R(\bmN_{\beta}(\bmx_{e,i})\bmA)\sigmab_0^+(\bmB_{\bmu}(\bmx_{e,i})\bmU_D)+
					\sigmab_0^-(\bmB_{\bmu}(\bmx_{e,i})\bmU_D)
					\bigg]=\bm{0}\,, \label{Eq.3.ResidualU}\\[2.ex] 
	& \bmR_{\beta}(\bmU,\bmA;\,\bmA_n) :=\sum_{\Omega_e^h\in\mathcal{T}^h}\;\sum_{i=1}^{n_{gp}}w_{e,i}j_{e,i}\bigg\{
		\bmN_{\beta}^{T}(\bmx_{e,i})\bigg[\left.\frac{dR}{d\beta}\right|_{\bmN_{\beta}(\bmx_{e,i})\bmA}\psi_0^+(\bmB_{\bmu}(\bmx_{e,i})(\bmU+\bmU_D)) \notag \\[2ex]
	&\phantom{xxxxxxx} 		+\frac{g_c}{\ell} \bmN_{\beta}(\bmx_{e,i})\bmA  
	+\frac{1}{\epsilon}\big[\bmN_{\beta}(\bmx_{e,i})(\bmA-\bmA_n)\big]_{-}\bigg]
	+g_c\,\ell\bmB_{\beta}^{T}(\bmx_{e,i}) \bmB_{\beta}(\bmx_{e,i})\bmA
                          \bigg\}=\bm{0}\,,\label{Eq.3.ResidualA}
\end{align}
\end{subequations}
with $w_{e,i}$ and $j_{e,i}$ the weight and the value of the Jacobian determinant at the Gauss point $\bmx_{e,i}$,  
respectively \cite{Bat96}.

Consistently with \eqref{AltMinVarForm}, for each time step $[t_n,\,t_{n+1}]$, we consider the solution of 
\eqref{Eq.3.Residual} separately with respect to $\bmU$ and $\bmA$ as follows

\begin{subequations}\label{Eq.3.AltMinDiscEq}
\begin{align}
	& \text{Let }\epsilon>0.\text{ Set }\bmA^0\in\mathbb{R}^{n_{\beta}},\,i=0		\nonumber\\[1.5ex]
	& \text{Find }\bmU^{i+1}\in\mathbb{R}^{n_{\bmU}}:\,\bmR_{\bmu}(t_{n+1},\bmU^{i+1},\,\bmA^{i};\,\bmA_n)=\bm{0}\,,\label{Eq.3.AltMinDiscEq.U}\\[1.5ex] 
	& \text{Find }\bmA^{i+1}\in\mathbb{R}^{n_{\bmA}}:\,\bmR_{\beta}(\bmU^{i+1},\,\bmA^{i+1};\,\bmA_n)=\bm{0}\,,	\label{Eq.3.AltMinDiscEq.A}	\\[1.5ex]
	& i\leftarrow i+1\,,	\nonumber
\end{align}
\end{subequations}
which represent the finite element equations of the stationariety conditions of the 
alternating minimization problems \eqref{Eq.AMMPen}. 
We solve  each of the equations \eqref{Eq.3.AltMinDiscEq} by applying a fully consistent Newton's method. 
The resulting scheme is given by the Algorithm \ref{Alg.NewAMM}.


\normalem 

\begin{algorithm}[H]
\DontPrintSemicolon
\KwData{$(\bmU_n,\,\bmA_n)$, $\varepsilon$, $tol_{\bmU}$, $tol_{\bmA}$}
\KwResult{$(\bmU_{n+1},\,\bmA_{n+1})$}
		\Set{
\lnl{Alg.NewAMM.Rem1}     $i=0$\; 
\lnl{Alg.NewAMM.Rem2}     $\bmA^{0}=\bmA_{n}$, $\bmU^{0}=\bmU_{n}$}
\lnl{Alg.NewAMM.Rem3} \Repeat{$\|\bmU^{i+1}-\bmU^{i}\|_\infty\leq\text{ tol}_{\bmU}$ and $\|\bmA^{i+1}-\bmA^{i}\|_\infty\leq\text{ tol}_{\bmA}$}{
	   \BlankLine
\lnl{Alg.NewAMM.Rem4} $\bmU^{i,0}=\bmU^{i}, k=1$ \;
\lnl{Alg.NewAMM.Rem5} \Repeat{$\displaystyle \|\Delta\bmU\|_{\ell^{\infty}}\leq tol_{\bmU}$}{
\lnl{Alg.NewAMM.Rem6}  $\displaystyle \Delta\bmU=-\left[\frac{d\bmR_{\bmu}}{d\bmU}(t_{n+1},\bmU^{i,k-1},\,\bmA^{i};\,\bmA_n)\right]^{-1}
		\bmR_{\bmu}(t_{n+1},\bmU^{i,k-1},\,\bmA^{i};\,\bmA_n)$\;  
\lnl{Alg.NewAMM.Rem7}  $\displaystyle \bmU^{i,k}=\bmU^{i,k-1}+\Delta\bmU$\;  
\lnl{Alg.NewAMM.Rem8}  $\displaystyle k\leftarrow k+1$\;
		\BlankLine}
	    \BlankLine
\lnl{Alg.NewAMM.Rem9} $\displaystyle \bmU^{i+1}=\bmU^{i,k}$ \; 
\lnl{Alg.NewAMM.Rem10} $\displaystyle \bmA^{i,0}=\bmA^{i}, k=1$ \;
	    \BlankLine
\lnl{Alg.NewAMM.Rem11} \Repeat{$\displaystyle \|\Delta\bmA\|_{\ell^{\infty}}\leq tol_{\bmA}$}{
\lnl{Alg.NewAMM.Rem12} $\displaystyle \Delta\bmA=-\left[\frac{d\bmR_{\beta}}{d\bmA}(\bmU^{i+1},\,\bmA^{i,k-1};\,\bmA_n)\right]^{-1}
		\bmR_{\beta}(\bmU^{i+1},\,\bmA^{i,k-1};\,\bmA_n)$\;  
\lnl{Alg.NewAMM.Rem13} $\displaystyle \bmA^{i,k}=\bmA^{i,k-1}+\Delta\bmA$\;  
\lnl{Alg.NewAMM.Rem14} $\displaystyle k\leftarrow k+1$\;
	    \BlankLine}
\lnl{Alg.NewAMM.Rem15} $\displaystyle\bmA^{i+1}=\bmA^{i,k}$ \;
\lnl{Alg.NewAMM.Rem16} $i\leftarrow i+1$\;
	    \BlankLine}
\caption{\label{Alg.NewAMM} Alternate Minimization Algorithm with Newton's Method}
\end{algorithm}

\ULforem 


\begin{remark}
Correspondingly to what already noted in Remark \ref{Rem.AltMin} about the continuous formulation, 
we can make a similar observation for the discrete scheme \eqref{Eq.3.AltMinDiscEq}. 
A more general initialization of Algorithm \ref{Alg.NewAMM} defined on line \ref{Alg.NewAMM.Rem2}
is given by taking $\bmA^{0}=\bmA^{\ast}$, $\bmU^{0}=\bmU_{n}$, with $\bmA^{\ast}\in\mathbb{R}^{n_{\bmA}}$.
In this manner, we distinguish the role of $\bmA^{0}$, which is used to start the alternating minimization of the 
functional $\mathcal{F}(t_{n+1},\bmU,\bmA)$, from the role of $\bmA_n$ that enters in to the definition
of the admissible set of $\mathcal{F}$. In the standard application (without backtracking)
of the Algorithm \ref{Alg.NewAMM} we take $\bmA^{\ast}=\bmA_n$, but we will see in the next section 
that when this scheme is combined with the backtracking, $\bmA^{\ast}$ might be different from $\bmA_n$.
\end{remark}


\subsection{A two-sided energy estimate based backtracking algorithm}

In Section \ref{Sec.AMMVF} we have observed that the solution of the 
alternate minimization \eqref{Eq.3.AltMinDiscEq}
represents, in general, an approximation of a critical point of the functional 
$\mathcal{F}(t,\varepsilonb,\beta)=\mathcal{E}(t,\varepsilonb,\beta)+\mathcal{D}(\beta_n,\,\beta)$, 
which might not be a global minimizer of $\mathcal{F}$. 
The global minimization model given by Problem \ref{GlobOptMatModel} is, indeed,
a crucial assumption of the theory of material behaviour we are applying. 

Given the particular structure of the problem 
at hand, the optimization landscape can change from one step 
increment to the other depending on whether damage occurs and, if so, on its
extension. If damage does not occur or does not change much, the function landscape maintains 
its shape without creation of other minima.
To avoid to resort to global optimization methods applied to Problem \eqref{GlobOptMatModel},
we propose here a numerical strategy 
where we still apply Newton's method but we change starting point which falls in the attraction basin of 
a stationariety point with lower energy. To ensure that this happens, we will use the
two-sided energy estimates \eqref{Eq.LwBnd} and \eqref{Eq.UpBnd} met by the solutions 
of Problem \ref{GlobOptMatModel}.

This will be realized by a backtracking strategy which is similar to 
the one used in \cite{Bou07,BFM08,CLR16,MRZ10,RKZ13} for related problems. The difference is that we now 
exploit our two-sided energy estimates \eqref{Eq.LwBnd} and \eqref{Eq.UpBnd} as necessary conditions 
of global optimality.
Consider the finite element approximation of the estimates \eqref{Eq.LwBnd} and \eqref{Eq.UpBnd}
which we write as follows

\begin{equation}\label{Disc.TwoSidIneq}
\begin{split}
	-\eta+LB(t_{n},\bmU_{n+1},\bmA_{n+1})\leq 
		\mathcal{E}(t_{n+1},\,\bmU_{n+1},\,\bmA_{n+1})&-\mathcal{E}(t_{n},\,\bmU_{n},\,\bmA_{n})
			+\mathcal{D}(\bmA_{n},\,\bmA_{n+1})	\\[1.5ex]
	&\leq UB(t_{n+1},\bmU_{n},\bmA_{n})+\eta\,.
\end{split}
\end{equation}
In \eqref{Disc.TwoSidIneq}, $\eta>0$ is an energy tolerance introduced to account for the approximated globality of the discrete solution,
which thus depends on the space and time discretization error,
whereas the discrete expressions of the energetic
terms and of the lower and upper bounds, $LB(t_{n},\bmU_{n+1},\bmA_{n+1})$ and $UB(t_{n+1},\bmU_{n},\bmA_{n})$ respectively, 
which appear in \eqref{Disc.TwoSidIneq}, are obtained from 
\eqref{Eq.LwBnd} and \eqref{Eq.UpBnd} by taking into account for 
\eqref{Eq.3.IM23.SF} and  \eqref{Eq.3.IM24.Grad}. 
Box \ref{BoundsFE} contains the steps needed for the implementation of \eqref{Disc.TwoSidIneq} as 
postprocessing step.

When the estimates \eqref{Disc.TwoSidIneq} are violated by the computed solution,
we go back over the time steps and restart the alternate minimization \eqref{Eq.3.AltMinDiscEq}
with a different inital value for $\bmA$ by taking one with a lower energy state.

To illustrate how actually such strategy works,
assume that $(\bmU_n,\,\bmA_n)$ is the computed solution corresponding to the time step $[t_{n-1},\,t_{n}]$
and it is such that the pairs $(\bmU_{n-1},\,\bmA_{n-1})$ and $(\bmU_n,\,\bmA_n)$ meet 
the two-sided inequality \eqref{Disc.TwoSidIneq}. By taking then the succesive time step $[t_{n},\,t_{n+1}]$,
the solution $(\bmU_{n+1},\,\bmA_{n+1})$ of \eqref{Eq.3.AltMinDiscEq} obtained with the initial value
$\bmA^0=\bmA_n$ is such that the pairs
$(\bmU_n,\,\bmA_n)$ and $(\bmU_{n+1},\,\bmA_{n+1})$ do not meet \eqref{Disc.TwoSidIneq}, even though, by construction, it is 
$\mathcal{F}(t_{n+1},\bmU_{n+1},\,\bmA_{n+1})\leq \mathcal{F}(t_{n+1},\bmU_n,\,\bmA_n)$. 
In this case then $(\bmU_{n+1},\bmA_{n+1})$ must be discharged. This might occur because when we did solve
\eqref{Eq.3.AltMinDiscEq}, we have used a starting value which falls in the attraction basin of a stationary point with 
a higher energy level. The idea is therefore to provide a better estimate of a starting value which could likely fall in 
the attraction basin of a stationary point with lower energy.
We therefore propose to go back one time step, that is, we solve again the time step $[t_{n-1},\,t_n]$,
even though the pairs $(\bmU_{n-1},\,\bmA_{n-1})$ and $(\bmU_n,\,\bmA_n)$
were meeting 
\eqref{Disc.TwoSidIneq}, but this time we use the starting value $\bmA^0=\bmA_{n+1}$.
The iteration over each previous step 
of the equilibrium path is repeated
until the estimates \eqref{Disc.TwoSidIneq} are met. The number of the backtracking steps will then clearly depend on the 
quality of the starting guess.
Algorithm \ref{Alg.BT} presents a conceptual implementation of the proposed strategy, whereas Figure \ref{Fig.BackTracking}
visualizes such algorithm, with possible situations for backtracking. The input data to start the algorithm are the total number 
$N$ of the time steps and the energy tolerance $\eta>0$ 
that enters the bound limit \eqref{Disc.TwoSidIneq}.
We also introduce the total number $K\geq 0$ of back steps by which we are willing to go back 
(by setting $K=0$ we do not activate the backtracking algorithm),
and the initial state $(\bmU_0,\,\bmA_0)$ at $t=0$. 
The parameter $K>0$ is not essential
for running the algorithm, but it is introduced only to allow the user to control the number of back steps. 
By setting a value of $K$ that is attained, we would accept discrete solutions that 
violate the estimates for some time steps.


\begin{figure}[H]
	\centering{
		\includegraphics[width=0.6\textwidth]{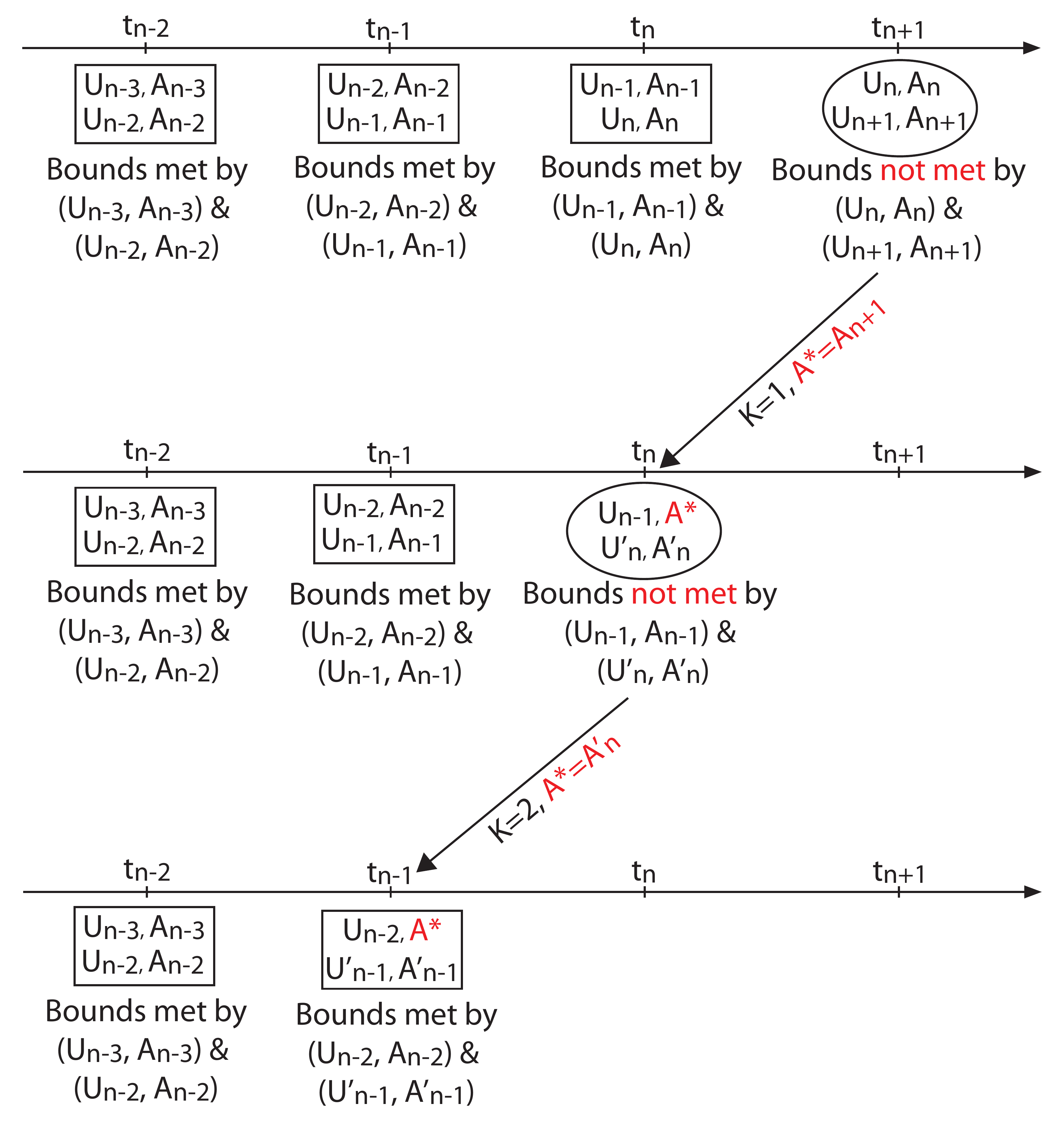}
		}
	\caption{\label{Fig.BackTracking} Visualization of the Backtracking Algorithm \ref{Alg.BT} 
	using the alternating minimization  
	\eqref{Eq.3.AltMinDiscEq} with $K=2$. In each square/circle we show the starting guess and the solution 
	of the incremental problem.
	}
\end{figure}



\boxx{0.95\textwidth}{BoundsFE}
{Implementation of \eqref{Disc.TwoSidIneq}}
{
	\begin{minipage}{\textwidth}
		\footnotesize{
	\noindent \textsc{function ERG}\\
		INPUT: $\bmU_1,\,\bmU_2,\,\bmA$\\
		COMPUTE:
			\[
				E=\sum_{\Omega_e^h\in\mathcal{T}^h}\;\sum_{i=1}^{n_{gp}} w_{e,i}j_{e,i}
				\bigg[ R\big(\bmN_{\beta}(\bmx_{e,i})\bmA\big)
				\psi_0^+\big(\bmB_{\bmu}(\bmx_{e,i})(\bmU_1+\bmU_2)\big)+
				\psi_0^-\big(\bmB_{\bmu}(\bmx_{e,i})(\bmU_1+\bmU_2)\big)
				\bigg]
			\]
	\noindent \textsc{function GRAD}\\
		INPUT: $\bmA$\\
		COMPUTE:
			\[
				G=\frac{g_c\ell}{2}\sum_{\Omega_e^h\in\mathcal{T}^h}\;\sum_{i=1}^{n_{gp}} w_{e,i}j_{e,i}
				\bigg[\bmB_{\beta}^T(\bmx_{e,i})\bmB_{\beta}(\bmx_{e,i})\bmA\cdot\bmA\bigg]
			\]
\noindent \textsc{function DIS}\\
		INPUT: $\bmA$\\
		COMPUTE:
			\[
				D=\frac{g_c}{2\ell}\sum_{\Omega_e^h\in\mathcal{T}^h}\;\sum_{i=1}^{n_{gp}} w_{e,i}j_{e,i}
					\bigg[\bmN_{\beta}^T(\bmx_{e,i})\bmN_{\beta}(\bmx_{e,i})\bmA\cdot\bmA\bigg]
			\]
\noindent To compute $\mathcal{E}(t_{n+1},\bmu_{n+1},\beta_{n+1})-\mathcal{E}(t_{n},\bmu_{n},\beta_{n})+\mathcal{D}(\beta_n,\beta_{n+1})$, 
use:
\begin{enumerate}
	\item \textsc{function ERG} with \textsc{INPUT:} $\bmU_1=\bmU_{n+1}$, $\bmU_2=\bmU_{D,n+1}$, $\bmA=\bmA_{n+1}$\\
		\phantom{xxxxxxx}\textsc{OUTPUT:} $E1$
	\item \textsc{function GRAD} with \textsc{INPUT:} $\bmA=\bmA_{n+1}$\\
		\phantom{xxxxxxx}\textsc{OUTPUT:} $G1$
	\item \textsc{function DIS} with \textsc{INPUT:} $\bmA=\bmA_{n+1}$\\
		\phantom{xxxxxxx}\textsc{OUTPUT:} $D1$
	\item \textsc{function ERG} with \textsc{INPUT:} $\bmU_1=\bmU_{n}$, $\bmU_2=\bmU_{D,n}$, $\bmA=\bmA_{n}$\\
		\phantom{xxxxxxx}\textsc{OUTPUT:} $E2$
	\item \textsc{function GRAD} with \textsc{INPUT:} $\bmA=\bmA_{n}$\\
		\phantom{xxxxxxx}\textsc{OUTPUT:} $G2$
	\item \textsc{function DIS} with \textsc{INPUT:} $\bmA=\bmA_{n}$\\
		\phantom{xxxxxxx}\textsc{OUTPUT:} $D2$
	\item \textsc{COMPUTE:} $(E1+G1+D1)-(E2+G2+D2)$
\end{enumerate}

\noindent To compute Upper Bound $UB$ (refer to Eq. \eqref{Eq.UpBnd} and Eq. \eqref{Eq.Bounds}), use:
\begin{enumerate}
	\item \textsc{function ERG} with \textsc{INPUT:} $\bmU_1=\bmU_{n}$, $\bmU_2=\bmU_{D,n+1}$, $\bmA=\bmA_{n}$\\
		\phantom{xxxxxxx}\textsc{OUTPUT:} $E1$
	\item \textsc{function ERG} with \textsc{INPUT:} $\bmU_1=\bmU_{n}$, $\bmU_2=\bmU_{D,n}$, $\bmA=\bmA_{n}$\\
		\phantom{xxxxxxx}\textsc{OUTPUT:} $E2$
	\item \textsc{COMPUTE:} $UB=E1-E2$
\end{enumerate}
\noindent To compute Lower Bound $LB$ (refer to Eq. \eqref{Eq.LwBnd} and Eq. \eqref{Eq.Bounds}), use:
\begin{enumerate}
	\item \textsc{function ERG} with \textsc{INPUT:} $\bmU_1=\bmU_{n+1}$, $\bmU_2=\bmU_{D,n+1}$, $\bmA=\bmA_{n+1}$\\
		\phantom{xxxxxxx}\textsc{OUTPUT:} $E1$
	\item \textsc{function ERG} with \textsc{INPUT:} $\bmU_1=\bmU_{D,n+1}$, $\bmU_2=\bmU_{D,n}$, $\bmA=\bmA_{n+1}$\\
		\phantom{xxxxxxx}\textsc{OUTPUT:} $E2$
	\item \textsc{COMPUTE:} $LB=E1-E2$
\end{enumerate}
		}
	\end{minipage}
}


\normalem 

\begin{algorithm}[H]
\SetKwInOut{Input}{input}\SetKwInOut{Output}{output}
\DontPrintSemicolon
\KwData{$N,\,K,\,\eta,\,(\bmU_0,\,\bmA_0)$}
\KwResult{$(\bmU_{n},\,\bmA_{n})$ $n=1,\ldots,N$}
		\Set{	
\lnl{Alg.BT.Rem1} 	$n=0$\;
\lnl{Alg.BT.Rem2} 	$\bmA^{0}=\bmA_0$, $\bmU^{0}=\bmU_0$}
\lnl{Alg.BT.Rem3} \Repeat{$n = N$}{
\lnl{Alg.BT.Rem4}		\Solve{
\lnl{Alg.BT.Rem5}		\Input{$\bmA^{0}$, $\bmU^{0}$, $\bmA_{n}$}
\lnl{Alg.BT.Rem6}		\textbf{Algorithm \ref{Alg.NewAMM}}: $(\bmU_{n+1},\,\bmA_{n+1})=\textsc{argmin }\mathcal{F}(t_{n+1},\bmU,\bmA;\,\bmA_n)$ \;
\lnl{Alg.BT.Rem7}		\Output{$\bmU_{n+1}$, $\bmA_{n+1}$}
			}%
		\Set{
\lnl{Alg.BT.Rem8}	     $\bmA^{0}=\bmA_{n+1}$, $\bmU^{0}=\bmU_{n+1}$}
\lnl{Alg.BT.Rem9} \eIf{inequality \eqref{Disc.TwoSidIneq} is met}{
\lnl{Alg.BT.Rem10} $n\leftarrow n+1$ (proceed to the next step)
						 }
			{
\lnl{Alg.BT.Rem11}        $b=0$ (back steps counter)\;
\lnl{Alg.BT.Rem12}	\Repeat{inequality \eqref{Disc.TwoSidIneq} is met or $b=K$}{
\lnl{Alg.BT.Rem13}		$n\leftarrow n-1$ (go back by one step)\; 
\lnl{Alg.BT.Rem14}		$b=b+1$ \;
\lnl{Alg.BT.Rem15}		\Solve{
\lnl{Alg.BT.Rem16}		\Input{$\bmA^{0}$, $\bmU^{0}$, $\bmA_{n}$}
\lnl{Alg.BT.Rem17}		\textbf{Algorithm \ref{Alg.NewAMM}}: $(\bmU_{n+1},\,\bmA_{n+1})=\textsc{argmin }\mathcal{F}(t_{n+1},\bmU,\bmA;\,\bmA_n)$ \;
\lnl{Alg.BT.Rem18}		\Output{$\bmU_{n+1}$, $\bmA_{n+1}$}
					}%
				\Set{
\lnl{Alg.BT.Rem19}		$\bmA^{0}=\bmA_{n+1}$, $\bmU^{0}=\bmU_{n+1}$}
				}
\lnl{Alg.BT.Rem20}	$n\leftarrow n+1$  (proceed to the next step)
			}
}
\caption{\label{Alg.BT} Backtracking Algorithm. }
\end{algorithm}


\section{Numerical examples}\label{Sec.NumEx}


In this section, we present representative numerical experiments to illustrate the perfomance of the energetic 
formulation and of the numerical procedure to obtain energetic solutions. We compare these solutions, which 
we will refer to as approximated energetic solutions, with those obtained by the standard procedure of 
simply solving the weak form of the Euler--Lagrange equations  \cite{MWH10,KWBNR20,URSS19,MHW10,AGDl15}. 
The problems that we consider are:
\begin{itemize}
	\item[$(i)$] Single edge notched tension test;
	\item[$(ii)$] Single edge notched shear test;
	\item[$(iii)$] $3d$ $L-$shaped panel test;
	\item[$(iv)$] $3d$ Symmetric bending test.
\end{itemize}
The first two are classical $2d$ benchmark problems where the specimens are assumed in plane strain conditions, whereas the 
last two are $3d$ bending tests of a concrete panel and a cement paste beam which we compare with experimental results.
All the numerical simulations are carried out with $\psi_0^{+}$ and $\psi_0^{-}$ given by

\[
	\psi_0^{\pm}(\varepsilonb)=\frac{\lambda}{2}(\tr\varepsilonb^{\pm})^2 + \mu\varepsilonb^{\pm}\colon\varepsilonb^{\pm}
\]
with $\lambda$ and $\mu$ the Lam\`{e} constants, whereas $\varepsilonb^{+}$ and 
$\varepsilonb^{-}$ are obtained by the spectral decomposition of $\varepsilonb$
introduced in \cite{MHW10,MWH10} as

\[
	\varepsilonb^+ =\sum_{a=1}^3\langle \varepsilon_a\rangle_+\bmn_a\otimes\bmn_a\quad\quad\quad
	\varepsilonb^- =\sum_{a=1}^3\langle \varepsilon_a\rangle_-\bmn_a\otimes\bmn_a\,,
\]
where $\varepsilon_a$, $\bmn_a$, $a=1,2,3 $, are  the principal strains and the principal strain directions of 
$\varepsilonb$, respectively, and for $x\in\mathbb{R}$, $\langle x\rangle_+=(x+\abs{x})/2$ and 
$\langle x\rangle_-=(x-\abs{x})/2$.
We also use $g(\beta)=(1-\beta)^2$ as degradation function according to what noted in Remark \ref{Rem:IncMinProAT}$(iii)$
and, for the residual stiffness, we set $k=10^{-4}$ (see \cite[Sec. 5]{AMM09} and discussion therein).
We apply monotonic displacement control by comparing the response for different displacement increments $\Delta w$. 
As for the value of the penalization factor $\epsilon$ to enforce crack irreversibility, 
this depends on the problem at hand. For its selection, we tested different values of $\epsilon$ and chose the one that 
was ensuring that the dissipation $\mathcal{D}(\beta_{n},\,\beta_{n+1})$, $n=0,\,\ldots,\,N-1$,
was non-negative and the problem conditioning was not jeopardized. The values of the tolerances 
$tol_{\bmU}$ and $tol_{\bmA}$ that control the convergence of Algorithm \ref{Alg.NewAMM} 
and the two-sided energy inequality tolerance $\eta$ have been all
set equal, in the respective units, to $10^{-5}$. 
We run our examples with two different values of $K$: $K=0$ to obtain discrete solutions
by the standard procedure such as in \cite{MHW10,DeG20} where one does not control the 
energy estimates \eqref{Disc.TwoSidIneq} and $K>0$ to obtain discrete solutions with 
the backtracking algorithm. In this latter case, the value of $K$ is chosen 
so that we can always go back as many time steps as necessary to obtain discrete solutions meeting \eqref{Disc.TwoSidIneq}
along the whole evolution. For the present simulations, for instance, by taking $K=50$ we could observe 
that we were never going back more than $10 \div 30$ steps. 
This will be easily verified for each of the numerical simulations by inspecting 
the plot of the terms that enter \eqref{Disc.TwoSidIneq} and noting that the estimates are met by the 
discrete solutions when the backtracking is active.


\subsection{Single edge notched tension test}\label{Sec.NumEx:Ex1}

The single edge notched tension (SENT) test is a classical benchmark problem which is used for a wide 
range of applications \cite{ASTM20} and is well studied also in the numerical literature 
\cite{MWH10,KWBNR20,URSS19,MHW10,AGDl15}. It consists of a square specimen with 
a single horizontal notch located at mid-height of the left edge with length equal to half the edge length, and is subject to constant tension
on the top edge.  In this paper, we consider the same mechanical model analysed in \cite{MWH10}. 
The geometric properties and boundary conditions of the specimen are shown in 
Figure \ref{NE1}$(a)$, with $\bmu=\bm{0}$ 
at the point of coordinates $(0,\,0)$ and $u_z=0$
on the bottom edge; $u_y=0$ and non--homogeneous Dirichlet condition  
$u_z=w$ on the top edge whereas all the other parts of the boundary including the slit are traction free. 
The elastic constants are chosen as $\lambda=121.1538\,\mathrm{kN/mm^2}$ and $\mu=80.7692\,\mathrm{kN/mm^2}$, 
the critical energy release rate as $g_c =2.7\, \,\mathrm{N/mm}$ and the internal length as $\ell =0.0175\,\mathrm{mm}$.
Figure \ref{NE1}$(b)$ displays the unstructured finite element mesh used for the simulations. 
We use linear finite elements for the approximation of the displacement field
and of the phase-field.
The mesh is thus formed by $6062$ triangular elements with $3088$ nodes. 
In order to capture properly the crack pattern, since under constant tension 
the crack propagates straight, we refine the mesh in this zone with an effective element size 
$h\approx 0.005\,\mathrm{mm}<\ell/2$ and for a bandwidth of about $2\,\mathrm{mm}$.

\begin{figure}[H]
	\centering{$\begin{array}{cc}
		\includegraphics[width=0.3\textwidth]{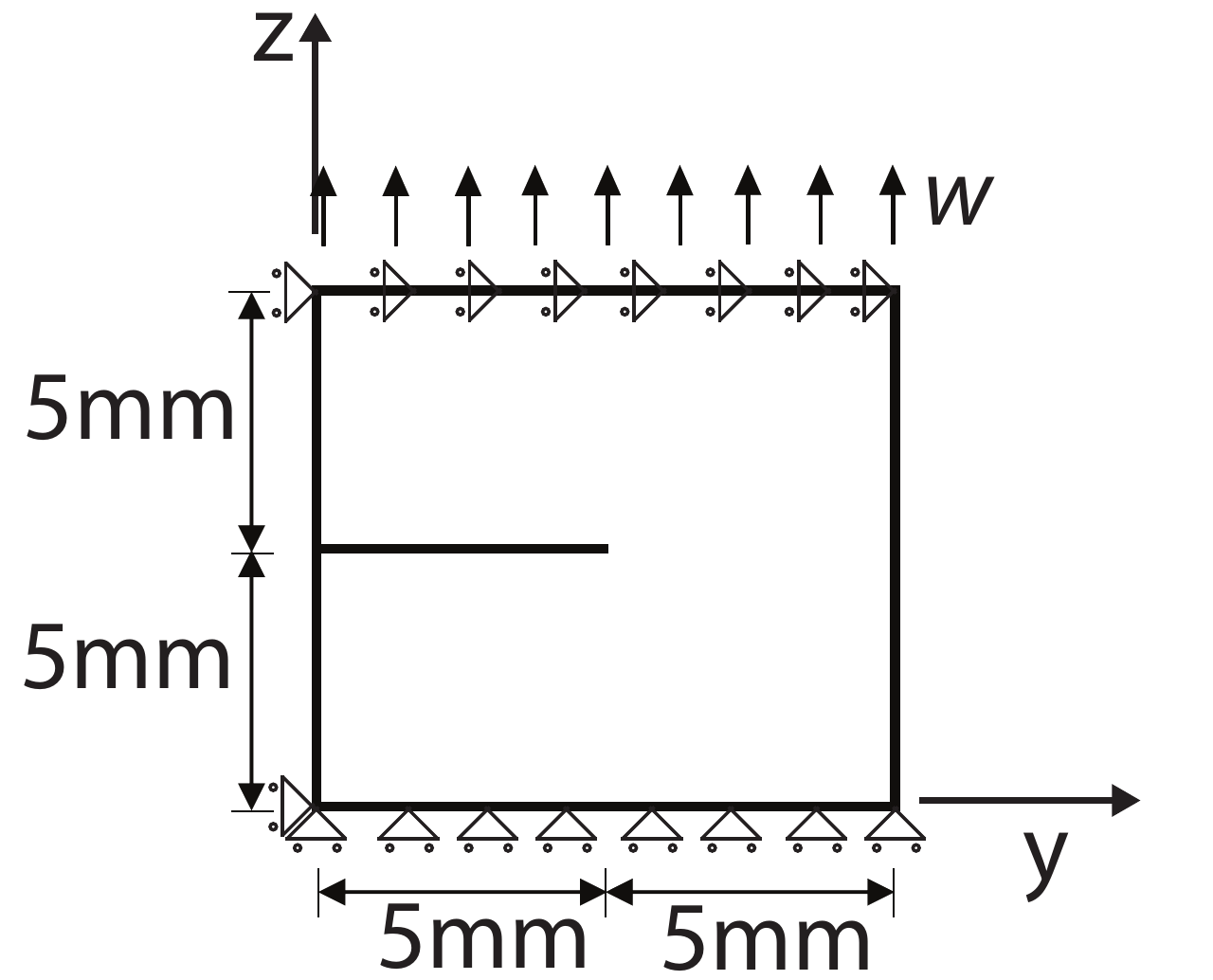}&
		\includegraphics[width=0.25\textwidth]{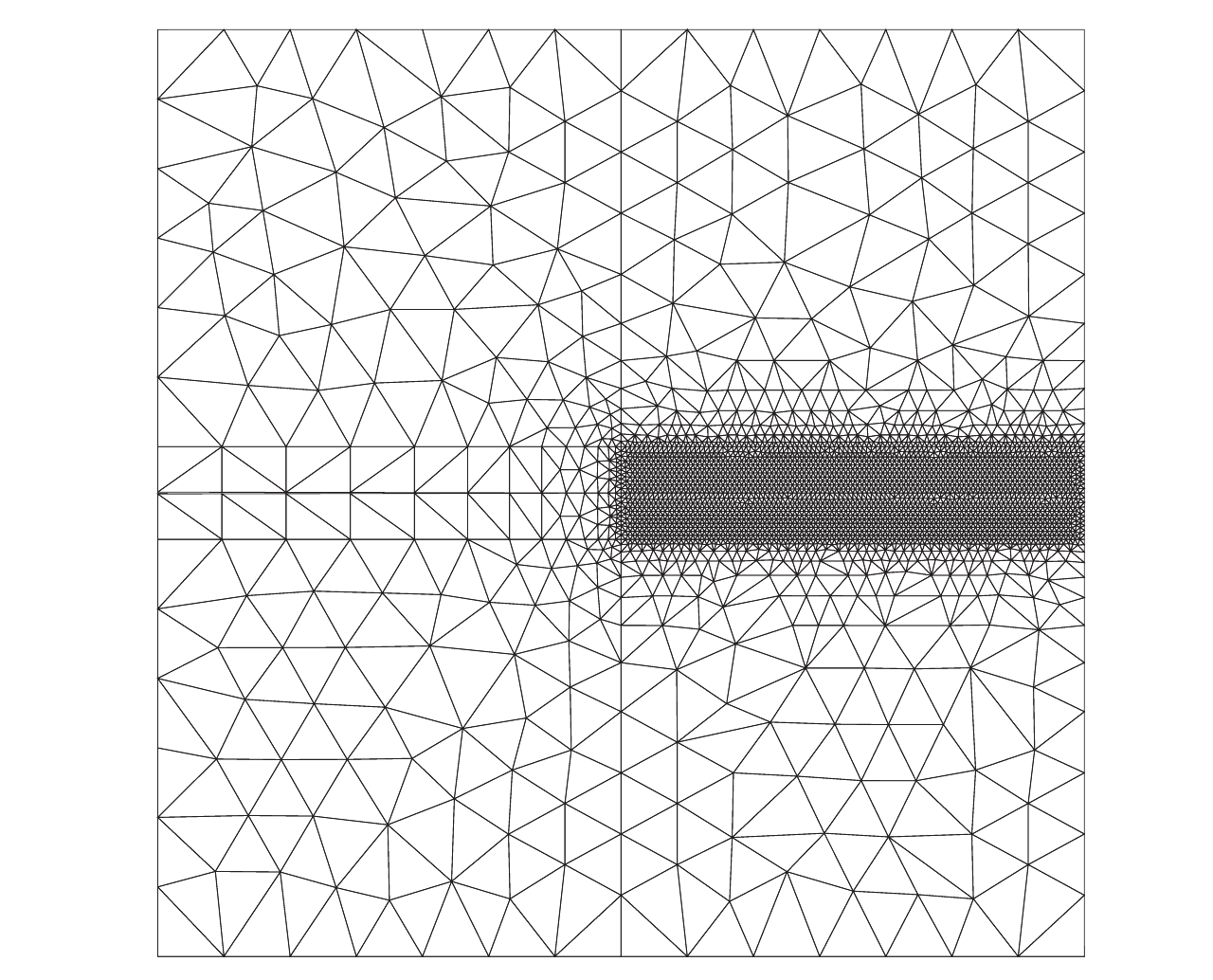}\\
		(a)&(b)
		\end{array}$
		}
	\caption{\label{NE1}
		Example \ref{Sec.NumEx:Ex1}. Single edge notched tension test. 
		$(a)$ Specimen geometry and boundary conditions. 
		$(b)$ Unstructured finite element mesh.
		}
\end{figure} 

We analyse the behaviour of the model for a monotone applied 
displacement $w$ resulting from the application of the following displacement increments: $\Delta~w~=~10^{-4}~\,~\mathrm{mm}$, 
$\Delta~w~=~10^{-5}~\,~\mathrm{mm}$ and $\Delta~w~=~10^{-6}~\,~\mathrm{mm}$. 
We take $\epsilon=10^{-6}$ and we
evaluate then the reaction force $F_z$ on the top edge $\Gamma_{top}\subseteq \partial\Omega$ given by

\[
	F_z=\int_{\Gamma_{top}}\,\sigmab\bmn\cdot\bmn\,\ds
\]
where $\bmn$ is the outward normal to this part of the boundary, and the energetic terms 
$\mathcal{E}(t_{n+1},\bmU_{n+1},\bmA_{n+1})$ and $\mathcal{D}(\bmA_{n},\bmA_{n+1})$, 
$n=0,1,\ldots, N-1$ with $N$ the total number of increments $\Delta w$. 

\begin{figure}[H]
	\centerline{\includegraphics[width=0.8\textwidth]{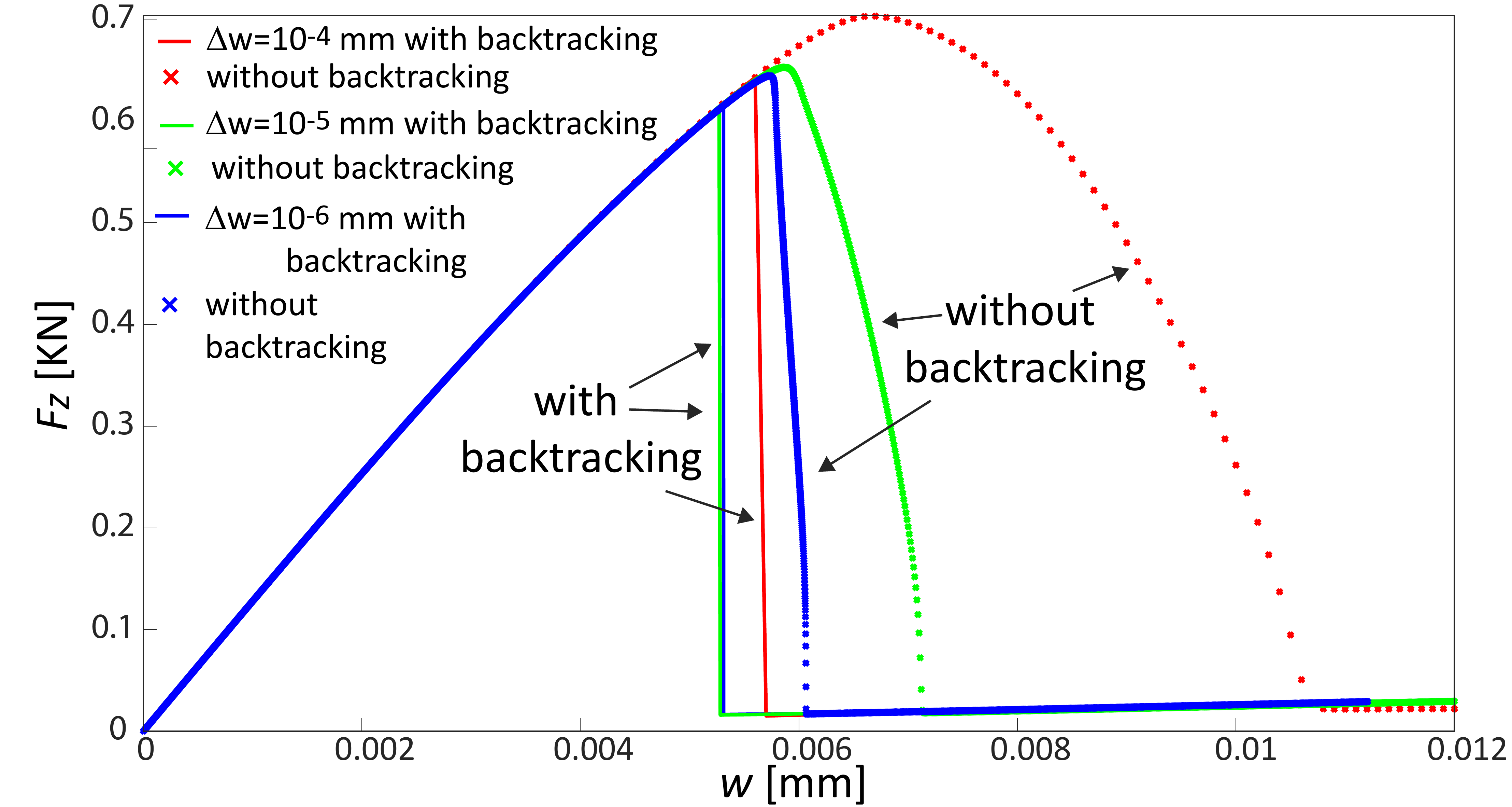}
		}
	\caption{\label{NE2} Example \ref{Sec.NumEx:Ex1}. Single edge notched tension test.
	Load--displacement curves for different displacement increments $\Delta w$ 
	and different schemes, using the backtracking algorithms describing the evolution 
	of the approximate energetic solutions and without applying the backtracking algorithm.
	}
\end{figure}

The variation of $F_z$ with $w$ for the different displacement increments and the different 
algorithms are displayed in Figure \ref{NE2}. By applying the backtracking algorithm (Algorithm \ref{Alg.BT} with $K>0$) 
the load-displacement curves display a similar response independent of the displacement increment $\Delta w$. This is in contrast 
with the behaviour associated with the solutions of Algorithm \ref{Alg.BT} with $K=0$, where the backtracking option is not active. In this case,
for the range of values used for $\Delta w$, the behaviour is sensitive with respect to $\Delta w$, though for small values of $\Delta w$
the response converges towards a definite configuration. Both numerical strategies identify 
a strong decreasing structural response, but the one 
described by the backtracking strategy occurs prior to that corresponding to the standard solution. Furthermore, for both type of solutions,
the load-displacement curve displays a residual force $F_z$ of the fully damaged specimen which is related to the value of $\delta$,
that defines the `residual' energy after complete damage. 

\begin{figure}[H]
	\centering{$\begin{array}{c}
		\includegraphics[width=0.60\textwidth]{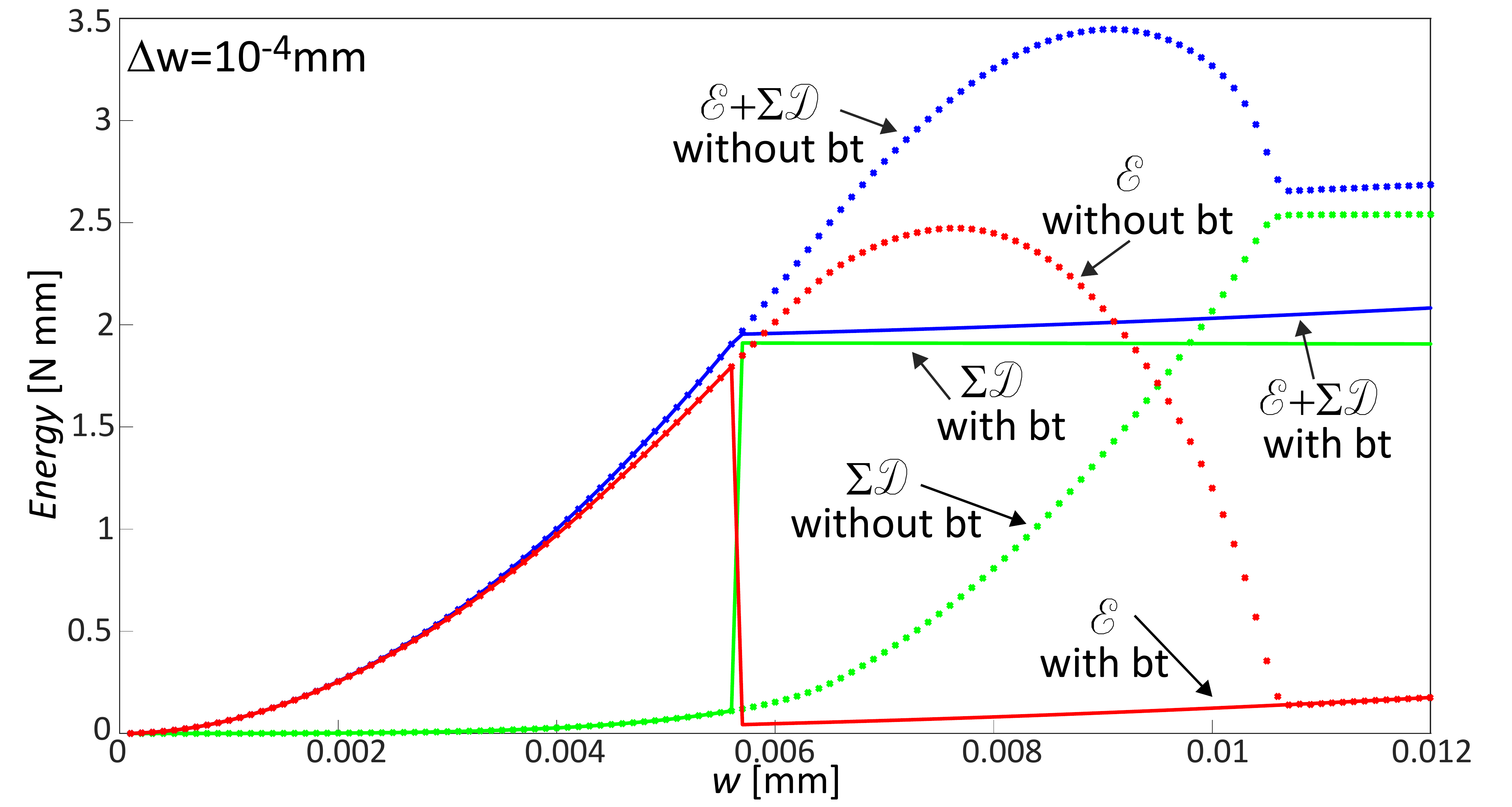}\\
		(a)\\
		\begin{array}{cc}
			\includegraphics[width=0.49\textwidth]{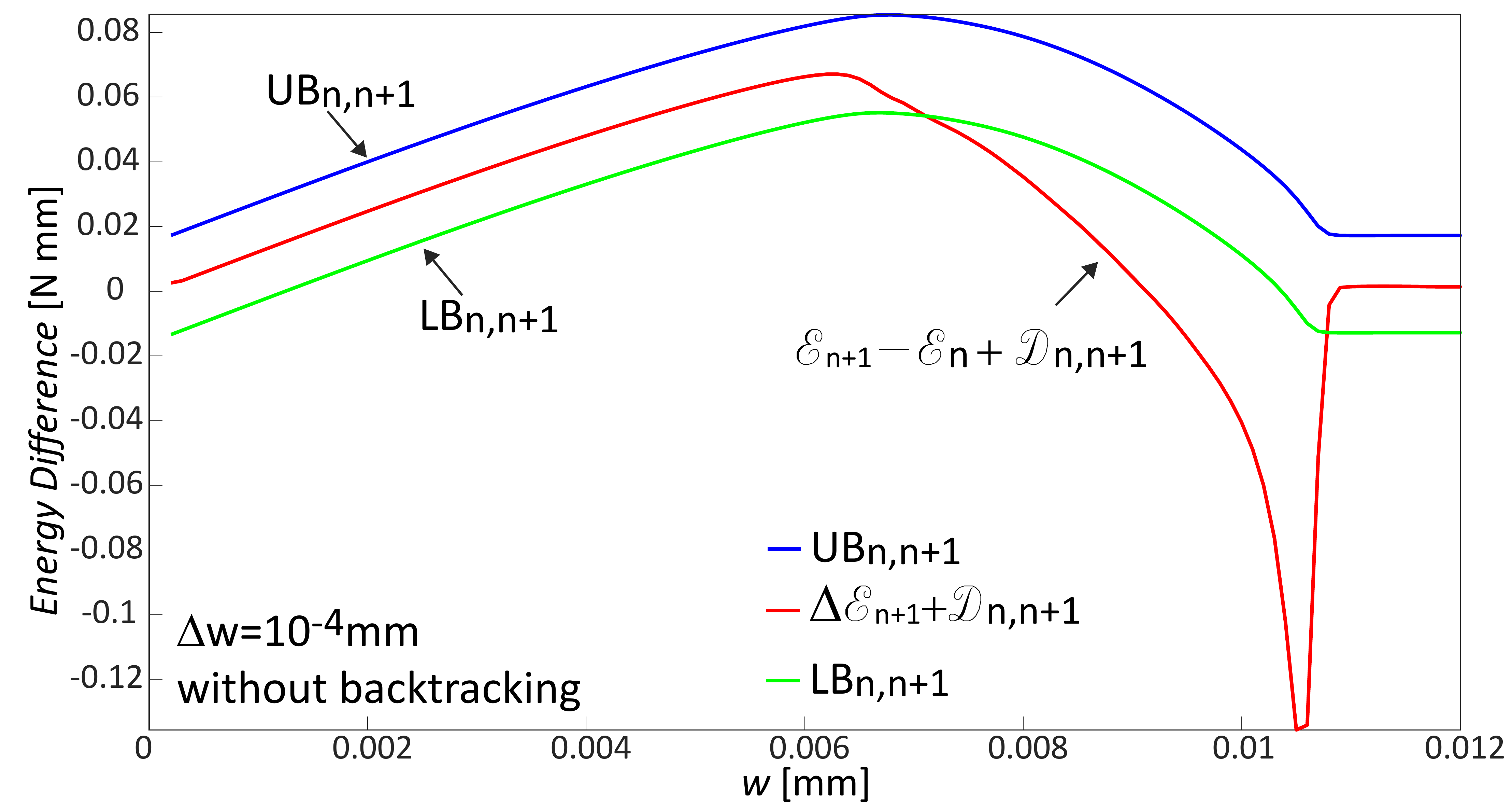}& 
			\includegraphics[width=0.49\textwidth]{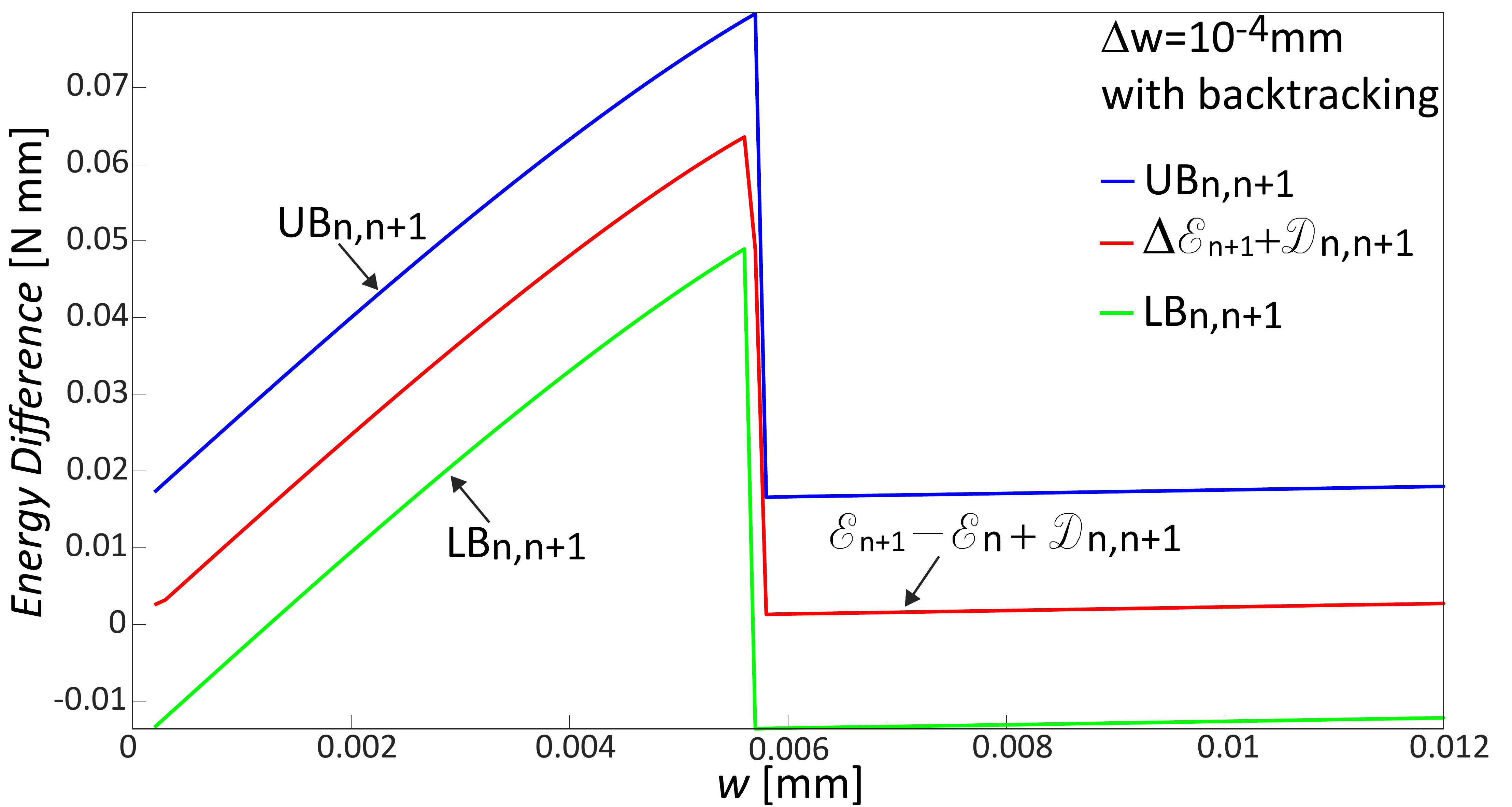}\\
			(b)&(c)
		\end{array}
		\end{array}$}
	\caption{\label{Ex:TensionTest:Enrg:4}
	Example \ref{Sec.NumEx:Ex1}. Single edge notched tension test.
	Results for $\Delta w=10^{-4}\,\mathrm{mm}$.
	$(a)$ Evolution of the total energy $\mathcal{E}(t_{n+1},\bmU_{n+1},\bmA_{n+1})+\sum_{i=0}^n\mathcal{D}(\bmA_i,\bmA_{i+1})$ 
	for $n=0,1,\ldots, N-1$, , without  backtracking $(K=0)$ and with backtracking.
	Evolution of the total incremental energy $\mathcal{E}_{n+1}-\mathcal{E}_{n}+\mathcal{D}_{n,n+1}$, 
	the lower bound $LB_{n,n+1}$ and the upper bound $UB_{n,n+1}$ which enter the two-sided energy estimate \eqref{Disc.TwoSidIneq},
	$n=0,1,\ldots, N-1$, for the scheme $(b)$ without backtracking and $(c)$ with backtracking.}
\end{figure}

\begin{figure}[H]
	\centering{$\begin{array}{c}
		\includegraphics[width=0.60\textwidth]{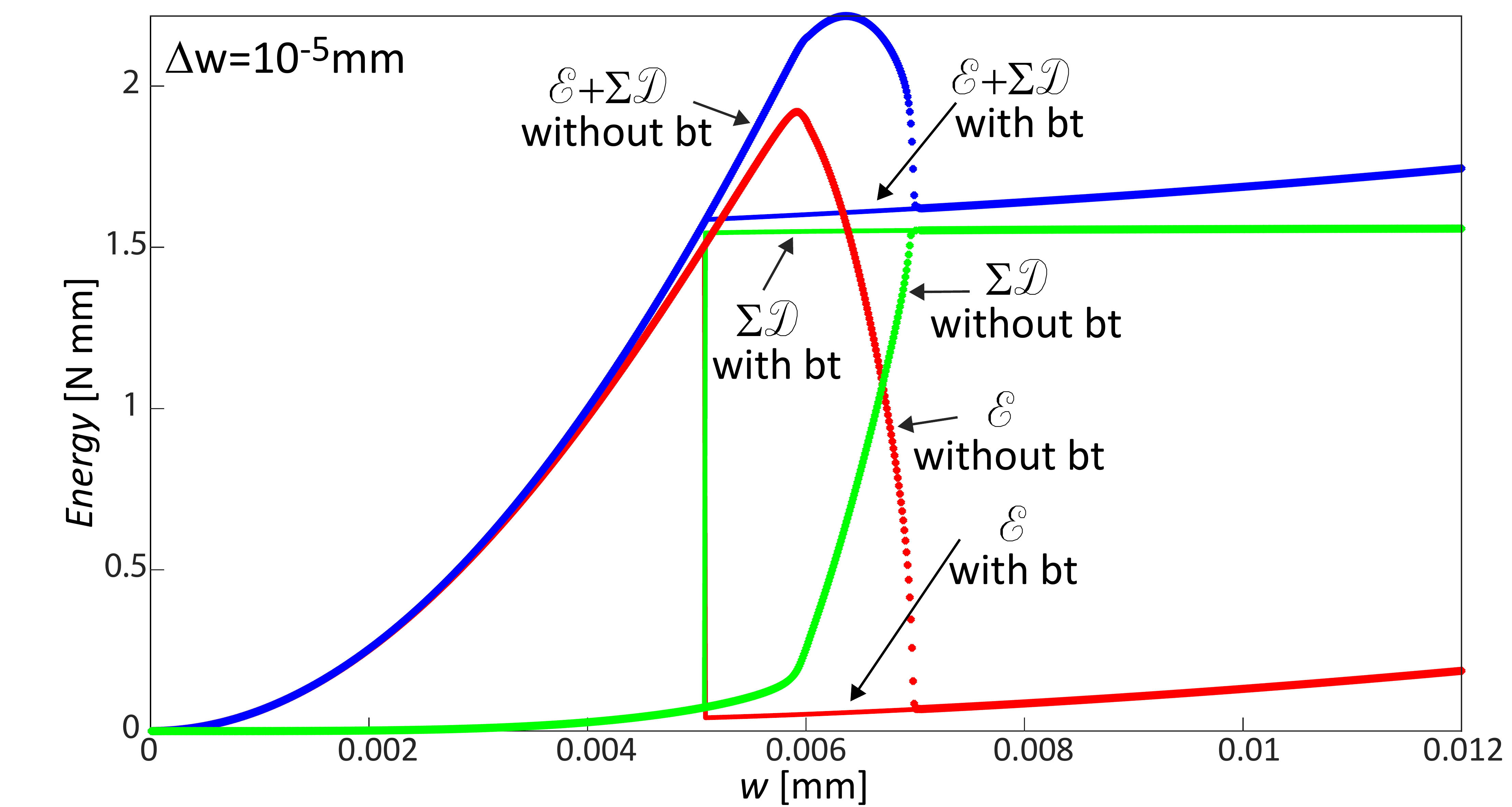}\\
		(a)\\
		\begin{array}{cc}
			\includegraphics[width=0.49\textwidth]{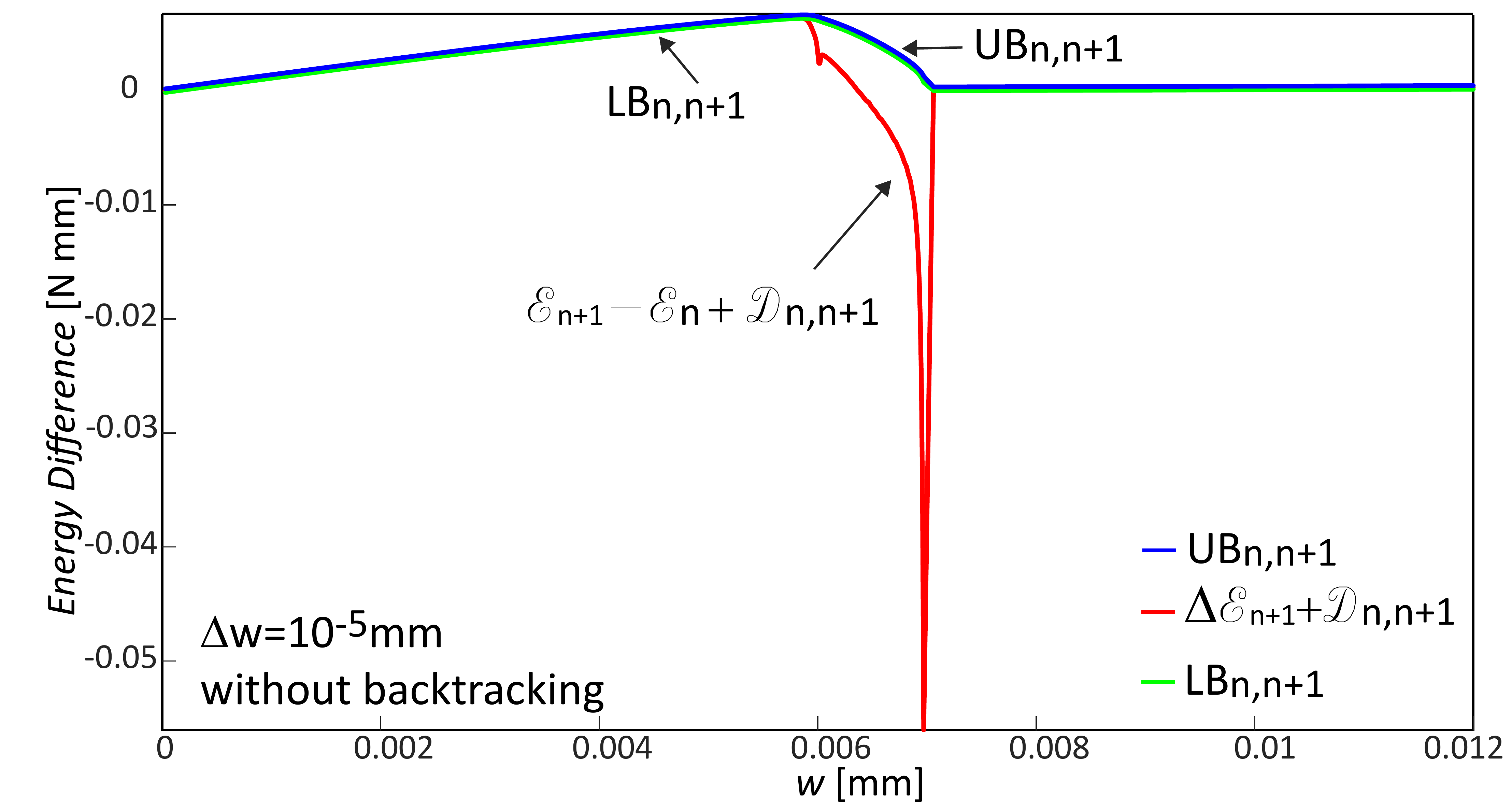}& 
			\includegraphics[width=0.49\textwidth]{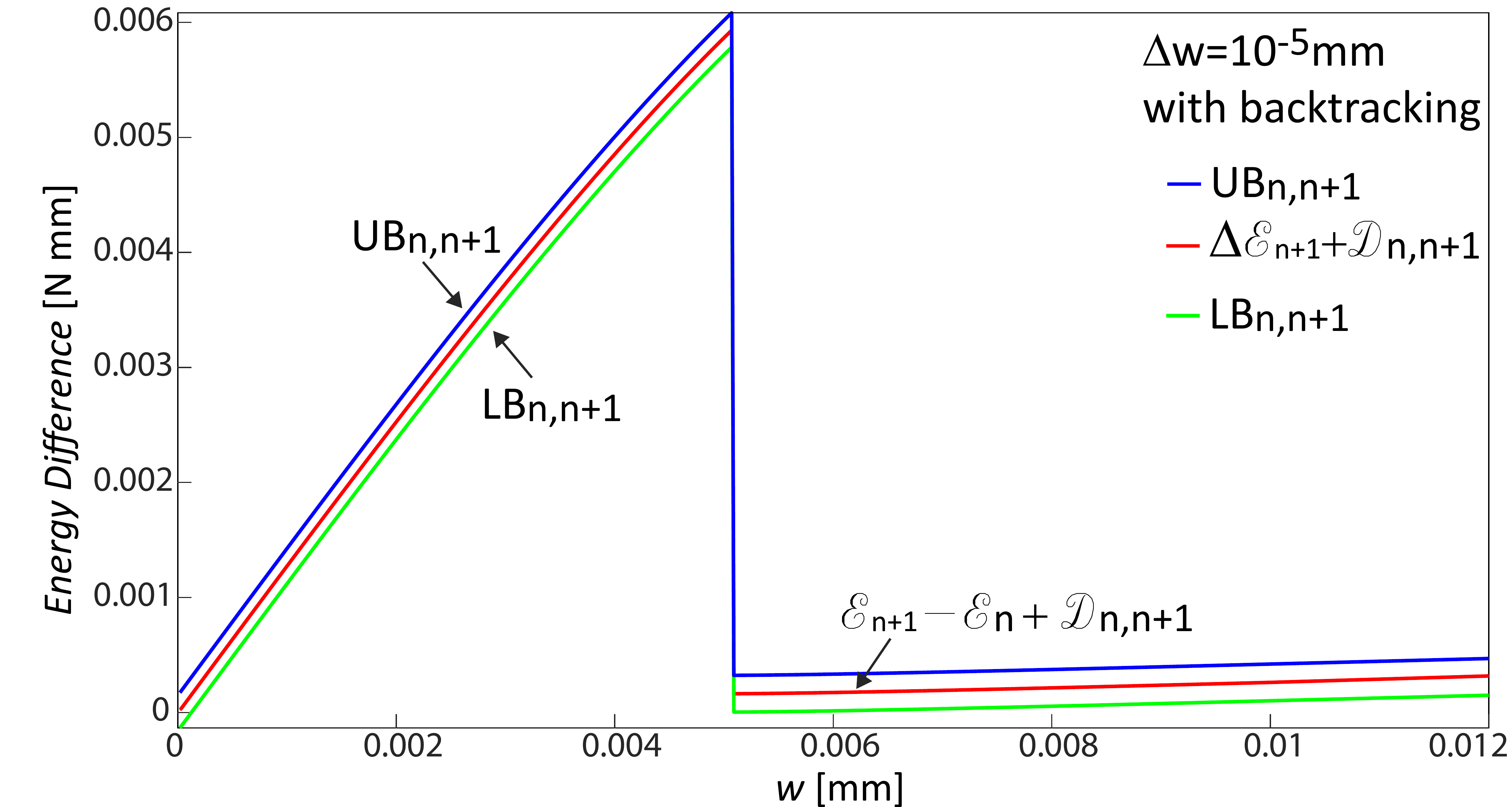}\\
			(b)&(c)
		\end{array}
		\end{array}$}
	\caption{\label{Ex:TensionTest:Enrg:5}
	Example \ref{Sec.NumEx:Ex1}. Single edge notched tension test.
	Results for $\Delta w=10^{-5}\,\mathrm{mm}$.
	$(a)$ Evolution of the total energy $\mathcal{E}(t_{n+1},\bmU_{n+1},\bmA_{n+1})+\sum_{i=0}^n\mathcal{D}(\bmA_i,\bmA_{i+1})$ 
	for $n=0,1,\ldots, N-1$, without  backtracking $(K=0)$ and with  backtracking.
	Evolution of the total incremental energy $\mathcal{E}_{n+1}-\mathcal{E}_{n}+\mathcal{D}_{n,n+1}$, 
	the lower bound $LB_{n,n+1}$ and the upper bound $UB_{n,n+1}$ which enter the two-sided energy estimate \eqref{Disc.TwoSidIneq},
	$n=0,1,\ldots, N-1$, for the scheme $(b)$ without backtracking and $(c)$ with backtracking.}
\end{figure}

\begin{figure}[H]
	\centering{$\begin{array}{c}
		\includegraphics[width=0.60\textwidth]{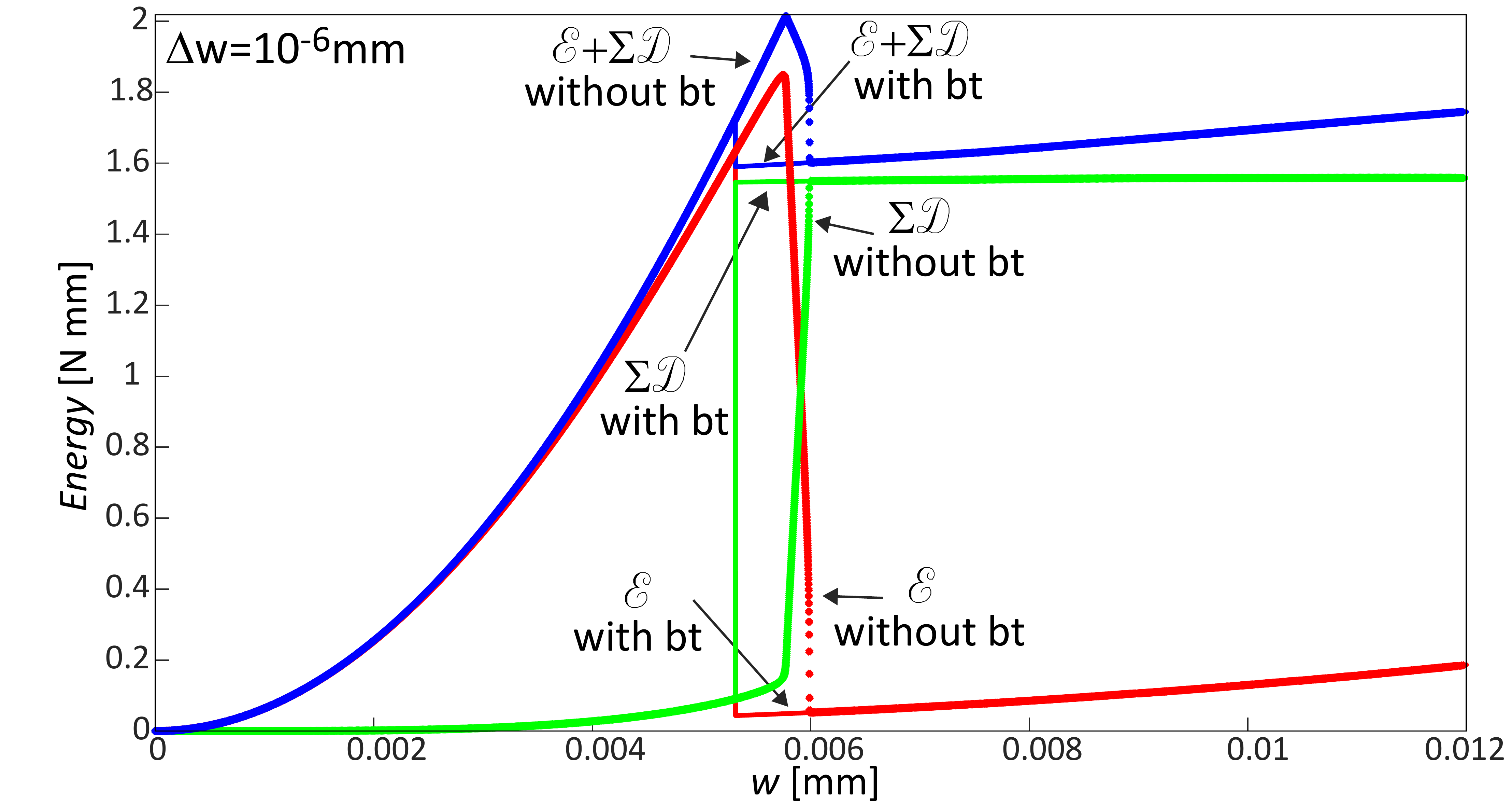}\\
		(a)\\
		\begin{array}{cc}
			\includegraphics[width=0.49\textwidth]{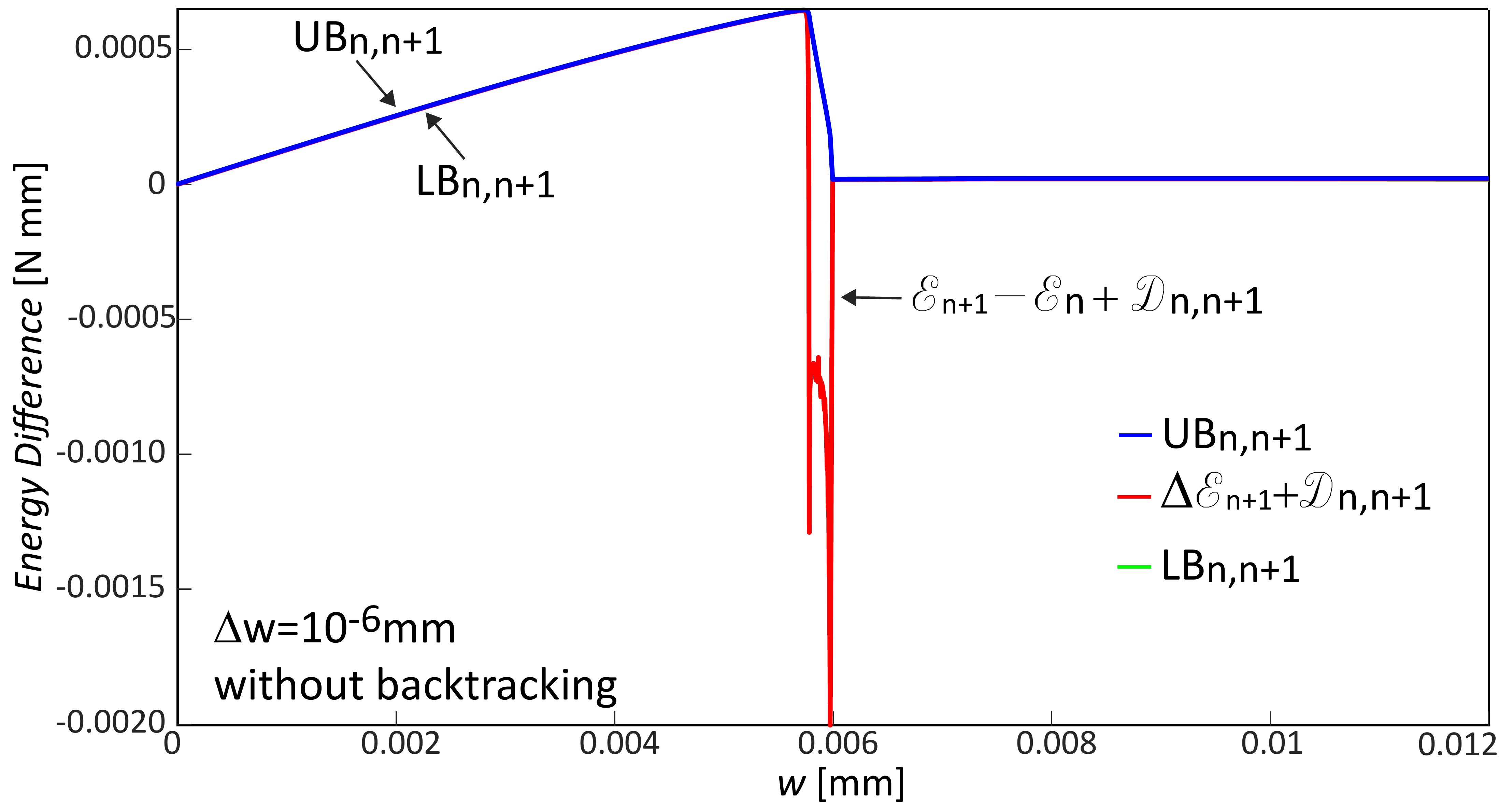}& 
			\includegraphics[width=0.49\textwidth]{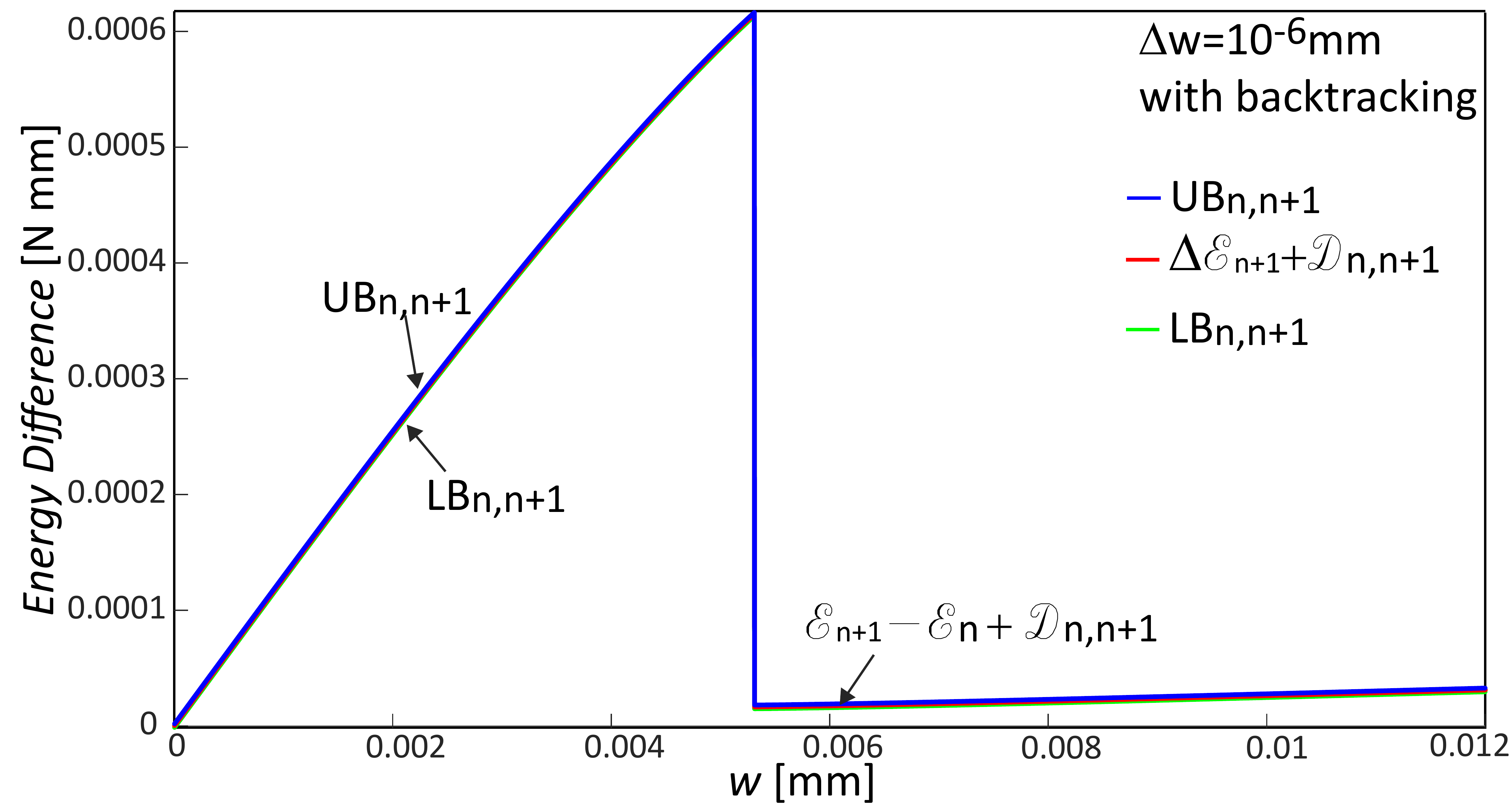}\\
			(b)&(c)
		\end{array}
		\end{array}$}
	\caption{\label{Ex:TensionTest:Enrg:6}
	Example \ref{Sec.NumEx:Ex1}. Single edge notched tension test.
	Results for $\Delta w=10^{-6}\,\mathrm{mm}$.
	$(a)$ Evolution of the total energy $\mathcal{E}(t_{n+1},\bmU_{n+1},\bmA_{n+1})+\sum_{i=0}^n\mathcal{D}(\bmA_i,\bmA_{i+1})$ 
	for $n=0,1,\ldots, N-1$, without  backtracking $(K=0)$ and with  backtracking.
	Evolution of the total incremental energy $\mathcal{E}_{n+1}-\mathcal{E}_{n}+\mathcal{D}_{n,n+1}$, 
	the lower bound $LB_{n,n+1}$ and the upper bound $UB_{n,n+1}$ which enter the two-sided energy estimate \eqref{Disc.TwoSidIneq},
	$n=0,1,\ldots, N-1$, for the scheme $(b)$ without backtracking and $(c)$ with backtracking.}
\end{figure}

The energy variation associated with the solutions computed
without the backtracking $(K=0)$ and with the backtracking option active $(K>0)$
are depicted in Figure \ref{Ex:TensionTest:Enrg:4} to Figure \ref{Ex:TensionTest:Enrg:6}
for the three displacement driven conditions, 
$\Delta w=10^{-4}\,\mathrm{mm}$, $\Delta w=10^{-5}\,\mathrm{mm}$ and $\Delta w=10^{-6}\,\mathrm{mm}$, respectively.
Figure \ref{Ex:TensionTest:Enrg:4}$(a)$,  
Figure \ref{Ex:TensionTest:Enrg:5}$(a)$ and Figure \ref{Ex:TensionTest:Enrg:6}$(a)$  
show the evolution of the total energy $\mathcal{E}(t_{n+1},\bmU_{n+1},\bmA_{n+1})+\sum_{i=0}^n\mathcal{D}(\bmA_i,\bmA_{i+1})$,
the current free energy $\mathcal{E}_{n+1}$ and the total dissipation 
$\sum_{i=0}^n\mathcal{D}_{i,i+1}$, for $n=0,1,\ldots, N-1$.
The energy paths obtained without backtracking display a bubble shape 
when damage starts to propagate due to the gradual
substantial reduction of the free energy because of the reduction of the elastic energy and of the increase of the 
dissipation energy. Such bubble is not present using the backtracking given that in this case 
the damage evolution is faster. For both the schemes, 
when the crack completes its propagation along half specimen, 
the total energy increases very little, due to the regularization parameter $\delta$, and is almost
equal to the accumulated dissipated energy.

Figure \ref{Ex:TensionTest:Enrg:4}$(b)$,  
Figure \ref{Ex:TensionTest:Enrg:5}$(b)$ and Figure \ref{Ex:TensionTest:Enrg:6}$(b)$ display
the total incremental energy  
$\mathcal{E}_{n+1}-\mathcal{E}_{n}+\mathcal{D}_{n,n+1}$, the upper bound $UB_{n,n+1}$ and 
the lower bound $LB_{n,n+1}$, for $n=0,\,\ldots, N-1$ associated with the solutions computed without activating 
the backtracking scheme $(K=0)$, whereas 
Figure \ref{Ex:TensionTest:Enrg:4}$(c)$,  
Figure \ref{Ex:TensionTest:Enrg:5}$(c)$ and Figure \ref{Ex:TensionTest:Enrg:6}$(c)$
contain the same type of plots relative to the approximate energetic solutions.
When damage starts to develop, the alternate minimization Algorithm \ref{Alg.NewAMM} 
fails to provide an appropriate 
energetic solution to the problem. As a result, the sequence of discrete solutions
evolves along a path of local minima, whose energy deviates substantially 
from the one associated with global minimization. 
The energetic bounds \eqref{Disc.TwoSidIneq}, and 
more specifically the lower bound, are then violated by the computed solutions. 
The two-sided energy inequality \eqref{Disc.TwoSidIneq}
is met only during the initial stage when the
specimen remains elastic and in the last stage of the cracking process showing that the algorithm
jumps back into a state of significantly lower energy. By contrast, 
with the backtracking option active, we avoid the wrong forward path obtained by the standard scheme,
for we restart with `better' local minima and we are able to obtain a
final path of the energy difference which lies between the two bounds.

\begin{figure}[H]
	\centering{$\begin{array}{cccc}
	\text{$(a)$ Without backtracking}&&&\\
		 \includegraphics[width=0.25\textwidth]{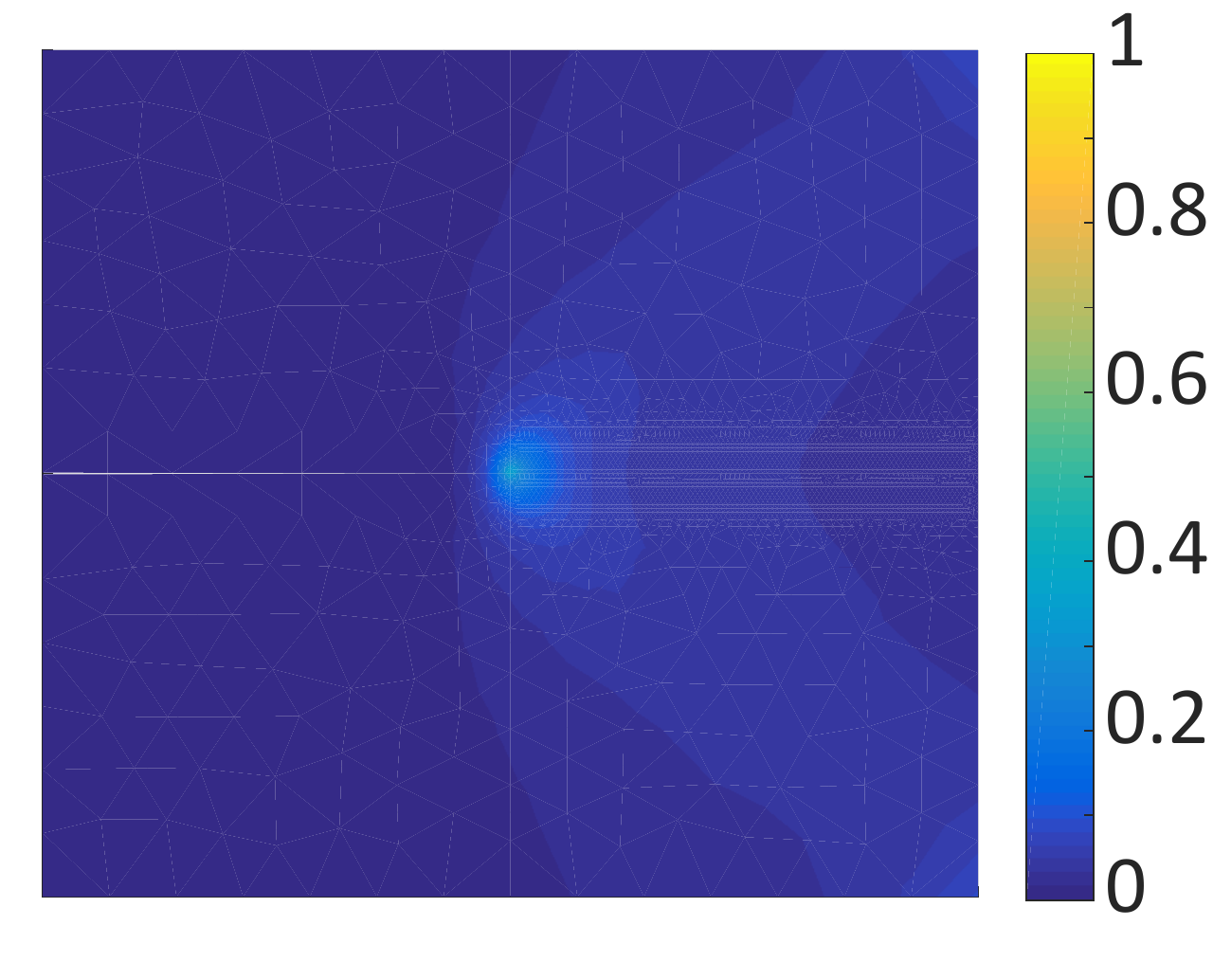}&\includegraphics[width=0.25\textwidth]{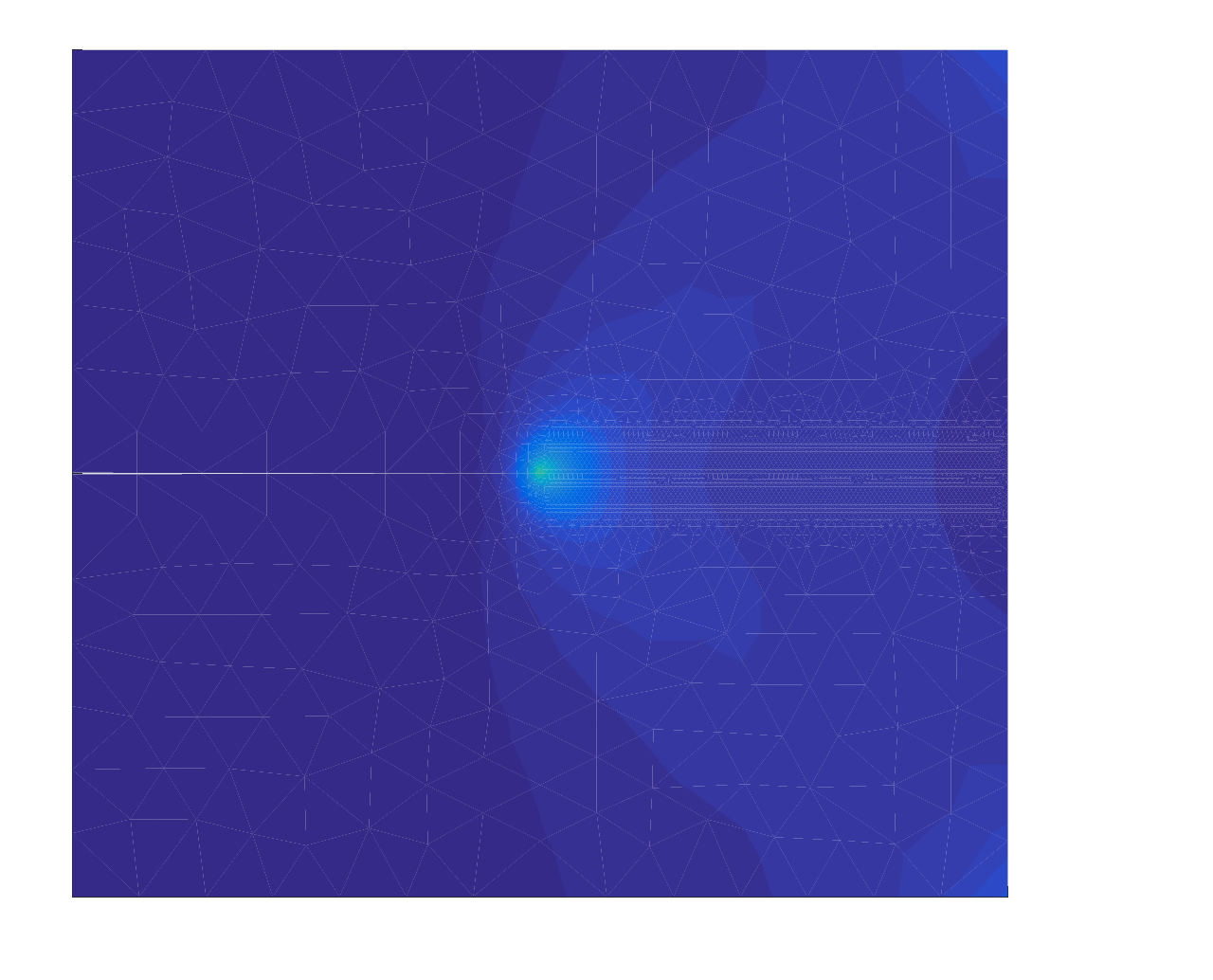}&		  
		 \includegraphics[width=0.25\textwidth]{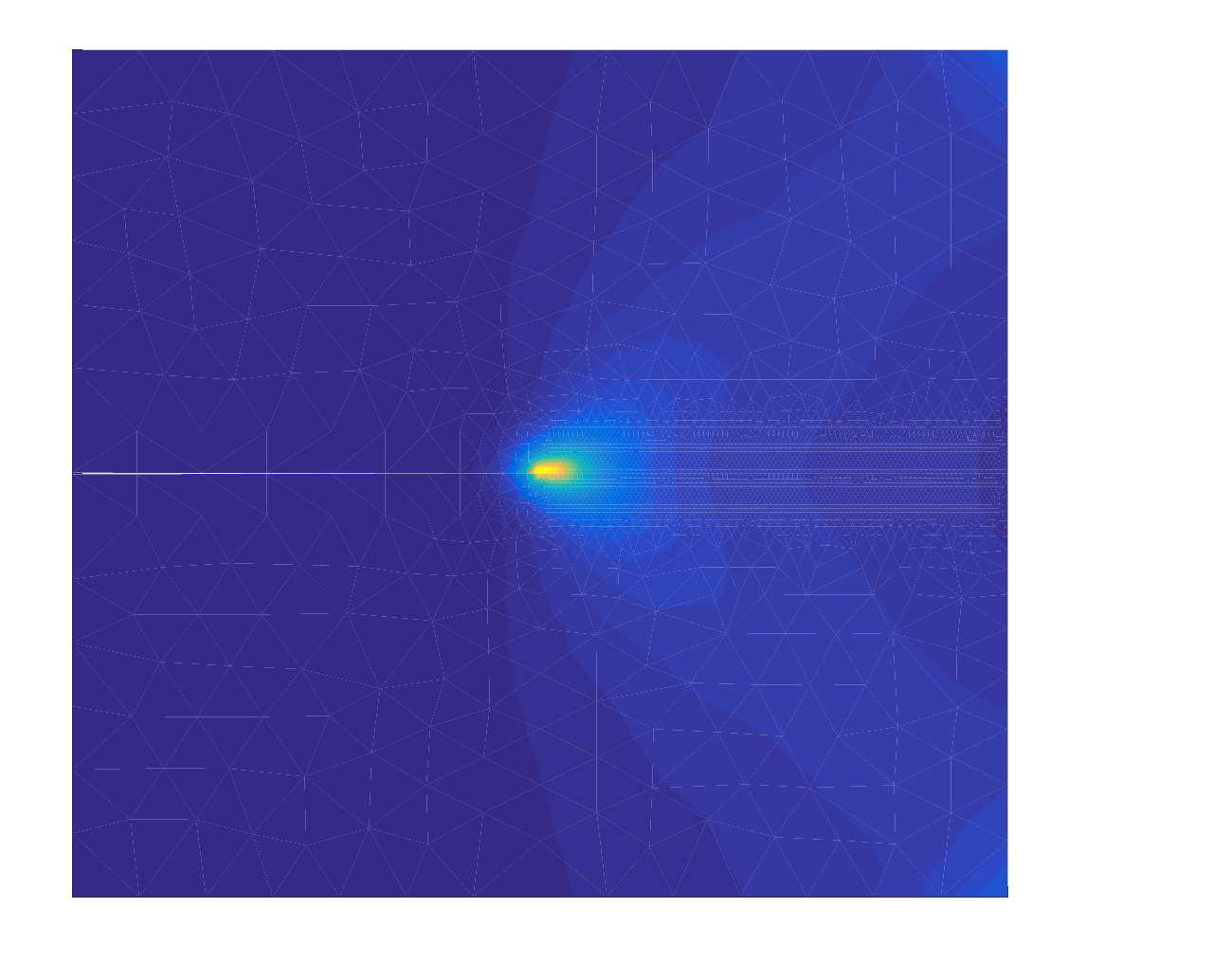}&\includegraphics[width=0.25\textwidth]{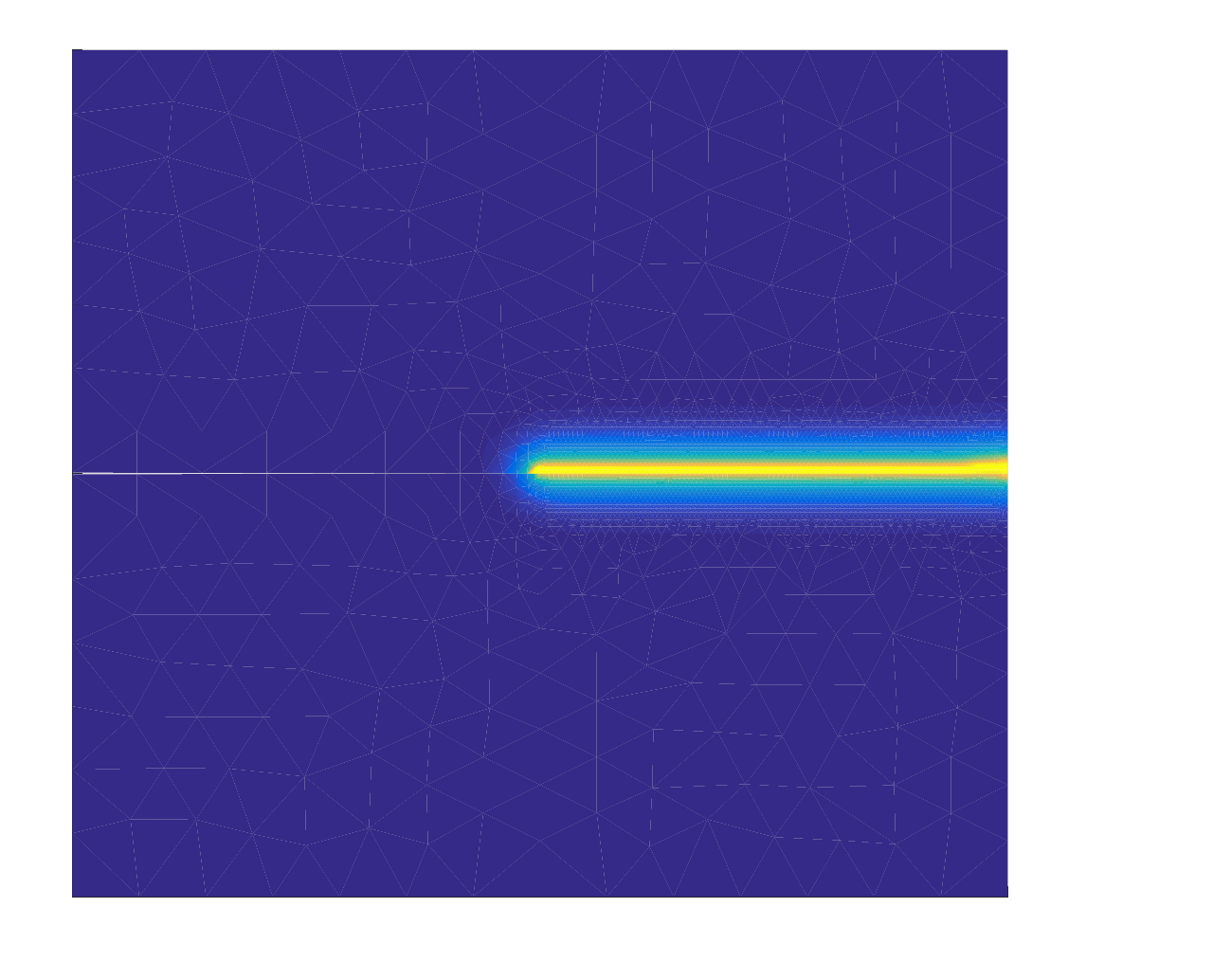}\\
		 w=0.005\,\mathrm{mm} & w=0.0055\,\mathrm{mm}  & w=0.006\,\mathrm{mm}  & w=0.007\,\mathrm{mm}\\[1.5ex]
	\text{$(b)$ With backtracking}&&&\\
		 \includegraphics[width=0.25\textwidth]{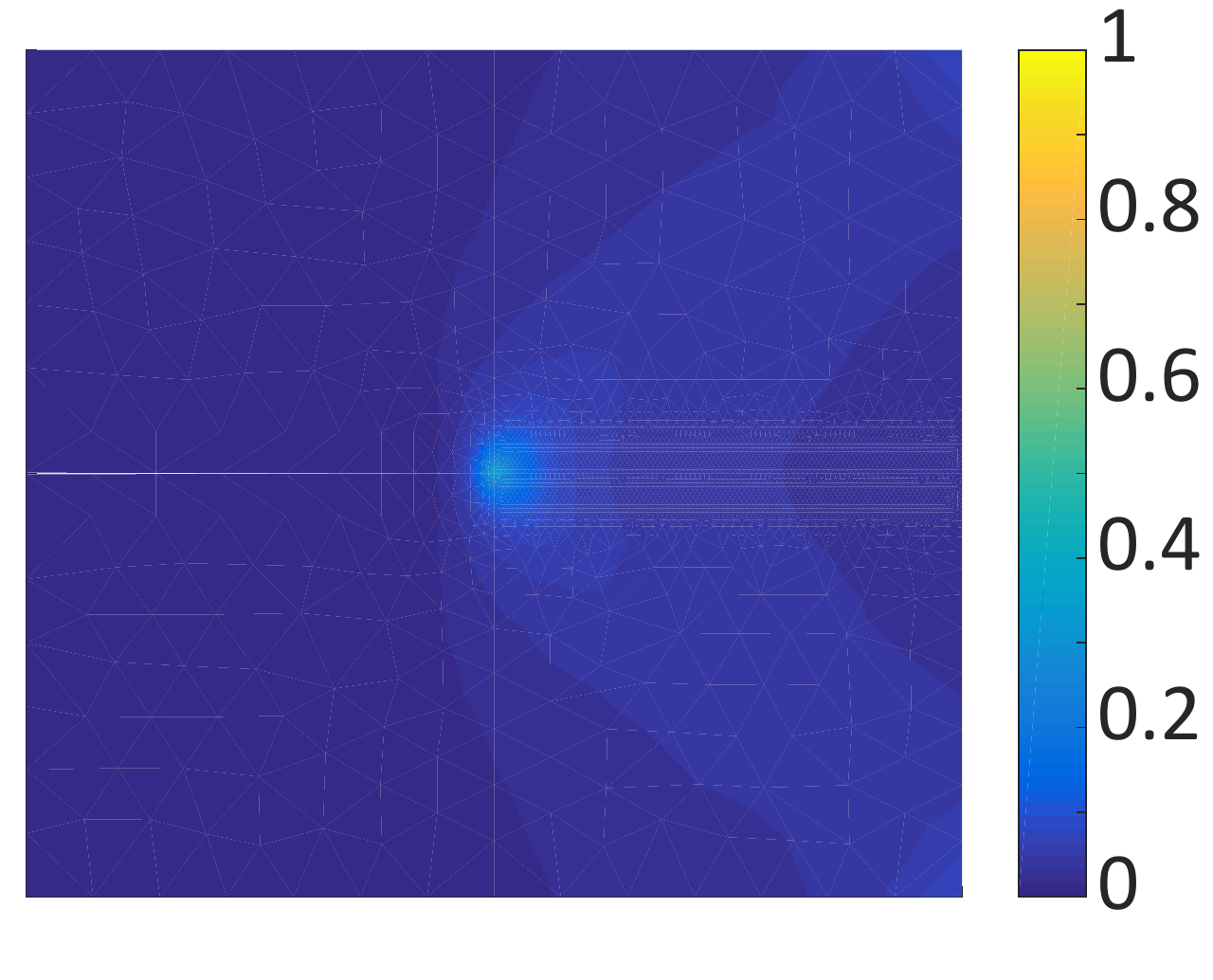}&\includegraphics[width=0.25\textwidth]{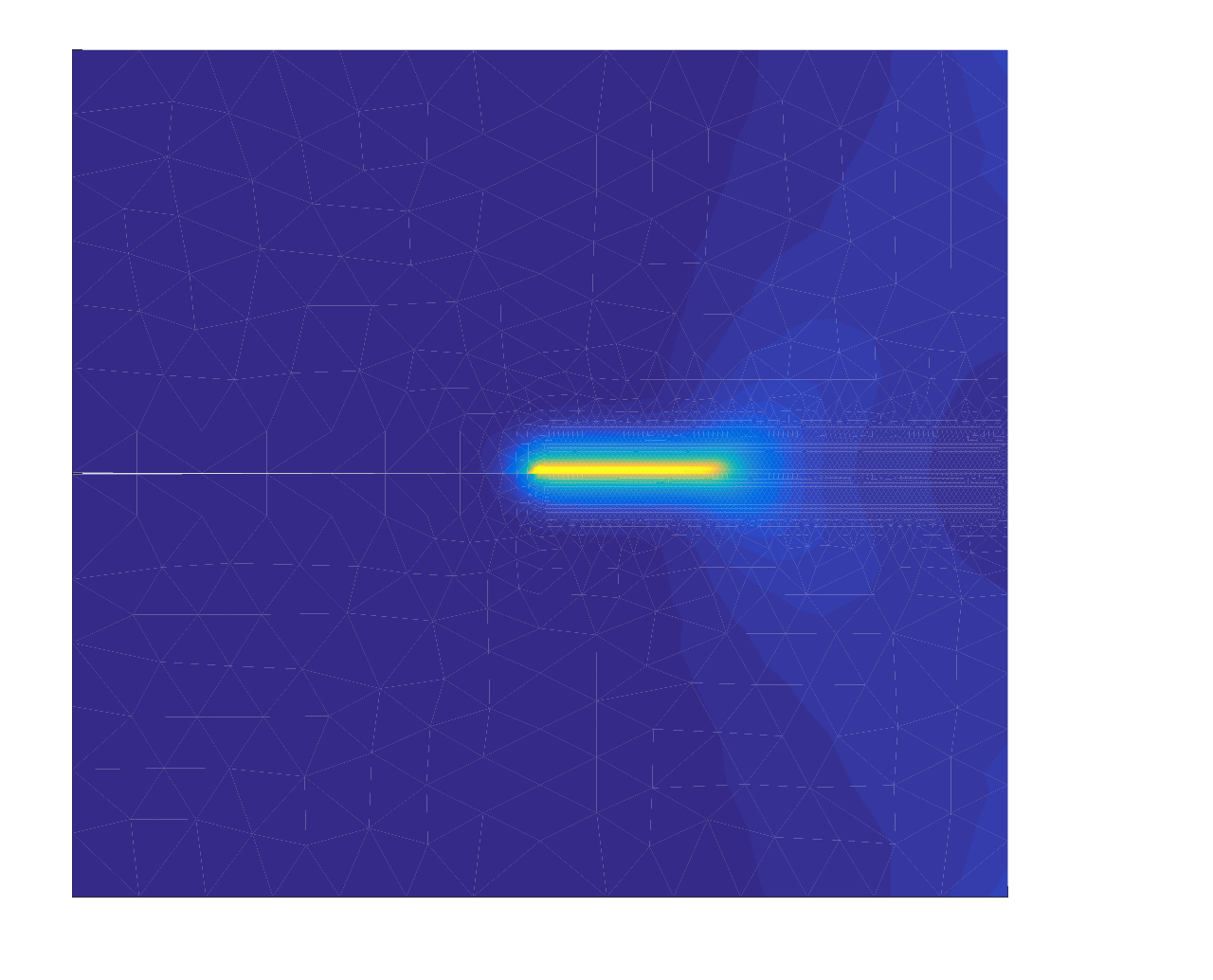}&		  
		 \includegraphics[width=0.25\textwidth]{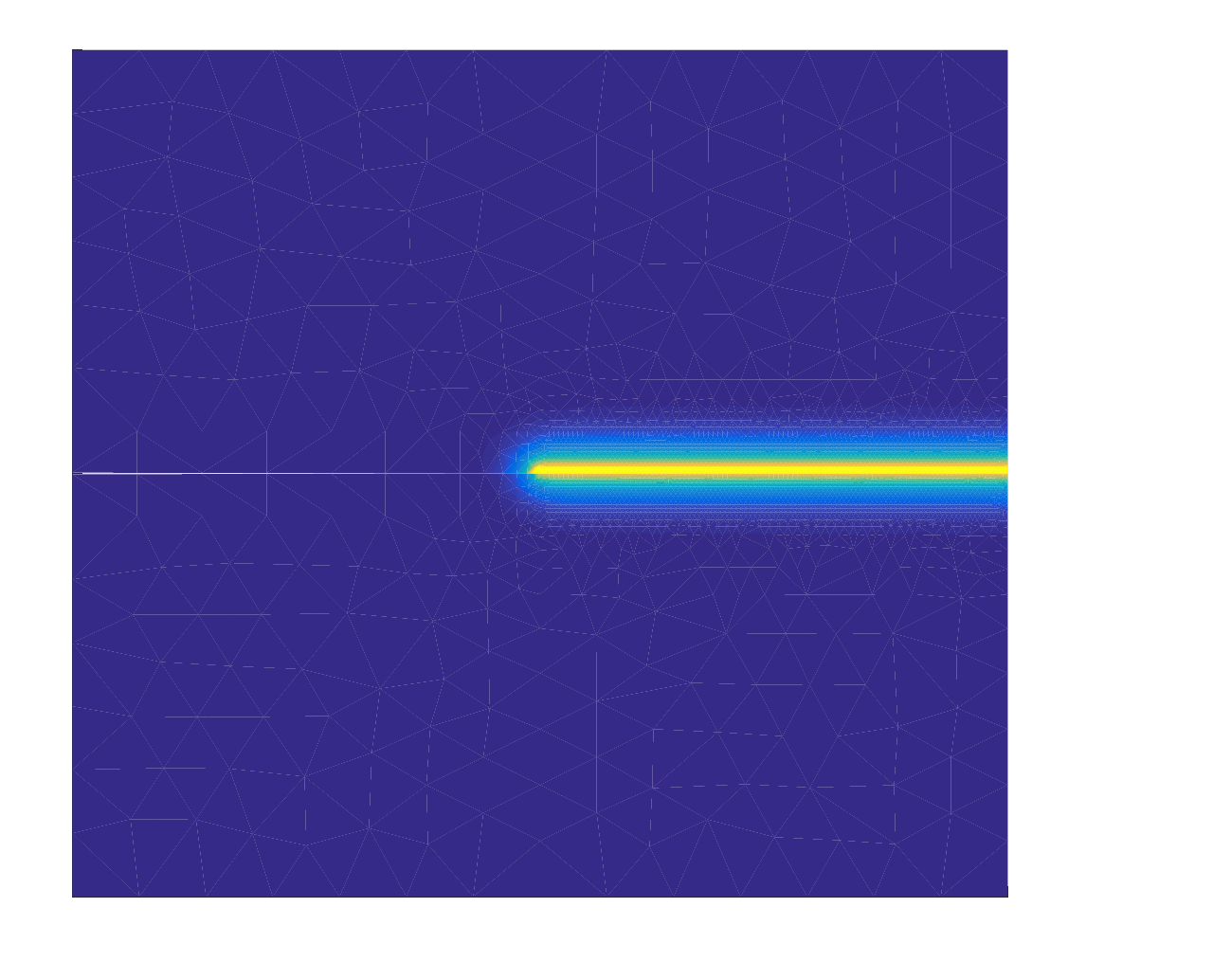}&\includegraphics[width=0.25\textwidth]{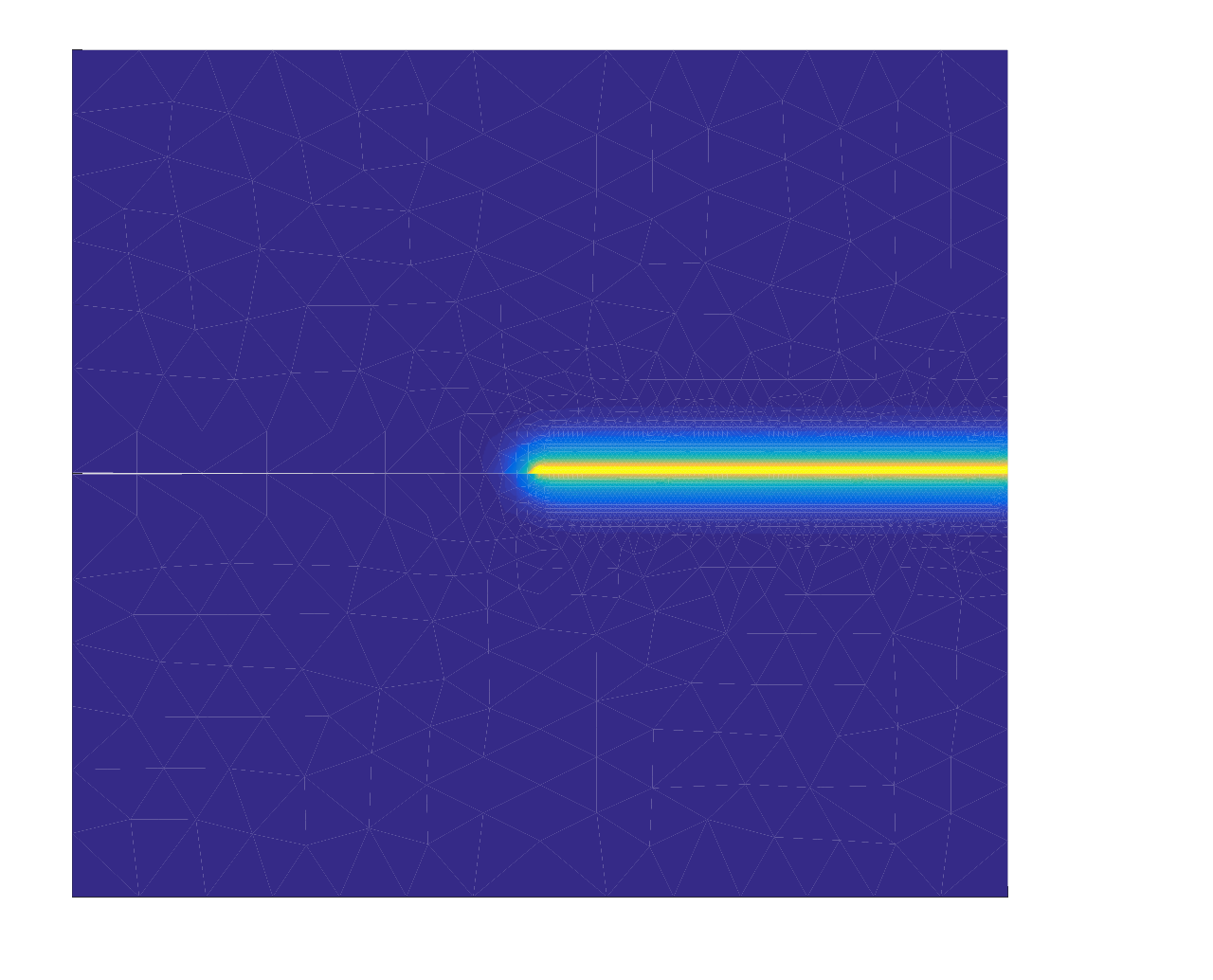}\\
		 w=0.005\,\mathrm{mm} & w=0.00537\,\mathrm{mm}  & w=0.0055\,\mathrm{mm}  & w=0.007\,\mathrm{mm}
               \end{array}$
		}
	\caption{\label{NE3} Example \ref{Sec.NumEx:Ex1}. Single edge notched tension test.
	Phase field distribution at different stages of the total displacement $w$ applied on the specimen top edge.
	Results for $\Delta w=10^{-5}\,\mathrm{mm}$. 
	In the damage maps, the yellow corresponds to $\beta=1-\delta$ 
	with $\delta=10^{-4}$ given that we are considering a partially damage profile, 
	whereas the blue corresponds to solid material for which $\beta=0$. 
	}
\end{figure}

\begin{figure}[H]
	\centering{$\begin{array}{ccc}
		\includegraphics[width=0.33\textwidth]{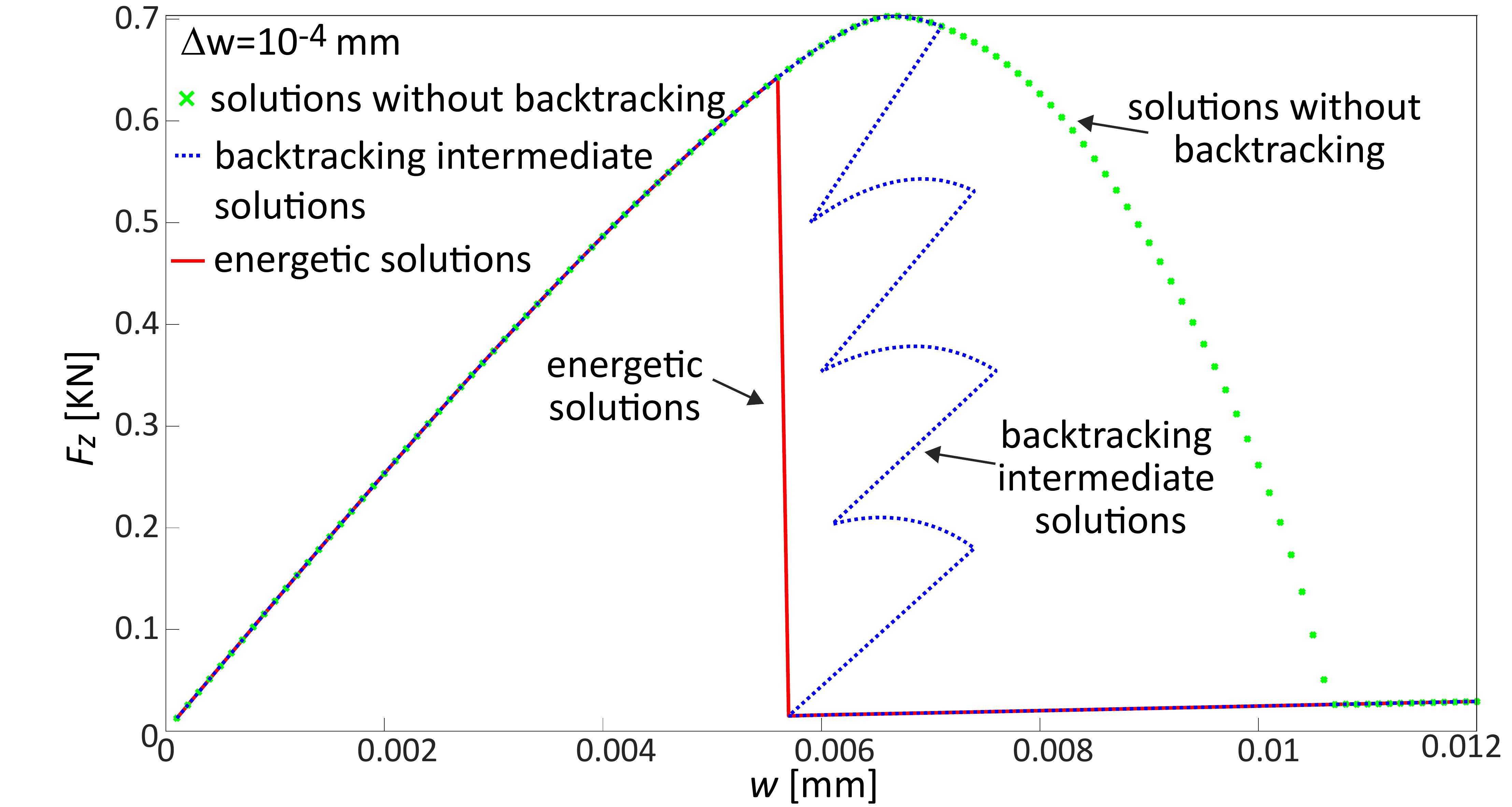}&	
		\includegraphics[width=0.33\textwidth]{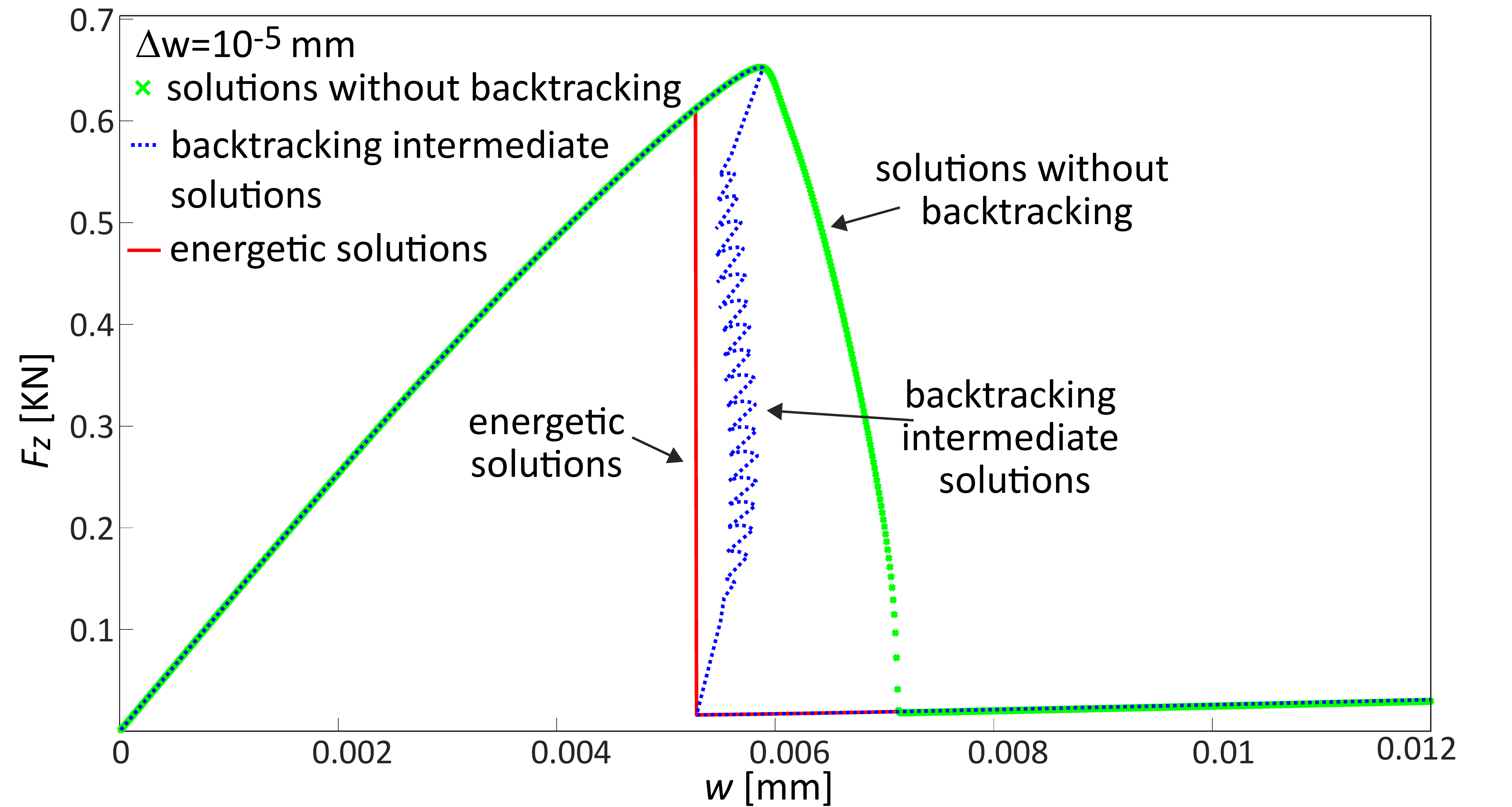}&	
		\includegraphics[width=0.33\textwidth]{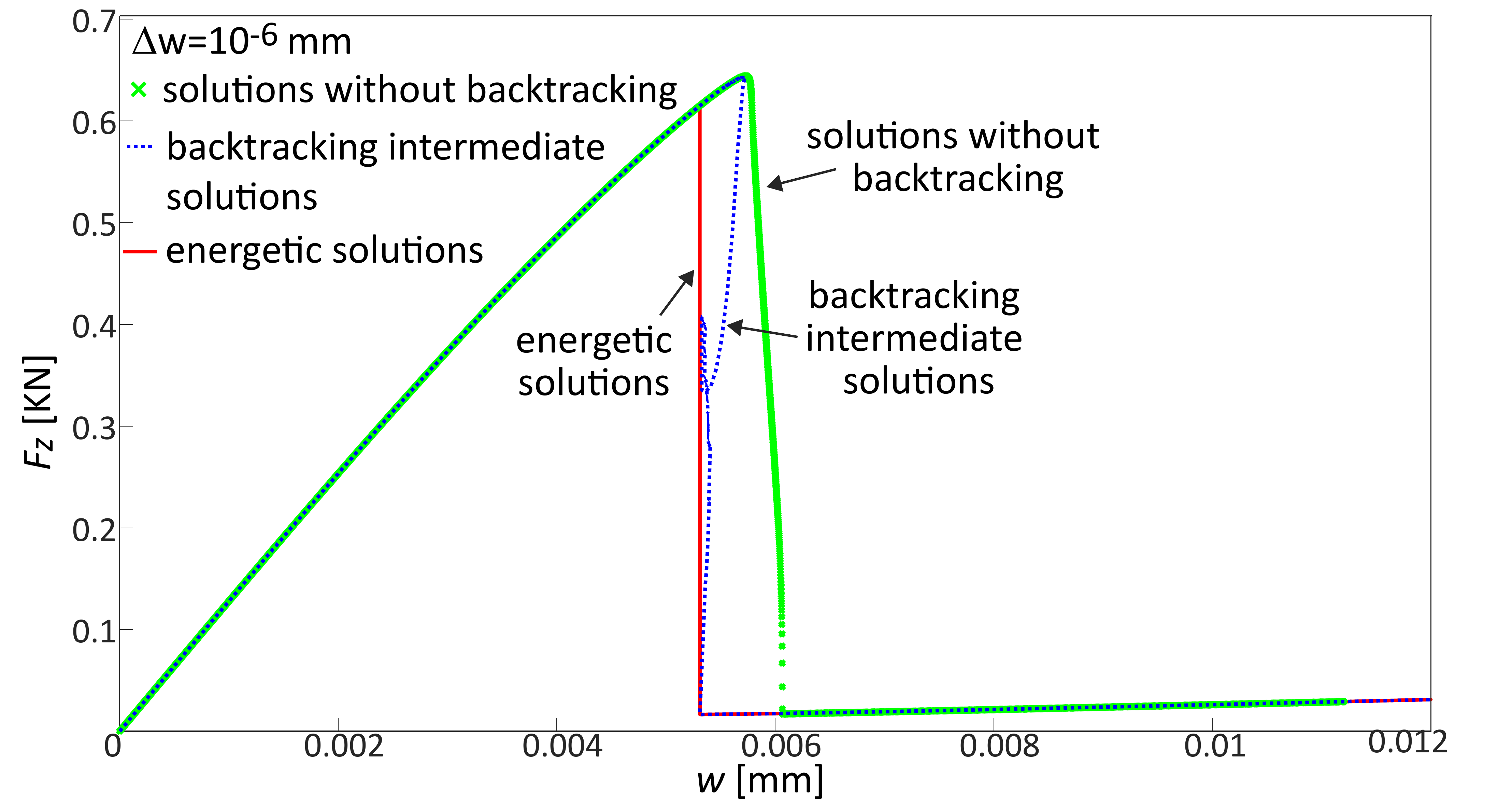}\\	
		 (a) & (b)& (c)
                \end{array}$
		}
	\caption{\label{NE4}
	Example \ref{Sec.NumEx:Ex1}. Single edge notched tension test.
	Load--displacement curves associated with the intermediate energetic solutions obtained with the backtracking algorithm
	and for the different displacement increments. 
	$(a)$ $\Delta w=10^{-4} \,\mathrm{mm}$; $(b)$ $\Delta w=10^{-5} \,\mathrm{mm}$ and $(c)$ $\Delta w=10^{-6} \,\mathrm{mm}$. 
	}
\end{figure}

To get some further insight on the structure of the energetic solutions, we recall that each loading step is solved by
the alternating minimization method with the last converged approximate energetic solution as starting guess. The solution
that we thus compute is in fact a local minimizer, unless it verifies the energetic bounds \eqref{Disc.TwoSidIneq}.
By the backtracking algorithm we are looking for a solution close to the starting guess which meets the energetic bounds, 
thus it is more likely to be a global minimizer.  As a result, if the energetic bounds are not met, the algorithm move one step backward 
(as we have skteched in Figure \ref{Fig.BackTracking}) and solves again the previous step 
but with different starting, guess given by the last computed damage value, thus defining a lower energy state. 
If also such solution does not meet the energetic bounds, the algorithm moves a further step back and the process is repeated 
until the bounds are met or 
we reach the maximum number $K$ that we have set to go backward. Only then the algorithm proceeds one step forward.
In our simulations we never reached such limit value for $K$ which was set equal to $50$.
Figure \ref{NE3}  displays the distribution of the 
phase-field $\beta$ at the different stages of the evolutive process 
as computed by the two numerical schemes. Consistently with the procedure described above,
the damage profile displays with the backtracking option active a faster evolution and higher dissipation
when compared with the basic variant. This behaviour is 
also confirmed by the numerical experiments of \cite{MRZ10}
where the \textsc{AT1} regularized formulation of fracture is considered without the
splitting of the free elastic energy $\psi_0$. 

Given that with the backtracking option active, we go backward and forward, obtaining energetic 
solutions for the same displacement but with increasing damage and therefore lower resultant load,
we save these solutions, for which the total energy remains between the two energetic bounds. 
Figure \ref{NE4} displays the load-displacement curve associated with such intermediate configurations. 
For instance, for the step increment $\Delta w=10^{-4}\,\mathrm{mm}$,
the path zigzags down to the curve because with such size of the increment, 
the two sided energetic bounds are more distant from each other, leaving more room to move 
within the bounds. Such range between the bounds is reduced  by reducing the displacement increment $\Delta w$
and so is the zigzaged path.


\subsection{Single edge notched shear test}\label{Sec.NumEx:Ex2}

We now consider the same square plate with horizontal notch as in the previous 
example but this time subject to pure shear deformation. This problem
has received a lot of attention in the literature on phase-field modelling of brittle fracture \cite{BFM00,MHW10,MWH10}
for its simple setup and for displaying an asymmetric failure pattern. Due to a non--trivial combination of
local tension-compression and loading--unloading processes, the crack propagates towards the lower right corner 
of the square plate. The geometric setup and boundary conditions are displayed in Figure \ref{NE6}$(a)$. 

\begin{figure}[H]\label{NE6}
	\centering{$\begin{array}{cc}
		\includegraphics[width=0.35\textwidth]{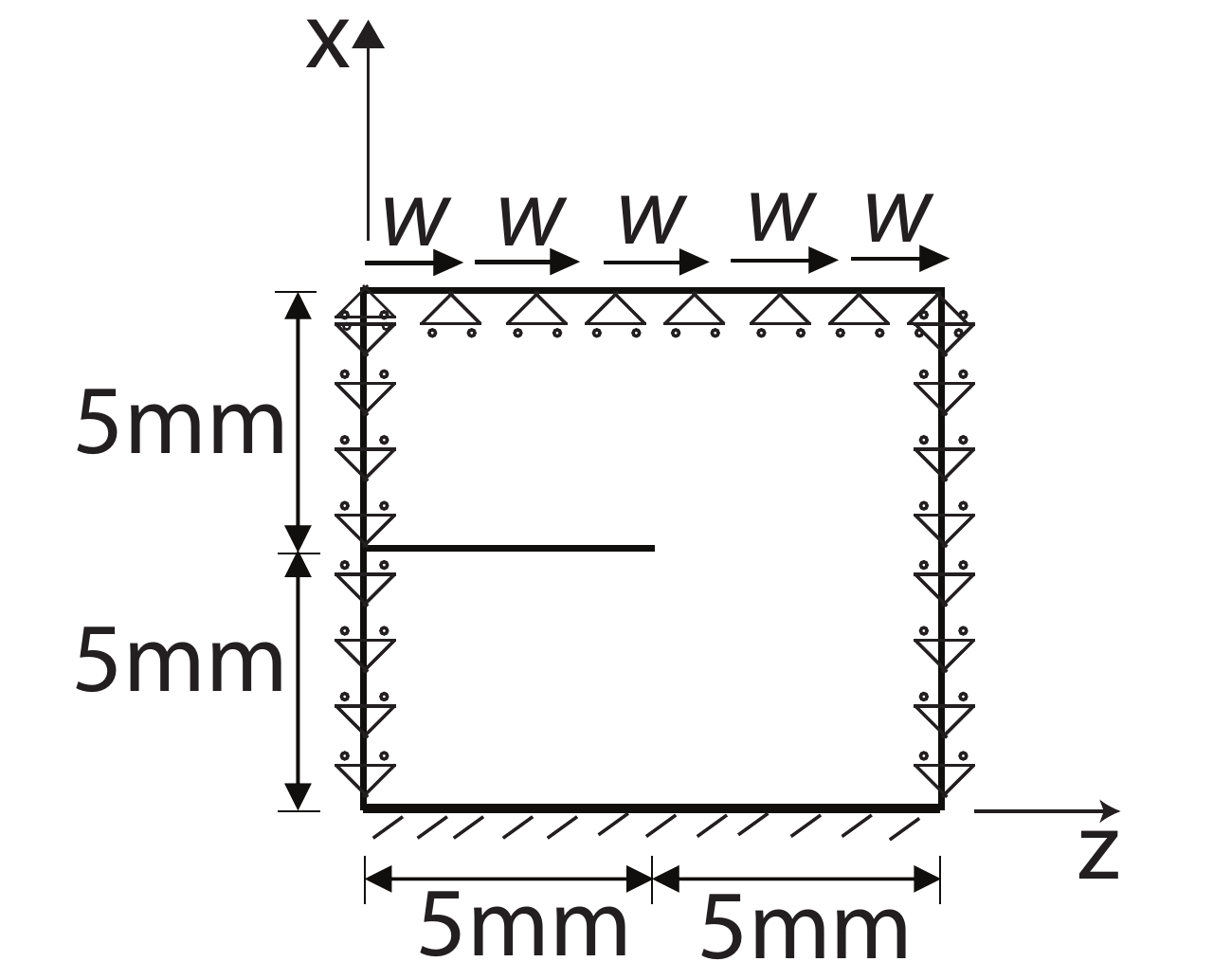}&    
		\includegraphics[width=0.20\textwidth]{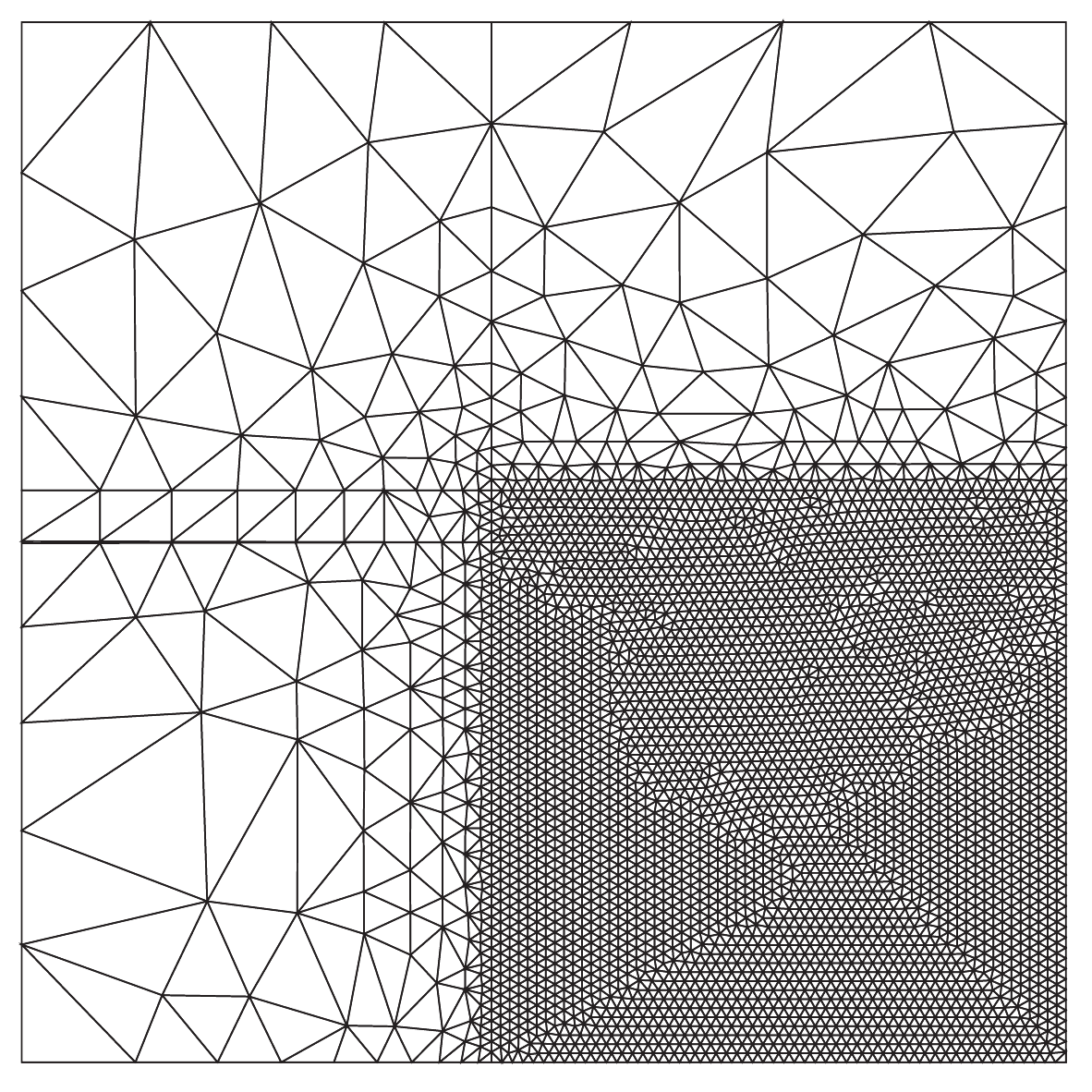}\\
		(a)& (b)
		\end{array}$
		}
	\caption{\label{NE6}
		Example \ref{Sec.NumEx:Ex2}. Single edge notched shear test.
		$(a)$ Geometry and boundary conditions. 
		$(b)$ Unstructured finite element mesh.
	}
\end{figure} 

The vertical displacement component is constrained on all four sides of the domain. The botton edge is also constrained
along the horizontal direction whereas the top edge presents a prescribed nonhomogeneous Dirichlet boundary condition $u_z=w$.
The same material properties are used as for the previous example. The characteristic length is now set equal to $\ell=0.001\,\mathrm{mm}$.  
Figure \ref{NE6}$(b)$ displays the unstructured finite element mesh with $7573$ triangular elements 
and $3878$ nodes which has been refined in the lower right part of the domain $\Omega$ where the crack 
is expected to propagate \cite{BFM00,MWH10}. The characteristic element size in this region is $h\approx 0.005\,\mathrm{mm}<\ell/2$.
The finite element approximations for the displacement and phase-field is the same as in the previous example. 
We consider displacement-driven loading by the application of two constant displacement increments $\Delta w=10^{-4}\,\mathrm{mm}$ 
and $\Delta w=10^{-5}\,\mathrm{mm}$, 
and take, for this example, the penalization factor $\epsilon$ equal to $10^{-5}$. We have then evaluated,
 for each case, the energetic terms that enter \eqref{Disc.TwoSidIneq}  
and the reaction force $F_z$ on the top edge $\Gamma_{top}\subseteq\partial\Omega$ given by

\[
	F_z=\int_{\Gamma_{top}}\,\sigmab\bmn\cdot\bmt\,\ds
\]
where $\bmt$ is the tangent to the top edge and $\bmn$ is the outward normal to this part of the boundary.

\begin{figure}[H]
	\centering{$\begin{array}{c}
		\includegraphics[width=0.5\textwidth]{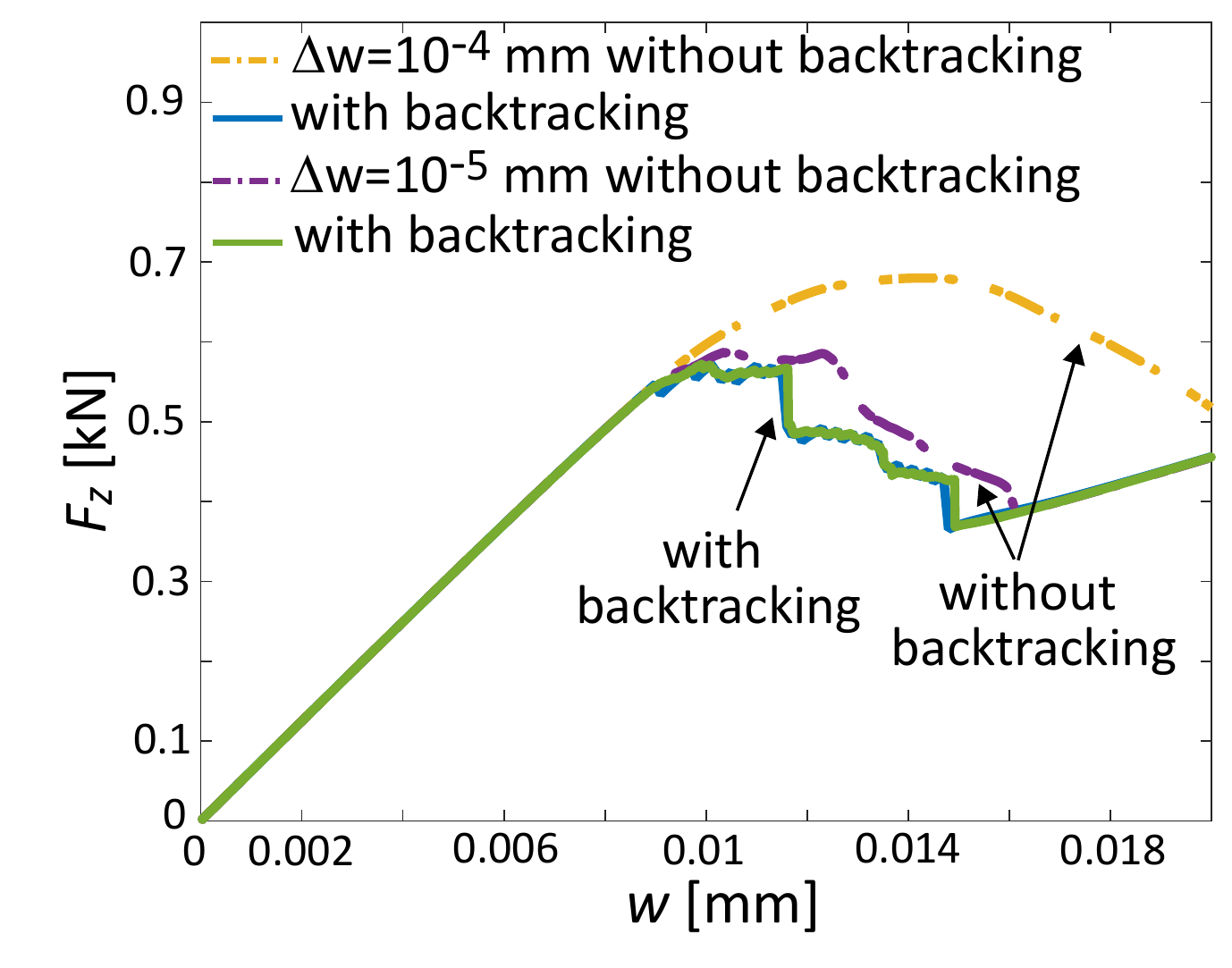} 
		\end{array}$}
	\caption{\label{NE7}
	Example \ref{Sec.NumEx:Ex2}. Single edge notched shear test. 
	Load--displacement curves for different displacement increments $\Delta w$ 
	and different schemes, using the backtracking algorithms describing the evolution 
	of the approximate energetic solutions and without applying the backtracking algorithm.}
\end{figure}

Figure \ref{NE7} displays the load-displacement curves corresponding to the approximate energetic solutions 
and to the solutions obtained without using the backtracking algorithm,
for the two different applications of $\Delta w$. Likewise the previous example, 
the structural response obtained by the approximate
energetic solutions is almost the same for $\Delta w=10^{-4}\,\mathrm{mm}$ and $\Delta w=10^{-5} \,\mathrm{mm}$. 

\begin{figure}[H]
	\centering{$\begin{array}{cc}
		\includegraphics[width=0.49\textwidth]{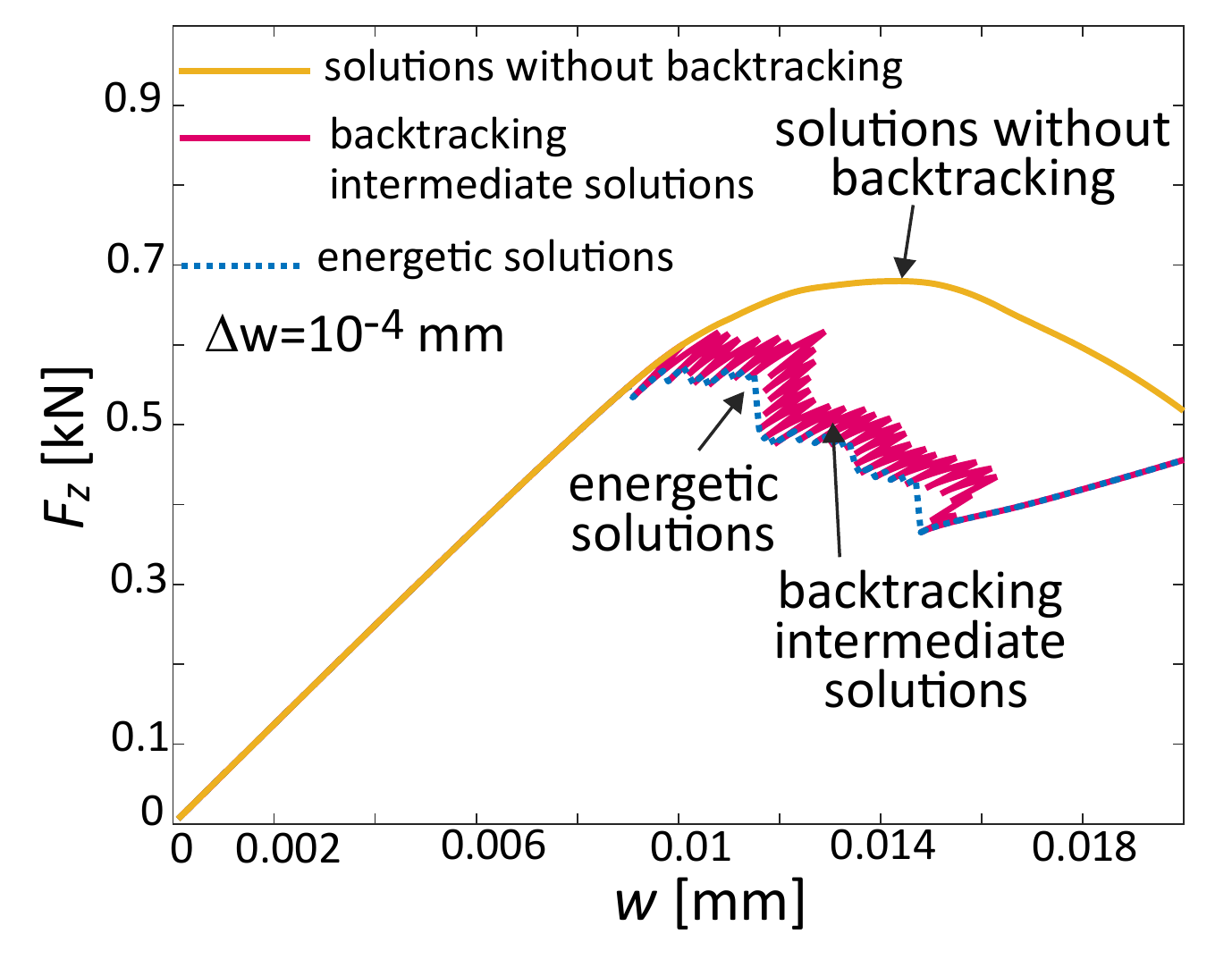} &\includegraphics[width=0.49\textwidth]{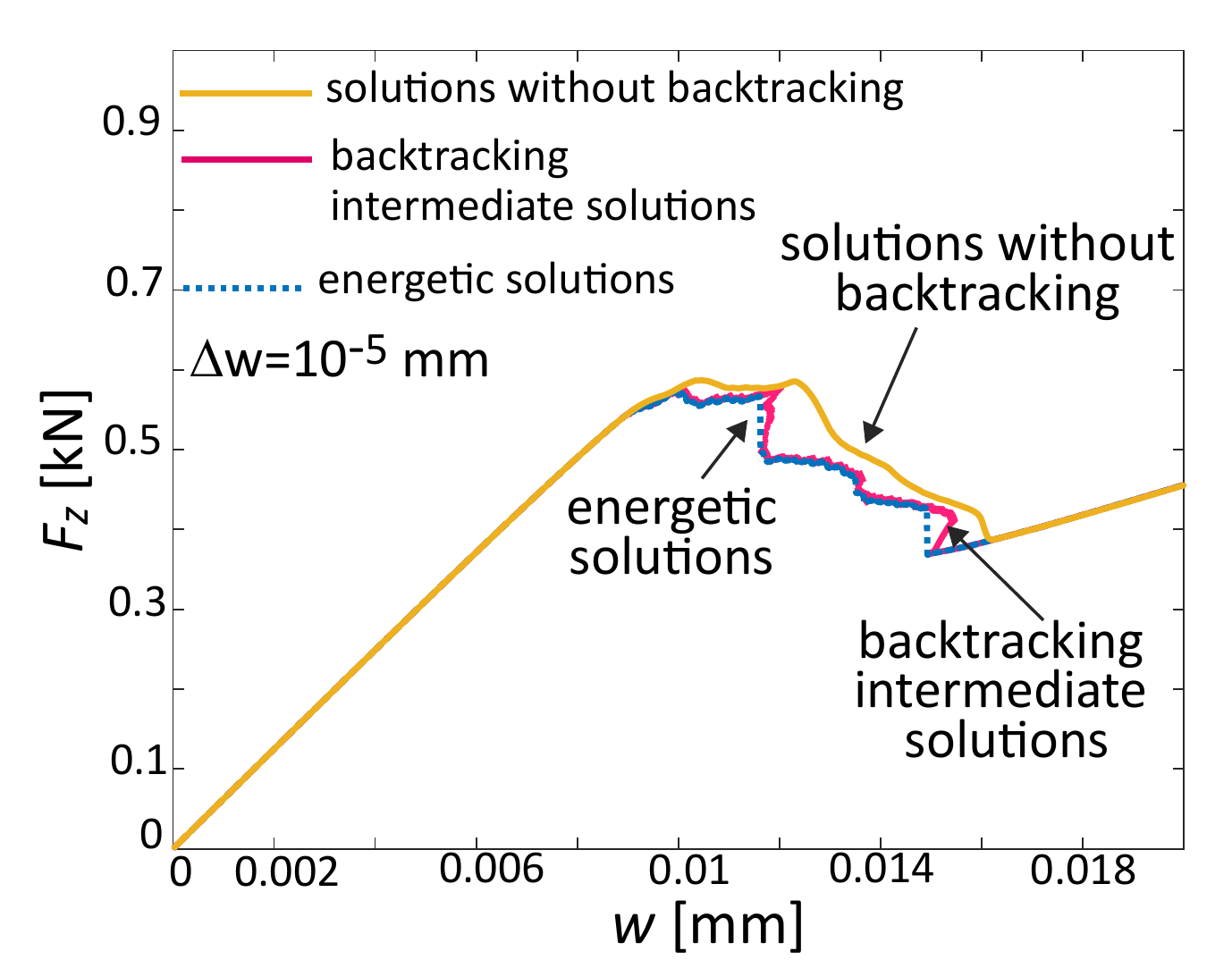} \\
		(a)&(b)
  \end{array}$}
	\caption{\label{NE8}
	Example \ref{Sec.NumEx:Ex2}. Single edge notched shear test.
	Load--displacement curves associated with the intermediate energetic solutions obtained with the backtracking algorithm
	for the displacement increment $(a)$ $\Delta w=10^{-4} \,\mathrm{mm}$ and $(b)$ $\Delta w=10^{-5} \,\mathrm{mm}$.}
\end{figure}

The load-displacement curves corresponding also to the intermediate solutions are, by contrast, displayed in Figure \ref{NE8}.
For the step increment $\Delta w=10^{-4}\, mm$ the curve zigzags towards the softenning part of the curve,
whereas for the smaller increment $\Delta~w=10^{-5}~\,~\mathrm{mm}$, the two bounds get closer and the curve 
results smoother with only two small jumps.
The variations of the total energies of the solutions computed with the two numerical schemes
and for $\Delta w=10^{-4}\,\mathrm{mm}$ and  $\Delta w=10^{-5}\,\mathrm{mm}$ are plotted in 
Figure \ref{Ex:ShearTest:Enrg:4} and Figure \ref{Ex:ShearTest:Enrg:5}, respectively. 
Figure \ref{Ex:ShearTest:Enrg:4}$(a)$
and  Figure \ref{Ex:ShearTest:Enrg:5}$(a)$ show, for the respective $\Delta w$, the evolution of the 
total energetic of the system, the total free energy and the accumulated dissipation.
Figure \ref{Ex:ShearTest:Enrg:4}$(b)$
and  Figure \ref{Ex:ShearTest:Enrg:5}$(b)$, and
Figure \ref{Ex:ShearTest:Enrg:4}$(c)$
and  Figure \ref{Ex:ShearTest:Enrg:5}$(c)$ display 
the total incremental energy  
$\mathcal{E}_{n+1}-\mathcal{E}_{n}+\mathcal{D}_{n,n+1}$, the upper bound $UB_{n,n+1}$ and 
the lower bound $LB_{n,n+1}$, for $n=0,\,\ldots, N-1$. 
We thus verify that also for this problem, 
the activation of the backtracking algorithm 
is needed to select the `right' forward path of the lowest energy content given by approximate energetic solutions. 
Figure \ref{NE10} finally displays the phase-field distribution 
at different stages of the displacement $w$ for the two numerical scheme showing that with the backtracking algorithm
we obtain a faster evolution when compared with the basic scheme.

\begin{figure}[H]
	\centering{$\begin{array}{c}
		\includegraphics[width=0.60\textwidth]{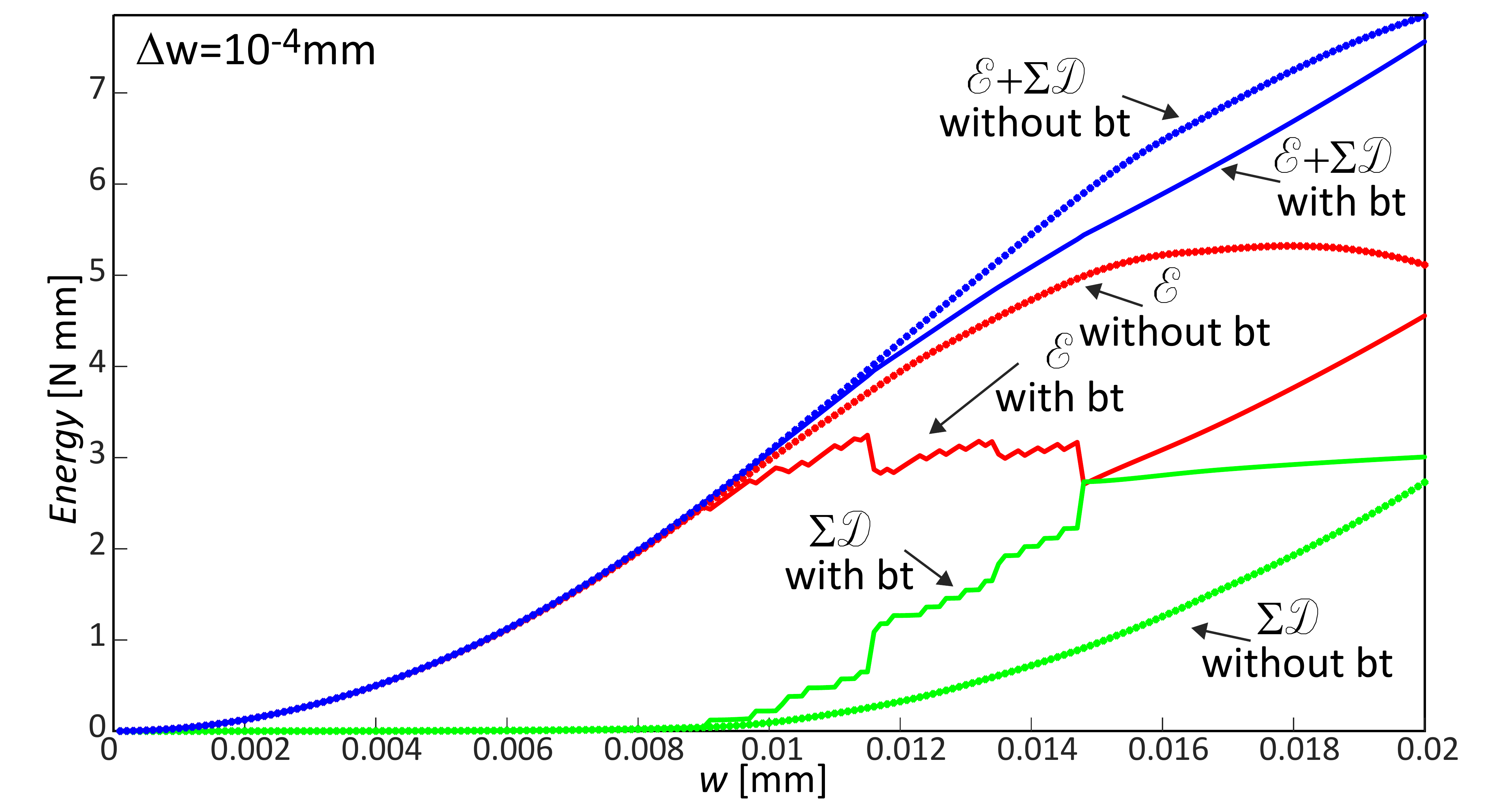}\\
		(a)\\
		\begin{array}{cc}
			\includegraphics[width=0.49\textwidth]{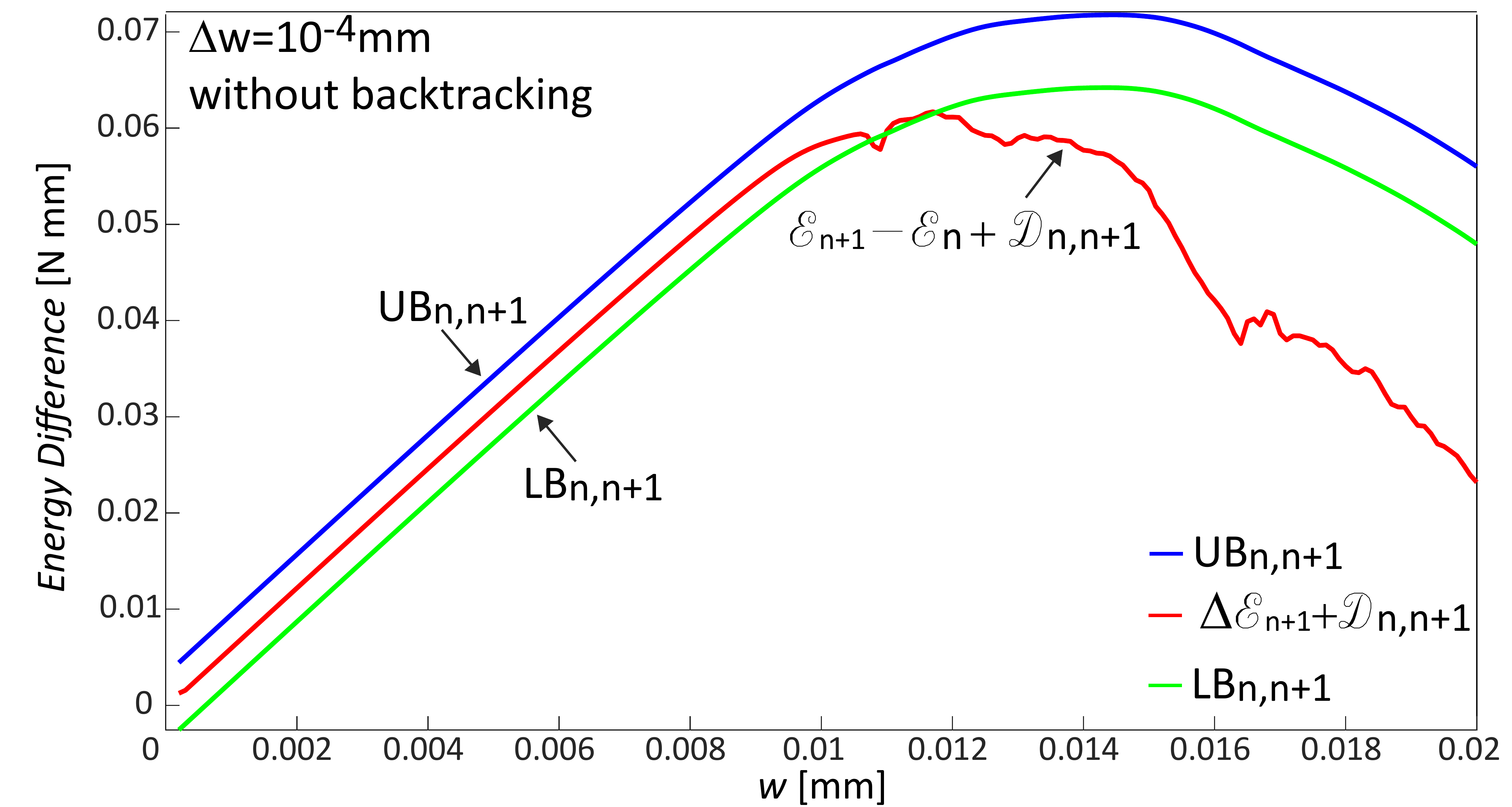}& 
			\includegraphics[width=0.49\textwidth]{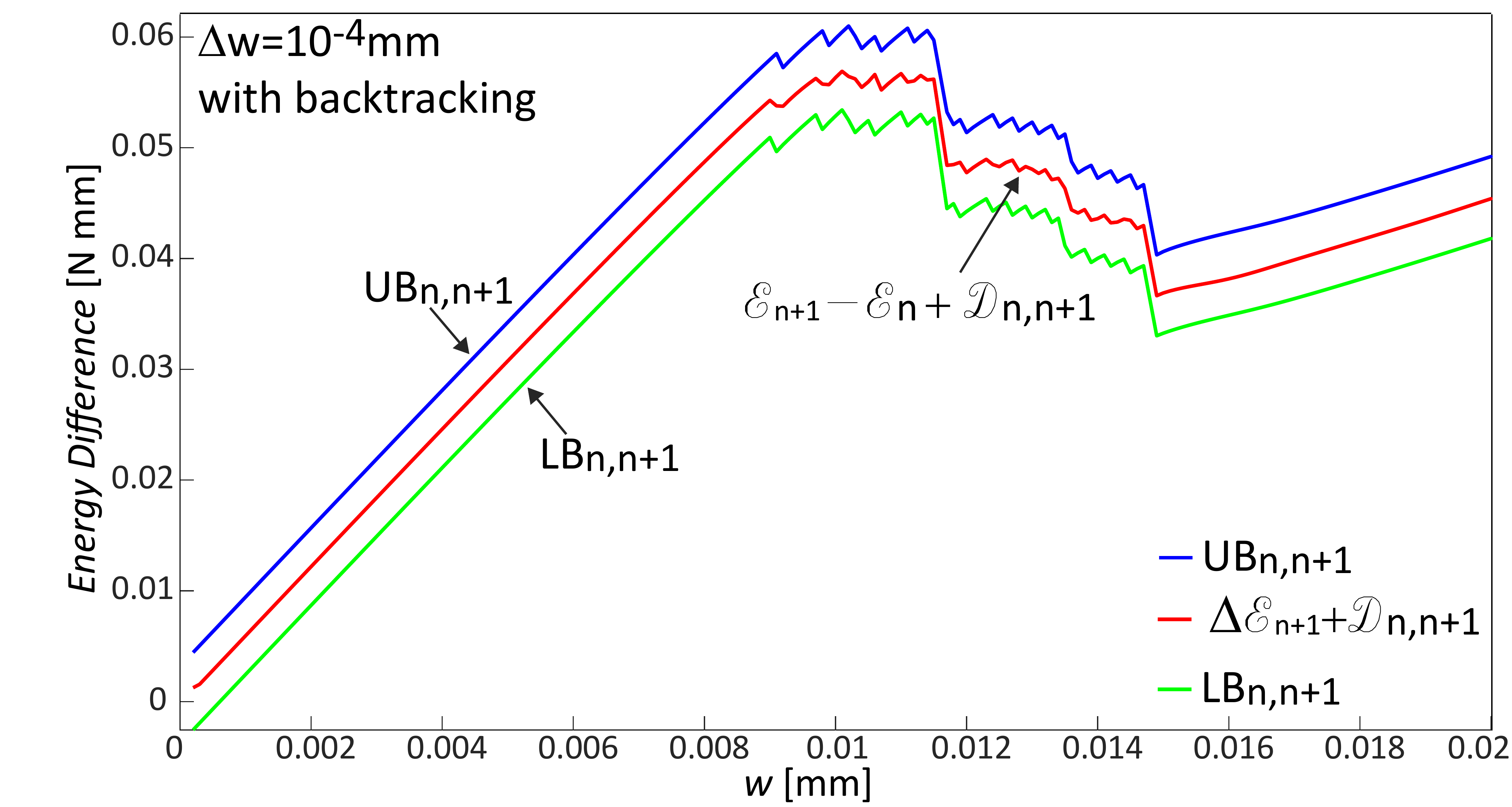}\\
			(b)&(c)
		\end{array}
		\end{array}$}
	\caption{\label{Ex:ShearTest:Enrg:4}
	Example \ref{Sec.NumEx:Ex2}. Single edge notched shear test.
	Results for $\Delta w=10^{-4}\,\mathrm{mm}$.
	$(a)$ Evolution of the total energy $\mathcal{E}(t_{n+1},\bmU_{n+1},\bmA_{n+1})+\sum_{i=0}^n\mathcal{D}(\bmA_i,\bmA_{i+1})$ 
	for $n=0,1,\ldots, N-1$, without backtracking $(K=0)$  and with  backtracking.
	Evolution of the total incremental energy $\mathcal{E}_{n+1}-\mathcal{E}_{n}+\mathcal{D}_{n,n+1}$, 
	the lower bound $LB_{n,n+1}$ and the upper bound $UB_{n,n+1}$ which enter the two-sided energy estimate \eqref{Disc.TwoSidIneq},
	$n=0,1,\ldots, N-1$, for the scheme $(b)$ without backtracking and $(c)$ with backtracking.}
\end{figure}

\begin{figure}[H]
	\centering{$\begin{array}{c}
		\includegraphics[width=0.60\textwidth]{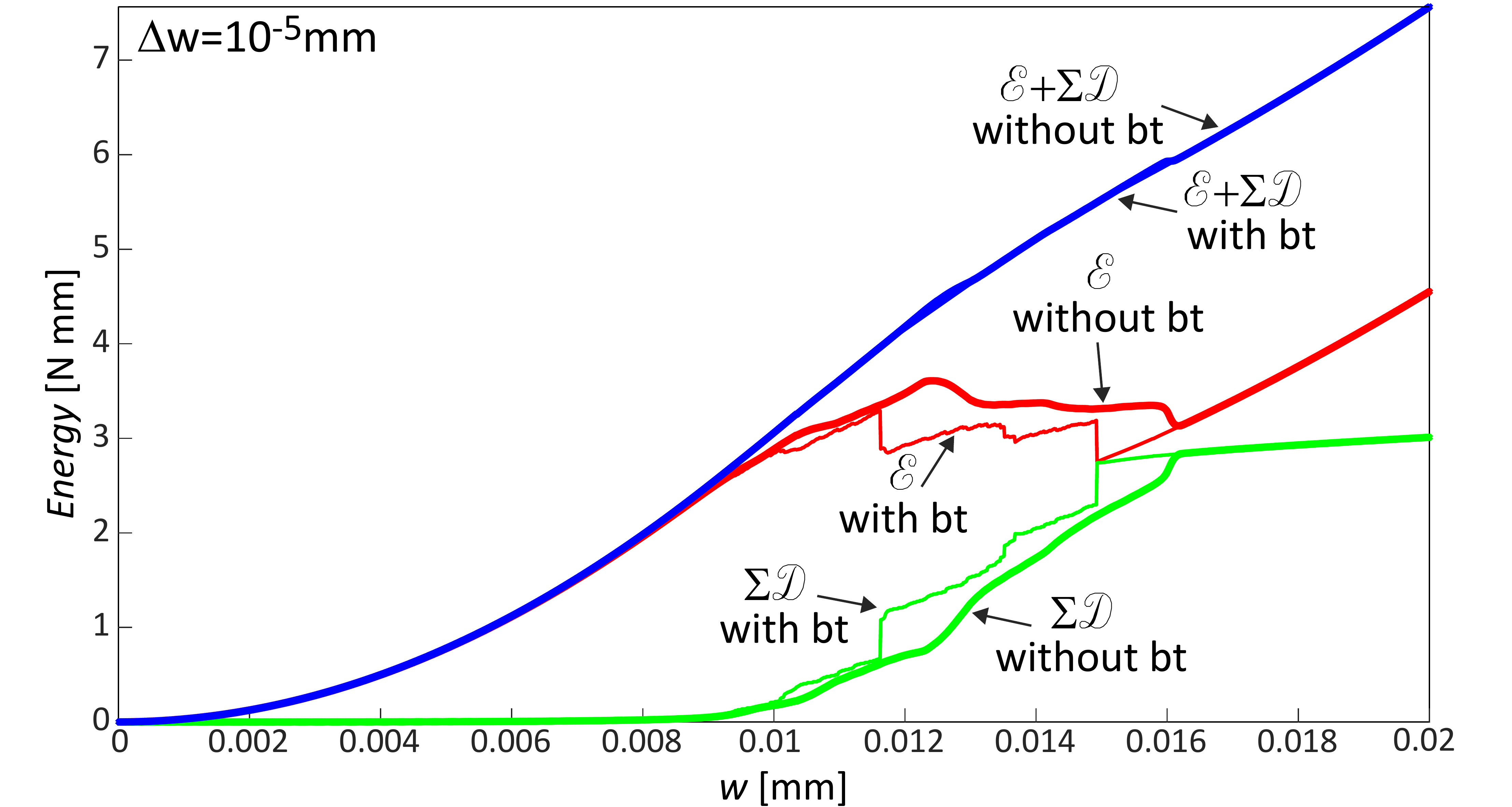}\\
		(a)\\
		\begin{array}{cc}
			\includegraphics[width=0.49\textwidth]{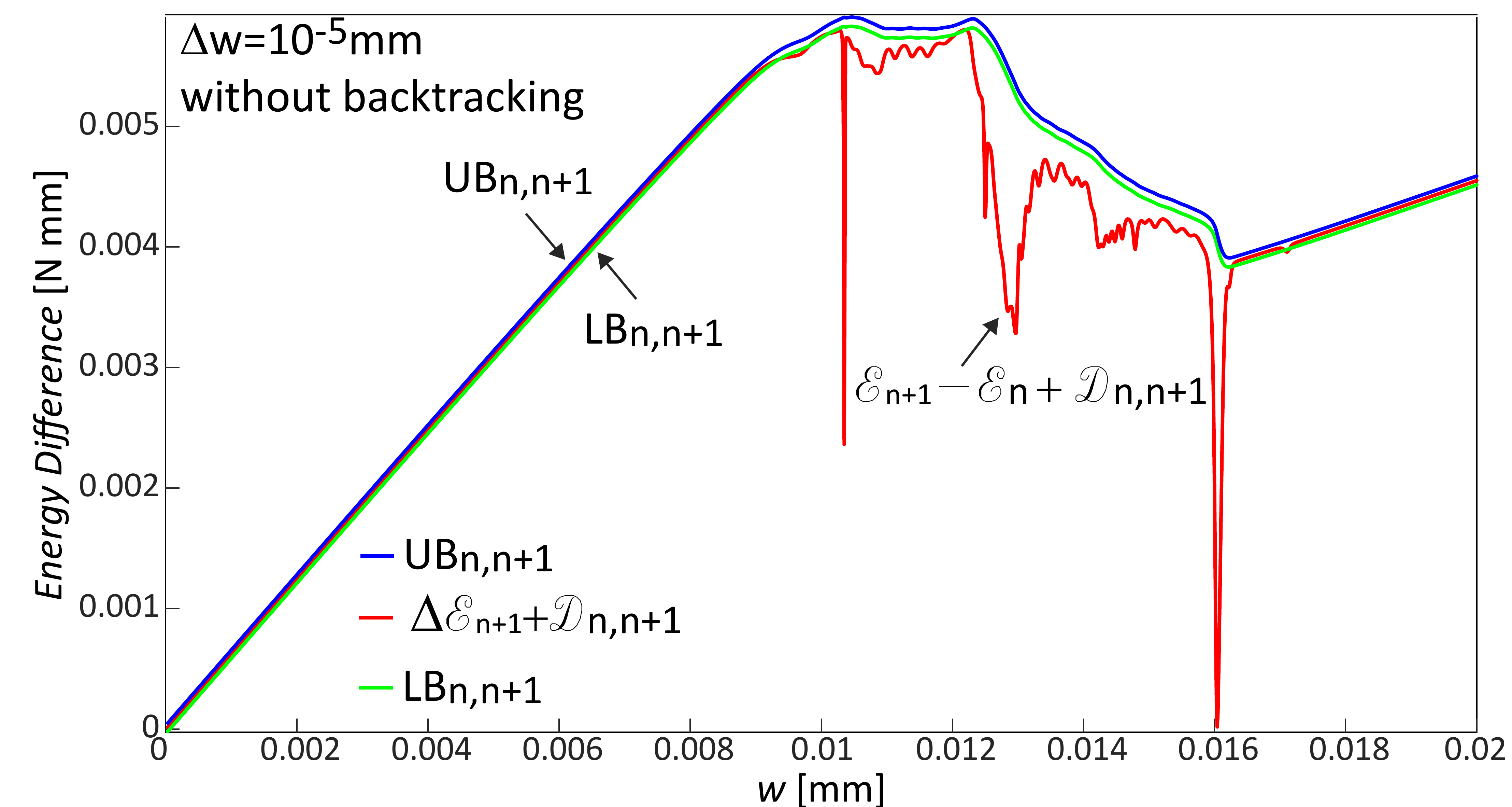}& 
			\includegraphics[width=0.49\textwidth]{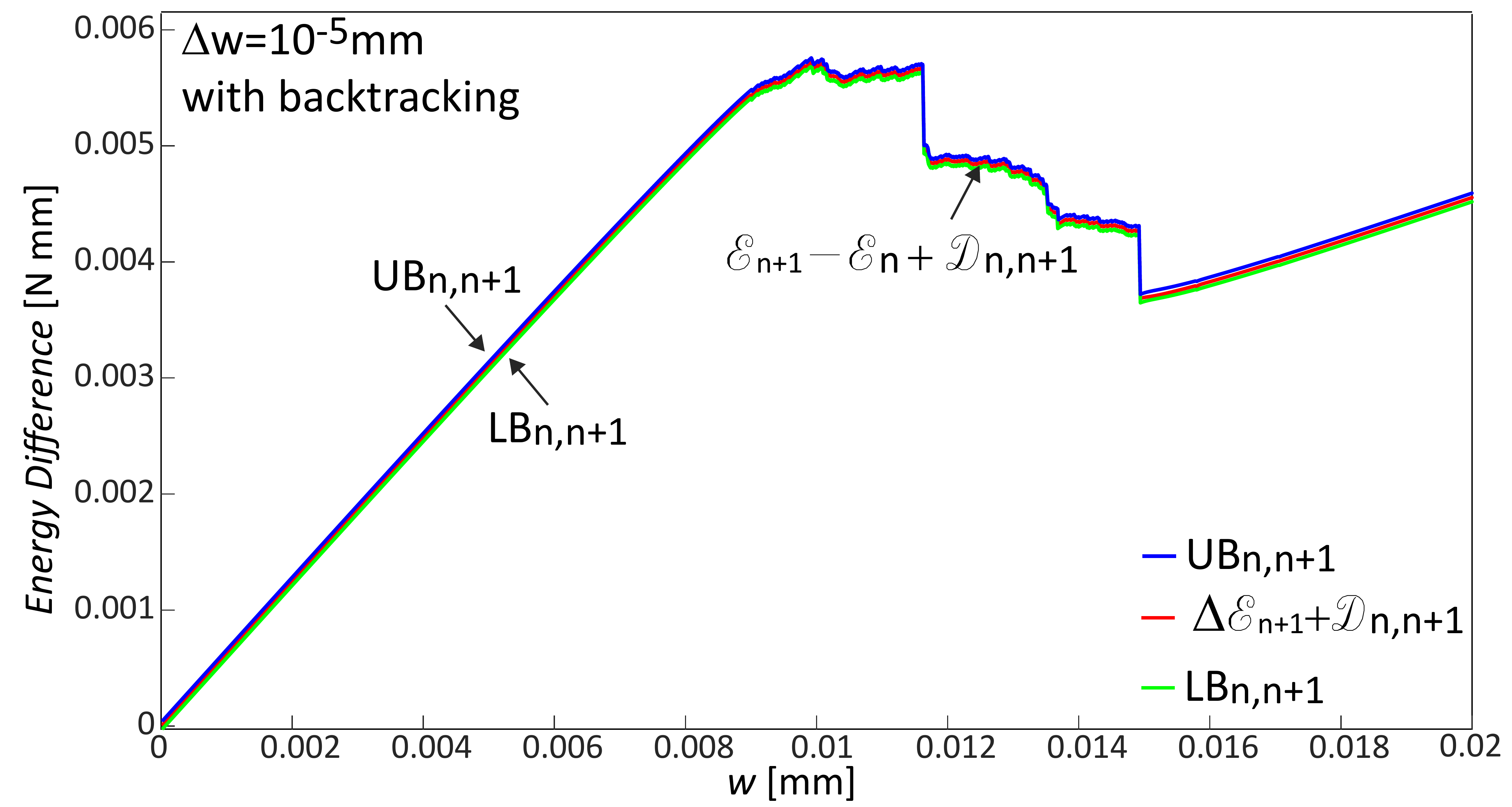}\\
			(b)&(c)
		\end{array}
		\end{array}$}
	\caption{\label{Ex:ShearTest:Enrg:5}
	Example \ref{Sec.NumEx:Ex2}. Single edge notched shear test.
	Results for $\Delta w=10^{-5}\,\mathrm{mm}$.
	$(a)$ Evolution of the total energy $\mathcal{E}(t_{n+1},\bmU_{n+1},\bmA_{n+1})+\sum_{i=0}^n\mathcal{D}(\bmA_i,\bmA_{i+1})$ 
	for $n=0,1,\ldots, N-1$,  without backtracking $(K=0)$  and with  backtracking.
	Evolution of the total incremental energy  $\mathcal{E}_{n+1}-\mathcal{E}_{n}+\mathcal{D}_{n,n+1}$, 
	the lower bound $LB_{n,n+1}$ and the upper bound $UB_{n,n+1}$ which enter the two-sided energy estimate \eqref{Disc.TwoSidIneq},
	$n=0,1,\ldots, N-1$, for the scheme $(b)$ without backtracking and $(c)$ with backtracking.}
\end{figure}

\begin{figure}[H]
	\centering{$\begin{array}{cccc}
	\text{$(a)$ Without backtracking}&&&\\
		\includegraphics[width=0.25\textwidth]{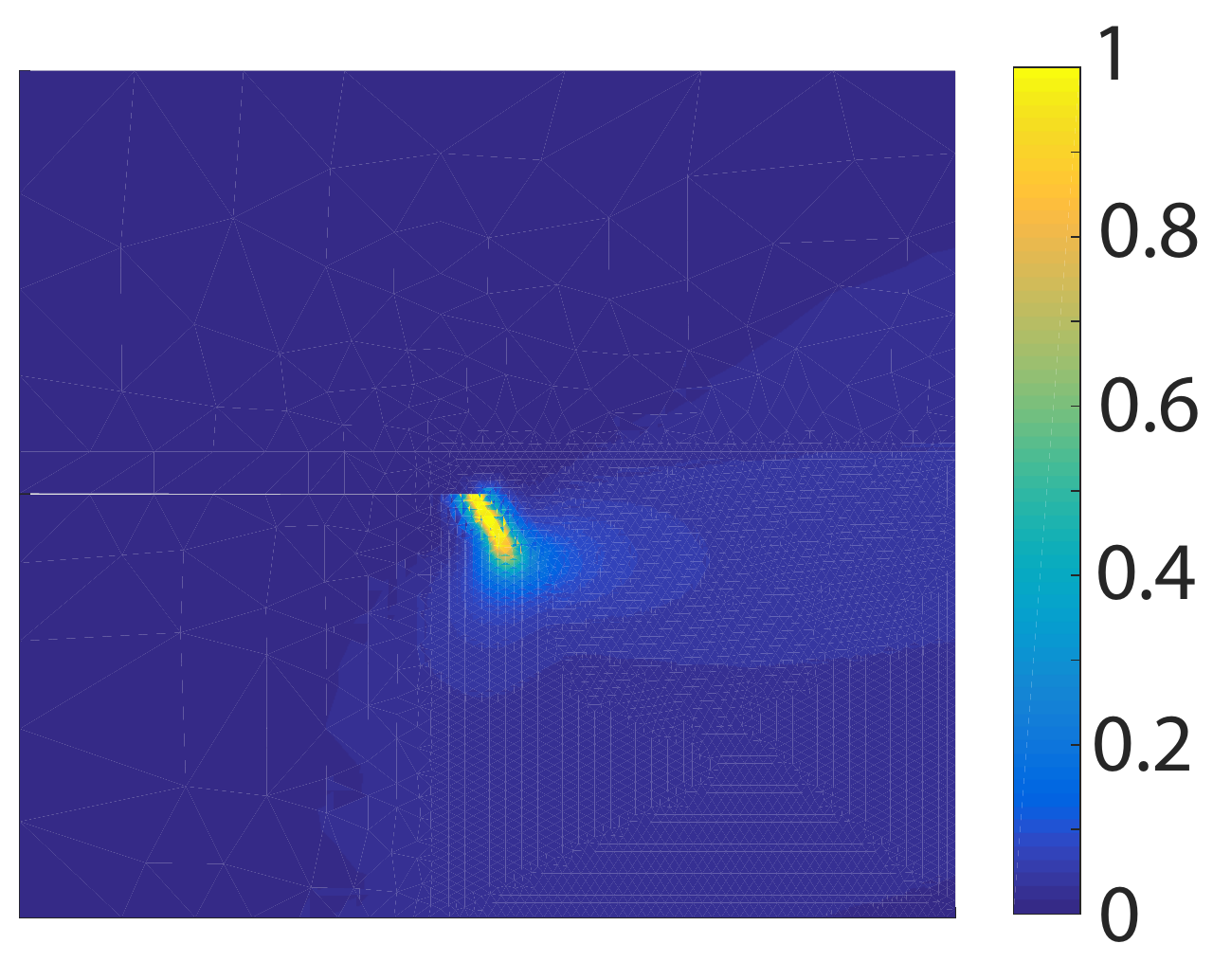}&\includegraphics[width=0.25\textwidth]{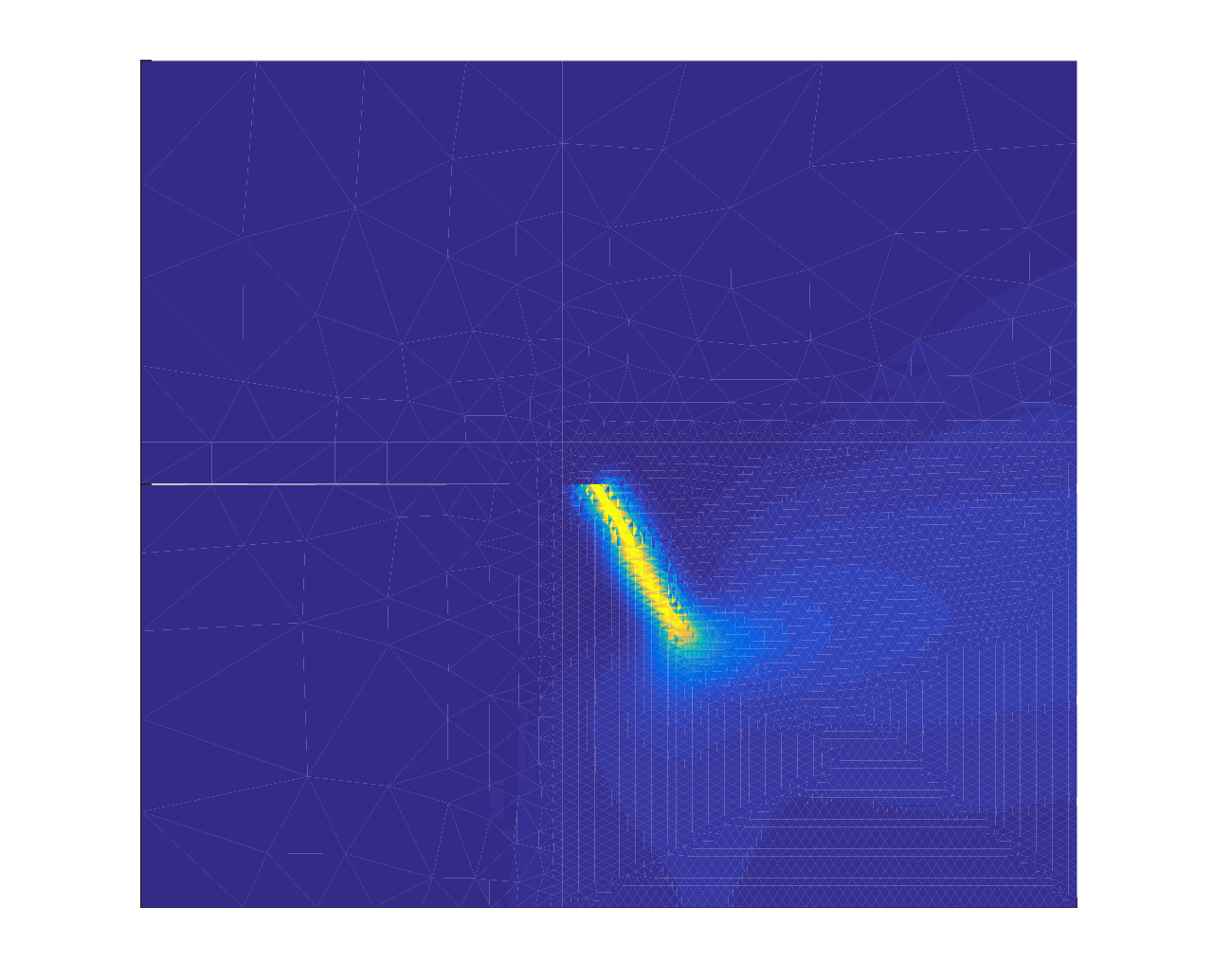}&		  
		\includegraphics[width=0.25\textwidth]{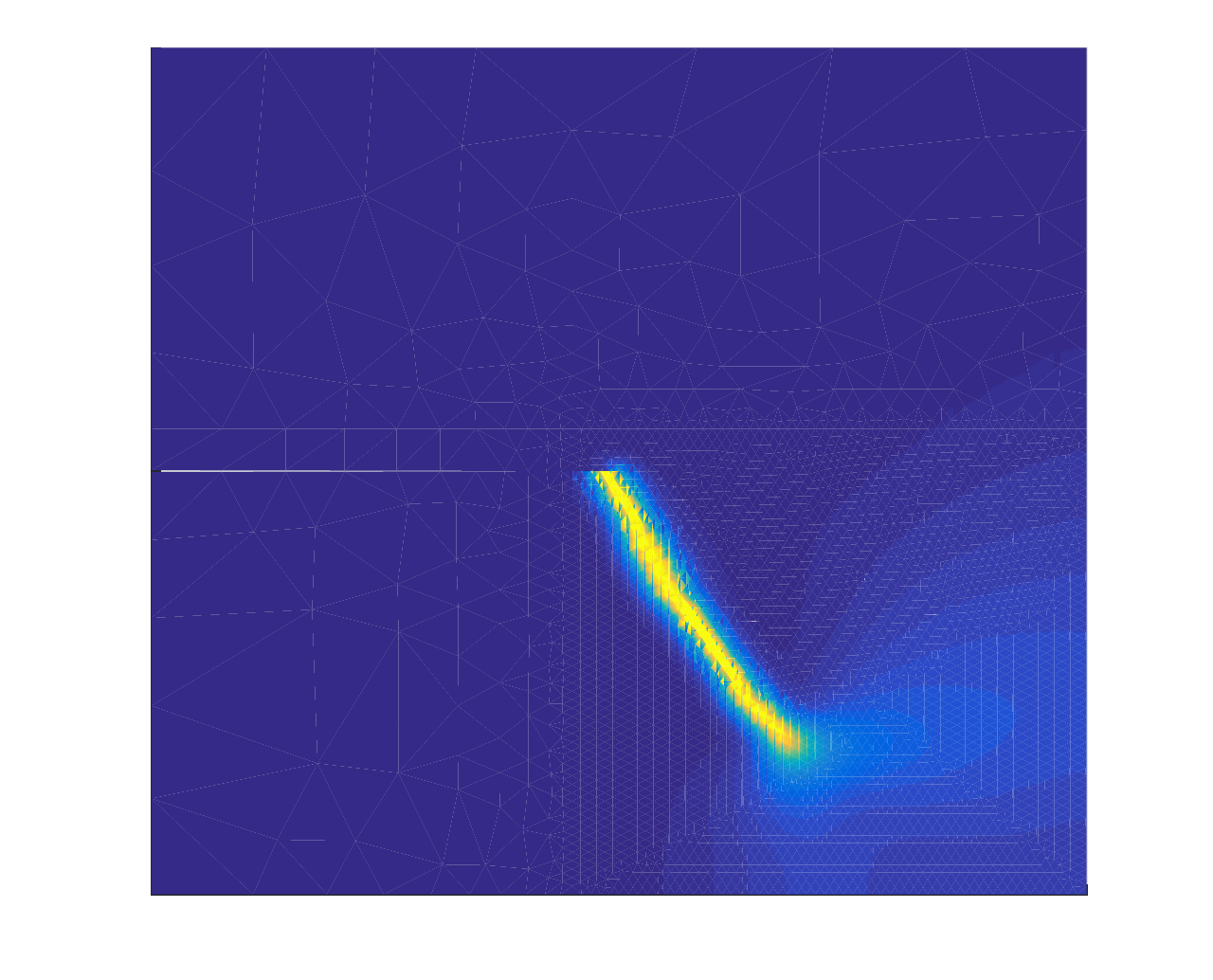}&\includegraphics[width=0.25\textwidth]{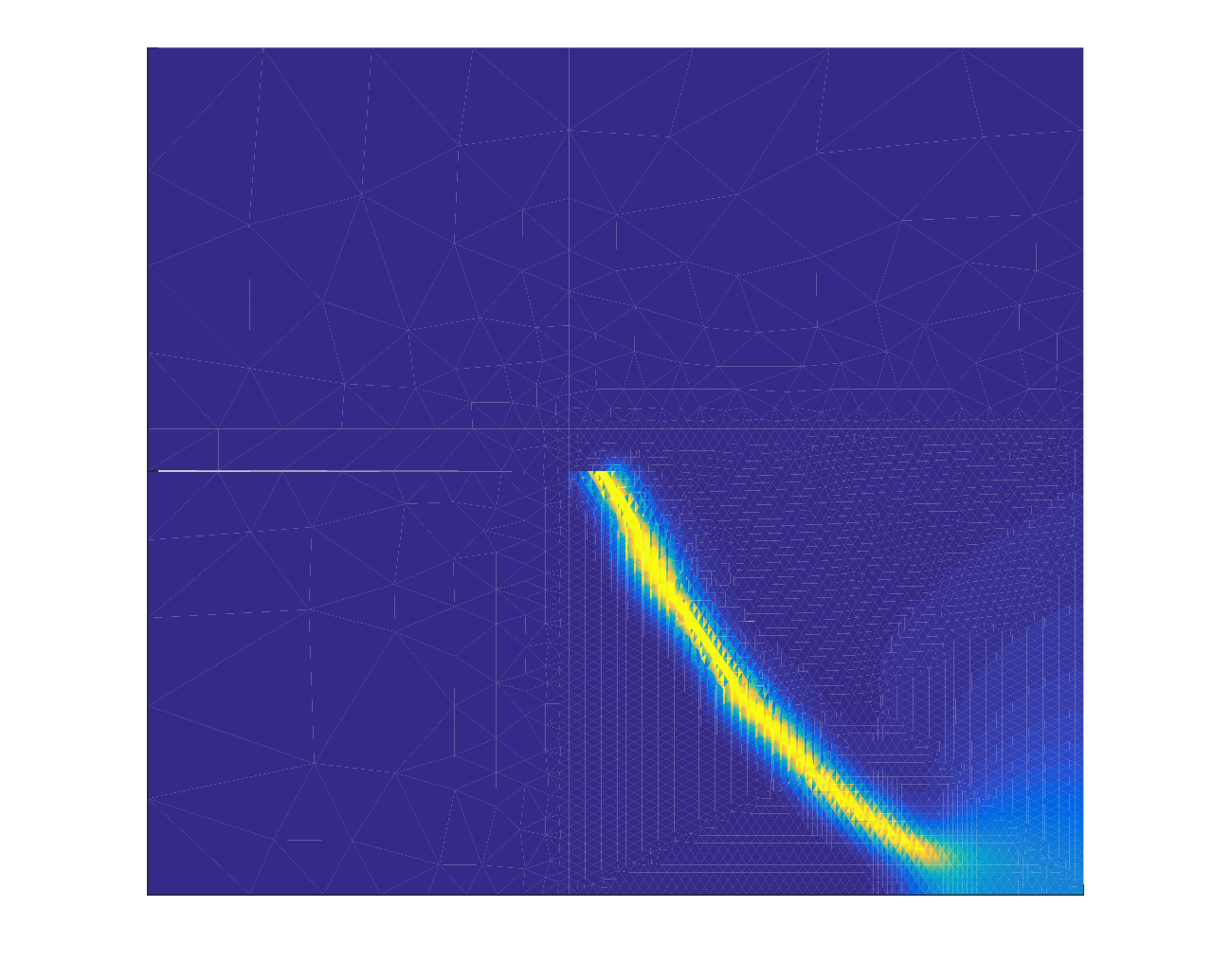}\\
		w=0.0125\,\mathrm{mm} & w=0.0150\,\mathrm{mm}  & w=0.0175\,\mathrm{mm}  & w=0.0200\,\mathrm{mm}\\[1.5ex]
	\text{$(b)$ With backtracking}&&&\\
		\includegraphics[width=0.25\textwidth]{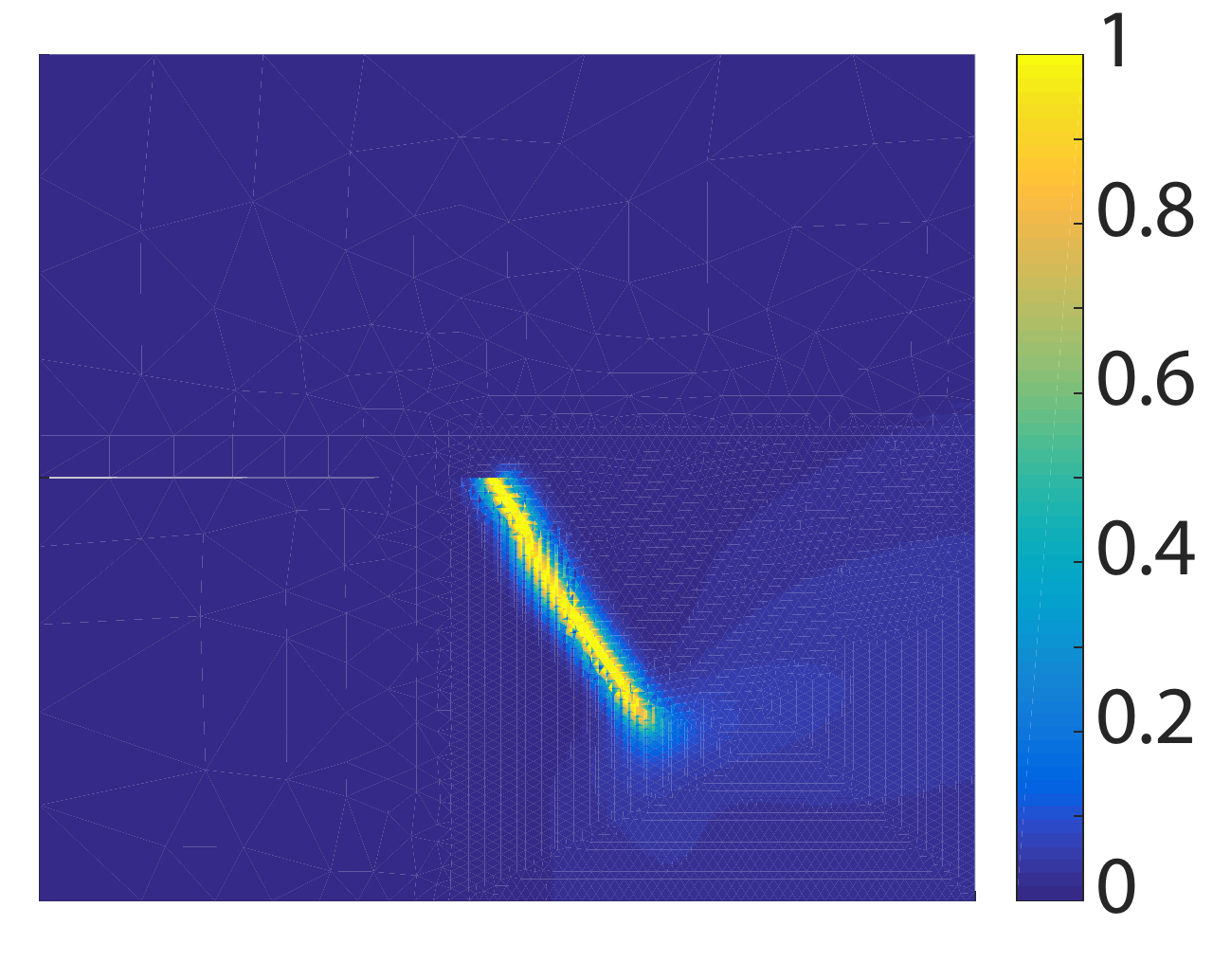}&\includegraphics[width=0.25\textwidth]{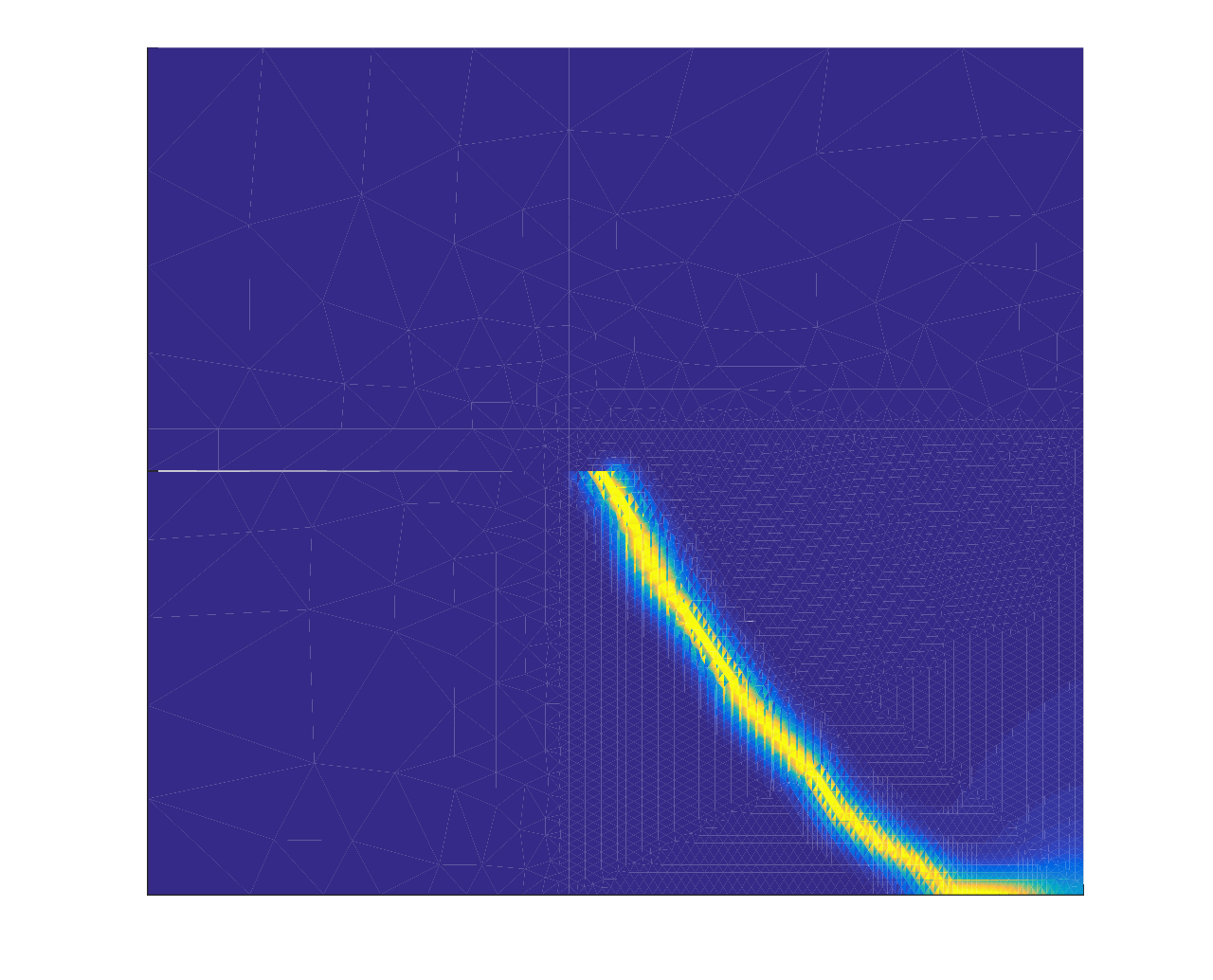}&		  
		\includegraphics[width=0.25\textwidth]{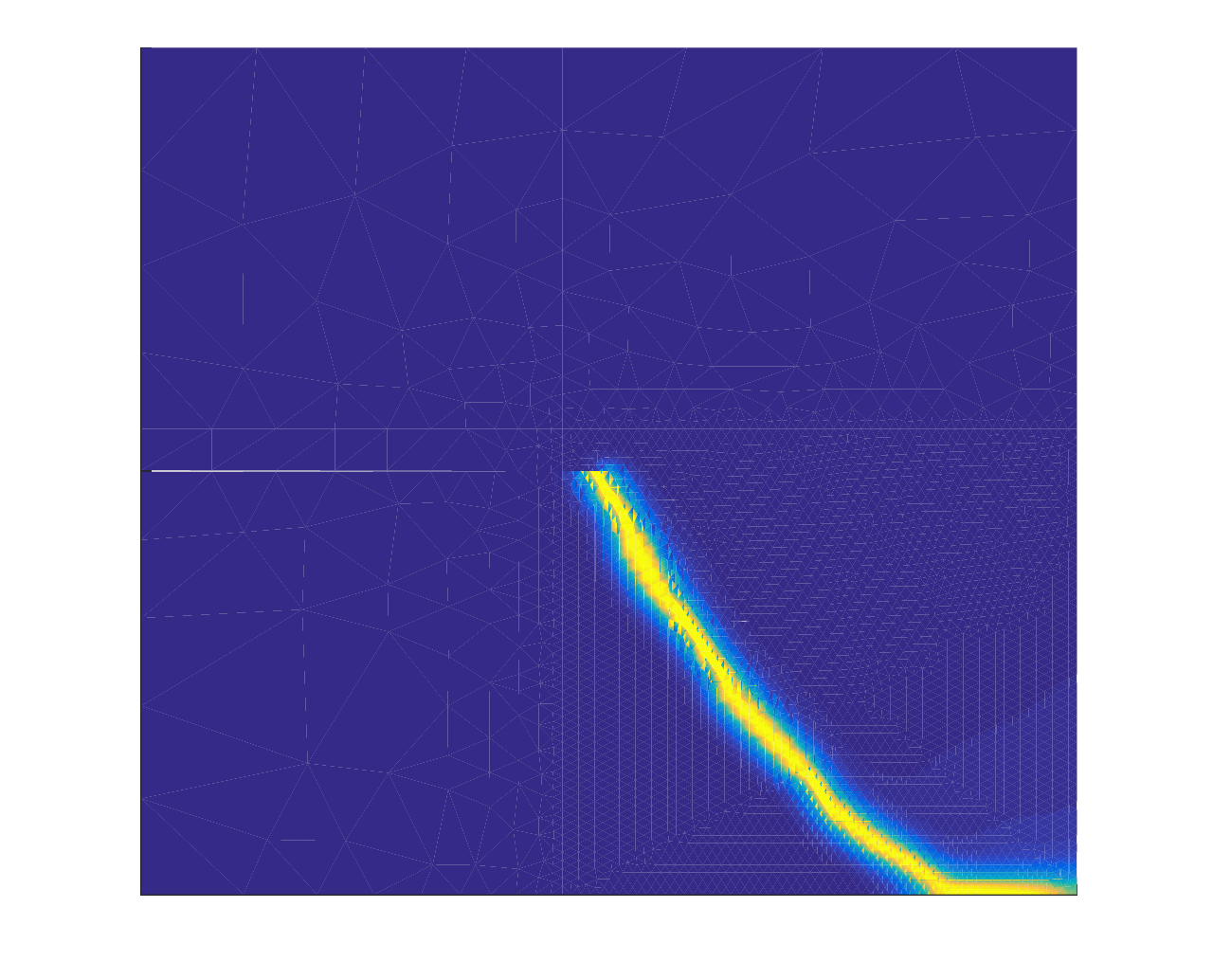}&\includegraphics[width=0.25\textwidth]{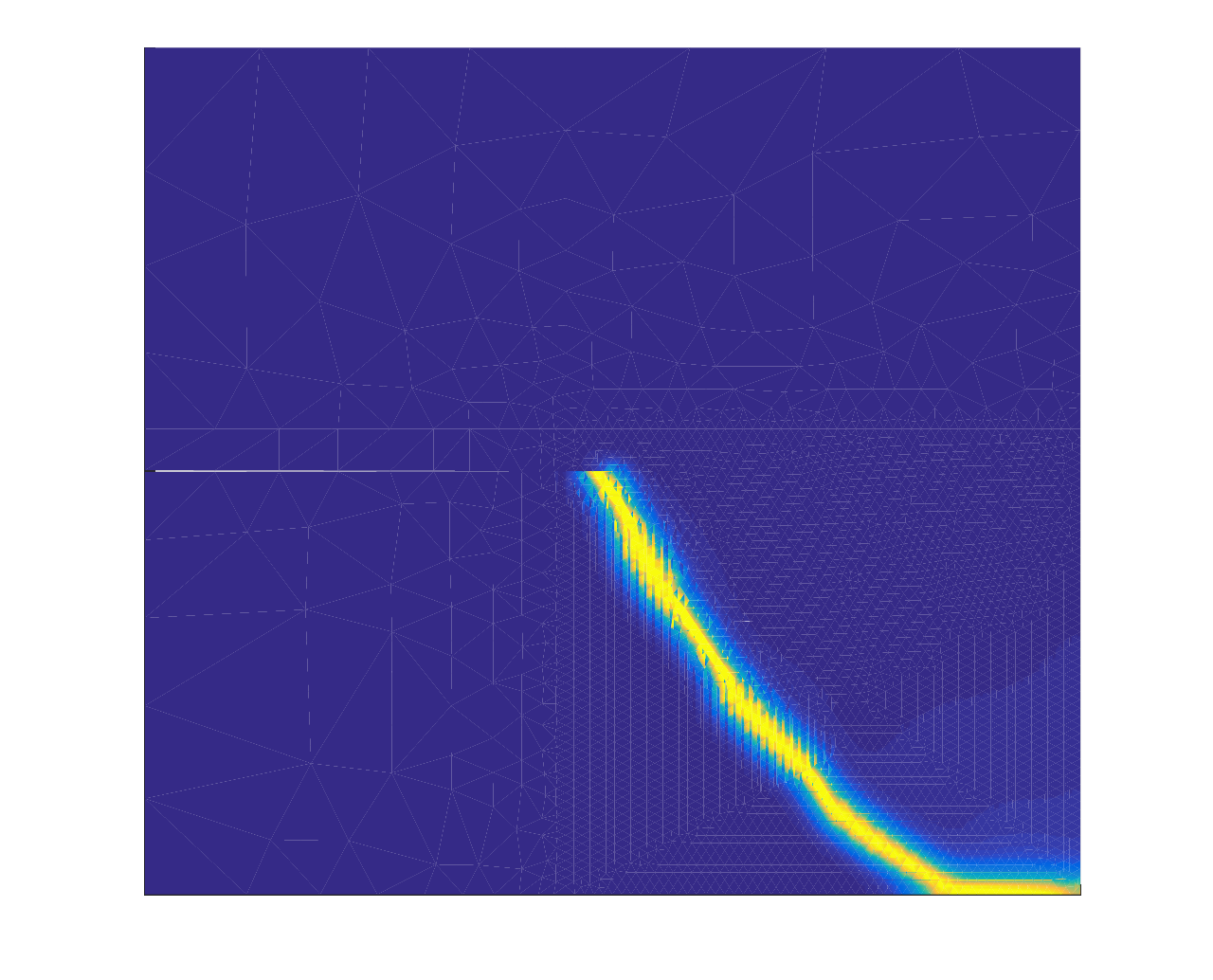}\\
		w=0.0125\,\mathrm{mm} & w=0.0150\,\mathrm{mm}  & w=0.0175\,\mathrm{mm}  & w=0.0200\,\mathrm{mm}\\
       \end{array}$
	}
	\caption{\label{NE10}
	Example \ref{Sec.NumEx:Ex2}. Single edge notched shear test.
	Phase field distribution at different stages of the total displacement $w$ applied on the specimen top edge.
	Results for $\Delta w=10^{-5}\,\mathrm{mm}$. 
	In the damage maps, the yellow corresponds to $\beta=1-\delta$ 
	with $\delta=10^{-4}$ given that we are considering a partially damage profile, 
	whereas the blue corresponds to solid material for which $\beta=0$. 
	}
\end{figure}


\subsection{Three dimensional $L-$shaped panel test}\label{Sec.NumEx:Lshape}

We analyze now the $3d$ crack propagation in an $L$--shaped concrete panel 
as benchmark for crack initiation  \cite{AGDl15,BWBNR20,GDl19,MBK15} and, likewise Example \ref{Sec.NumEx:Ex2}, 
to demonstrate the ability 
of the phase-field variational formulation to describe curved crack patterns.
The geometry and boundary conditions are displayed in 
Figure \ref{Ex:Lshape:GeoMesh}$(a)$, and correspond to the experimental setup given in \cite{Win01}.
All the points of the face of equation $y=0$ are fully restrained whereas those belonging to the line
of equation $x=470\,\mathrm{mm}$, $y=250\,\mathrm{mm}$ and $0\,\mathrm{mm}\leq z\leq 100\,\mathrm{mm}$ present prescribed values for $v$
and free the other degrees of freedom.

\begin{figure}[H]
	\centerline{$\begin{array}{cc}
		\includegraphics[width=0.45\textwidth]{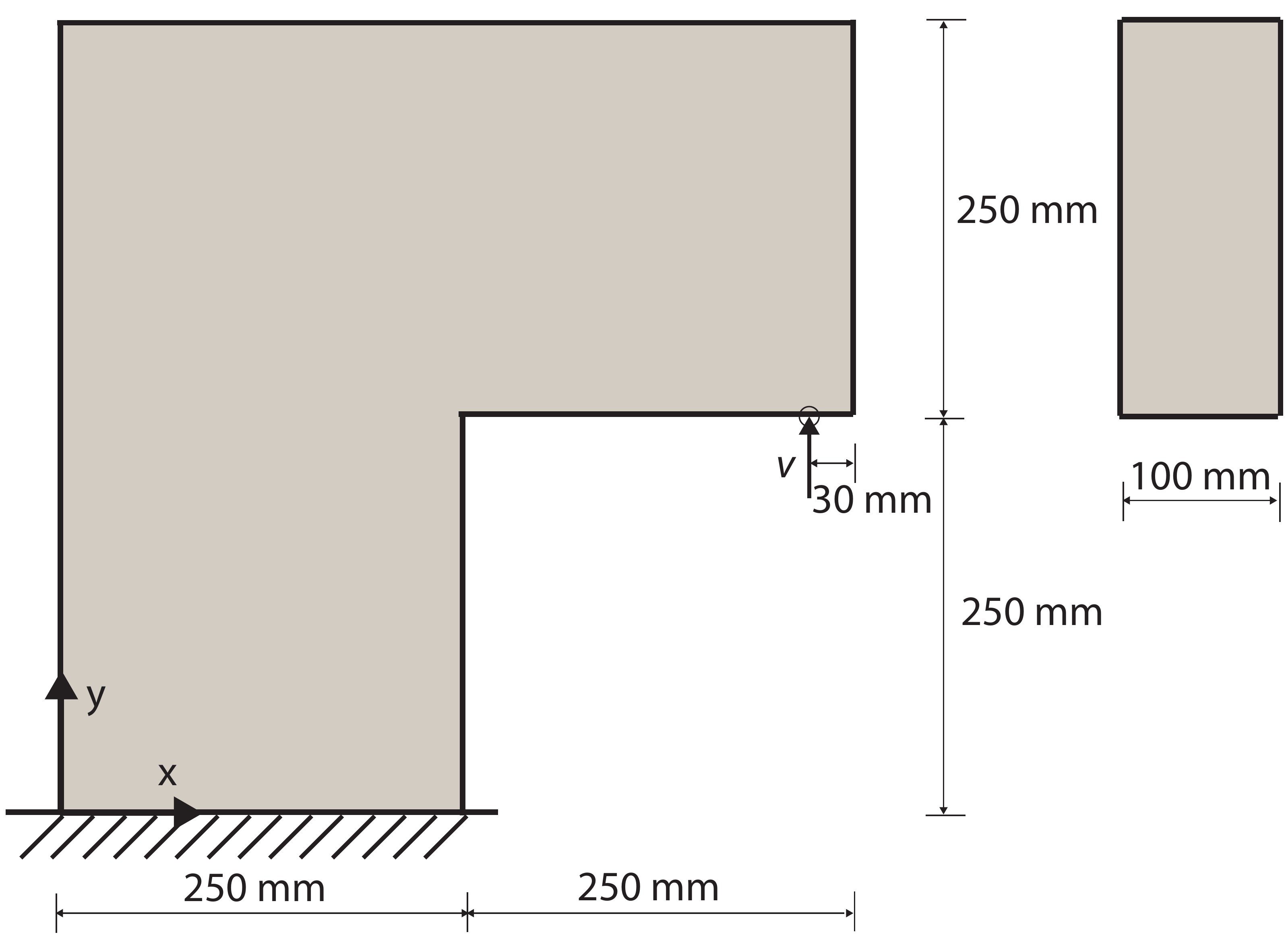}&
		\includegraphics[width=0.40\textwidth]{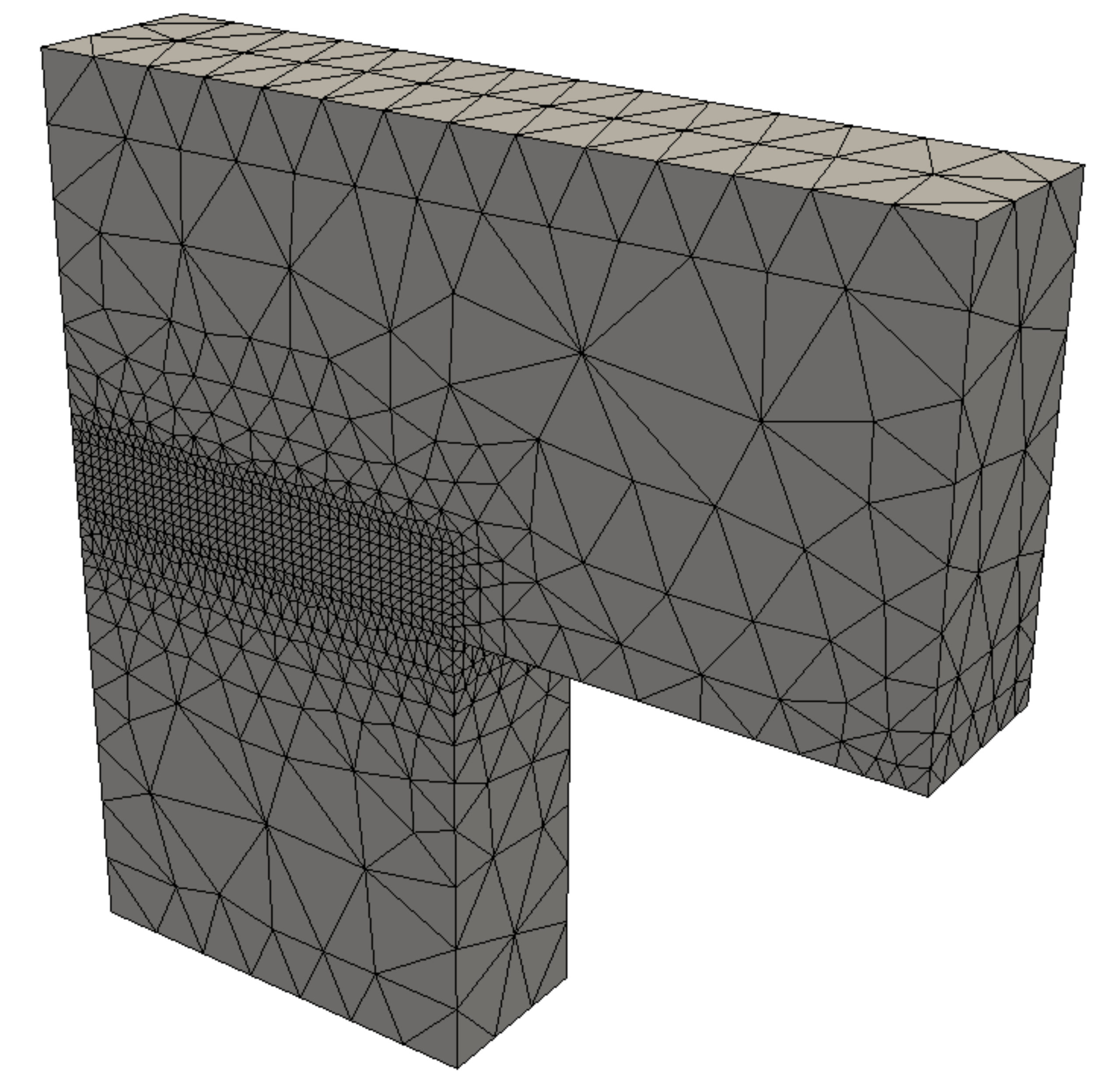}\\
		(a)&(b)
		\end{array}$
		}
	\caption{\label{Ex:Lshape:GeoMesh}
		Example \ref{Sec.NumEx:Lshape}. $3d$ $L-$shaped panel test.
		$(a)$ Geometric setup.
		$(b)$ Unstructured finite element mesh.
		}
\end{figure}

The concrete material properties chosen are the same as given in \cite{Win01} with the Young modulus  
$E=25.85\,\mathrm{kN/mm^2}$, the Poisson ratio $\nu=0.18$, the critical energy release rate $g_c=0.095\,\,\mathrm{N/mm}$ and the internal length 
$\ell=20\,\mathrm{mm}$.
The unstructured finite element mesh is shown in Figure \ref{Ex:Lshape:GeoMesh}$(b)$ and consists of $44880$ tetrahedral 
elements and $68470$ nodes. No initial crack is prescribed. However, since
we expect that this starts at the interior corner of the $L-$shape, we 
have refined therein the mesh with a characteristic finite element length equal to $h=6.25\,\mathrm{mm}<\ell/2$
in order to resolve properly the crack pattern. 

\begin{figure}[H]
	\centering{$\begin{array}{c}
		\includegraphics[width=0.65\textwidth]{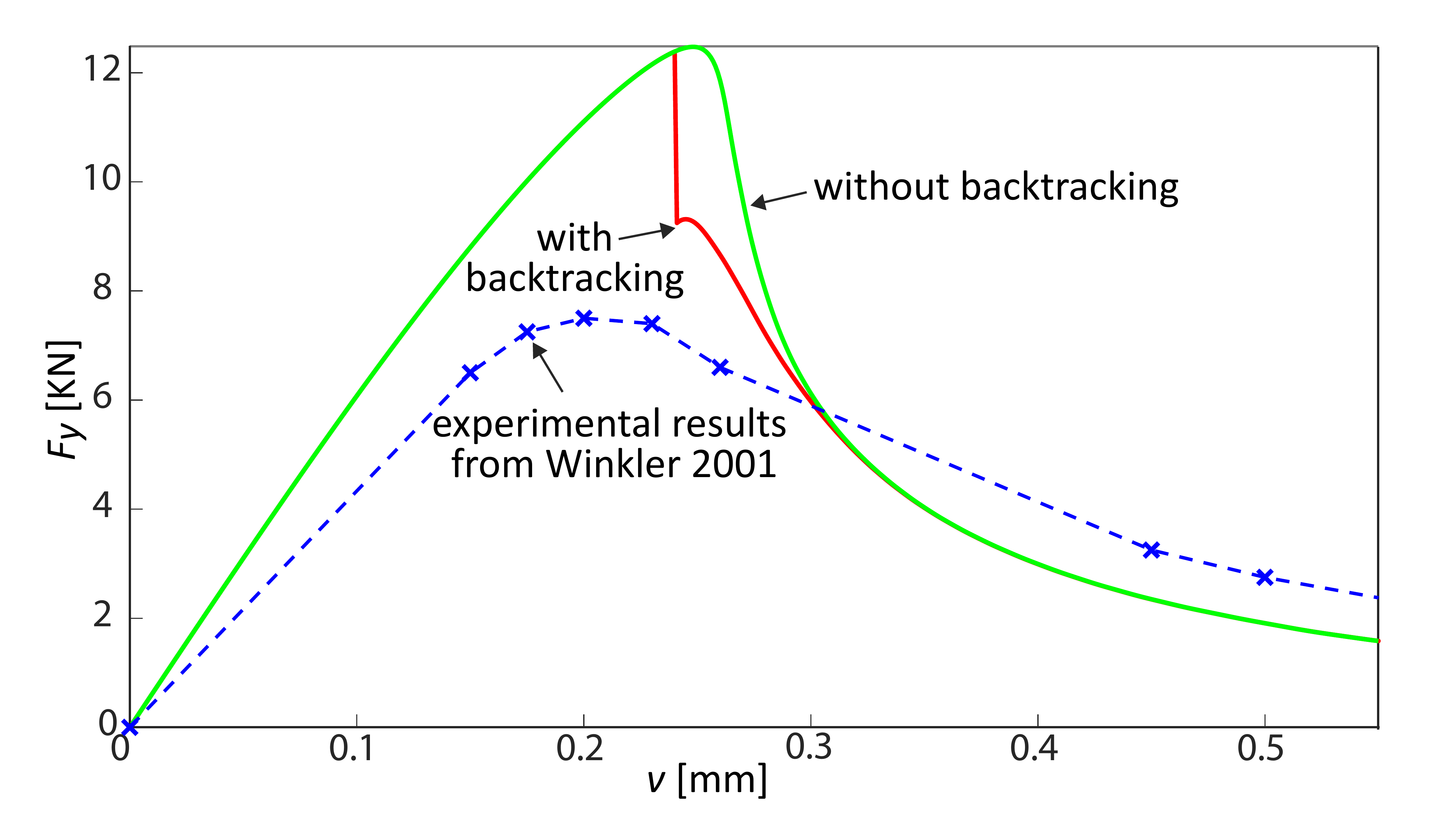} 
		\end{array}$}
	\caption{\label{Ex:Lshape:FvDis}
	Example \ref{Sec.NumEx:Lshape}. $3d$ $L-$shaped panel test.
	Load--displacement curves in the case the backtracking algorithm is activated and for the case 
	without applying the backtracking algorithm.}
\end{figure}

The numerical simulations have been carried out by applying a monotone loading history for the nonhomogeneous Dirichlet boundary 
condition $v$ by means of the application of constant displacement increments $\Delta v=10^{-3}\,\mathrm{mm}$.
For this example, the penalization factor $\epsilon$ has been set equal to $10^{-4}$.
Figure \ref{Ex:Lshape:FvDis} displays the resulting load-displacement curves with $(K>0)$ and without $(K=0)$ the backtracking option active.
We observe that the two curves are practically identical until the peak, but they then differentiate each other for a short 
range of the applied displacement $v$ in the post peak, where we verify only a minor occurrence of backtracking,
for then to display again the same behaviour starting from around $v=0.3\,\mathrm{mm}$. Our numerical results compare quite well
with those obtained by \cite{MBK15,BWBNR20}, but they all differentiate in a relevant manner from the experimental findings of \cite{Win01}
in the detection of the peak value and of the residual load. This behaviour was also noted in \cite{MBK15}.
We ascribe the difference of results to the quasi--brittle model we have used for the concrete which does not account for
plastic deformations prior to the damage and for cohesive forces on the crack surfaces. By applying a mixed-mode cohesive crack model but with
an energy--based crack criterion, on the other hand, \cite{MB07} can  obtain good agreement 
with the experiments of \cite{Win01}.

\begin{figure}[H]
	\centering{$\begin{array}{c}
		\includegraphics[width=0.49\textwidth]{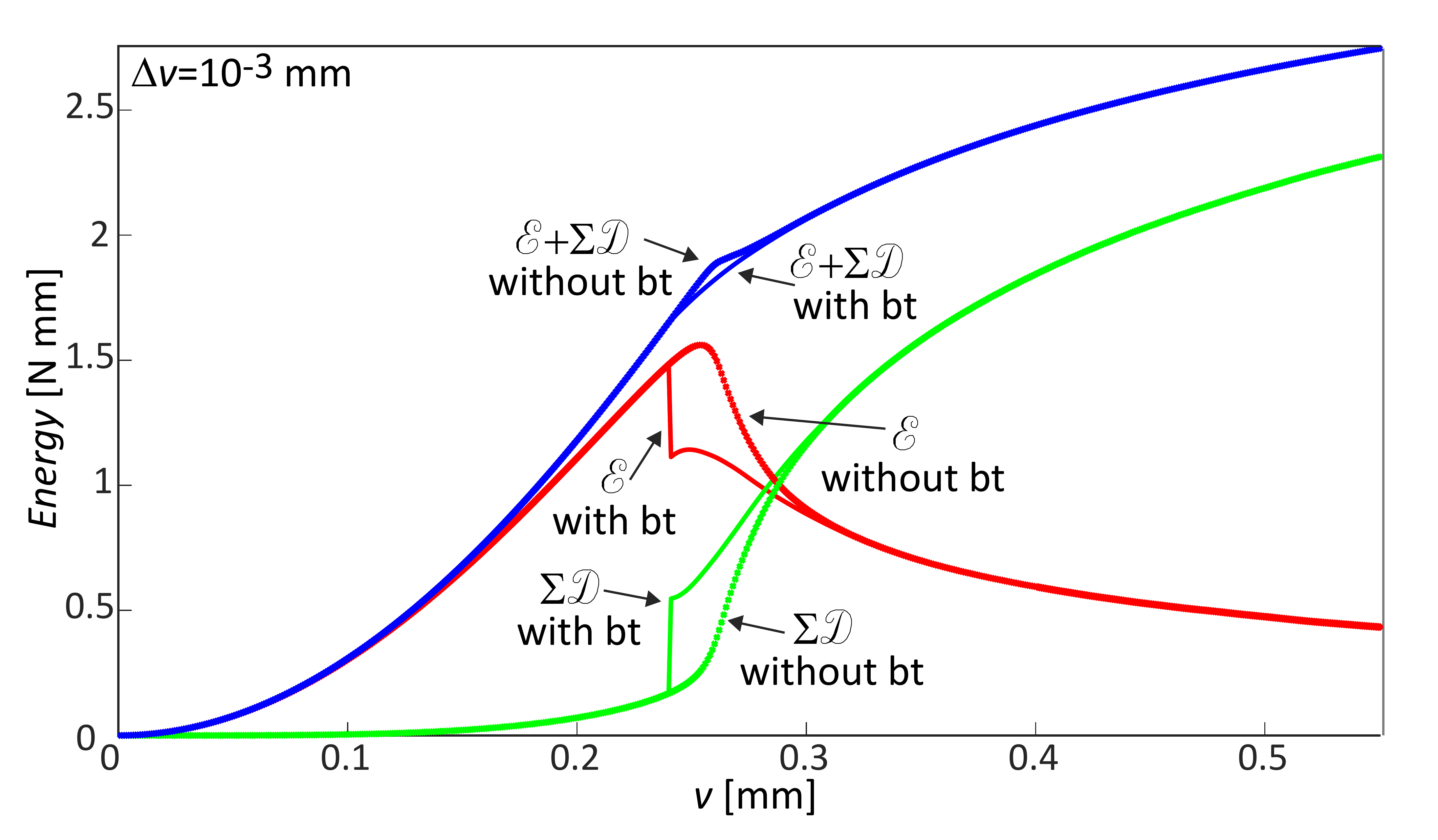}\\
		(a)\\
		\begin{array}{cc}
			\includegraphics[width=0.49\textwidth]{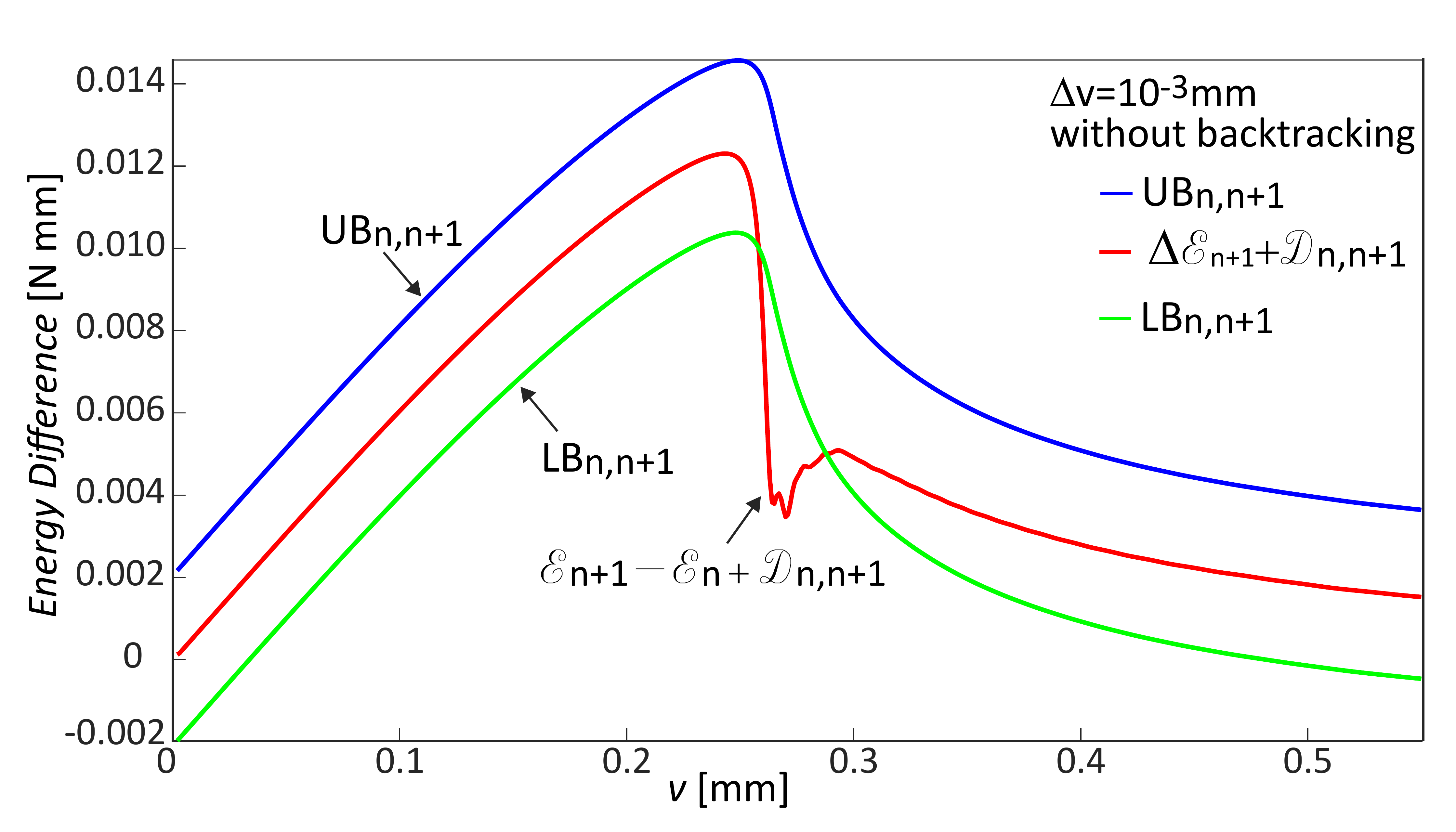}& 
			\includegraphics[width=0.49\textwidth]{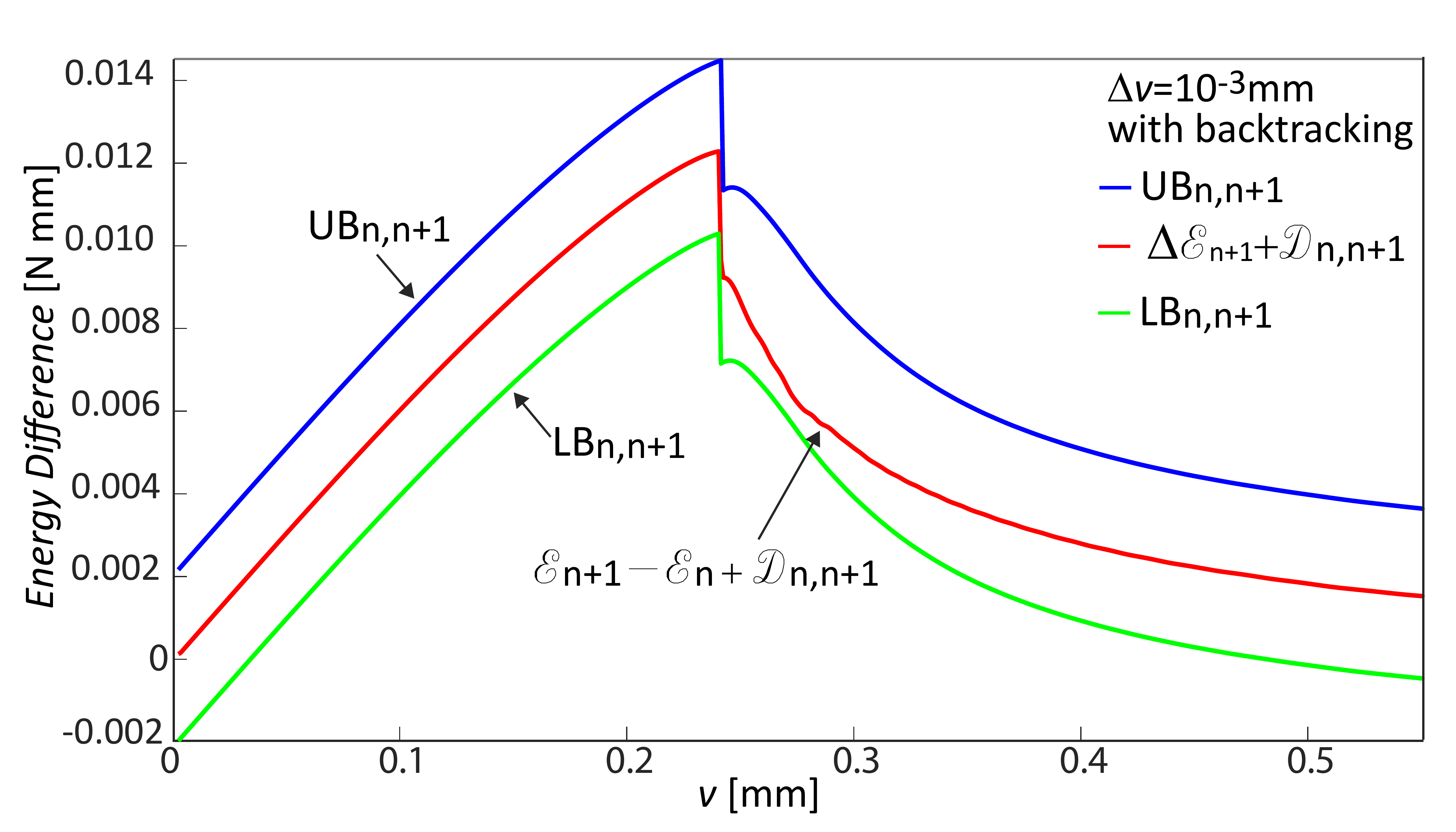}\\
			(b)&(c)
		\end{array}
		\end{array}$}
	\caption{\label{Ex:Lshape:EnrgVrDis}
	Example \ref{Sec.NumEx:Lshape}. $3d$ $L-$shaped panel test.
	Results for $\Delta v=10^{-3}\,\mathrm{mm}$.
	$(a)$ Evolution of the total energy $\mathcal{E}(t_{n+1},\bmU_{n+1},\bmA_{n+1})+\sum_{i=0}^n\mathcal{D}(\bmA_i,\bmA_{i+1})$ 
	for $n=0,1,\ldots, N-1$, without  backtracking $(K=0)$ and with backtracking.
	Evolution of the total incremental energy given by $\mathcal{E}_{n+1}-\mathcal{E}_{n}+\mathcal{D}_{n,n+1}$ 
	and of the lower $LB_{n,n+1}$ and upper bound $UB_{n,n+1}$ which enter the two-sided energy estimate \eqref{Disc.TwoSidIneq},
	$n=0,1,\ldots, N-1$, for the scheme $(b)$ without backtracking and $(c)$ with backtracking.
	}
\end{figure}

The confirmation of the aforementioned behaviour is obtained by analysing the evolution of the total energetics of the 
computed solutions displayed in Figure \ref{Ex:Lshape:EnrgVrDis}. The discrete computed solutions obtained by the 
the alternate minimization method without backtracking fails to yield approximate energetic solutions.
Figure \ref{Ex:Lshape:EnrgVrDis}$(b)$ shows that the two-sided energy inequality \eqref{Disc.TwoSidIneq} 
is satisfied only in the initial stage when the specimen stays mainly elastic 
and in the last stage when the specimen experiences the same damage pattern, that is, when the algorithm
fall back to lower energy states, whereas it is violated for other values of $v$.
With the backtracking option active, by contrast, 
Figure \ref{Ex:Lshape:EnrgVrDis}$(c)$ shows that the alternate minimization is capable of detecting a lower energy 
path during the whole evolution which is defined by the approximate energetic solutions 
that meet the two--sided energy  inequality.
Finally, Figure \ref{Ex:Lshape:DmgMap} displays the phase-field distribution on the plane $z=50\,\mathrm{mm}$ 
at different stages of the displacement $v$ for the two numerical schemes verifying 
a faster evolution of the damage with the backtracking algorithm when compared with the basic scheme. 

\begin{figure}[H]
	\centering{$\begin{array}{ccc}
	\text{$(a)$ Without backtracking}&&\\
		\includegraphics[width=0.28\textwidth]{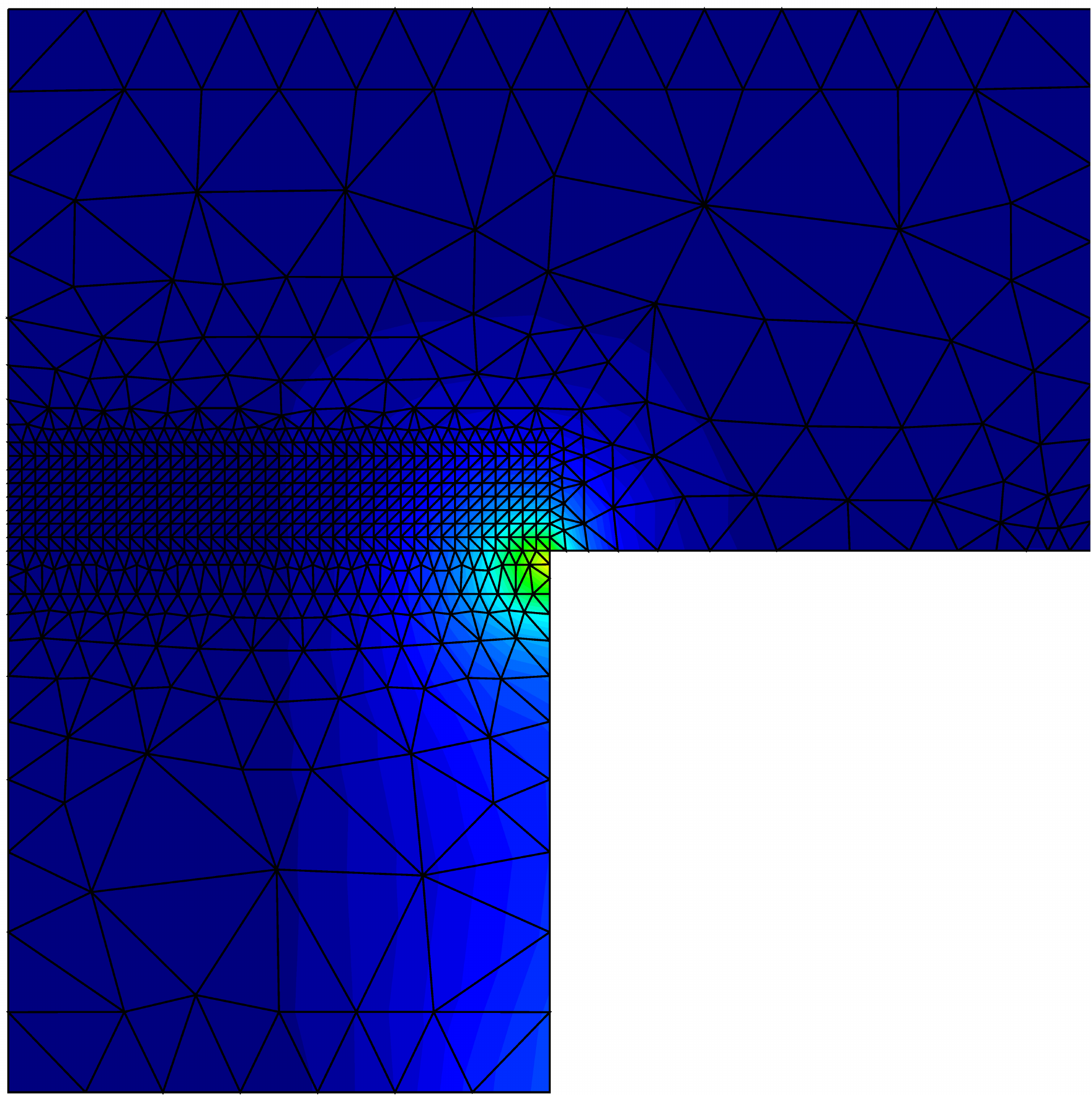}
		\includegraphics[width=0.045\textwidth]{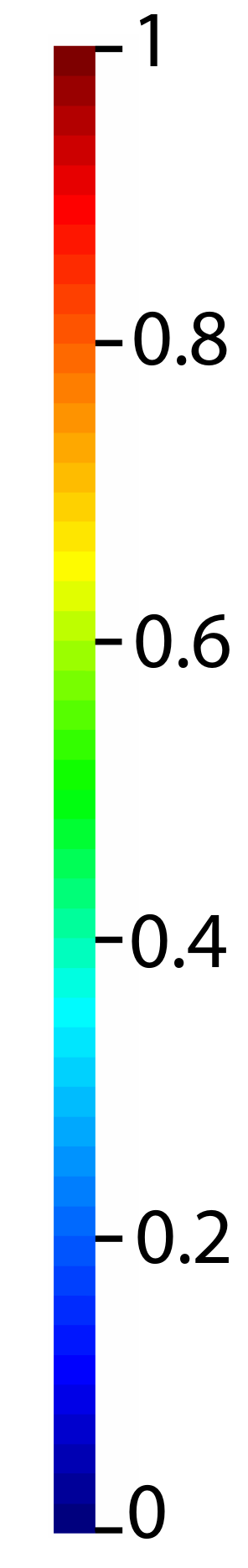}&
		\includegraphics[width=0.28\textwidth]{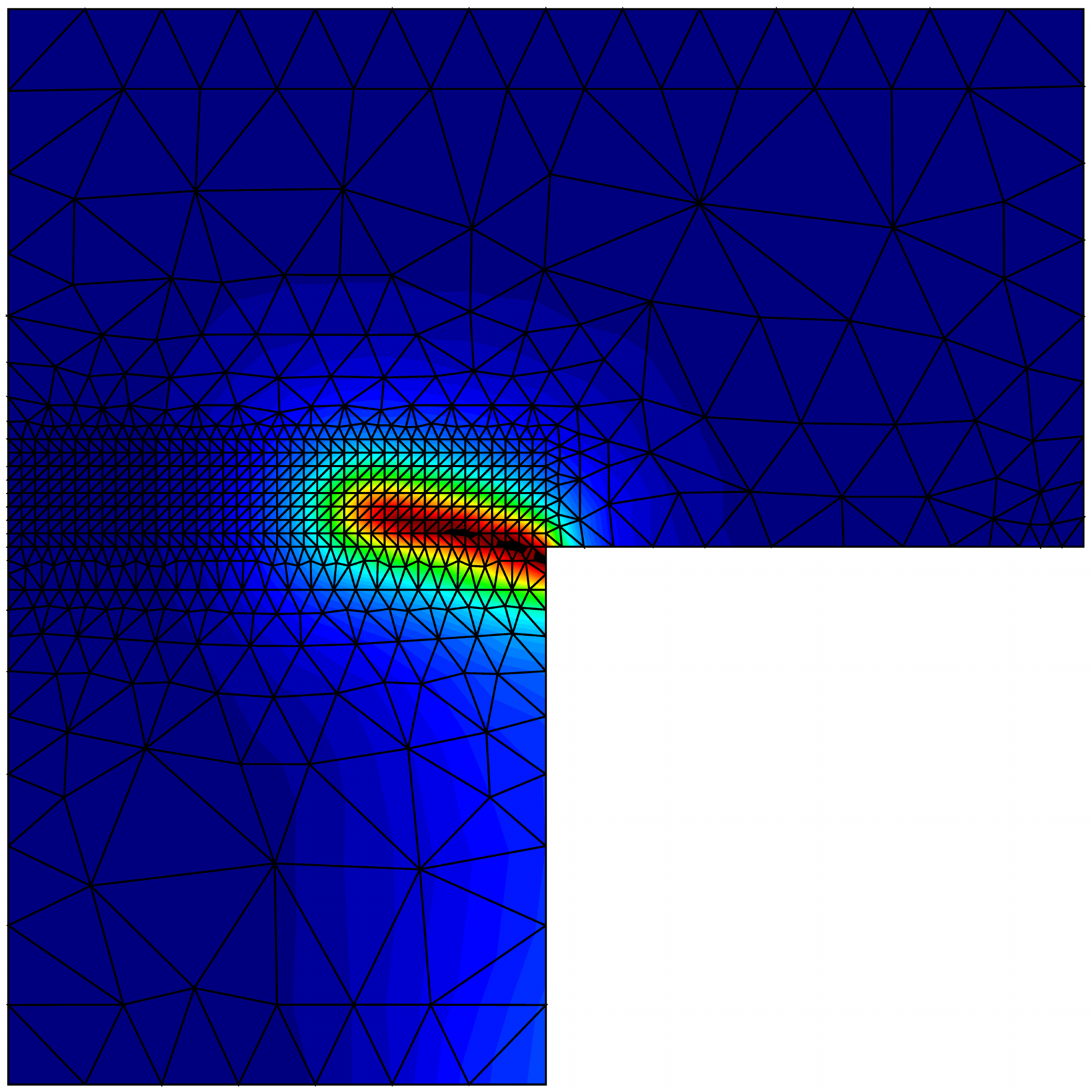}&		  
		\includegraphics[width=0.28\textwidth]{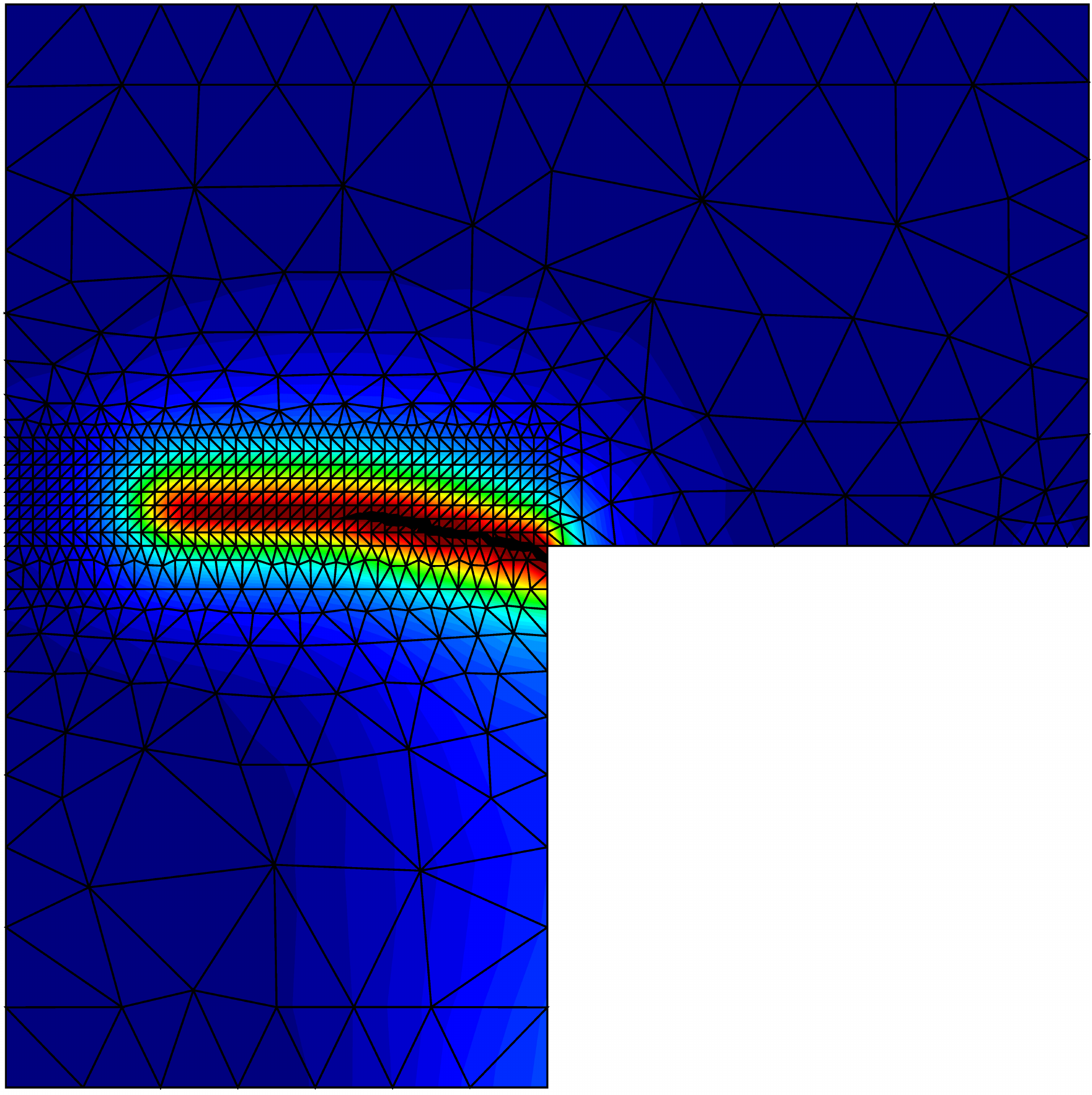}\\
		\includegraphics[width=0.28\textwidth]{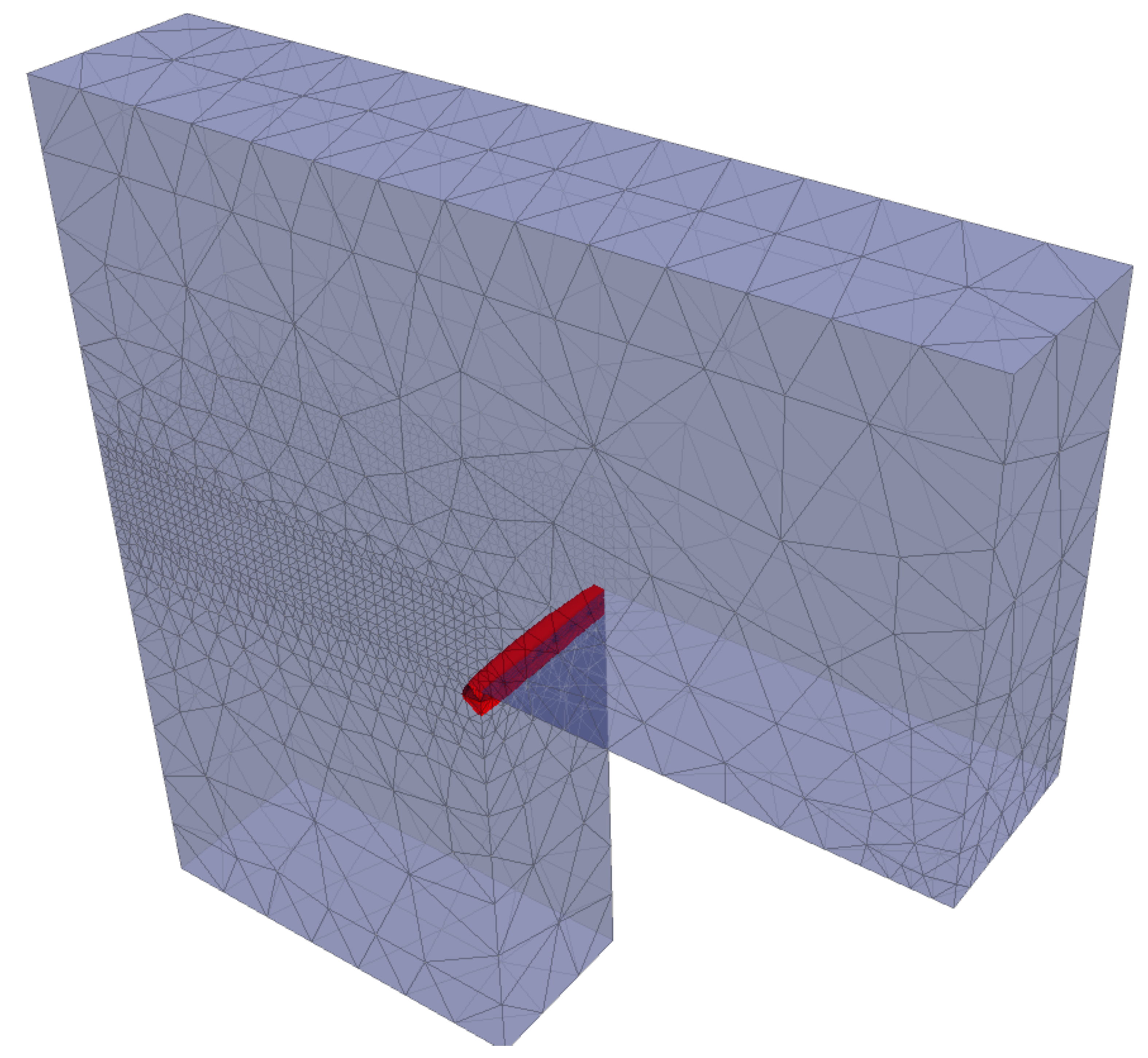}&
		\includegraphics[width=0.28\textwidth]{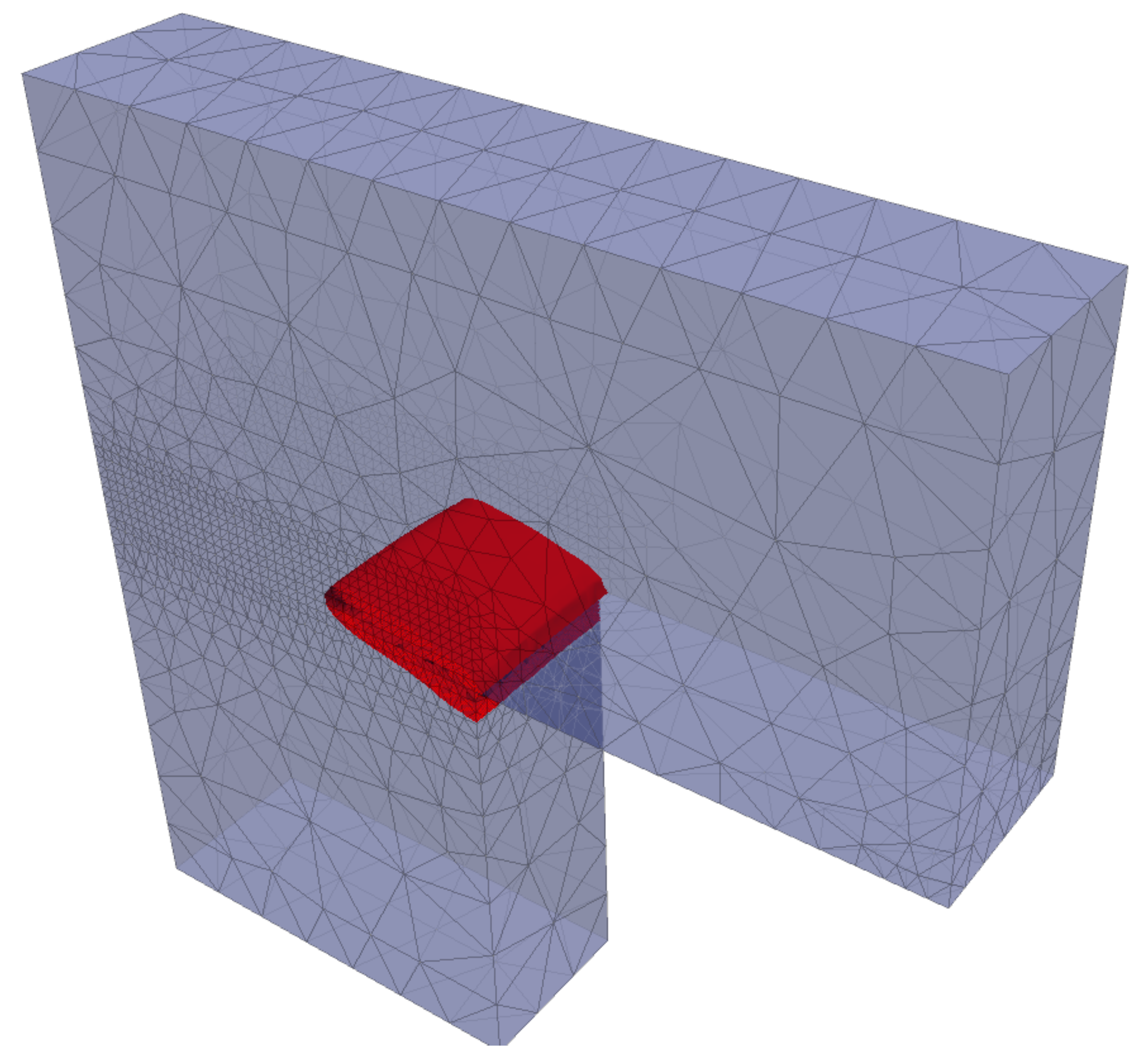}&
		\includegraphics[width=0.28\textwidth]{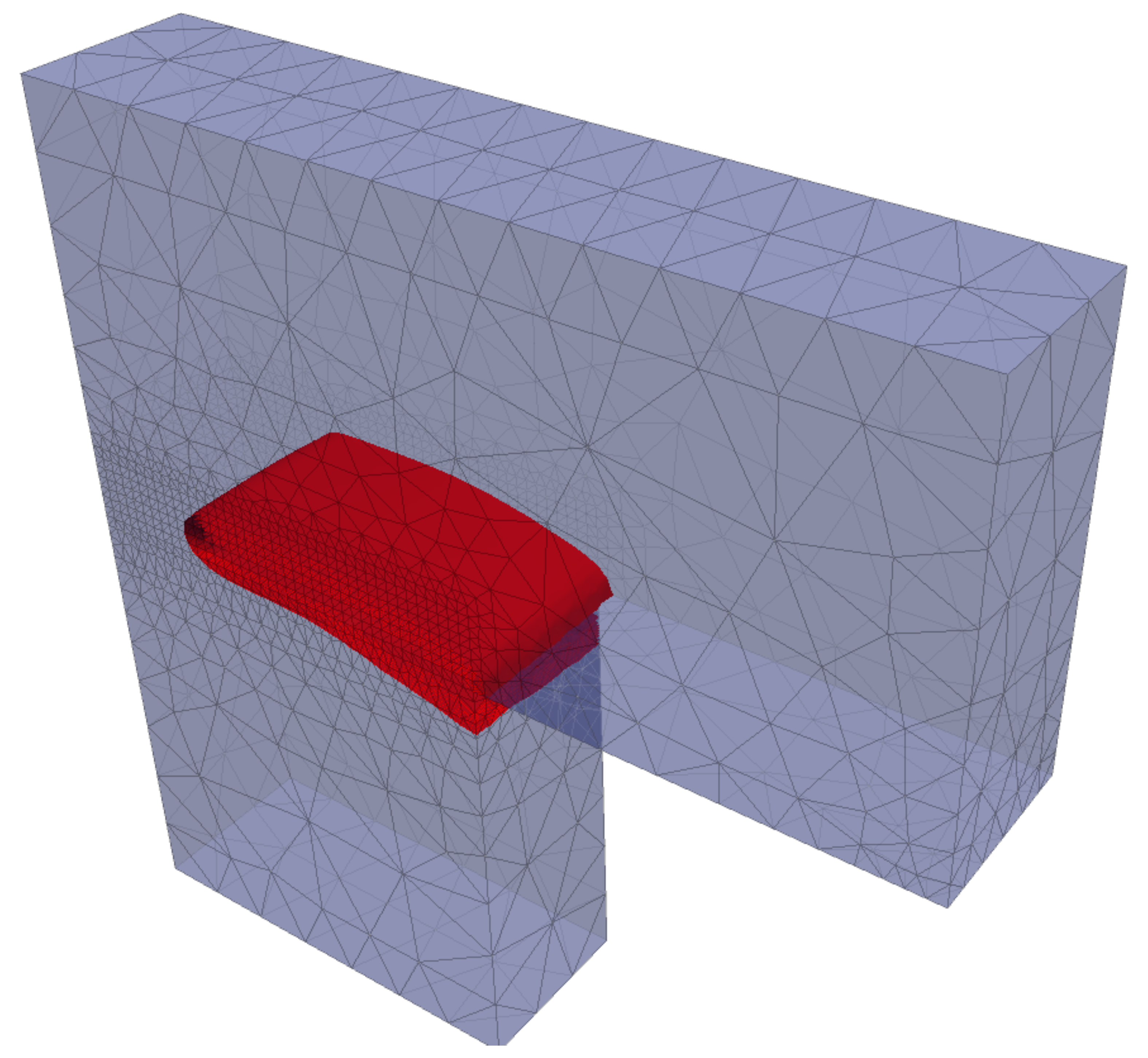}\\		
		v=0.25\,\mathrm{mm} & v=0.30\,\mathrm{mm}  & v=0.5\,\mathrm{mm}  \\\\[1.5ex]
	\text{$(b)$ With backtracking}&&\\
		\includegraphics[width=0.28\textwidth]{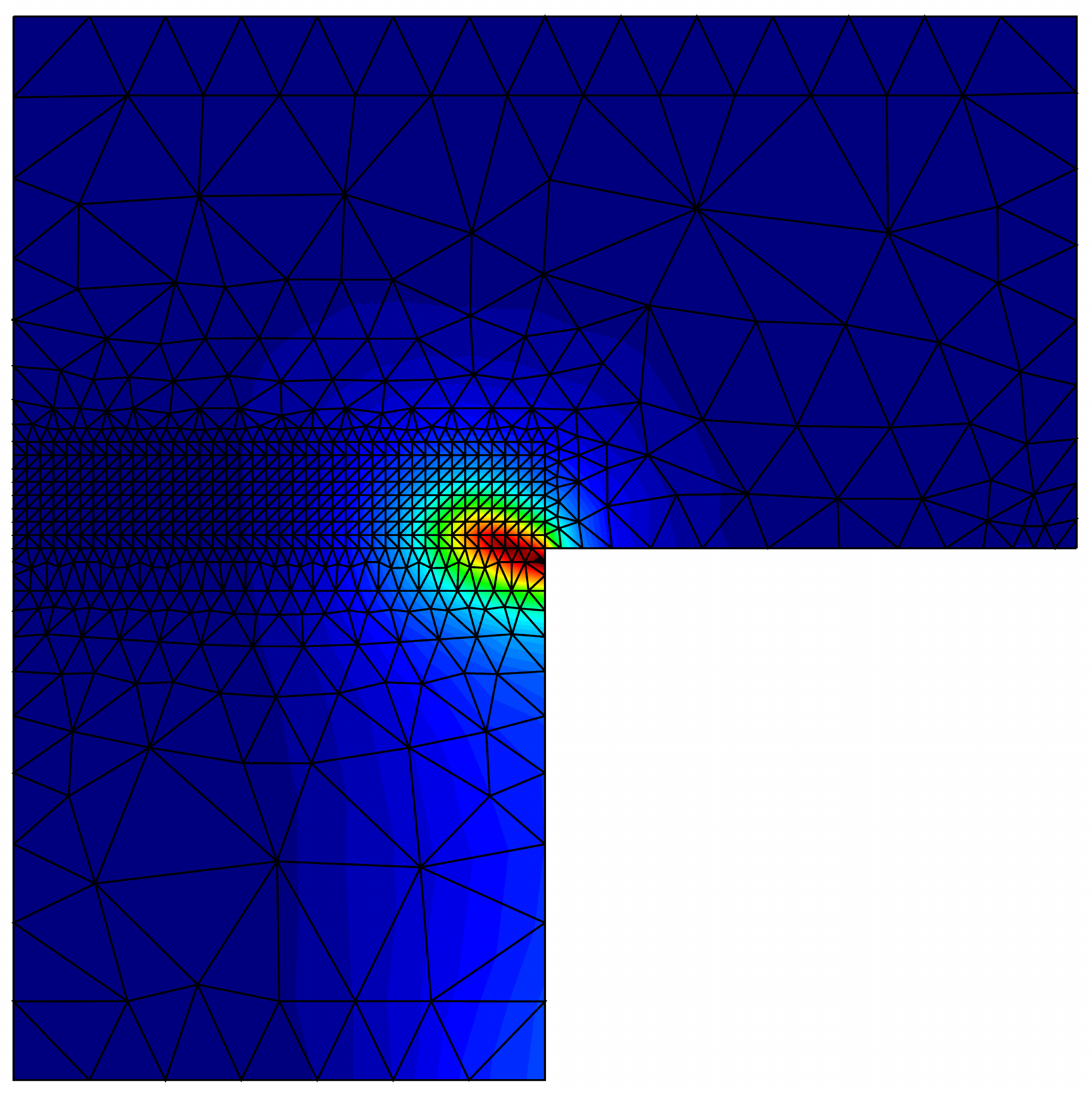}
		\includegraphics[width=0.045\textwidth]{LegendDamageMap3d.pdf}&
		\includegraphics[width=0.28\textwidth]{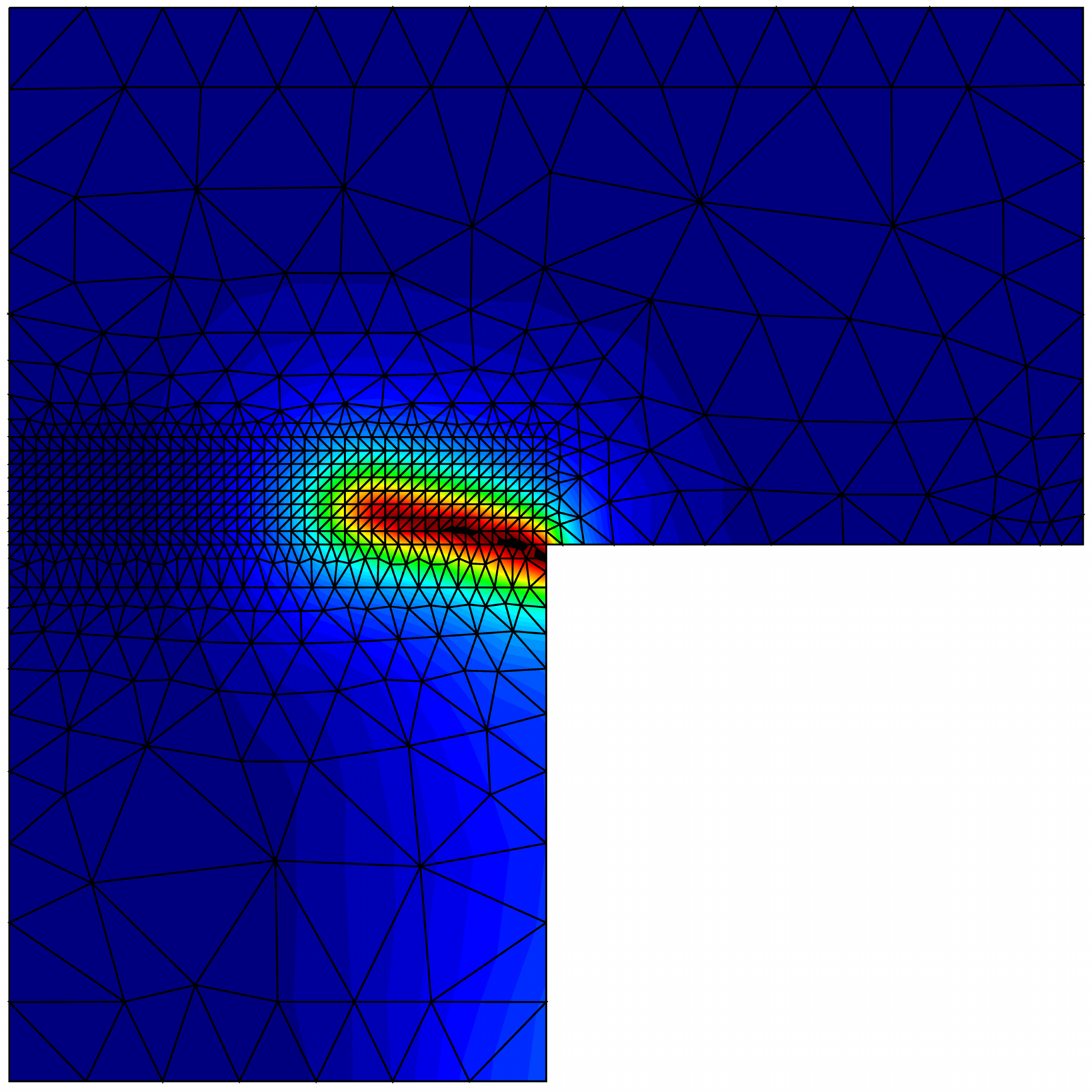}&		  
		\includegraphics[width=0.28\textwidth]{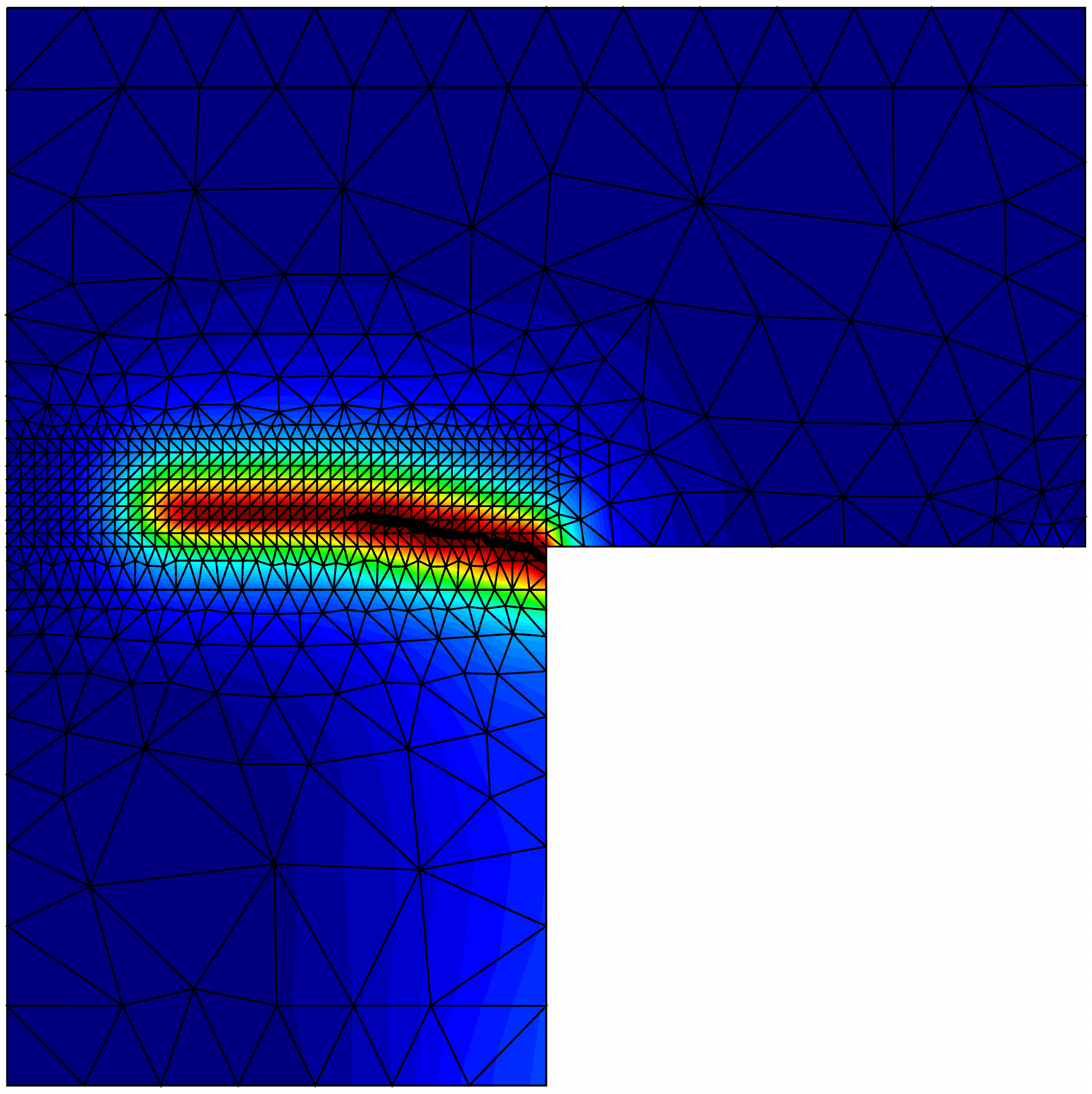}\\
		\includegraphics[width=0.28\textwidth]{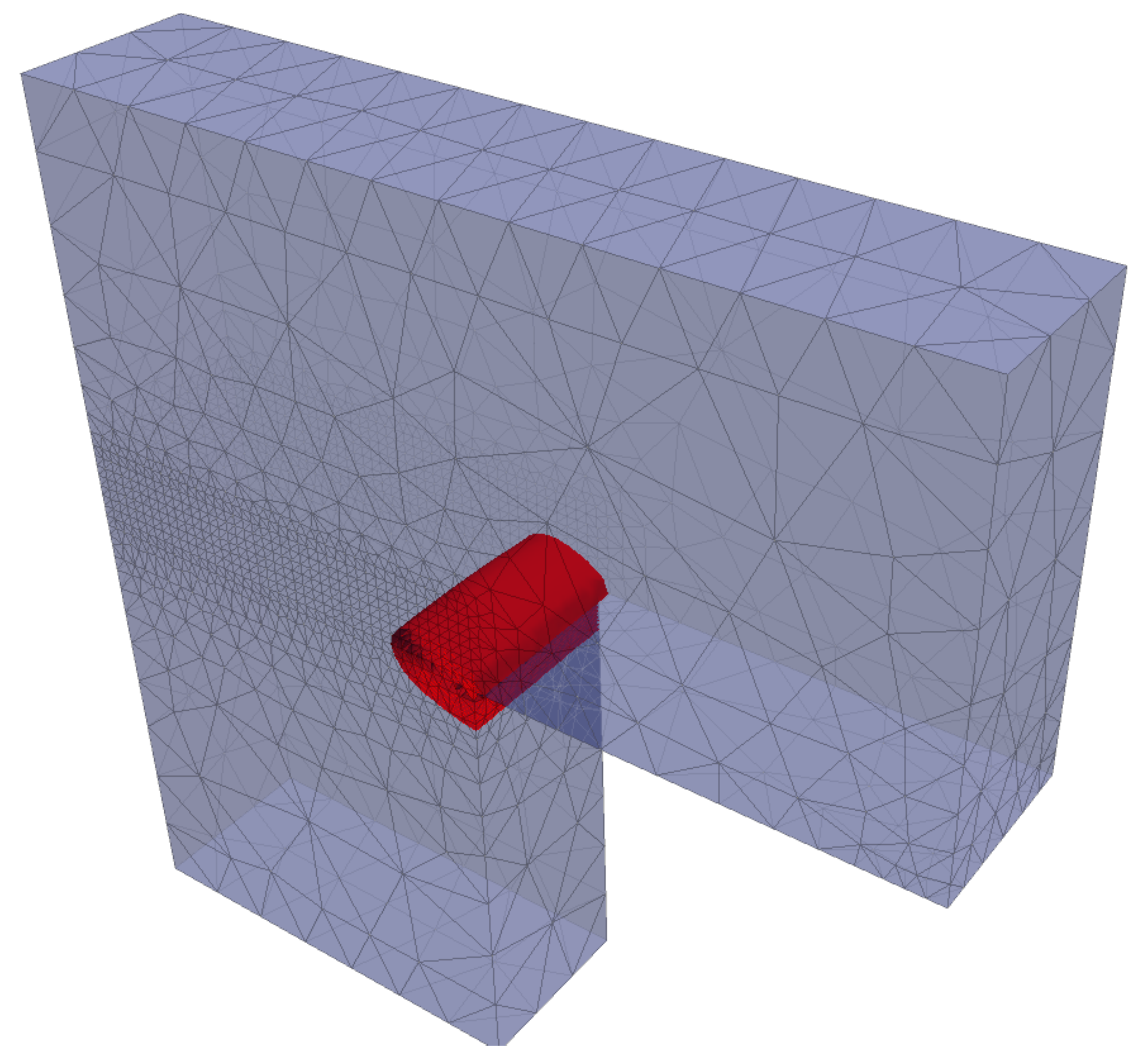}&
		\includegraphics[width=0.28\textwidth]{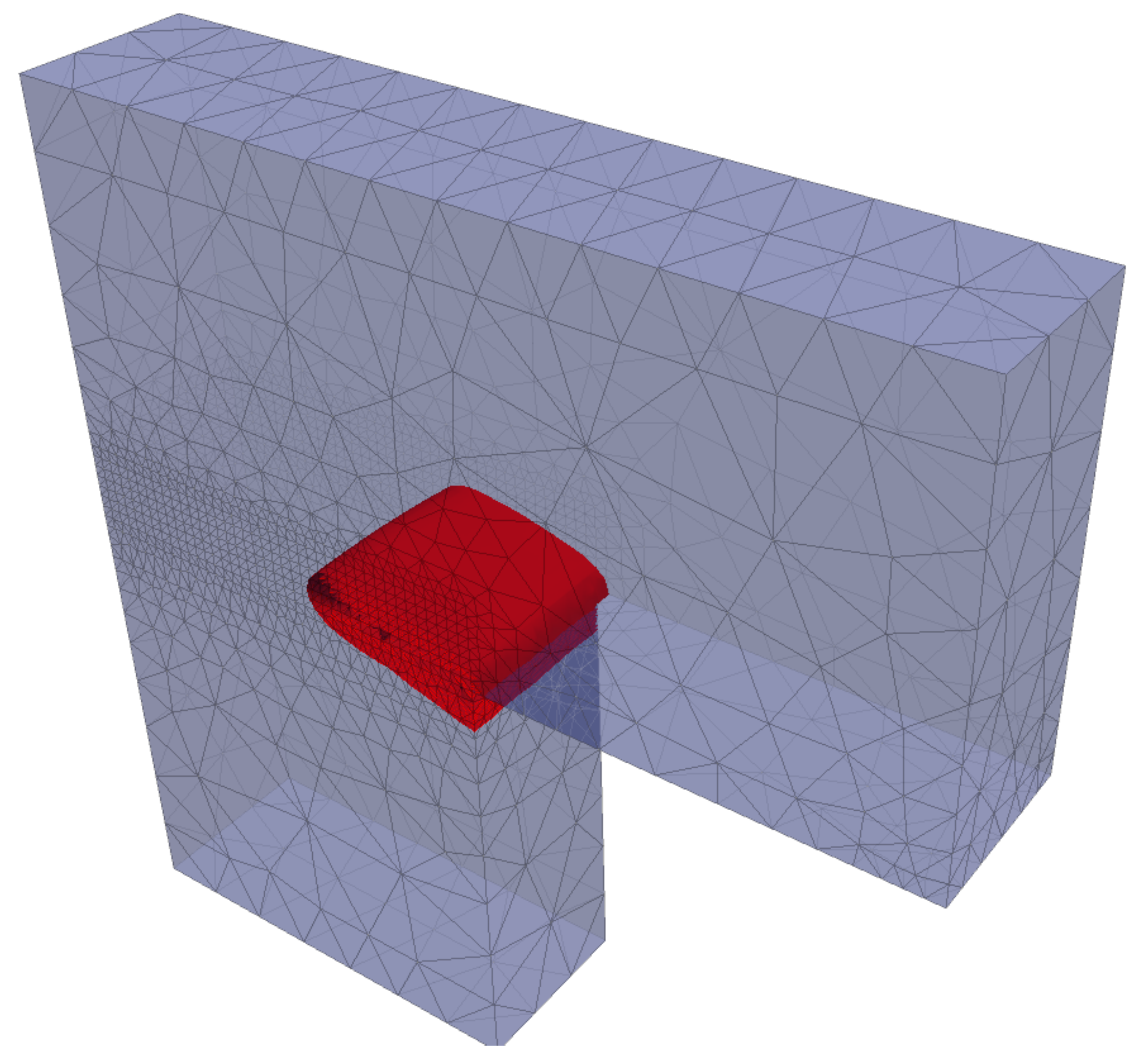}&
		\includegraphics[width=0.28\textwidth]{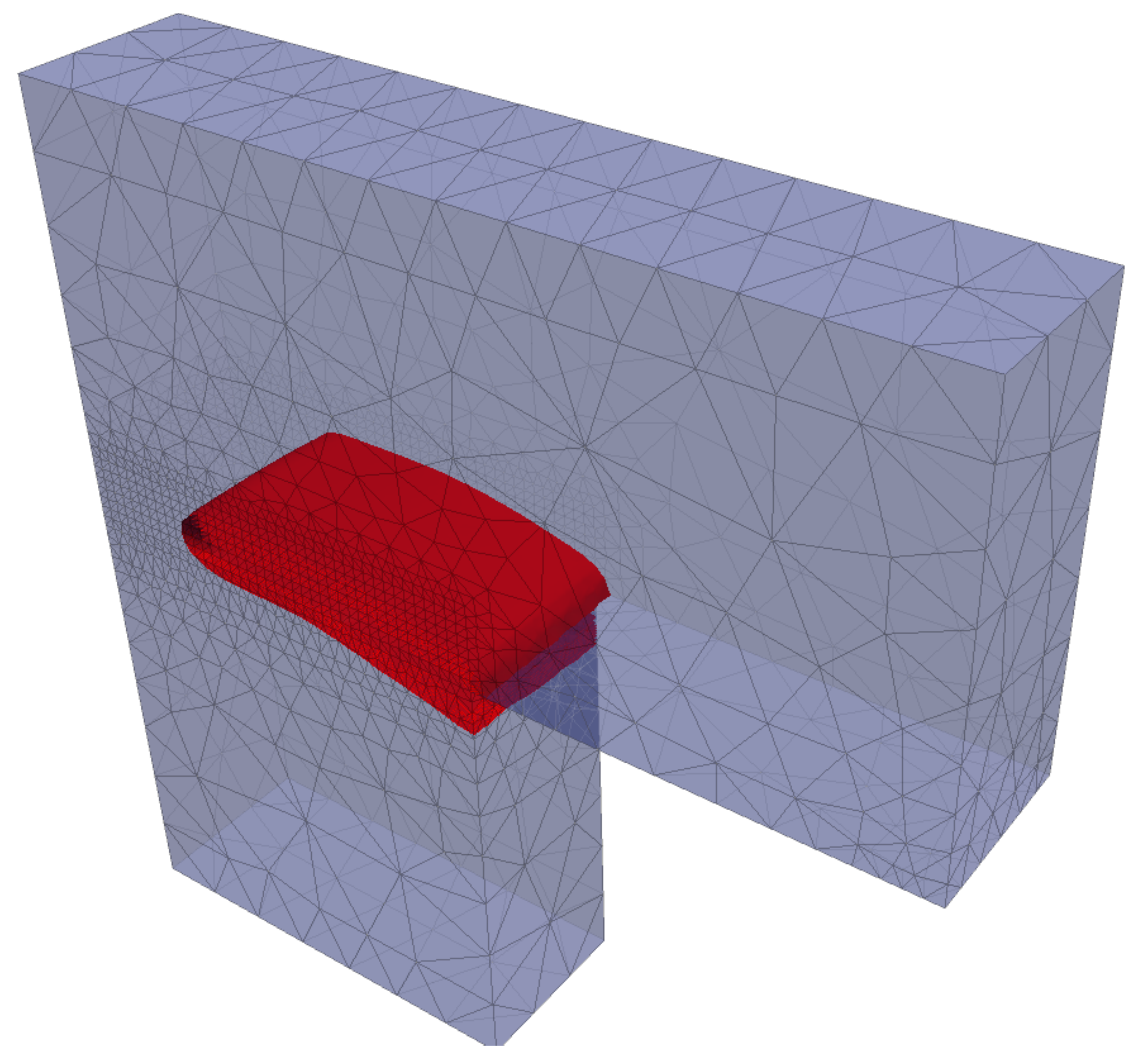}\\
		v=0.25\,\mathrm{mm} & v=0.3\,\mathrm{mm}  & v=0.5\,\mathrm{mm}  \\
       \end{array}$
	}
	\caption{\label{Ex:Lshape:DmgMap}
	Example \ref{Sec.NumEx:Lshape}. $3d$ $L-$shaped panel test.
	Phase field distribution on the plane $z=50\,\mathrm{mm}$ with $3d$ views of the 
	phase-field isolevel lines.
	In the damage maps, 
	the darkest colour corresponds to $\beta=1-\delta$ 
	with $\delta=10^{-4}$ given that we are considering a partially damage profile, 
	whereas the blue corresponds to solid material for which $\beta=0$. 
	}
\end{figure}


\subsection{Three dimensional symmetric bending test}\label{Sec.NumEx:Ex4}

We conclude this section with the $3d$ numerical simulation of the three--point bending test of a mortar notched beam,
normally used in applications to determine the fracture energy \cite{RILEM85}. We compare
our numerical results with the experimental findings of \cite{GRZ93}. The geometric setup conforms with the 
specifications of \cite{RILEM85} and is displayed in Figure \ref{NE16}$(a)$. The height notch is equal to half the 
beam height and its width is not greater than $10\,\mathrm{mm}$. The elastic constants are chosen as 
$E=39.0\,\mathrm{kN/mm^2}$ and $\nu=0.15$, the critical energy release rate as $g_c=0.04\,\,\mathrm{N/mm}$ and the internal length as $\ell=15\,\mathrm{mm}$.

\begin{figure}[H]
	\centerline{$\begin{array}{cc}
		\includegraphics[width=0.45\textwidth]{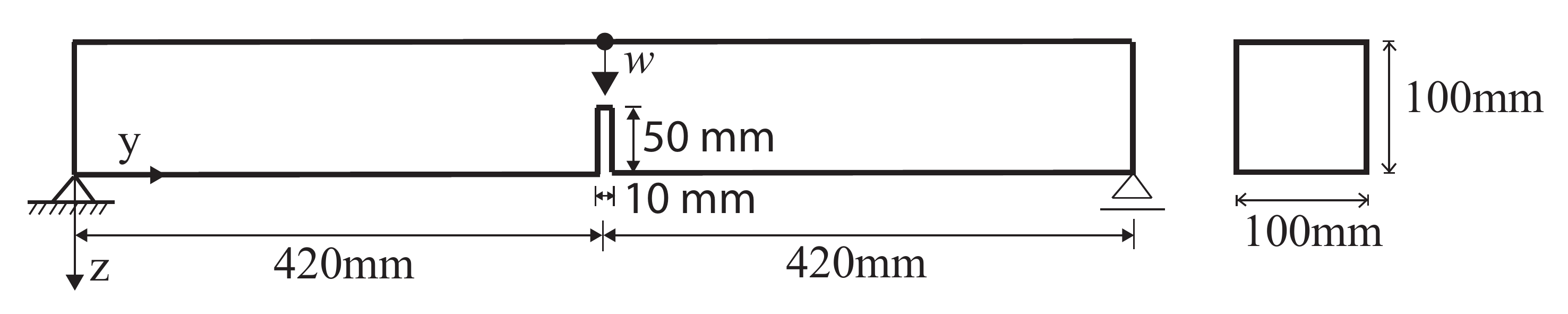}&
		\includegraphics[width=0.45\textwidth]{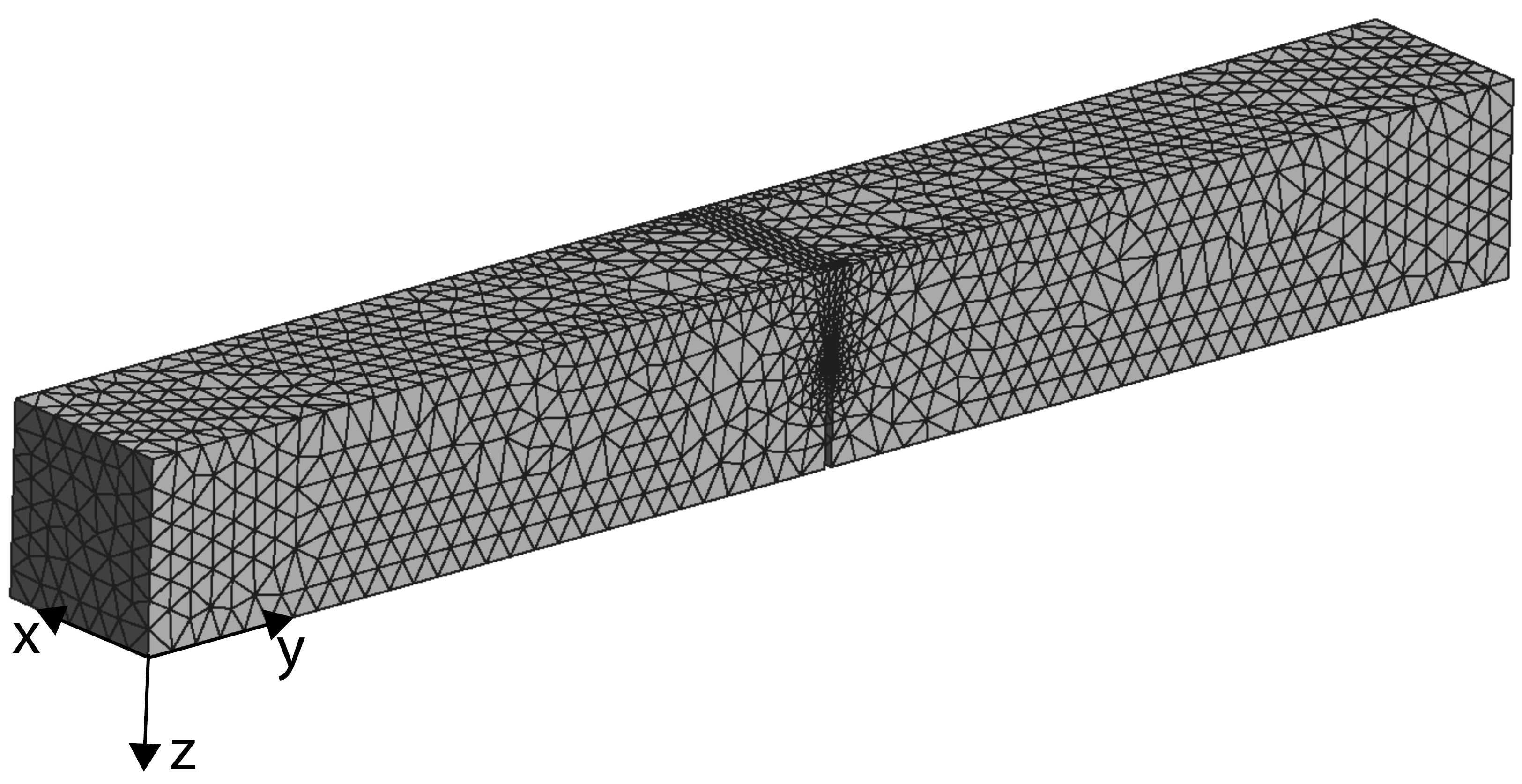}\\
		(a)&(b)
		\end{array}$
		}
	\caption{\label{NE16}
		Example \ref{Sec.NumEx:Ex4}. $3d$-symmetric bending test.
		$(a)$ Geometric setup.
		$(b)$ Unstructured finite element mesh.
		}
\end{figure}

The finite element mesh is shown in Figure \ref{NE16}$(b)$ and is formed by $33824$ tetrahedral elements 
and $6601$ nodes. In order to capture properly the crack pattern, the mesh has been refined in the region
where the crack is expected to propagate with a characteristic finite element length equal to $h=1\,\mathrm{mm}<\ell/2$.  
The tests are performed by applying a deformation controlled loading of the central 
line of equation $y=420\,\mathrm{mm},\,z=-100\,\mathrm{mm}$ by constant displacement increments $\Delta w=10^{-3}\,\mathrm{mm}$
and we set the penalization factor $\epsilon$ equal to $10^{-4}$.
If we denote by $F$ the 
resultant reaction force of the non-homogeneous Dirichlet boundary condition $w$, which is prescribed on the top edge,
the load-displacement curve without the backtracking option active is displayed in Figure \ref{NE17} 
showing good agreement with the experimental findings of \cite{GRZ93}. However, unlike the previous examples,
the variation of the total energy of the computed solutions
displayed in Figure \ref{Ex4:3dbeamEnergy} 
shows that, in this case, the standard scheme of the alternating minimization is capable of identifying 
the energetic solutions when damage starts to manifest without resorting to backtacking, given that 
the computed solutions meet the two-sided energetic inequality. This occurs because
the bound limits are quite ample. 
Finally, Figure \ref{Ex:3dbeam:DmgMap} shows the damage distribution on the cross section $y=420\,\mathrm{mm}$ at several stages of the 
deformation which is consistent with the
description of the experimental results by \cite{GRZ93}.

\begin{figure}[H]
	\centerline{\includegraphics[width=0.5\textwidth]{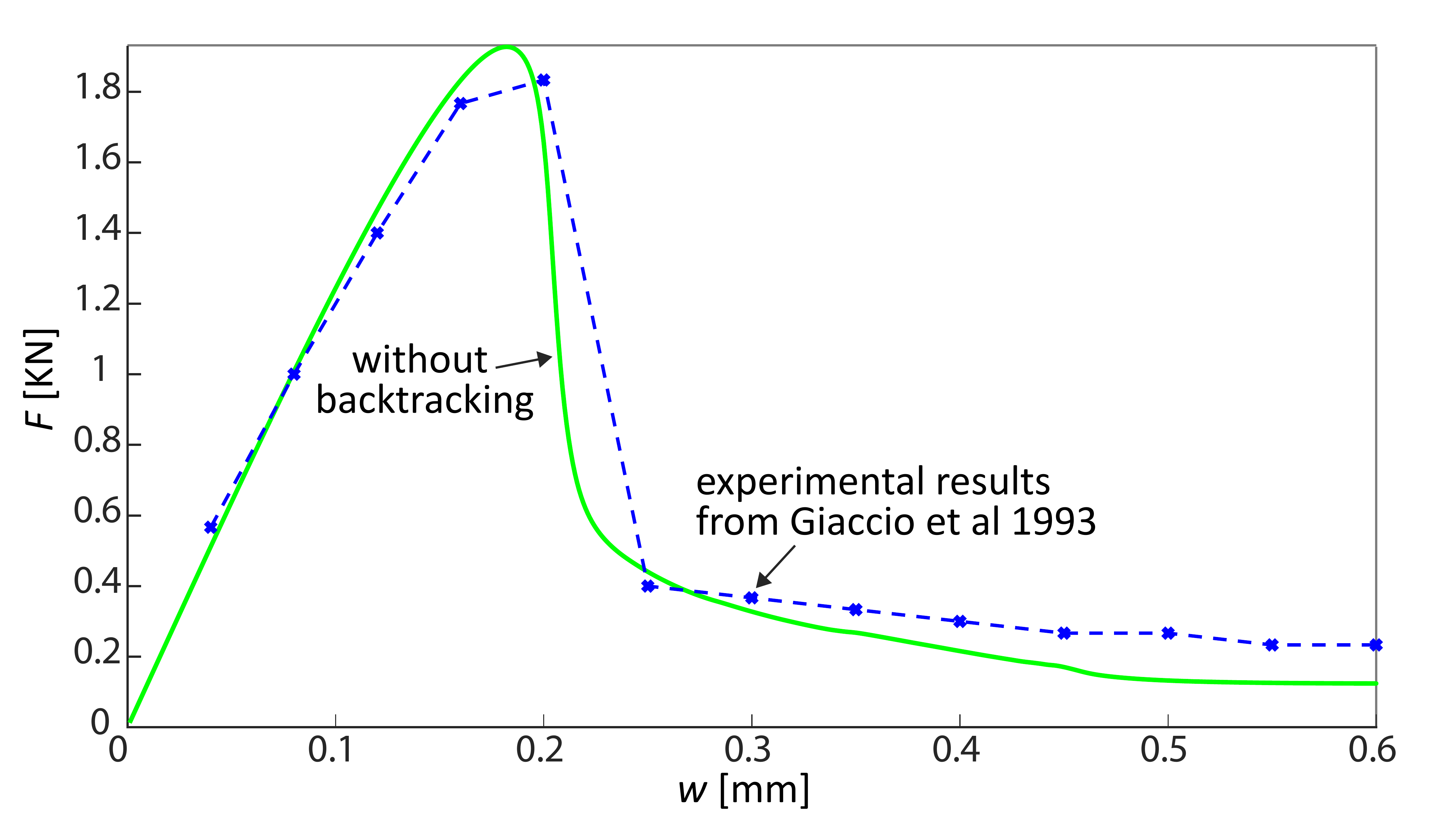}
		}
	\caption{\label{NE17}
	Example \ref{Sec.NumEx:Ex4}. $3d$-symmetric bending test.
	Load-displacement curve associated with the evolution 
	of the approximate energetic solutions. For this case,
	standard application of the alternating minimization identifies
	the energetic solutions without resorting to the backtracking strategy.
	}
\end{figure}

\begin{figure}[H]
	\centering{$\begin{array}{cc}
		\includegraphics[width=0.45\textwidth]{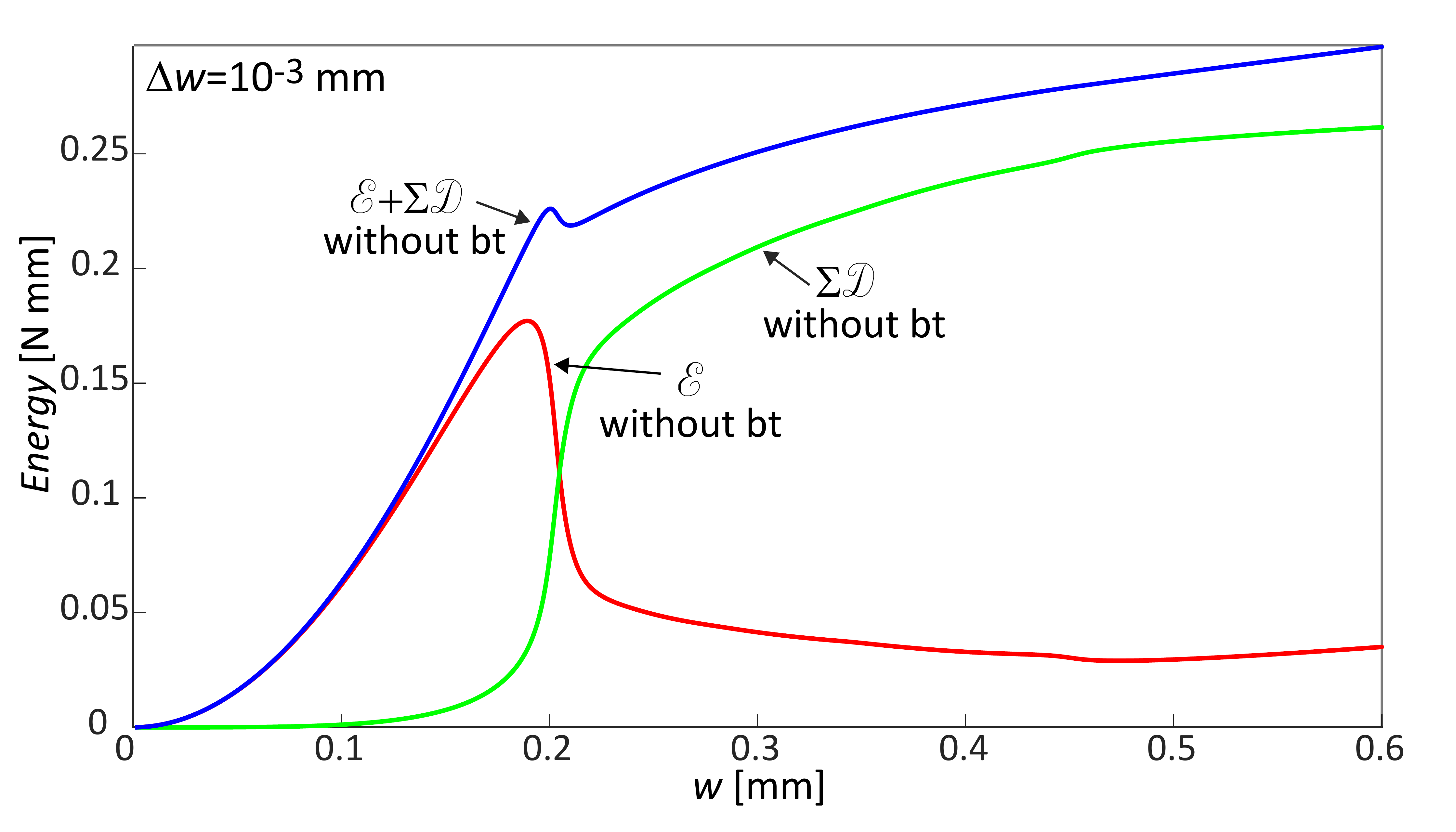}&
		\includegraphics[width=0.45\textwidth]{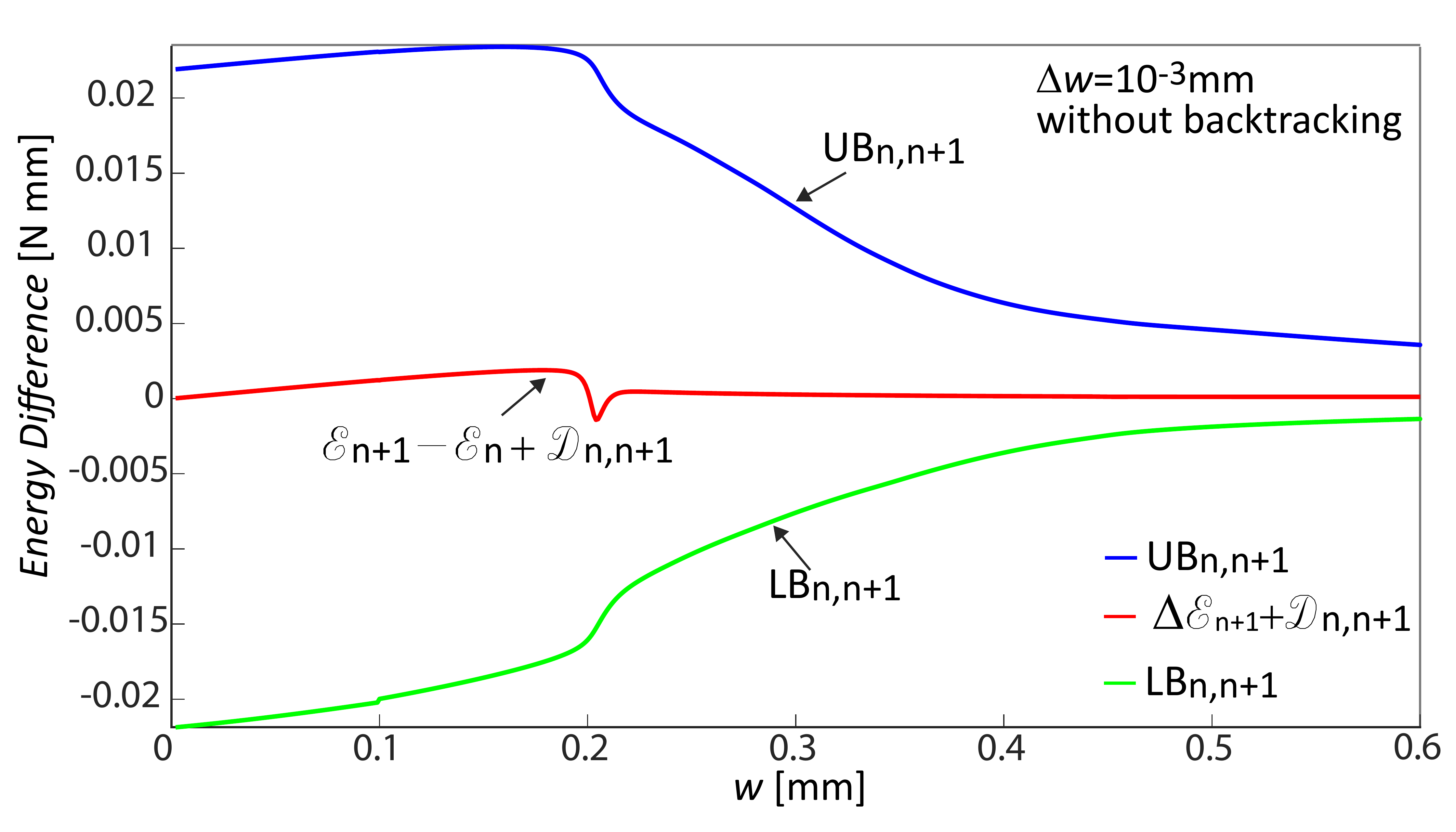}\\
		(a)&(b)
		\end{array}$}
	\caption{\label{Ex4:3dbeamEnergy}
	Example \ref{Sec.NumEx:Ex4}. $3d$-symmetric bending test.
	Results for $\Delta w=10^{-3}\,\mathrm{mm}$.
	Evolution of: $(a)$ the total energy $\mathcal{E}(t_{n+1},\bmU_{n+1},\bmA_{n+1})+\sum_{i=0}^n\mathcal{D}(\bmA_i,\bmA_{i+1})$ 
	for $n=0,1,\ldots, N-1$ and $(b)$ of the total incremental energy $\mathcal{E}_{n+1}-\mathcal{E}_{n}+\mathcal{D}_{n,n+1}$, 
	the lower bound $LB_{n,n+1}$ and the upper bound $UB_{n,n+1}$ which enter the two-sided energy estimate \eqref{Disc.TwoSidIneq},
	$n=0,1,\ldots, N-1$. For this problem, standard application of the alternating minimization identifies
	the energetic solutions without resorting to the backtracking strategy.
	}
\end{figure}

 \begin{figure}[H]
	\centering{$\begin{array}{ccc}
		\includegraphics[width=0.30\textwidth]{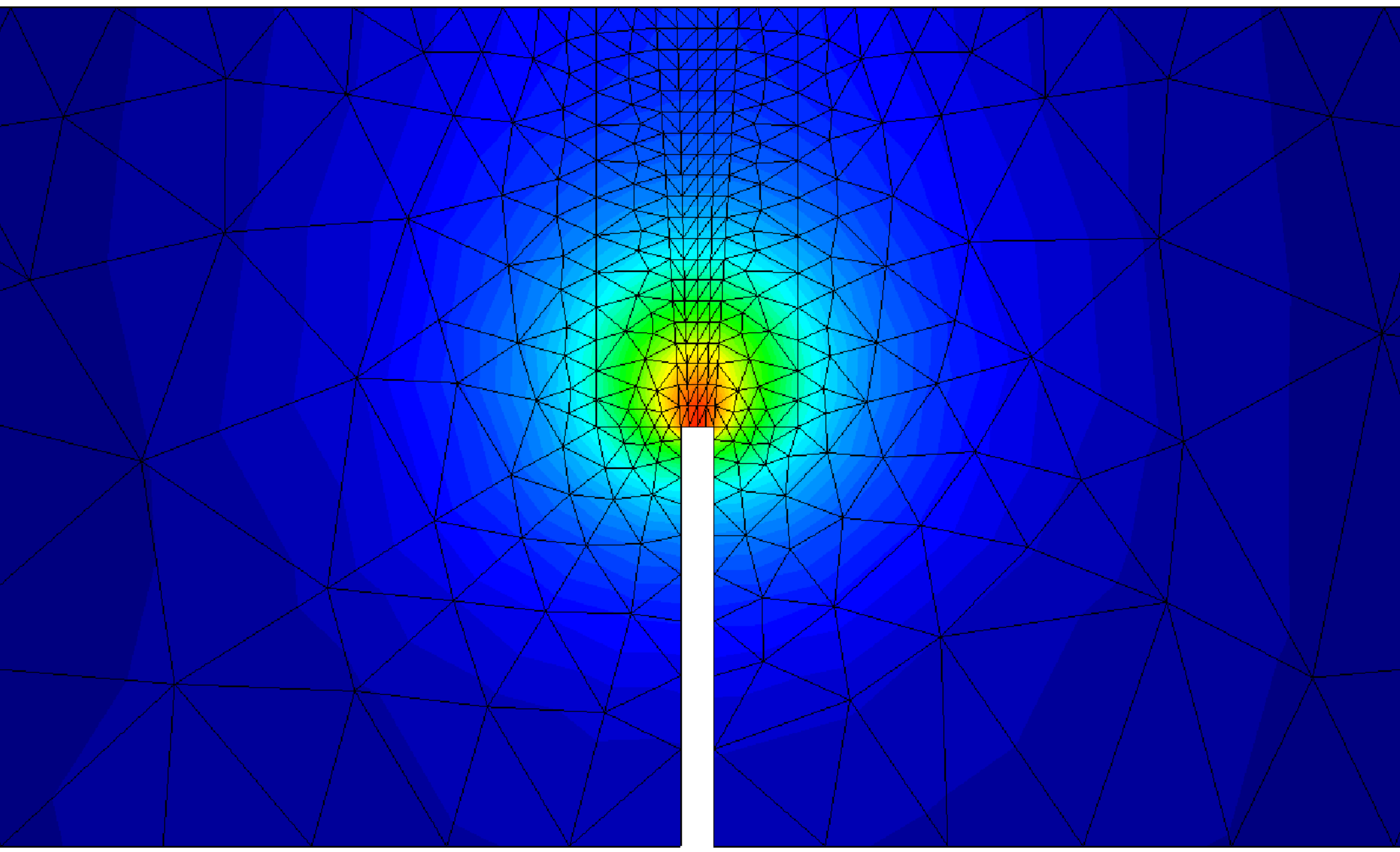}
		\includegraphics[width=0.03\textwidth]{LegendDamageMap3d.pdf}&
		\includegraphics[width=0.30\textwidth]{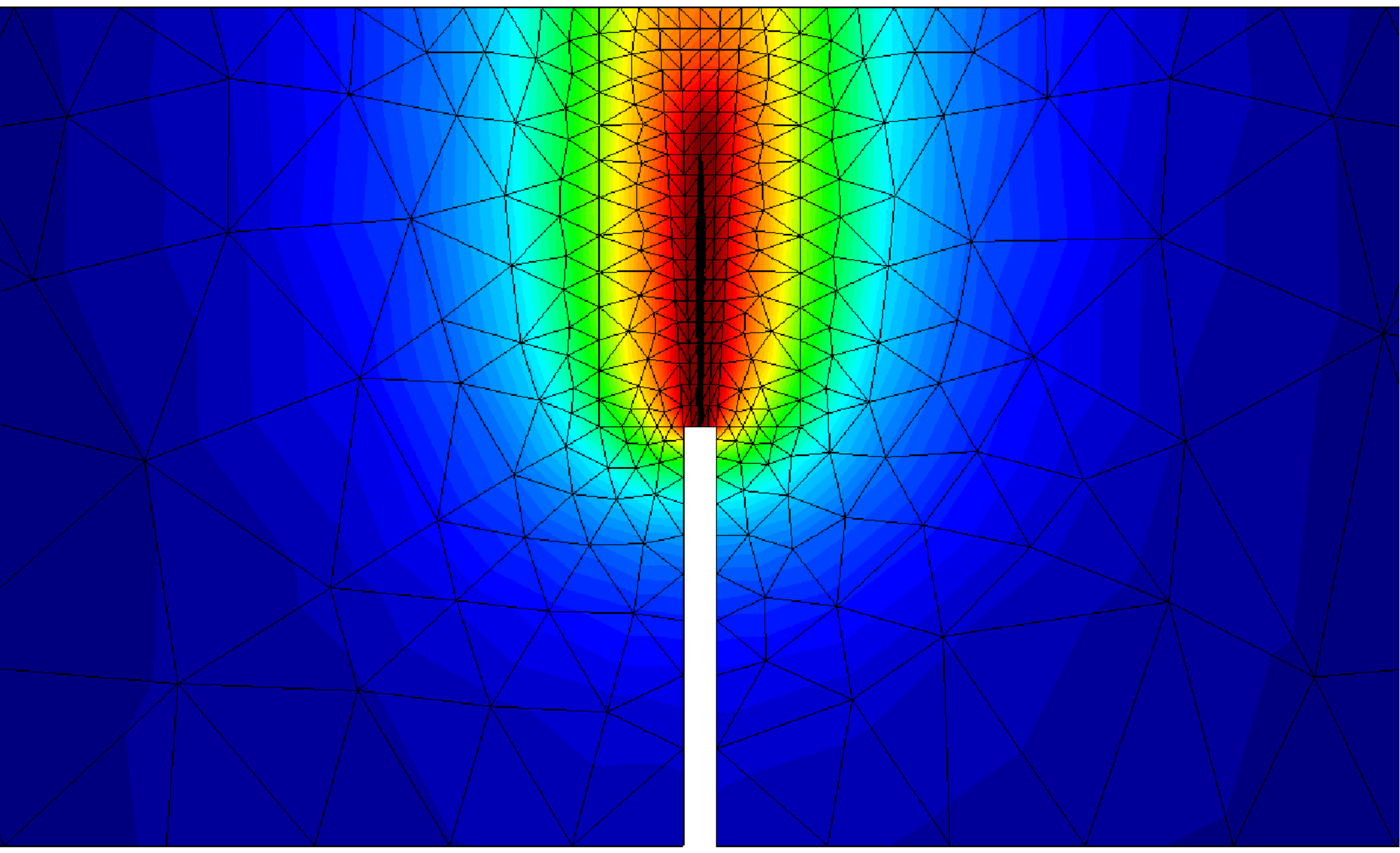}&
		\includegraphics[width=0.30\textwidth]{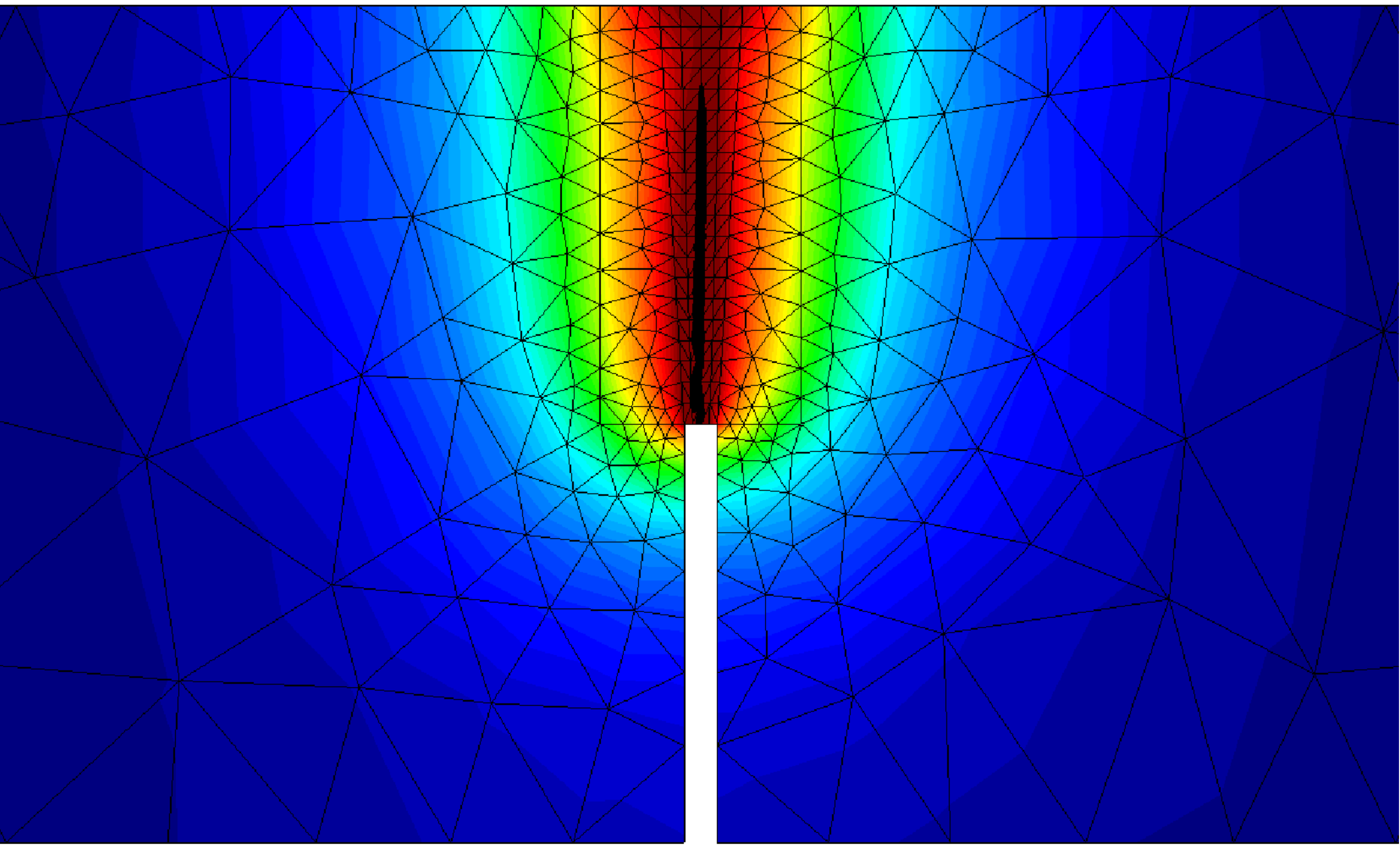}\\
		\includegraphics[width=0.30\textwidth]{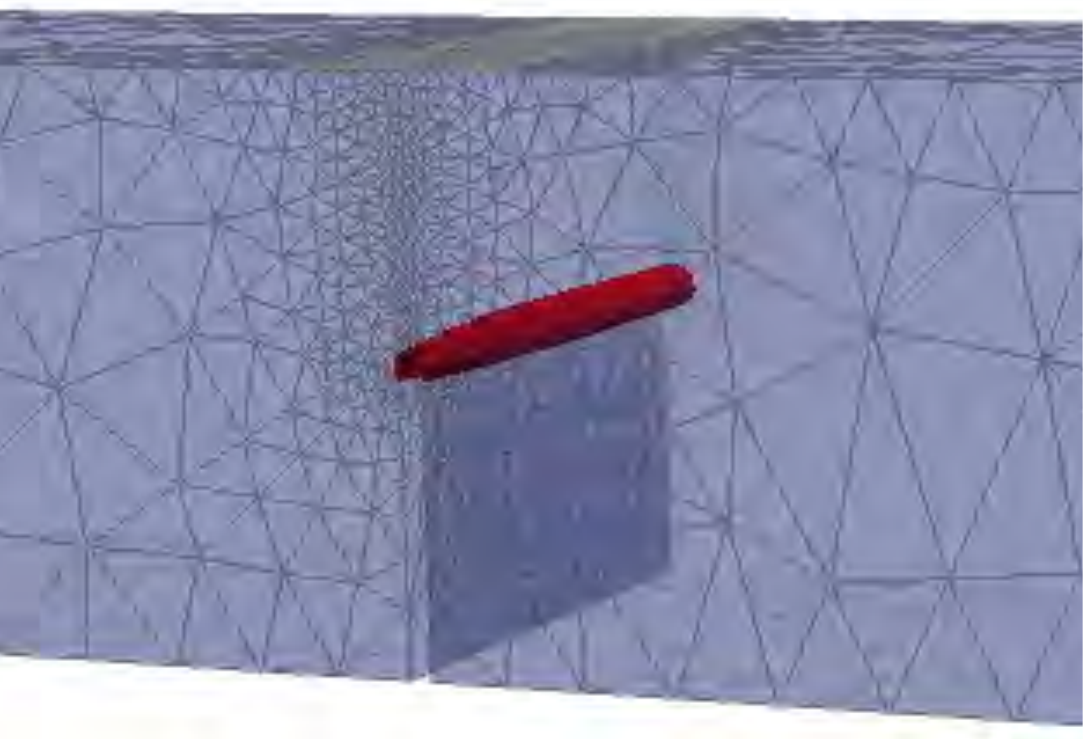}&
		\includegraphics[width=0.30\textwidth]{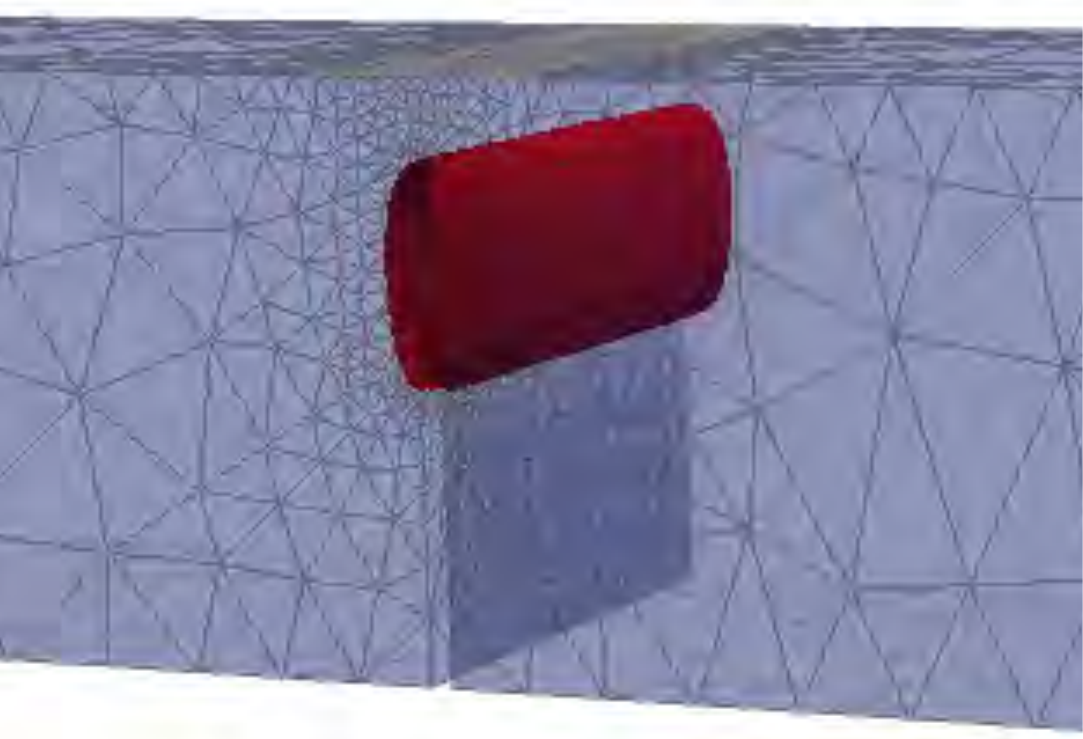}&
		\includegraphics[width=0.30\textwidth]{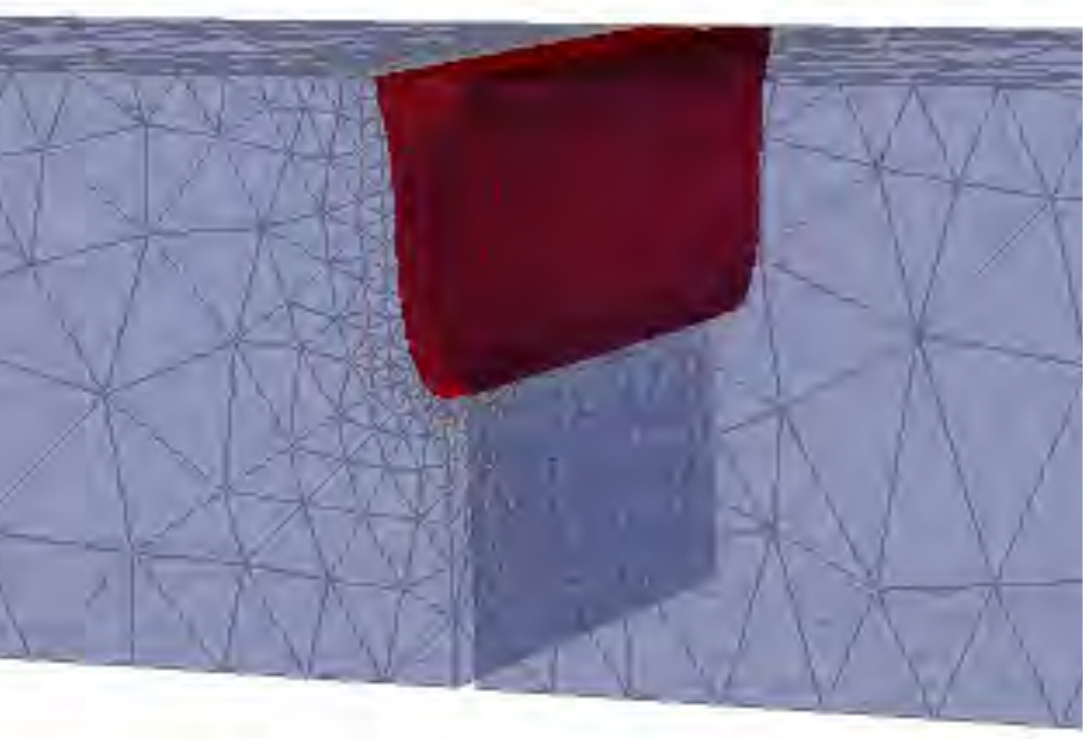}\\
		w=0.171\,\mathrm{mm}   & w=0.291\,\mathrm{mm} & w=0.591\,\mathrm{mm}
               \end{array}$
		}
	\caption{\label{Ex:3dbeam:DmgMap}
	Example \ref{Sec.NumEx:Ex4}. $3d$-symmetric bending test.
	Phase field distribution on the cross section $y=420\,\mathrm{mm}$ with
	correspondent $3d$ views of the phase-field isolevel lines.
	In the damage maps, the darkest colour corresponds to $\beta=1-\delta$ 
	with $\delta=10^{-4}$ given that we are considering a partially damage profile, 
	whereas the blue corresponds to solid material for which $\beta=0$. 
	}
 \end{figure}


\section{Conclusions}\label{Sec.Conc}

In this paper we have proposed an algorithm that computes the solutions of the energetic formulation of an 
anisotropic phase-field model of quasi--brittle fracture characterized by 
a different behaviour at traction and compression 
and by a state dependent dissipation potential. Through the simulation of $2d$ and $3d$ benchmark problems,
our results show that the standard procedure of simply 
solving the weak form of the Euler-Lagrange equations detects generally solutions that
violate the modelling assumption of evolution via global minima of the discrete functionals
which underpins the variational formulation of fracture and its corresponding energetic formulation.
We have verified this assumption by checking an additional optimality condition of the global minimizers.
Such condition 
has the form of a two--sided energy estimate and 
consists 
in checking that within each time step $[t_n,\,t_{n+1}]$, $n=1,\ldots, N$, the sum of the variation of the 
stored energy and of the dissipation between the 
state of the system at the time instants $t_n$ and $t_{n+1}$
is bounded above and below. As an alternative and first approximation to 
the application of global optimization algorithms, which would not be viable in our case, given 
their high numerical complexity,
we have designed a feasible numerical aproach that automatically enforces the meeting of the 
two-sided energetic estimates and allows the computation of 
improved energetic solutions consistent with the basic assumptions. The implementation 
of the energy  estimates has been done within a backtracking strategy by which one goes back 
over past time steps when the estimates are violated, and restarts the simulation of the incremental 
problem with different initial condition 
for the phase-field variable. We noted, however, that there might be cases where the backtracking
procedure is not activated, which occur especially in those situations where the lower energy bound is small and the  
two energy bounds are well apart from each other. This has occurred, for instance, in the 
$3d$ simulation of the symmetric bending beam whereas in the other cases, the check of the 
bounds has been determinant to find different energetic solutions which comply with the energetic bounds.
In this case we cannot infer anything about the computed solution, though we believe, 
that, still within the field of application of methods of local optimization, 
a backtracking strategy based on sharper bounds would select more reliable energetic solutions.
Finally, following the validation of the formulation to model crack in concrete,
we believe that a better modelling of the cracking process upon different loading conditions
can be obtained by including plasticity and cohesive effects, thus elaborating, for instance, the formulation
advanced in \cite{AMV15,AMMV18} where we account of what we have developed in our present work. 
This is part of an ongoing research.


\section*{Acknowledgements} 
The authors are extremely grateful to the anonymous referees, whose constructive and generous comments
on earlier versions of the manuscript have contributed to produce a better version of the paper and
made the authors appreciate subtleties of this fascinating topic.
The authors wish also to thank the partial financial support of the Argentinian Research Council (CONICET),
the Argentinian Ministry of Science, Technology \& Development (MINCyT)
and the National University of Tucum\'{a}n, Argentina for the financial support through the Projects CONICET PIP 101,
PICT 2016-105 and PIUNT CX-E625, respectively.

\end{document}